\documentclass{amsart}
\usepackage{a4wide}

\usepackage{latexsym,amsfonts,amssymb,exscale,enumerate,amscd}
\usepackage{amsmath,amsthm}
\usepackage{hyperref}
\usepackage{subfigure}

\usepackage{pstricks}

\psset{linewidth=0.3pt,dimen=middle}
\psset{xunit=.70cm,yunit=0.70cm}
\psset{arrowsize=1pt 5,arrowlength=0.6,arrowinset=0.7}

\usepackage[all]{xy}
\SelectTips{cm}{}

\usepackage{graphicx}

\newcommand{\Ucatt}{\cal{U}_{\chi}}

\newcommand{\symm}{{\rm \Lambda}}
\newcommand{\Mat}{{\rm Mat}}

\newcommand{\onen}{{\mathbf 1}_{n}}
\newcommand{\onenn}[1]{{\mathbf 1}_{#1}}

\newcommand{\BNC}{NH}

\newcommand{\ep}{\underline{\epsilon}}

\newcommand{\refequal}[1]{\xy {\ar@{=}^{#1}
(-1,0)*{};(1,0)*{}};
\endxy}

\newcommand{\UDnm}{ { _m\UcatD_n}}
\newcommand{\Unm}{ {1_m(\UA)1_n}}

\newcommand{\U}{\dot{{\bf U}}}
\newcommand{\UA}{{_{\cal{A}}\dot{{\bf U}}}}

\newcommand{\Ucat}{\cal{U}}

\newcommand{\UcatD}{\dot{\cal{U}}}
\newcommand{\B}{\dot{\mathbb{B}}}

\newcommand{\xsum}[2]{
  \vcenter{\xy
  (0,.4)*{\sum};
  (0,3.8)*{\scs #2};
  (0,-3.2)*{\scs #1};
  \endxy}
}

\newcommand{\Uup}{\xy {\ar (0,-3)*{};(0,3)*{} };(0,0)*{\bullet};(2,0)*{};(-2,0)*{};\endxy}

\newcommand{\Ucross}{\xy {\ar (2.5,-2.5)*{};(-2.5,2.5)*{}}; {\ar (-2.5,-2.5)*{};(2.5,2.5)*{} };
(4,0)*{};(-4,0)*{};\endxy}

\newcommand{\xUup}{
    \xy {\ar (0,-3)*{};(0,3)*{} };(1.5,0)*{};(-1.5,0)*{};\endxy}

\newcommand{\xUdown}{
    \xy {\ar (0,3)*{};(0,-3)*{} };(1.5,0)*{};(-1.5,0)*{};\endxy}

\newcommand{\xUupdot}{
   \xy {\ar (0,-3)*{};(0,3)*{} };(0,0)*{\bullet};(1.5,0)*{};(-1.5,0)*{};\endxy}

\newcommand{\xUdowndot}{
   \xy {\ar (0,3)*{};(0,-3)*{} };(0,0)*{\bullet};(1.5,0)*{};(-1.5,0)*{};\endxy}

\newcommand{\xUcupl}{\;\;
    \vcenter{\xy (2,3)*{}; (-2,3)*{} **\crv{(2,-1) & (-2,-1)}?(1)*\dir{>};
            (2,-3)*{};(-2,3)*{}; \endxy} \;\; }

\newcommand{\xUcapr}{\;\;
    \vcenter{\xy (-2,-3)*{}; (2,-3)*{} **\crv{(-2,1) & (2,1)}?(1)*\dir{>};
            (2,-3)*{};(-2,3)*{}; \endxy} \;\; }

\newcommand{\cat}[1]{\ensuremath{\mbox{\bfseries {\upshape {#1}}}}}

\newcommand{\BOX}{\hbox {$\sqcap$ \kern -1em $\sqcup$}}

\newcommand{\To}{\Rightarrow}

\newcommand{\Hom}{{\rm Hom}}
\newcommand{\HOM}{{\rm HOM}}
\newcommand{\END}{{\rm END}}

\renewcommand{\to}{\rightarrow}
\newcommand{\maps}{\colon}

\newcommand{\End}{{\rm End}}

\newcommand{\im}{{\rm im\ }}

\newcommand{\la}{\langle}
\newcommand{\ra}{\rangle}
\newcommand{\sla}{\langle}
\newcommand{\sra}{\rangle}

\newcommand{\scs}{\scriptstyle}

\theoremstyle{definition}
\newtheorem{thm}{Theorem}[section]
\newtheorem{cor}[thm]{Corollary}

\newtheorem{rem}[thm]{Remark}
\newtheorem{prop}[thm]{Proposition}
\newtheorem{defn}[thm]{Definition}
\newtheorem{example}[thm]{Example}
\newtheorem{examples}[thm]{Examples}

\numberwithin{equation}{section}

\def\emph#1{{\sl #1\/}}

\let\phi=\varphi
\let\epsilon=\varepsilon

\usepackage{bbm}
\def\C{{\mathbbm C}}
\def\N{{\mathbbm N}}
\def\R{{\mathbbm R}}
\def\Z{{\mathbbm Z}}
\def\Q{{\mathbbm Q}}

\def\cal#1{\mathcal{#1}}%
\def\1{\mathbbm{1}}%
\def\nn{\notag}

\def\pmod{{\mathrm{-pmod}}}  
\def\gdim{{\mathrm{gdim}}}

\def\Id{\mathrm{Id}}

\def\mf{\mathfrak}
\def\shuffle{\,\raise 1pt\hbox{$\scriptscriptstyle\cup{\mskip
               -4mu}\cup$}\,}

\newgray{whitegray}{.9}

\newcommand{\chern}[1]{\begin{pspicture}(-0.5,-0.5)(0.5,0.5)
    \rput(0,0){\psframebox[framearc=.4,fillstyle=solid, linewidth=.8pt]{\small $\scriptstyle #1$}} \end{pspicture}
}

\newcommand{\lowrru}[1]{\xybox{%
  (-8,0)*{};
  (8,0)*{};
  (-6,-18)*{};(6,-9)*{} **\crv{(-6,-13) & (6,-15)} ?(1)*\dir{>};
  (6,-9)*{};(6,0)*{}  **\dir{-} ?(.3)*\dir{ }+(2,0)*{\scs {\bf j}};
}}

\newcommand{\lowllu}[1]{\xybox{%
  (-8,0)*{};
  (8,0)*{};
  (6,-18)*{};(-6,-9)*{} **\crv{(6,-13) & (-6,-15)} ?(1)*\dir{>};
  (-6,-9)*{};(-6,0)*{}  **\dir{-} ?(.3)*\dir{ }+(-2,0)*{\scs {\bf j}};
}}

\newcommand{\bbe}[1]{\xybox{%
  (-2,0)*{};
  (2,0)*{};
  (0,0);(0,-18) **\dir{-}; ?(.5)*\dir{<}+(2.3,0)*{\scriptstyle{#1}};
}}

\newcommand{\bbsid}{\xybox{%
  (-2,0)*{};
  (2,0)*{};
  (0,10);(0,4) **\dir{-};
}}
\newcommand{\bbpef}[1]{\xybox{%
  (-6,0)*{};
  (6,0)*{};
  (-4,0)*{}="t1";
  (4,0)*{}="t2";
  "t1";"t2" **\crv{(-4,-6) & (4,-6)}; ?(.15)*\dir{>} ?(.9)*\dir{>}
   ?(.5)*\dir{}+(0,-2)*{\scriptstyle{#1}};
}}
\newcommand{\bbpfe}[1]{\xybox{%
  (-6,0)*{};
  (6,0)*{};
  (-4,0)*{}="t1";
  (4,0)*{}="t2";
  "t2";"t1" **\crv{(4,-6) & (-4,-6)}; ?(.15)*\dir{>} ?(.9)*\dir{>}
  ?(.5)*\dir{}+(0,-2)*{\scriptstyle{#1}};
}}

\newcommand{\bbcfe}[1]{\xybox{%
  (-6,0)*{};
  (6,0)*{};
  (-4,0)*{}="t1";
  (4,0)*{}="t2";
  "t1";"t2" **\crv{(-4,6) & (4,6)}; ?(.15)*\dir{>} ?(.9)*\dir{>}
  ?(.5)*\dir{}+(0,2)*{\scriptstyle{#1}};
}}
\newcommand{\bbcef}[1]{\xybox{%
  (-6,0)*{};
  (6,0)*{};
  (-4,0)*{}="t1";
  (4,0)*{}="t2";
  "t2";"t1" **\crv{(4,6) & (-4,6)}; ?(.15)*\dir{>}
  ?(.9)*\dir{>} ?(.5)*\dir{}+(0,2)*{\scriptstyle{#1}};
}}

\newcommand{\ccbub}[1]{
\xybox{%
 (-6,0)*{};
  (6,0)*{};
  (-4,0)*{}="t1";
  (4,0)*{}="t2";
  "t2";"t1" **\crv{(4,6) & (-4,6)};
  ?(.05)*\dir{>} ?(1)*\dir{>};
  "t2";"t1" **\crv{(4,-6) & (-4,-6)};
   ?(.3)*\dir{}+(0,0)*{\bullet}+(0,-3)*{\scs {#1}};
}}
\newcommand{\cbub}[1]{
\xybox{%
 (-6,0)*{};
  (6,0)*{};
  (-4,0)*{}="t1";
  (4,0)*{}="t2";
  "t2";"t1" **\crv{(4,6) & (-4,6)};?
   ?(0)*\dir{<} ?(.95)*\dir{<};
  "t2";"t1" **\crv{(4,-6) & (-4,-6)};
   ?(.3)*\dir{}+(0,0)*{\bullet}+(0,-3)*{\scs {#1}};
}}

\newcommand{\ncbub}{
\xybox{%
 (-6,0)*{};
  (6,0)*{};
  (-4,0)*{}="t1";
  (4,0)*{}="t2";
  "t2";"t1" **\crv{(4,6) & (-4,6)}; ?(0)*\dir{<} ?(.95)*\dir{<};
  "t2";"t1" **\crv{(4,-6) & (-4,-6)}; ?(.3)*\dir{};
}}

\newcommand{\bbdl}[1]{\xybox{%
  (2,0);(0,-8) **\crv{(2,-2)&(0,-6)}; ?(.5)*\dir{>}
}}
\newcommand{\bbdlu}[1]{\xybox{%
  (2,0);(0,-8) **\crv{(2,-2)&(0,-6)}; ?(.5)*\dir{<}
}}
\newcommand{\bbdr}[1]{\xybox{%
  (-2,0);(0,-8) **\crv{(-2,-2)&(0,-6)}; ?(.5)*\dir{>}
}}
\newcommand{\bbdru}[1]{\xybox{%
  (-2,0);(0,-8) **\crv{(-2,-2)&(0,-6)}; ?(.5)*\dir{<}
}}

\newcommand{\sccbub}[1]{%
\xybox{%
 (-6,0)*{};
  (6,0)*{};
  (-4,0)*{}="t1";
  (4,0)*{}="t2";
  "t2";"t1" **\crv{(4,6) & (-4,6)}; ?(.05)*\dir{>} ?(1)*\dir{>};
  "t2";"t1" **\crv{(4,-6) & (-4,-6)}; ?(.3)*\dir{}+(2,-1)*{\scs #1};
}}

\begin{document}
%

\title[Diagrammatic algebra and categorified quantum $sl(2)$]{An introduction to diagrammatic algebra
and categorified quantum $\mathfrak{sl}_{2}$}

\author{Aaron D. Lauda}
\address{Department of Mathematics, University of Southern California, Los Angeles, CA 90089, USA}
\email{lauda@usc.edu}

\date{August 10, 2011}

\maketitle

\begin{abstract}
This expository article explains how planar diagrammatics naturally arise in the study of categorified quantum groups with a focus on the categorification of quantum $\mf{sl}_2$.  We derive the definition of categorified quantum $\mathfrak{sl}_2$ and highlight some of the new structure that arises in categorified quantum groups. The expert will find a discussion of rescaling isomorphisms for categorified quantum $\mathfrak{sl}_2$, a proof that cyclotomic quotients of the nilHecke algebra are isomorphic to matrix rings over the cohomology ring of Grassmannians, and an interpretation of `fake bubbles' using symmetric functions.
\end{abstract}

\setcounter{tocdepth}{2}
\tableofcontents

%
\section{Introduction} \label{sec_intro}
%

%
\subsection{Categorification}
%

What does it mean to categorify an algebraic object?  This is a very common question, and part of the confusion lies in the fact that there is no universal definition of categorification that applies in all contexts.  The term categorification originated in work of Crane and Frenkel~\cite{CF} on algebraic structures in topological quantum field theories, but  it has become increasingly clear that categorification is a broad mathematical phenomenon with applications extending far beyond these original motivations.

The general mantra of categorification is to replace sets by categories, functions by functors, and equations by natural isomorphisms of functors. While this is sufficiently vague to be broadly applicable, it is not clear that it conveys to non-experts in the field a sufficiently accurate picture of the goals and applicability of the ideas behind categorification.  It takes some thought to put the most interesting examples into this framework, and it may hide some of the more exciting features these examples possess.

Categorification can be thought of as the (not necessarily unique) process of enhancing an algebraic object to a more sophisticated one.  There is a precise but context dependent notion of ``decategorification" -- the process of reducing the categorified object back to the simpler original object.  A useful categorification should possess a richer ``higher level" structure not seen in the underlying object. This new structure provides new insights and gives rise to a deeper understanding of the original object.

A very simple example is categorifying a natural number $n \in \N$ by lifting it to an $n$-dimensional $\Bbbk$-vector space $V$ with $\dim V = n$.   Here decategorification is the well-defined process of taking the dimension of the vector space. Put another way, the {\em set} of natural numbers $\N$ is categorified by the {\em category} of finite dimensional $\Bbbk$-vector spaces $\cat{FinVect}_{\Bbbk}$. To decategorify the category  $\cat{FinVect}_{\Bbbk}$, we identify isomorphic objects and forget all additional structure except for the dimension. This decategorification recovers the natural numbers $\N$ since every finite dimensional vector space is isomorphic to $\Bbbk^n$ for some $n$.  What makes this example interesting is that both the additive and multiplicative structures on $\N$ are also categorified by the operations of direct sum and tensor product of vector spaces:
\begin{equation}
  \dim (V \oplus W) = \dim V + \dim W, \qquad \dim (V \otimes W) = \dim V \times \dim W. \nn
\end{equation}

Similar in spirit is the example of categorifying $\N[q,q^{-1}]$ using the category of $\Z$-graded vector spaces.  Given a graded vector space $V=\oplus_{n \in \Z}V_n$, we can decategorify $V$ by taking its graded dimension
\begin{equation}
  \gdim V = \sum_{n \in \Z} q^{n} \dim V_n. \nn
\end{equation}
Addition and multiplication in $\N[q,q^{-1}]$ naturally lift to direct sum and tensor product of graded vector spaces.

Many of the key ideas behind categorification can be understood by considering another well-known example.  The Euler characteristic of a finite CW-complex $X$ can be defined as the alternating sum
\begin{equation}
  \chi(X) = \sum_{i=0}^{\infty}(-1)^i k_i(X) \nn
\end{equation}
where the  $k_i(X)$ is the number of cells of dimension $i$ in the complex $X$.  While Euler characteristic is an invariant of the topological space $X$, it is somewhat unsatisfying since given a continuous map $f \maps X \to Y$ it is not obvious how to relate $\chi(X)$ with $\chi(Y)$.  However, the Euler characteristic has a well known categorification which carries information about continuous maps between spaces.  The Euler characteristic is just the shadow of a richer invariant associated to the topological space $X$, namely the homology groups of $X$.

For each $i$, the homology group $H_i(X)$ is an invariant of $X$. These homology groups are in general stronger invariants than the Euler characteristic.  The homology groups can be viewed as a categorification of the Euler characteristic in the sense that
\begin{equation}
  \chi(X) = \sum_{i=0}^{\infty} (-1)^i \dim H_i(X). \nn
\end{equation}
In this way, we have lifted the numerical invariant $\chi(X)$ to more sophisticated algebraic invariants $H_i(X)$.  This categorification possesses a higher structure not seen at the level of the Euler characteristic.  Indeed, homology groups are {\em functorial} so that a continuous map $f \maps X \to Y$ gives a homomorphism of abelian groups
\begin{equation}
  f \maps H_i(X) \to H_i(Y) \nn
\end{equation}
for each $i$.  For some other examples of categorification see~\cite{BD,CY}.

A more recent example that is similar in spirit to the one above is Khovanov's categorification of the Jones polynomial~\cite{Kh1,Kh2}.  The Jones polynomial $J(K)$ of a knot (or link) $K$ is a Laurent polynomial in $\Z[q,q^{-1}]$ that is an invariant of the knot $K$.  Khovanov categorified this link invariant by introducing a graded homology theory  giving rise to a collection of graded vector spaces $Kh^a(K)$ whose graded Euler characteristic
\begin{equation}
  \chi (Kh(K)) := \sum_{a\in \Z} (-1)^a \gdim Kh^a(K) \nn
\end{equation}
agrees with a suitably normalized Jones polynomial $J(K)$.

Khovanov homology is a strictly stronger knot invariant~\cite{BN}, but even more interestingly, it is a functorial knot invariant~\cite{Kh2,Jac1,CMW,Cap}.  A surface embedded in 4-dimensional space whose boundary consists of a pair of knots is a cobordism from one knot to the other~\cite{CRS,CS,BL}.  Functoriality of Khovanov homology means that knot cobordisms induce maps between Khovanov homologies. (There is a similar story for tangles.)  This was used by Rassmussen to give a purely combinatorial proof of the Milnor conjecture~\cite{Ras}, demonstrating the strength of this categorification.

Quantum knot invariants such as the Jones polynomial, its generalizations to the coloured Jones polynomial, and the HOMFLYPT polynomial can all be understood in a unified framework using the representation theory of an algebraic structure called a quantum group.  A quantum group is a Hopf algebra obtained by $q$-deforming the universal enveloping algebra $\mathbf{U}(\mf{g})$ of a Lie algebra $\mf{g}$.  One does not need to understand deformation theory, or a tremendous amount of Lie theory, to begin to study the quantum group $\mathbf{U}_q(\mf{g})$.  One can define these $\Q(q)$-algebras explicitly using generators and relations.

Associated to a quantum group is a family of knot invariants called Reshetikhin-Turaev invariants~\cite{RT}; the quantum knot invariants mentioned above are all special cases of these invariants.  They are determined by certain representations of a quantum group $\mathbf{U}_q(\mf{g})$ associated to highest weights for the Lie algebra $\mf{g}$. We will review these representations in the context of $\mf{sl}_2$ below, but it suffices to know that they are indexed by positive weights $\lambda$ in the weight lattice associated to $\mf{g}$.

Given a tangle diagram $T$, the Reshetikhin-Turaev invariant is defined by colouring the strands in a tangle diagram by highest weights $\lambda$.  To the endpoints of the tangle diagram we associate a tensor product of representations indexed by the highest weights labelling the strands.
\[
\xy
 (0,0)*{\includegraphics[width=1.5in]{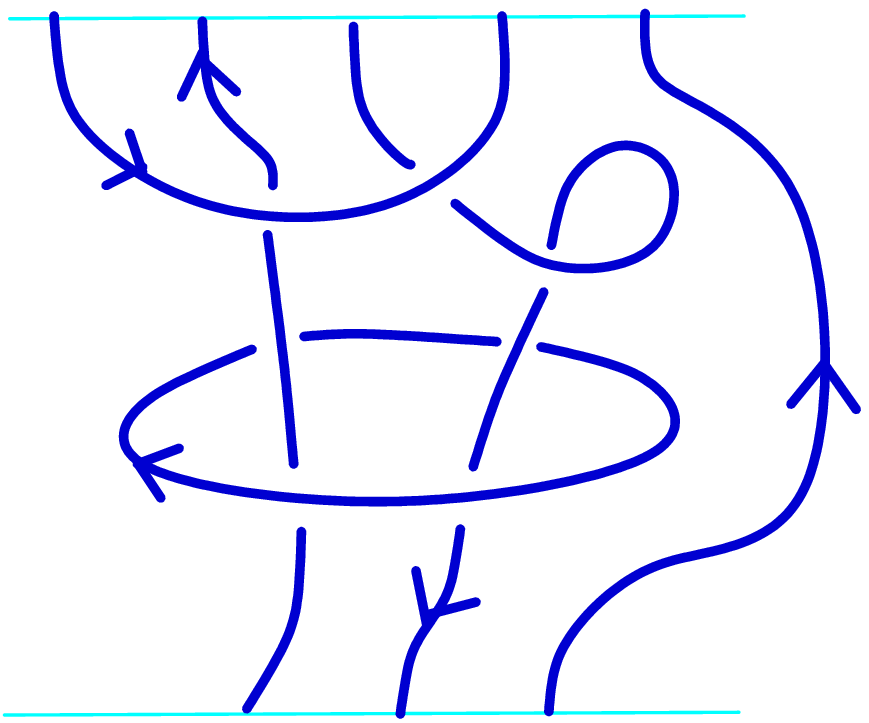}};
 (-2,-18)*{\lambda_1 \;\; \cdots \;\; \lambda_n};
  (-17,18)*{\mu_1};
   (0,18)*{\cdots};
    (10,18)*{ \mu_m};
\endxy
\qquad
\xy
 (0,16)*+{V_{\mu_1} \otimes \cdots \otimes V_{\mu_m}}="1";
 (0,-16)*+{V_{\lambda_1} \otimes \cdots \otimes V_{\lambda_n}}="2";
 (35,16)*{\textcolor[rgb]{0.00,0.50,0.00}{\text{$\mathbf{U}_q(\mathfrak{g})$-module}}}="3";
 (35,-16)*{\textcolor[rgb]{0.00,0.50,0.00}{\text{$\mathbf{U}_q(\mathfrak{g})$-module}}}="4";
 (5,0)*{\textcolor[rgb]{1.00,0.00,0.00}{f(T)}};
 (35,0)*{\textcolor[rgb]{1.00,0.00,0.00}{\text{RT-invariant}}};
\textcolor[rgb]{0.00,0.00,0.63}{ {\ar "2";"1"}}
\endxy
\]
The Reshetikhin-Turaev invariant $f(T)$ is a map between these representations that intertwines the action of $\mathbf{U}_q(\mf{g})$, i.e. it is a $\mathbf{U}_q(\mf{g})$-module homomorphism.

In the special case when our tangle diagram has no endpoints we associate the tensor product of zero copies of the highest weight representations, namely, the ground field $\Q(q)$:
\begin{equation}
 \xy
  (0,0)*{\includegraphics[width=1.5in]{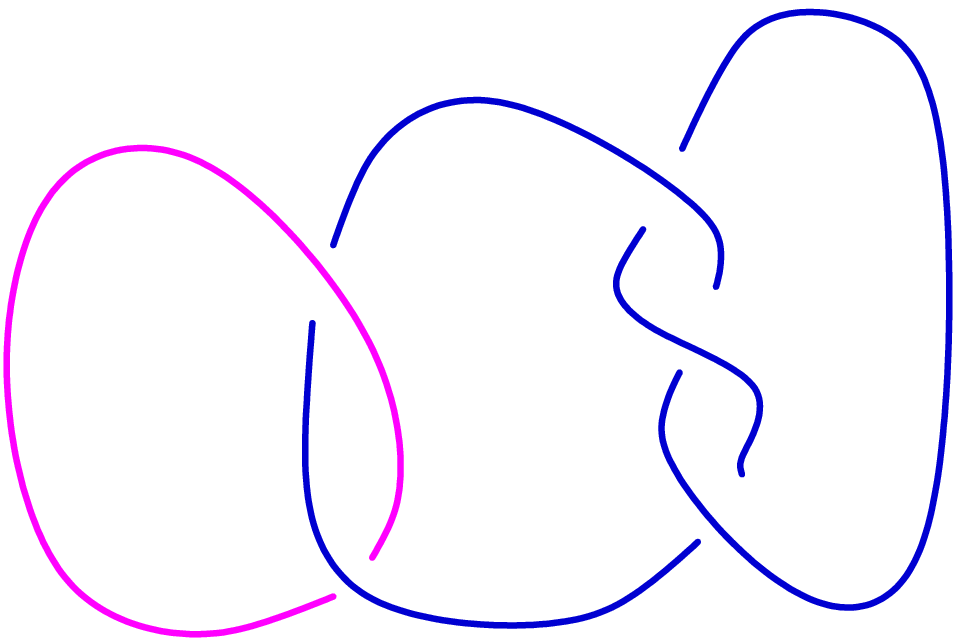}};
  (22,10)*{\textcolor[rgb]{0.00,0.00,1.00}{\lambda}};
  (-20,0)*{\textcolor[rgb]{1.00,0.00,0.50}{\mu}};
 \endxy
 \qquad
\xy
 (0,16)*+{\Q(q)}="1";
 (0,-16)*+{\Q(q)}="2";
 (35,16)*{\textcolor[rgb]{0.00,0.50,0.00}{\text{$\mathbf{U}_q(\mathfrak{g})$-module}}}="3";
 (35,-16)*{\textcolor[rgb]{0.00,0.50,0.00}{\text{$\mathbf{U}_q(\mathfrak{g})$-module}}}="4";
 (5,0)*{\textcolor[rgb]{1.00,0.00,0.00}{f(T)}};
 (35,0)*{\textcolor[rgb]{1.00,0.00,0.00}{\text{RT-invariant}}};
\textcolor[rgb]{0.00,0.00,0.63}{ {\ar "2";"1"}}
\endxy \nn
\end{equation}
But in this case, an intertwiner from $\Q(q)$ to itself is just multiplication by an element in $\Q(q)$.  One can show that these elements arising in the Reshetikhin-Turaev construction actually lie in $\Z[q,q^{-1}]$ so that we get a Laurent polynomial when evaluating the Reshetikhin-Turaev invariant on a link.

 \begin{example} \hfill
 \begin{itemize}
  \item For $\mf{g}=\mf{sl}_2$ and $V$ the defining two-dimensional representation described below, the Reshetikhin-Turaev invariant is the Jones polynomial.  When we take $V$ to be the $N+1$ dimensional representation of $\mf{g}$ we get an invariant called the coloured Jones polynomial.

 \item Taking $\mf{g}=\mf{sl}_n$ and $V$ the defining representation gives specializations of the HOMFLYPT polynomial.
 \end{itemize}
 \end{example}

One of the topological motivations for studying categorifications of quantum groups stems from the goal of trying to categorify all of these Reshetikhin-Turaev invariants for any quantum group $\mathbf{U}_q(\mf{g})$.  The existence of Khovanov homology and the subsequent categorifications of other quantum link invariants suggested that it should be possible to categorify these invariants.  These categorified invariants require as basic algebraic input categorifications of quantum groups, their highest weight representations, and their tensor products.

In this article we will focus on the simplest quantum group $\mathbf{U}_q(\mf{sl}_2)$ defined in the next section.  The categorification of this quantum group provides the representation theoretic explanation of Khovanov homology.   We hope that by carefully explaining the definition and motivation behind categorified quantum $\mf{sl}_2$  the reader will gain sufficient insight to understand the categorifications of all quantum groups, and their role in Webster's approach to categorification of all Reshetikhin-Turaev invariants~\cite{Web2}.

%
\subsection{Quantum $\mathfrak{sl}_2$ and its representations}
%

The first interesting Lie algebra is the Lie algebra $\mathfrak{sl}_2$ of traceless two-by-two complex matrices. This Lie algebra has a basis given by the matrices
\[
E=\left(
    \begin{array}{cc}
      0 & 1 \\
      0 & 0 \\
    \end{array}
  \right), \qquad
F = \left(
    \begin{array}{cc}
      0 & 0 \\
      1 & 0 \\
    \end{array}
  \right),
\qquad
H=\left(
    \begin{array}{cc}
      1 & 0 \\
      0 & -1 \\
    \end{array}
  \right),
\]
with the commutation relations
\[
 [E,F]=H, \qquad [H,E] = 2E, \qquad [H,F] = -2F.
\]
The universal enveloping algebra ${\bf U}(\mathfrak{sl}_2)$ of the Lie algebra $\mathfrak{sl}_2$ is the associative algebra with generators $E$, $F$, $H$ and relations
\begin{eqnarray*}
HE-EH =2E, \qquad   HF-FH =-2F, \qquad EF-FE =H.
\end{eqnarray*}

The quantum deformation ${\bf U}_q(\mathfrak{sl}_2)$ of ${\bf U}(\mathfrak{sl}_2)$ is an algebra over the ring $\Q(q)$ of rational functions in the indeterminant $q$ given by generators $E$, $F$, $K$, $K^{-1}$ and relations
\begin{eqnarray}
  KK^{-1}&=&K^{-1}K \; =1, \label{eq_UqI}\\
  KE &=& q^2EK, \label{eq_UqII}\\
  KF&=&q^{-2}FK, \label{eq_UqIII}\\
  EF-FE&=&\frac{K-K^{-1}}{q-q^{-1}}. \label{eq_UqIV}
\end{eqnarray}

Just as any finite-dimensional representation $V$ of ${\bf U}(\mathfrak{sl}_2)$ can be decomposed into eigenspaces $V_n$ for the action of $H$, with $v \in V_n$ if and only if
\[
 H v = n v,
\]
we can also decompose a finite-dimensional representation $V$ of  ${\bf U}_q(\mathfrak{sl}_2)$ into eigenspaces $V_n$ for the action of $K$
\[
 K v = q^n v , \qquad v \in V_{n}.
\]
One can show that these eigenspaces $V_n$ are indexed by integers $\Z$.  In this case $\Z$ is the weight lattice for $\mathfrak{sl}_2$. The vector space $V_n$ is called the $n$th weight space of $V$, and $v\neq 0$ in $V_n$ for some $n$ is called a weight vector.  Here we will only be interested weight representations, or representations that admit a decomposition
\[
 V = \bigoplus_{n \in \Z} V_n
\]
into weight spaces.

Given a weight vector $v \in V_n$ the weights of $Ev$ and $Fv$ are determined using the relations
\[
K(Ev) = q^2EKv = q^{n+2}(Ev), \qquad K(Fv) = q^{-2} F Kv = q^{n-2} (Fv),
\]
so that $E \maps V_n \to V_{n+2}$ and $F \maps V_n \to V_{n-2}$.
\[
 \xy
 {\ar@{-} (-30,0)*{} ; (30,0)*{}};
 (0,-1,5)*{};(0,1.5)*{} **\dir{-};
 (8,-1,5)*{};(8,1.5)*{} **\dir{-};
 (16,-1,5)*{};(16,1.5)*{} **\dir{-};
 (24,-1,5)*{};(24,1.5)*{} **\dir{-};
 (-8,-1,5)*{};(-8,1.5)*{} **\dir{-};
 (-16,-1,5)*{};(-16,1.5)*{} **\dir{-};
  (-24,-1,5)*{};(-24,1.5)*{} **\dir{-};
  (0,0)*{\bullet}+(0,-4)*{v};
  (16,0)*{\bullet}+(0,-4)*{Ev};
  (-16,0)*{\bullet}+(0,-4)*{Fv};
  {\ar@/^1pc/^E (1,2)*{}; (16,0)*++{}};
  {\ar@/_1pc/_F (-1,2)*{}; (-16,0)*++{}};
  (-34,-10)*{\text{\textcolor[rgb]{0.00,0.50,0.00}{Weight}}};
  (-16,-10)*{\textcolor[rgb]{0.00,0.50,0.00}{n-2}};
    (0,-10)*{\textcolor[rgb]{0.00,0.50,0.00}{n}};
      (16,-10)*{\textcolor[rgb]{0.00,0.50,0.00}{n+2}};
 \endxy
\]
Therefore, a weight representation of ${\bf U}_q(\mathfrak{sl}_2)$ can be thought of as a collection of vector spaces $V_n$ for $n\in \Z$ where $E$ maps the $n$th weight space to the $n+2$ weight space and $F$ maps the $n$th weight space to the $n-2$ weight space
\[
 \xy
 (-58,0)*{\cdots};
 (58,0)*{\cdots};
  (-50,0)*+{V_{-N}}="1";
  (-16,0)*+{V_{n-2}}="2";
  (0,0)*+{V_n}="3";
  (16,0)*+{V_{n+2}}="4";
  (50,0)*+{V_{N}}="5";
    {\ar@/^0.7pc/^E "1";(-34,0)*+{}};
    {\ar@/^0.7pc/^E "2";"3"};
    {\ar@/^0.7pc/^E "3";"4"};
    {\ar@/^0.7pc/^E (34,1)*+{}; "5"};
    {\ar@/^0.7pc/^F (-34,-1)*+{}; "1"};
    {\ar@/^0.7pc/^F "3";"2"};
    {\ar@/^0.7pc/^F "4";"3"};
    {\ar@/^0.7pc/^F "5";(34,-1)*+{}};
  (27,0)*{\cdots};
  (-27,0)*{\cdots};
 \endxy
\]
such that the main $\mathfrak{sl}_2$ relation $EF-FE=\frac{K-K^{-1}}{q-q^{-1}}$ holds.  Note that on a weight vector $v \in V_n$ this relation takes the form
\[
(EF-FE)v=\frac{K-K^{-1}}{q-q^{-1}}v = \frac{Kv - K^{-1}v}{q-q^{-1}} = \frac{q^n-q^{-n}}{q-q^{-1}}v = [n]v
\]
since $Kv =q^nv$. The rational function $[n]=\frac{q^n-q^{-n}}{q-q^{-1}}$ appearing in the expression above is called the quantum integer $n$. One can check that
\begin{equation}
  [n] := q^{n-1}+q^{n-3}+ \dots +  q^{1-n}. \nn
\end{equation}

Of particular importance are the irreducible $N+1$ dimensional representations $V^N$.  These representations are generated by a highest weight vector $v$, that is a vector $v \in V_N$ such that $Ev=0$.  They have a basis given by vectors $v_k = \frac{F^k}{[k]!}v$, for $0 \leq k \leq N$, where $[k]!$ denotes the quantum factorial $[k]!=[k][k-1]\dots[1]$.  Each nonzero weight space in $V^N$ is therefore 1-dimensional.

For applications to knot theory and low dimensional topology we are primarily interested in representations of ${\bf U}_q(\mathfrak{sl}_2)$ that admit such a decomposition into weight spaces.  Consider a modified version $\U$ of ${\bf U}_q(\mathfrak{sl}_2)$ where the unit element is substituted by a collection of mutually orthogonal idempotents $1_n$ for $n \in \Z$,
\begin{equation} \label{eq_orthog_idemp}
  1_n1_m=\delta_{n,m}1_n,
\end{equation}
that project onto the $n$th weight space. Since $K$ acts on the $n$th weight space by $q^n$ this implies that
\begin{equation}
K1_n = q^n 1_n.
\end{equation}
Furthermore, since $E$ must increase the weight by $2$ and $F$ must decrease the weight by $2$ equations \eqref{eq_UqII} and \eqref{eq_UqIII} take the form
\begin{equation}
 E1_n = 1_{n+2}E = 1_{n+2}E1_n, \qquad F1_n = 1_{n-2}F = 1_{n-2}F 1_n,
\end{equation}
and the main $\mathfrak{sl}_2$ relation \eqref{eq_UqIV} becomes
\begin{equation} \label{eq_Udot_relation}
 EF1_n-FE1_n = [n]1_n.
\end{equation}
The algebra $\U$ was introduced by Beilinson, Lusztig and MacPherson~\cite{BLM} for $\mathfrak{sl}_n$ and was generalized to an idempotented form $\U(\mathfrak{g})$ for arbitrary symmetrizable Kac-Moody algebra $\mathfrak{g}$ by Lusztig.

A key property of $\U$ is that the category of $\U$-modules is equivalent to the category of ${\bf U}$-modules that admit a weight decomposition.

%
\subsection{Categorical actions}
%
%
\subsubsection{Quantum groups acting on categories}
%

A representation theoretic reason to suspect that categorified quantum groups should exist is the existence of categorical quantum group actions.  For an ordinary representation of $\U$ we specify a vector space $V=\oplus_{n \in \Z} V_n$ together with linear maps
\begin{equation}
1_n \maps V_n \to V_n, \quad E1_n \maps V_n \to V_{n+2}, \quad F1_n \maps V_n \to V_{n-2}, \nn
\end{equation}
that satisfy the quantum $\mathfrak{sl}_2$ relation \eqref{eq_Udot_relation}.

Motivated by various geometric constructions it is natural to consider categorical $\U$-actions.  In a categorical $\U$-action the vector spaces $V_n$ are replaced by additive {\em categories} $\cal{V}_n$.  These categories are required to be graded or triangulated so that they are equipped with an auto-equivalence $\{1\} \maps \cal{V}_n \to \cal{V}_n$. We denote by $\{ s \}$ the auto-equivalence obtained by applying $\{ 1 \}$ $s$ times.   Linear maps are replaced by {\em functors}
\begin{equation}
\onen \maps \cal{V}_n \to \cal{V}_n, \quad \cal{E}\onen \maps \cal{V}_n \to \cal{V}_{n+2}, \quad \cal{F}\onen \maps \cal{V}_n \to \cal{V}_{n-2}, \nn
\end{equation}
that commute with the grading shift functor $\{1\}$. These functors must satisfy the main $\mathfrak{sl}_2$ relation up to natural isomorphisms of functors
\begin{eqnarray}
  \cal{E}\cal{F}\onen \cong \cal{F}\cal{E}\onen \oplus \onen^{\oplus_{[n]}}  & \qquad & \text{for $n \geq 0$}, \nn\\
  \cal{F}\cal{E}\onen  \cong \cal{E}\cal{F}\onen\oplus\onen^{\oplus_{[-n]}} & \qquad & \text{for $n \leq 0$,} \nn
\end{eqnarray}
where we write
\begin{eqnarray}
  \onen^{\oplus_{[n]}}\; :=\; \onen\{n-1\} \oplus \onen\{n-3\} \oplus \cdots \oplus
  \onen\{1-n\}. \nn
\end{eqnarray}

In a {\em categorified representation} of $\U$ the weight categories $\cal{V}_n$  categorify the $\Q(q)$-vector spaces $V_n$ of some representation $V = \oplus_{n \in \Z} V_n$ of $\U$.  The precise definition of categorification, or rather decategorification, can vary depending on the example.  Here we will be primarily interested in {\em additive} categorifications where we identify  additive categories $\cal{V}_n$ and `decategorification' means that
\begin{equation}
  K_0(\cal{V}_n) \otimes_{\Z[q,q^{-1}]} \Q(q) \cong V_n \nn
\end{equation}
where $K_0$ is the split Grothendieck group of the additive category $\cal{V}_n$.  We will remind the reader about split Grothendieck groups in Section~\ref{subsec_what-structure}, but for now it is enough to know that it is a procedure for turning additive categories into abelian groups, or $\Z[q,q^{-1}]$-modules when the categories $\cal{V}_n$ have the additional structure of a $\Z$-grading on objects.  The additive functors $\cal{E}\onen$ and $\cal{F}\onen$ on $\cal{V}$ are required to induce the action of $E1_n$ and $F1_n$ on the split Grothendieck group,
\begin{equation}
  [\cal{E}\onen], [\cal{F}\onen] \maps K_0(\cal{V}) \to K_0(\cal{V}), \nn
\end{equation}
so that all of the structure of the $\U$-module $V$ arises from the categorified representation $\cal{V}$.

There are many examples of categorified representations in the literature.  Many of these are based on geometric constructions~\cite{BLM,GL1,zheng}, while others have a more algebraic flavour~\cite{BFK,Strop,FKS,Sussan,FSS,FKS}.  A nice survey of various geometric categorifications is given in~\cite{Kam}, and a review of categorifications in the context of abelian categories is given in~\cite{KMS}, see also \cite{Maz}.  We will look at a simple combinatorial categorification  of the irreducible  representation $V^N$ of $\U$ in Section~\ref{subsec_flag}.

%
\subsubsection{Higher structure of categorical actions}
%

The existence of categorical quantum group actions hints at a new level of structure that could not be seen with traditional representations of quantum groups on vector spaces.  In particular, it now makes sense to ask what natural transformations can exist between composites of functors $\onen$, $\cal{E}\onen$, and $\cal{F}\onen$.  This higher level structure is a phenomenon that only exists for categorical representations.  It is natural to wonder what aspects of this higher structure of natural transformations, if any, is common among all examples of categorical quantum group actions.

Igor Frenkel conjectured that this higher level structure governing natural transformations in these categorical actions is governed by the existence of a categorification of the quantum group $\U$ itself.  This categorification of quantum groups should have an additional level of structure allowing for morphisms going between quantum group generators.  In Section~\ref{subsec_what-structure} we elaborate on what led Frenkel to conjecture that a categorification of $\U$ should exist.  One key idea is the existence of the canonical basis for $\U$.  Figure~\ref{fig:catrep} summarizes these ideas.
\begin{figure}[htc]
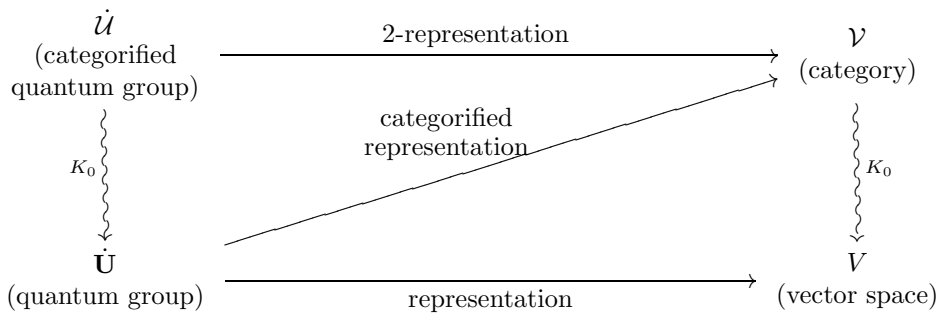

\[
 \xy
(0,18)*{};
 (-50,-15)*+{
 \begin{array}{c}  \U \\ \text{(quantum group)}
 \end{array}  }="b1";
 (50,-15)*+{
 \begin{array}{c}
   V \\
  \text{(vector space)}
 \end{array} }="b2";
 (50,15)*+{
 \begin{array}{c}
   \cal{V}\\
   \text{(category)}
 \end{array} }="t2";
(-50,15)*+{
\begin{array}{c}
  \UcatD \\
  \txt{(categorified \\quantum group)}
\end{array}
}="t1";
 {\ar_-{\txt{representation}} "b1";"b2" };
{\ar^{\txt{categorified \\representation}\;\;} "b1";"t2" };
{\ar^{\txt{2-representation}\;\;} "t1";"t2" };
{\ar@{~>}^{K_0} "t2";"b2"};
{\ar@{~>}_{K_0} "t1";"b1"};
 \endxy
\]
\caption{The above diagram schematically illustrates the various structures arising in categorified representation theory. The bottom arrow represents a traditional representation of $\U$ acting on a vector space by linear transformations.  The diagonal arrow indicates a categorical action of $\U$ on a category $\cal{V}$, with generators acting by functors.  The top arrow indicates that categorical actions can be understood as a 2-representation of a structure that categorifies the algebra $\U$.  The natural transformations between functors acting on $\cal{V}$ are governed by the structure of $\UcatD$.  } \label{fig:catrep}
\end{figure}

Chuang and Rouquier began the systematic study of natural transformations between functors arising in categorified representations of the non-quantum enveloping algebra $\mathbf{U}(\mathfrak{sl}_2)$~\cite{CR}.  They identified some structures common to a large class of examples and used this structure to prove an important result for modular representation theory of the symmetric group called the abelian defect conjecture.  While this work identified structures on natural transformations between functors in large class of examples of categorified representations, this work stopped short of categorifying the enveloping algebra $\mathbf{U}(\mathfrak{sl}_2)$ itself, and it was not clear how to extend their results to the quantum setting.

In 2006 the author began working with Mikhail Khovanov to categorify the algebra $\U$ directly and uncover the underlying structure governing categorical quantum group actions.  The author found a solution in the case of $\mathfrak{sl}_2$ in \cite{Lau1,Lau2}.  Together with Khovanov~\cite{KL,KL2,KL3}, this work was then extended to  $\U(\mathfrak{sl}_n)$ and conjecturally to all symmetrizable Kac-Moody algebras.  This conjectural extension appears in Webster~\cite{Web}.
The purpose of this article is to motivate the definition of the categorification of $\U(\mathfrak{sl}_2)$ and the techniques used in this approach to categorified representation theory.

%
\subsubsection{Diagrammatics to encode combinatorial complexity}
%

If a categorified quantum group is going to successfully describe the higher structure of all examples of categorical quantum group actions, including those coming from geometric constructions using perverse or coherent sheaves, as well as those involving category $\cal{O}$, we can expect this structure to be a combinatorially  complex object.  We will need a way of expressing this object in a meaningful way and this is where diagrammatic algebra enters the story.

From the discussion of categorical quantum group actions it would seem that a categorification of a quantum group should be a structure combining three levels of structure. Categorical actions have categories associated to weights for $\mathfrak{sl}_2$, functors between weight categories corresponding to quantum group generators $E1_n$ and $F1_n$, and the new level of structure taking the form of natural transformations between functors.

A 2-category is an algebraic structure that can be used to keep track of these three levels of structure, the combinatorial information of sources and targets for morphisms and morphisms between morphisms, and all the various composition operations.  In Section~\ref{subsec_2cat} we will define 2-categories and explain how they provide a natural framework for transforming algebra into planar diagrammatics. See \cite{BD} for a discussion on the relationship between higher categories and higher dimensional diagrams and \cite{Kh5} for a survey on the relationship between planar diagrammatics and categorification.

A family of algebras called KLR-algebras arose from the study of categorified quantum groups~\cite{KL,KL2,Rou2}.  These algebras categorify the positive half $\U^+$ of arbitrary symmetrizable Kac-Moody algebras.  We hope that this article helps to clarify the diagrammatic origins of these algebras.

\bigskip

In Section~\ref{sec_diagrammatics} we define 2-categories and explain how these algebraic structures provide a natural framework for studying planar diagrammatics. In Section~\ref{sec_catU} we describe the ideas leading to the 2-category $\UcatD$ categorifying the algebra $\U$. We essentially derive the definition using the structure of a semilinear form on $\U$ and a relationship between $\U$ and partial flag varieties.   The expert will find a discussion of rescaling the 2-category $\Ucat$ to allow degree zero bubbles to take arbitrary values (see Section~\ref{subsubsec_rescaling}) and an interpretation of fake bubbles using symmetric functions (see Section~\ref{subsubsec_fake_symm}).

In Section~\ref{sec_higher} we describe some of the consequences of the new higher structure present in the 2-category $\UcatD$.  Finally, in Section~\ref{sec_irreps} we look at a categorification of irreducible representations of $\U$ called cyclotomic quotients.  The reader who is new to categorified quantum groups may find this section more challenging.  This section contains a proof that cyclotomic quotients of nilHecke algebras are isomorphic to a matrix ring with coefficients in the cohomology ring of a Grassmannian.

\bigskip
\noindent {\bf Acknowledgments:}
The author is tremendously grateful to his collaborator and friend Mikhail Khovanov for sharing his insights and ideas over the years.  The author is also grateful to John Baez for introducing him to the diagrammatic algebra in the language of 2-categories while he was an undergraduate at UC Riverside and is also grateful to Sabin Cautis, Alex P. Ellis, Mustafa Hajij, You Qi, Joshua Sussan and Heather Russell for helpful comments on an earlier version of this article.

The author was partially supported by the NSF grants DMS-0739392 and DMS-0855713. The author would like to thank the mathematics department at LSU for their hospitality during his visit in the Spring of 2011.

%
\section{Diagrammatic algebra }  \label{sec_diagrammatics}
%

2-categories can be thought of as an algebraic world for problems that are inherently 2-dimensional.   In this world algebra and topology begin to merge as we see algebraic computations taking the form of planar diagrammatics.

%
\subsection{2-categories} \label{subsec_2cat}
%
A category is an algebraic structure combining two levels of combinatorial data:
objects, and morphisms between objects.   In a category there is an associative composition operation allowing morphisms to be composed when their sources and targets agree, and there is an identity morphism for each object.   A 2-category can  be thought of as a more sophisticated combinatorial structure combining three levels of structure: objects, morphisms between objects, and 2-morphisms going between morphisms.  Much of the definition of a 2-category can be guessed by drawing some simple pictures to keep track of the relationships between this combinatorial data.

\begin{defn}
 A {\em strict 2-category} $\cal{K}$ is given by the following data:
\begin{itemize}
 \item A collection of objects which we denote by $x, y, z \dots$
 \item For each pair of objects $x,y$ a category $\cal{K}(x,y)$.  We call the objects $\xymatrix@1{y & \ar[l]_{f} x}$ of  $\cal{K}(x,y)$ 1-morphisms of $\cal{K}$ with source $x$ and target $y$. A morphism $\alpha \maps f \To g$
\begin{equation} \label{eq_globular}
\xy (8,0)*+{\scs x}="4"; (-8,0)*+{\scs y}="6"; {\ar@/^1.65pc/^{f}
"4";"6"}; {\ar@/_1.65pc/_{g} "4";"6"}; {\ar@{=>}_<<<{\scriptstyle
\alpha} (-.5,-3)*{};(-.5,3)*{}} ;
\endxy
\end{equation}
 of the category $\cal{K}(x,y)$ is called a 2-morphism in $\cal{K}$.
\end{itemize}
\begin{itemize}
\item
There is an associative composition operation for 1-morphisms
\[
\xymatrix{ z & \ar[l]_{g}   y & \ar[l]_f  x} =
 \xymatrix{ z & \ar[l]_{gf}  x}
\]
so that given a triple of composable 1-morphisms
\[
\xymatrix{ w & \ar[l]_{h}  z & \ar[l]_{g}
   y & \ar[l]_f  x}
\]
we have $(hg)f = h(gf)$.
\end{itemize}
As suggested by the ``globular" representation \eqref{eq_globular} of 2-morphisms there are two ways to compose 2-morphisms.
\begin{itemize}
  \item 2-morphisms can be composed horizontally
  \[
\xy (16,0)*+{\scs x}="4"; (0,0)*+{\scs y}="6"; {\ar@/^1.65pc/^{f} "4";"6"};
{\ar@/_1.65pc/_{g} "4";"6"}; {\ar@{=>}_<<<{\scriptstyle \alpha}
(8,-3)*{};(8,3)*{}} ; (0,0)*+{\scs y}="4"; (-16,0)*+{\scs z}="6";
{\ar@/^1.65pc/^{f'} "4";"6"}; {\ar@/_1.65pc/_{g'} "4";"6"};
{\ar@{=>}_<<<{\scriptstyle \beta} (-8,-3)*{};(-8,3)*{}} ;
\endxy
\qquad = \qquad \xy (8,0)*+{\scs x}="4"; (-8,0)*+{\scs z}="6";
{\ar@/^1.65pc/^{f' f} "4";"6"}; {\ar@/_1.65pc/_{g'  g} "4";"6"};
{\ar@{=>}_<<<{\scriptstyle \beta \ast \alpha} (-.5,-3)*{};(-.5,3)*{}} ;
\endxy
\]
and this operation is associative in the sense that when
   \[
   \xy
   (-16,0)*+{\scs x}="4";
   (0,0)*+{\scs y}="6";
    {\ar@/^1.65pc/^{f_1} "6";"4"};
    {\ar@/_1.65pc/_{f_2} "6";"4"};
    {\ar@{=>}^<<<{\scriptstyle \gamma} (-8,-3)*{};(-8,3)*{}} ;
    (0,0)*+{\scs y}="4";
    (16,0)*+{\scs z}="6";
    {\ar@/^1.65pc/^{g_1} "6";"4"};
    {\ar@/_1.65pc/_{g_2} "6";"4"};
    {\ar@{=>}^<<<{\scriptstyle \beta} (8,-3)*{};(8,3)*{}} ;
    (16,0)*+{\scs z}="4";
    (32,0)*+{\scs w}="6";
    {\ar@/^1.65pc/^{g_1} "6";"4"};
    {\ar@/_1.65pc/_{g_2} "6";"4"};
    {\ar@{=>}^<<<{\scriptstyle \alpha} (24,-3)*{};(24,3)*{}} ;
    \endxy
   \qquad \text{ we have} \qquad
    (\gamma \ast \beta) \ast \alpha = \gamma \ast (\beta \ast
    \alpha) .
    \]
\item The composition operation in the category $\cal{K}(x,y)$ is called vertical composition in $\cal{K}$, and we depict it as follows:
    \[
\xy (8,0)*+{\scs x}="4"; (-8,0)*+{\scs y}="6"; {\ar^(.7){g} "4";"6"}; {\ar@/^1.75pc/^{f}
"4";"6"}; {\ar@/_1.75pc/_{h} "4";"6"}; {\ar@{=>}_<<{\scriptstyle \alpha}
(0,-6)*{};(0,-1)*{}} ; {\ar@{=>}_<<{\scriptstyle \beta} (0,1)*{};(0,6)*{}} ;
\endxy
\qquad = \qquad \xy (8,0)*+{\scs x}="4"; (-8,0)*+{\scs y}="6"; {\ar@/^1.65pc/^{f}
"4";"6"}; {\ar@/_1.65pc/_{h} "4";"6"}; {\ar@{=>}_<<<{\scriptstyle \beta
\alpha} (-.5,-3)*{};(-.5,3)*{}} ;
\endxy.
\]
Vertical composition is associative in the sense that for a composable triple of 2-morphisms
\[
     \xy
    (-9,0)*+{\scs y}="1";
    (9,0)*+{\scs x}="2";
    {\ar@/^1.1pc/^<<<<g "2";"1"};
    {\ar@/_1.1pc/_<<<<h "2";"1"};
    {\ar@/^2.5pc/^f "2";"1"};
    {\ar@/_2.5pc/_i "2";"1"};
    {\ar@{=>}^{\beta} (0,-1.5)*{};(0,1.5)*{}} ;
    {\ar@{=>}^{\alpha} (0,-9)*{};(0,-6)*{}} ;
    {\ar@{=>}^{\gamma} (0,6)*{};(0,9)*{}} ;
\endxy
   \qquad  \text{we have} \qquad
  (\gamma  \beta)  \alpha = \gamma  ( \beta
  \alpha).
   \]
\end{itemize}

Just as categories have identity axioms for composing a morphism with an identity morphism, a 2-category has identity axioms for {\em each} form of composition.
\begin{itemize}
  \item For every object $x$ there is a 1-morphism $\xymatrix{ x  & \ar[l]_-{1_x}x }$ such that
    \[
    \xymatrix{ y  & x \ar[l]_{f} & \ar[l]_{1_x}x} \quad =\quad
    \xymatrix{ y  & \ar[l]_{f}x }\quad =\quad
    \xymatrix{ y  & \ar[l]_{1_y}y  & \ar[l]_{f}x} .
    \]
  \item For every 1-morphism $\xymatrix{ y  & \ar[l]_{f}x }$ there exists a 2-morphism $\xy
    (6,0)*+{\scriptstyle x}="1"; (,8)*+{};
    (-6,0)*+{\scriptstyle y}="2"; {\ar@/^1.2pc/^f "1";"2"}; {\ar@/_1.2pc/_f "1";"2"};
    {\ar@{=>}_{1_f} (0,-2);(0,2)};
    \endxy$ such that
    \[
\xy (8,0)*+{\scs x}="4"; (-8,0)*+{\scs y}="6"; {\ar^<<<<g "4";"6"};
{\ar@/^1.75pc/^{f} "4";"6"}; {\ar@/_1.75pc/_{g} "4";"6"};
{\ar@{=>}_<<{\scriptstyle \alpha} (0,-6)*{};(0,-1)*{}} ;
{\ar@{=>}_<<{\scriptstyle 1_g} (0,1)*{};(0,6)*{}} ;
\endxy
\quad =\quad
\xy (8,0)*+{\scs x}="4"; (-8,0)*+{\scs y}="6"; {\ar@/^1.65pc/^{f}
"4";"6"}; {\ar@/_1.65pc/_{g} "4";"6"}; {\ar@{=>}_<<<{\scriptstyle \alpha}
(.5,-3)*{};(.5,3)*{}} ;
\endxy
\quad =\quad
 \xy (8,0)*+{\scs x}="4"; (-8,0)*+{\scs y}="6"; {\ar^<<<<f
"4";"6"}; {\ar@/^1.75pc/^{f} "4";"6"}; {\ar@/_1.75pc/_{g} "4";"6"};
{\ar@{=>}^<<{\scriptstyle 1_f} (0,-6)*{};(0,-1)*{}} ; {\ar@{=>}^<<{\scriptstyle
\alpha} (0,1)*{};(0,6)*{}} ;
\endxy
\]
and
\[
\xy (16,0)*+{\scs x}="4"; (0,0)*+{\scs y}="6"; {\ar@/^1.65pc/^{f} "4";"6"};
{\ar@/_1.65pc/_{g} "4";"6"}; {\ar@{=>}_<<<{\scriptstyle \alpha}
(8,-3)*{};(8,3)*{}} ; (0,0)*+{\scs y}="4"; (-16,0)*+{\scs y}="6";
{\ar@/^1.65pc/^{1_y} "4";"6"}; {\ar@/_1.65pc/_{1_y} "4";"6"};
{\ar@{=>}_<<<{\scriptstyle 1_{1_y}} (-8,-3)*{};(-8,3)*{}} ;
\endxy
\quad = \quad \xy (8,0)*+{\scs x}="4"; (-8,0)*+{\scs y}="6"; {\ar@/^1.65pc/^{f}
"4";"6"}; {\ar@/_1.65pc/_{g} "4";"6"}; {\ar@{=>}_<<<{\scriptstyle \alpha}
(.5,-3)*{};(.5,3)*{}} ;
\endxy
\quad = \quad
  \xy (16,0)*+{\scs x}="4"; (0,0)*+{\scs x}="6";
{\ar@/^1.65pc/^{1_x} "4";"6"}; {\ar@/_1.65pc/_{1_x} "4";"6"};
{\ar@{=>}_<<<{\scriptstyle 1_{1_x}} (8,-3)*{};(8,3)*{}} ; (0,0)*+{\scs x}="4";
(-16,0)*+{\scs y}="6"; {\ar@/^1.65pc/^{f} "4";"6"}; {\ar@/_1.65pc/_{g} "4";"6"};
{\ar@{=>}_<<<{\scriptstyle \alpha} (-8,-3)*{};(-8,3)*{}} ;
\endxy .
\]
\end{itemize}
The identity morphisms must be compatible with composition in the sense that
\[
 \xy
    (-8,0)*+{\scs x}="4"; (8,0)*+{\scs y}="6";
    {\ar@/^1.65pc/^{g} "6";"4"};
    {\ar@/_1.65pc/_{g} "6";"4"};
    {\ar@{=>}^<<<{\scriptstyle 1_g} (0,-3)*{};(0,3)*{}} ;
\endxy
 \ast
  \xy
    (-8,0)*+{\scs y}="4"; (8,0)*+{\scs z}="6";
    {\ar@/^1.65pc/^{f} "6";"4"};
    {\ar@/_1.65pc/_{f} "6";"4"};
    {\ar@{=>}^<<<{\scriptstyle 1_f} (0,-3)*{};(0,3)*{}} ;
\endxy
 =
  \xy
    (-8,0)*+{\scs x}="4"; (8,0)*+{\scs z}="6";
    {\ar@/^1.65pc/^{g  f} "6";"4"};
    {\ar@/_1.65pc/_{g  f} "6";"4"};
    {\ar@{=>}^<<<{\scriptstyle 1_{g  f}} (1,-3)*{};(1,3)*{}} ;
\endxy.
\]

The most interesting axiom in the definition of a 2-category is the interchange law:
\begin{equation} \label{eq_interchange}
  \xy
    (-16,0)*+{\scs z}="4";
    (0,0)*+{\scs y}="6";
    (16,0)*+{\scs x}="8";
    {\ar "6";"4"};
    {\ar "8";"6"};
    {\ar@/^1.75pc/^{f'} "6";"4"};
    {\ar@/_1.75pc/_{g'} "6";"4"};
    {\ar@/^1.75pc/^{f} "8";"6"};
    {\ar@/_1.75pc/_{g} "8";"6"};
    {\ar@{=>}^<<{\scriptstyle \gamma} (-8,-6)*{};(-8,-1)*{}} ;
    {\ar@{=>}^<<{\scriptstyle \delta} (-8,1)*{};(-8,6)*{}} ;
    {\ar@{=>}^<<{\scriptstyle \alpha} (8,-6)*{};(8,-1)*{}} ;
    {\ar@{=>}^<<{\scriptstyle \beta} (8,1)*{};(8,6)*{}} ;
\endxy  \qquad (\delta  \gamma) \ast (\beta  \alpha) \;\; = \;\;
 (\delta \ast \beta)  (\gamma \ast \alpha),
\end{equation}
which we may picture as
\[
  \xy 
    (-20,0)*+{\scs x}="4";
    (-4,0)*+{\scs y}="6";
      {\ar "6";"4"};
      {\ar@/^1.75pc/ "6";"4"};
      {\ar@/_1.75pc/ "6";"4"};
      {\ar@{=>}^{\scs \gamma} (-12,-5)*{};(-12,-2)*{}} ;
      {\ar@{=>}^{\scs \delta} (-12,2)*{};(-12,5)*{}} ;
    (4,0)*+{\scs y}="4";
    (20,0)*+{\scs z}="6";
      {\ar "6";"4"};
      {\ar@/^1.75pc/ "6";"4"};
      {\ar@/_1.75pc/ "6";"4"};
      {\ar@{=>}^{\scs \alpha}  (12,-5)*{};(12,-2)*{}} ;
      {\ar@{=>}^{\scs \beta} (12,2)*{};(12,5)*{}} ;
  (0,0)*{\ast};
 \endxy \qquad = \qquad
 \xy 
    (-16,5)*+{\scs x}="1";
    (0,5)*+{\scs y}="2";
      {\ar@/_1.75pc/ "2";"1"};
      {\ar "2";"1"}; (-0,5)*+{\scs y}="3";
    (16,5)*+{\scs z}="4";
      {\ar "4";"3"};
      {\ar@/_1.75pc/ "4";"3"};
      {\ar "4";"3"};
      {\ar@{=>}^{\scs \beta} (8,7)*{};(8,10)*{}} ;
      {\ar@{=>}^{\scs \delta} (-8,7)*{};(-8,10)*{}} ;
    (-16,-5)*+{\scs x}="1";
    (0,-5)*+{\scs y}="2";
      {\ar "2";"1"};
      {\ar@/^1.75pc/ "2";"1"};
      (-0,-5)*+{\scs y}="3";
    (16,-5)*+{\scs z}="4";
      {\ar "4";"3"}; {\ar@/^1.75pc/ "4";"3"};
       {\ar "4";"3"};
       {\ar@{=>}^{\scs \alpha} (8,-10)*{};(8,-7)*{}} ;
       {\ar@{=>}^{\scs \gamma} (-8,-10)*{};(-8,-7)*{}} ;
(0,0)*{\circ};
\endxy .
\]
\end{defn}

Many examples of 2-categories that appear in the literature are not strict, meaning that the composition of 1-morphisms and 1-morphism identity axioms only hold up to coherent isomorphism.  Such 2-categories are called {\em weak} 2-categories or {\em bicategories} (see for example~\cite{Bor}). By a coherence theorem for bicategories, every bicategory is equivalent to a strict one in a suitable sense.  So for the purposes of representing 2-categories using planar diagrammatics the difference between strict and weak 2-categories will not be important here.

\begin{examples} \hfill \label{examp-monoidal}
\begin{enumerate}
  \item $\cat{Cat}$ is the 2-category whose objects are (small) categories, morphisms are functors, and 2-morphisms are natural transformations between functors.

  \item $\Sigma(\cal{M})$: Let $\cal{M}$ be a strict monoidal category, that is, a category equipped with an associative tensor product functor allowing us to tensor objects $x \otimes y$ and tensor morphisms $f \otimes g$. We can think of $\cal{M}$ as a 2-category $\Sigma(\cal{M})$ by shifting our perspective.  Define $\Sigma(\cal{M})$ to be the 2-category that has just one object $*$. The 1-morphisms of $\Sigma(\cal{M})$ are given by the objects of $\cal{M}$.  Since our 2-category has just one object we should be able to compose any pair of one morphisms of $\Sigma(\cal{M})$.  This composition operation is just the tensor product in $\cal{M}$.  We then think of the morphisms in $\cal{M}$ as 2-morphisms in $\Sigma(\cal{M})$ with vertical composition coming from the composition in $\cal{M}$ and horizontal composition coming from the tensor product of morphisms in $\cal{M}$.  This idea is summarized in the table below:
  \[
\begin{tabular}{|l|l|}
  \hline
   The 2-category $\Sigma(\cal{M})$ &   \\ \hline
  objects & $*$ \\
  morphisms & objects of $\cal{M}$  \\
  composition of 1-morphisms & tensor product of objects of $\cal{M}$  \\
  identity 1-morphism & unit for tensor product in $\cal{M}$ \\
  2-morphisms & morphisms of $\cal{M}$  \\
  vertical composition & composition of morphisms in $\cal{M}$  \\
  horizontal composition  & tensor product of morphisms in $\cal{M}$  \\
  \hline
\end{tabular}
\]

Unfortunately, most of the interesting examples of categories with a tensor product, like the category of vector spaces with the usual tensor product $(\cat{Vect}_{\Bbbk}, \otimes_{\Bbbk})$, do not have a strictly associative tensor structures since for example, $(V\otimes_{\Bbbk} W) \otimes_{\Bbbk} Z$ is only isomorphic to $V\otimes_{\Bbbk} (W \otimes_{\Bbbk} Z)$.  However, we can still do the trick described above, but we get a weak 2-category rather than a strict one.

\item Another example of a 2-category that we cannot help but describe here is the 2-category $\cat{Bim}$.  The objects of $\cat{Bim}$ are $\Bbbk$-algebras.  A 1-morphism between a pair of $\Bbbk$-algebras $R$ and $S$ is an $(S,R)$-bimodule $_{S}M_{R}$.  Composition is performed using the tensor product of bimodules
\[
 \xymatrix{T &&\ar[ll]_-{_TN_{S}} S && \ar[ll]_-{_SM_{R}}R}  = \xymatrix{T && \ar[ll]_{_TN_{S} \otimes_S {}_SM_{R}} R}.
\]
Note that composition is not strictly associative since the tensor product is only associative up canonical isomorphism, so $\cat{Bim}$ is a weak 2-category.  A 2-morphism in $\cat{Bim}$ is a bimodule homomorphism.  Note that bimodule homomorphisms can be composed (vertical composition) and tensored (horizontal composition).
\end{enumerate}
\end{examples}

While the globular diagrams we have used above in the definition of a 2-category already have a 2-dimensional flavour stemming from the two types of composition, it turns out we will be primarily interested in a different flavour of diagrams for 2-categories called string diagrams.

%
\subsection{String diagrams} \label{sec_string}
%

String diagrams as an informal tool for computations existed for some time before the idea was formalized in the language of higher category theory~\cite{js2}. String diagrams are diagrammatic representations of 2-morphisms that are Poincar\'{e} dual to the more traditional globular diagrams used in the previous section.  Vertices in the globular diagrams labelled by objects of $\cal{K}$ become regions in the plane in a string diagram. The 1-morphisms in $\cal{K}$ are still represented by 1-dimensional edges in a string diagram, but these edges go in an orthogonal direction to the edges in the globular diagram. For 2-morphisms the dimensions are again flipped so that the 2-dimensional globular region between two edges is transformed into 0-dimensional vertex.
\[
\xy (8,0)*+{\scs x}="4"; (-8,0)*+{\scs y}="6"; {\ar@/^1.65pc/^{f}
"4";"6"}; {\ar@/_1.65pc/_{g} "4";"6"}; {\ar@{=>}_<<<{\scriptstyle
\alpha} (-.5,-3)*{};(-.5,3)*{}} ;
\endxy
\qquad \rightsquigarrow \quad
  \xy
  (-10,0)*{y};
  (8,0)*{x};
  (0,10);(0,-10) **\dir{-};
  (0,0)*{\bullet}+(2.5,1)*{ \alpha};
  (0,-12)*{\scs f};
  (0,12)*{\scs g};
  \endxy
\]
The diagram on the right is a string diagram representing a 2-morphism $\alpha \maps f \To g$ between 1-morphisms $f,g \maps x \to y$.  Our convention here is that string diagrams are read from bottom to top and from right to left.  Notice that all the combinatorial source and target information can be immediately read off from the string diagram using this convention.

More interesting 2-morphisms between composites of
1-morphisms, such as a 2-morphism $\alpha$ between the composite $\xymatrix@1{z
& \ar[l]_{f_2}y  &\ar[l]_{f_1} x }$ and the composite $\xymatrix@1{z
& \ar[l]_{g_2}w  & \ar[l]_{g_1}x }$, can be represented as
\[
  \xy
 (22,0)*+{x}="L";
(0,0)*+{y}="M";
(11,10)*+{w}="LT";
(11,-10)*+{z}="RB";
{\ar@/_.4pc/_{g_1} "L";"LT" };
{\ar@/^.4pc/^{f_1} "L";"RB" };
{\ar@/_.4pc/_{g_2} "LT";"M" };
{\ar@/^.4pc/^{f_2} "RB";"M" };
{\ar@{=>}^{\alpha} (11,-5)*{};(11,5)*{}};
\endxy
\qquad \rightsquigarrow \qquad
 \xy (0,12)*{};
  (0,0)*{\bullet}="X";
  (-4,10)*{}="TL"; (4,10)*{}="TR";(-4,-10)*{}="BL"; (4,-10)*{}="BR";
  "TL"; "X" **\crv{(-4,4)};
  "TR"; "X" **\crv{(4,4)} ;
  "X"; "BL" **\crv{(-4,-4)};
  "X"; "BR" **\crv{(4,-4)} ;
  (3,0)*{\alpha};
  (-10,0)*{y};
  (10,0)*{x};
  (0,6)*{w};
  (0,-6)*{z};
  (4,-12)*{\scs f_1};
  (-4,-12)*{\scs f_2};
  (4,12)*{\scs g_1};
  (-4,12)*{\scs g_2};
\endxy
\]
and between more general composites as
\begin{equation} \label{eq_general_nattrans}
  \xy
 {\ar@{=>}^{\alpha'}
 (0,-14)*+{\scs f_n\cdots f_2f_1};
 (0,14)*+{\scs g_m\cdots g_2g_1}};
 \endxy
 \quad \rightsquigarrow \qquad
  \xy
  (8,8)*{w_1};
  (0,8)*{w_2};
  (12,0)*{x};
  (5,-8)*{y_1};
  (-16,0)*{z};
  (0,0)*{\bullet}="x";
  (-8,6)*{\cdots};
  (-2,-6)*{\cdots};
  (14,14)*+{\scs g_1}; "x" **\crv{(12,4)};
  (5,14)*+{\scs g_2}; "x" **\crv{(4,4)};
  (-5,14)*+{\scs g_3}; "x" **\crv{(-4,4)};
  (-16,14)*+{\scs g_m}; "x" **\crv{(-14,4)};
  "x";(8,-14)*+{\scs f_1} **\crv{(12,-6)};
  "x";(2,-14)*+{\scs f_2} **\crv{};
  "x";(-12,-14)*+{\scs f_n} **\crv{(-12,-6)};
  \endxy
\end{equation}
where the composites $f_n\cdots f_2f_1$ and $g_m\cdots g_2g_1$ are composable
strings of 1-morphisms mapping $x$ to $z$.  In string notation we can think of a 1-morphism as being represented by a sequence of dots on a line corresponding to the endpoints of the edges in the string diagrams.  Composition of 1-morphisms is represented in string notation by placing these endpoints side by side. We could have drawn the string diagram above in the form:
\[  \xy
  (-8,-4)*{z};
  (8,-4)*{x};
  (0,10);(0,-10) **\dir{-};
  (0,0)*{\bullet}+(3.5,1)*{ \alpha'};
  (0,-12)*{\scs \scs f_n\cdots f_2f_1};
  (0,12)*{\scs g_m\cdots g_2g_1};
  \endxy
\]
but we will use the form in \eqref{eq_general_nattrans} because it makes the composition structure more transparent.

Identity 1-morphisms are not depicted in string diagrams.  For example, let $f,g \maps x \to y$ be 1-morphisms in $\cal{K}$ and consider a 2-morphism $\alpha \maps f \To g$. By the identity axioms for identity 1-morphisms,  this 2-morphism is equal to the obvious 2-morphisms $\alpha \maps 1_yf \To g$ and $\alpha \maps f1_x \To g$.  In string notation this implies an equality of string diagrams
\[
  \xy (0,12)*{};
  (0,0)*{\bullet}="X";
  (0,10)*{}="TL"; (-4,-10)*{}="BL"; (4,-10)*{}="BR";
  "TL"; "X" **\dir{-};
  "X"; "BL" **\crv{(-4,-4)};
  "X"; "BR" **\crv{(4,-4)} ;
  (3,0)*{\alpha};
  (-10,0)*{y};
  (10,0)*{x};
  (0,-6)*{y};
  (4,-12)*{\scs f};
  (-4,-12)*{\scs  1_y};
  (0,12)*{\scs g};
\endxy
\quad = \quad
  \xy (0,12)*{};
  (0,0)*{\bullet}="X";
  (0,10)*{}="TL";  (0,-10)*{}="BR";
  "TL"; "X" **\dir{-};
  "X"; "BR" **\dir{-} ;
  (3,0)*{\alpha};
  (-10,0)*{y};
  (10,0)*{x};
  (0,-12)*{\scs f};
  (0,12)*{\scs g};
\endxy
\quad = \quad
  \xy (0,12)*{};
  (0,0)*{\bullet}="X";
  (0,10)*{}="TL"; (-4,-10)*{}="BL"; (4,-10)*{}="BR";
  "TL"; "X" **\dir{-};
  "X"; "BL" **\crv{(-4,-4)};
  "X"; "BR" **\crv{(4,-4)} ;
  (3,0)*{\alpha};
  (-10,0)*{y};
  (10,0)*{x};
  (0,-6)*{x};
  (4,-12)*{\scs 1_x};
  (-4,-12)*{\scs f};
  (0,12)*{\scs  g};
\endxy
\]
But since we do not draw identity 1-morphisms in string notation we can think of the axiom for identity 1-morphisms as saying that
\[
  \xy (0,12)*{};
  (0,0)*{\bullet}="X";
  (0,10)*{}="TL"; (-4,-10)*{}="BL"; (4,-10)*{}="BR";
  "TL"; "X" **\dir{-};
  "X"; "BR" **\crv{(4,-4)} ;
  (3,0)*{\alpha};
  (-10,0)*{y};
  (10,0)*{x};
  (0,-6)*{y};
  (4,-12)*{\scs f};
  (0,12)*{\scs g};
\endxy
\quad = \quad
  \xy (0,12)*{};
  (0,0)*{\bullet}="X";
  (0,10)*{}="TL";  (0,-10)*{}="BR";
  "TL"; "X" **\dir{-};
  "X"; "BR" **\dir{-} ;
  (3,0)*{\alpha};
  (-10,0)*{y};
  (10,0)*{x};
  (0,-12)*{\scs f};
  (0,12)*{\scs g};
\endxy
\quad = \quad
  \xy (0,12)*{};
  (0,0)*{\bullet}="X";
  (0,10)*{}="TL"; (-4,-10)*{}="BL"; (4,-10)*{}="BR";
  "TL"; "X" **\dir{-};
  "X"; "BL" **\crv{(-4,-4)};
  (3,0)*{\alpha};
  (-10,0)*{y};
  (10,0)*{x};
  (0,-6)*{x};
  (-4,-12)*{\scs f};
  (0,12)*{\scs  g};
\endxy
\]
so that the relative horizontal position of the top and bottom endpoints in a string diagram is not relevant: any placement of the endpoints gives rise to a string diagram representing the same 2-morphism in $\cal{K}$.

\begin{examples} Below we give several examples of 2-morphisms between various composite 1-morphisms to help illustrate this convention for depicting identity 1-morphisms.
\begin{enumerate}[1)]
  \item $
  \xy (8,0)*+{\scs x}="4"; (-8,0)*+{\scs x}="6"; {\ar@/^1.65pc/^{1_x}
"4";"6"}; {\ar@/_1.65pc/_{g} "4";"6"}; {\ar@{=>}_<<<{\scriptstyle
\alpha} (-.5,-3)*{};(-.5,3)*{}} ;
\endxy \quad \rightsquigarrow \quad
  \vcenter{\xy
  (0,4)*{\bullet}="X";
   (0,10)*{}="BL";
  "X"; "BL" **\dir{-};
  (3,4)*{\alpha};
  (10,0)*{x};
  (0,12)*{\scs  g};
\endxy}
$
  \item
$\xy
 (8,0)*+{\scs x}="4";
 (-8,0)*+{\scs x}="6";
 (0,-6)*+{\scs y}="b";
 {\ar@/^.35pc/^{f} "4";"b"};
 {\ar@/^.35pc/^{g} "b";"6"};
 {\ar@/_1.65pc/_{1_x} "4";"6"};
 {\ar@{=>}_<<<{\scriptstyle
\alpha} (-.5,-3)*{};(-.5,3)*{}} ;
\endxy
\quad \rightsquigarrow \quad
  \xy
  (0,0)*{\bullet}="X";
   (-4,-10)*{}="BL"; (4,-10)*{}="BR";
  "X"; "BL" **\crv{(-4,-4)};
  "X"; "BR" **\crv{(4,-4)} ;
  (3,0)*{\alpha};
  (10,4)*{x};
  (0,-6)*{y};
  (4,-12)*{\scs f};
  (-4,-12)*{\scs  g};
\endxy$
\end{enumerate}
\end{examples}

Strings diagrams also give a natural way to depict the various composition operations for 2-morphisms.  Horizontal composition is achieved by placing string diagrams side by side:
  \[
\xy (16,0)*+{\scs x}="4"; (0,0)*+{\scs y}="6"; {\ar@/^1.65pc/^{f} "4";"6"};
{\ar@/_1.65pc/_{g} "4";"6"}; {\ar@{=>}_<<<{\scriptstyle \alpha}
(8,-3)*{};(8,3)*{}} ; (0,0)*+{\scs y}="4"; (-16,0)*+{\scs z}="6";
{\ar@/^1.65pc/^{f'} "4";"6"}; {\ar@/_1.65pc/_{g'} "4";"6"};
{\ar@{=>}_<<<{\scriptstyle \beta} (-8,-3)*{};(-8,3)*{}} ;
\endxy
\quad
\rightsquigarrow \quad
  \xy
  (-18,6)*{z};
  (-5,6)*{y};
  (8,6)*{x};
  (0,10);(0,-10) **\dir{-};
  (0,-12)*{\scs f};
  (0,12)*{\scs g};
  (-10,10);(-10,-10) **\dir{-};
  (-10,-12)*{\scs f'};
  (-10,12)*{\scs g'};
  (0,0)*{\bullet};
  (-10,0)*{\bullet};
  (2.5,0)*{\scs \alpha};
  (-7.5,0)*{\scs \beta};
  \endxy
\]
By the associativity of horizontal composition we get a well defined 2-morphism in $\cal{K}$ obtained from placing three string diagrams side by side.  The order in which we regard this horizontal composite does not matter.

Vertical composition of 2-morphisms is achieved by stacking diagrams on top of each other
\[
\xy (8,0)*+{\scs x}="4"; (-8,0)*+{\scs y}="6"; {\ar^(.7){g} "4";"6"}; {\ar@/^1.75pc/^{f}
"4";"6"}; {\ar@/_1.75pc/_{h} "4";"6"}; {\ar@{=>}_<<{\scriptstyle \alpha}
(0,-6)*{};(0,-1)*{}} ; {\ar@{=>}_<<{\scriptstyle \beta} (0,1)*{};(0,6)*{}} ;
\endxy
\quad \rightsquigarrow \quad
   \xy
  (-10,0)*{y};
  (8,0)*{x};
  (0,10);(0,-10) **\dir{-};
  (0,-12)*{\scs f};
  (-2,0)*{\scs g};
  (0,12)*{\scs h};
  (0,-5)*{\bullet};
  (3.5,-5)*{\scs \alpha};
  (0,5)*{\bullet};
  (3.5,5)*{\scs \beta};
  \endxy
\]

Because identity 2-morphisms can be removed from composites using the identity axioms in a 2-category we simplify the presentation of string diagrams by not drawing identity 2-morphisms either.  We write
\[
  \xy
  (-8,0)*{y};
  (8,0)*{x};
  (0,10);(0,-10) **\dir{-};
  (0,0)*{\bullet}+(3,1)*{\scs 1_f};
  (0,-12)*{\scs f};
  (0,12)*{\scs f};
  \endxy
  \qquad \quad  \text{as} \qquad \quad
   \xy
  (-8,0)*{y};
  (8,0)*{x};
  (0,10);(0,-10) **\dir{-};
  (0,-12)*{\scs f};
  (0,12)*{\scs f};
  \endxy
\]
for simplicity.

\begin{examples} Below we give several examples to illustrate compositions of 2-morphisms involving identity morphisms and 2-morphisms.
\begin{enumerate}[1)]
  \item $
 \xy
 (12,0)*+{\scs x}="4";
 (-12,0)*+{\scs x}="6";
 (0,0)*+{\scs y}="m";
 {\ar^{f} "4";"m"}; {\ar^{g} "m";"6"};
 {\ar@/^1.95pc/^{1_x} "4";"6"};
 {\ar@/_1.95pc/_{1_x} "4";"6"};
 {\ar@{=>}_<<{\scriptstyle \alpha}(0,-7)*{};(0,-2)*{}}; {\ar@{=>}_<<{\scriptstyle \beta} (0,2)*{};(0,7)*{}};
\endxy \quad \rightsquigarrow \quad
  \xy
  (0,5)*{\bullet}="X";
   (0,-5)*{\bullet}="B";
  "X"; "B" **\crv{(-7,0)};
  "X"; "B" **\crv{(7,0)} ;
  (1,-7)*{\scs \alpha}; (1,7)*{\scs \beta};
  (8,6)*{x};
  (0,0)*{y};
  (6,0)*{\scs f};
  (-6,0)*{\scs  g};
\endxy$
\item The identity 2-morphism of an identity 1-morphism $1_x \maps x \to x$ just appears as a region labelled by the object $x$
\[\xy (8,0)*+{\scs x}="4"; (-8,0)*+{\scs x}="6"; {\ar@/^1.65pc/^{1_x}
"4";"6"}; {\ar@/_1.65pc/_{1_x} "4";"6"}; {\ar@{=>}_<<<{ 1_{1_x}} (-.5,-3)*{};(-.5,3)*{}} ;
\endxy \quad \rightsquigarrow \quad
  \xy
  (-6,6)*{}="X"; (6,-6)*{}="X";
  (4,4)*{x};
\endxy\]
\item $
 \xy
 (0,0)*+{\scs y}="4";
 (-24,0)*+{\scs y}="6";
 (-12,0)*+{\scs z}="m";
 {\ar^{f} "4";"m"}; {\ar^{g} "m";"6"};
 {\ar@/^1.95pc/^{h} "4";"6"};
 {\ar@/_1.95pc/_{1_y} "4";"6"};
 {\ar@{=>}_<<{\scriptstyle \alpha}(-12,-7)*{};(-12,-2)*{}}; {\ar@{=>}_<<{\scriptstyle \beta} (-12,2)*{};(-12,7)*{}};
 (16,0)*+{\scs x}="r";
 {\ar_{k} "r";"4"};
\endxy
\quad = \quad
 \xy
 (0,0)*+{\scs y}="4";
 (-24,0)*+{\scs y}="6";
 (-12,0)*+{\scs z}="m";
 {\ar^{f} "4";"m"}; {\ar^{g} "m";"6"};
 {\ar@/^1.95pc/^{h} "4";"6"};
 {\ar@/_1.95pc/_{1_y} "4";"6"};
 {\ar@{=>}_<<{\scriptstyle \alpha}(-12,-7)*{};(-12,-2)*{}}; {\ar@{=>}_<<{\scriptstyle \beta} (-12,2)*{};(-12,7)*{}};
 (16,0)*+{\scs x}="r";
{\ar^(.3){k} "r";"4"};
{\ar@/^1.75pc/^{k} "r";"4"};
{\ar@/_1.75pc/_{k} "r";"4"}; {\ar@{=>}_<<{\scriptstyle 1_k}
(7,-6)*{};(7,-1)*{}} ; {\ar@{=>}_<<{\scriptstyle 1_k} (7,1)*{};(7,6)*{}} ;
\endxy \quad \rightsquigarrow \quad
  \xy
  (0,5)*{\bullet}="X";
  (0,-5)*{\bullet}="B";
  (0,-10)*{}="BB";
    "X"; "B" **\crv{(-7,0)};  "X"; "B" **\crv{(7,0)} ;
  "B";"BB" **\dir{-};
  (2.5,-7)*{\scs \alpha}; (1,7)*{\scs \beta};
  (8,6)*{y}; (20,6)*{x};
  (0,0)*{z};
  (6,0)*{\scs f}; (0,-12)*{\scs h};
  (-6,0)*{\scs  g};
  (14,10)*{}; (14,-10)*{}; **\dir{-};
  (14,-12)*{\scs  k};
\endxy$
\end{enumerate}
\end{examples}

To convert a string diagram $D$ back into a globular diagram we must be a bit more precise with our set up. Regard the string diagram $D$ as living in the infinite strip $\R \times [0,1]$. The boundary of the string diagram $D$ is mapped to the top and bottom of the strip with the source 1-morphisms mapping to $\R \times \{0\}$ and the target to $\R \times \{1\}$.

Each vertex of $D$ corresponding to a nonidentity 2-morphism occurs at some horizontal line $\R \times \{t\}$ with $0 < t < 1$. The string diagram $D$ then intersects the line $\R \times \{t\}$ at points $(n_1,t), (n_2,t), \dots, (n_a,t)$ with $n_1 < n_2 < \dots < n_a$.  For each point $(n_s,t)$, with $1 \leq  s \leq a$, in the intersection, we get a 2-morphism $\alpha_{(n_s,t)}$ depending on wether the intersection point $(n_s,t)$ is a vertex of $D$, or a point lying on a vertical line. If the intersection point is a vertex of $D$, then let $\alpha_{(n_s,t)}$ be the 2-morphism labelling this vertex, otherwise let  $\alpha_{(n_s,t)}=\Id_{f_{(n_s,t)}}$ where $f_{(n_s,t)}$ is the 1-morphism labelling the line intersecting $n_{s} \times \{t\}$.  To the intersection $D \cap (\R \times \{t\})$ we can then associate the horizontal composite $\alpha_t:= \alpha_{(n_1,t)} \ast \alpha_{(n_2,t)} \ast \dots \ast \alpha_{(n_a,t)}$.  If the all vertices of $D$ occur at heights $t_1$, $t_2$, \dots $t_b$ with $0<t_1 < t_2 < \dots < t_b$, then the 2-morphism corresponding to the string diagram $D$ is then the vertical composite $\alpha_{t_b} \dots \alpha_{t_2}\alpha_{t_1}$.

The height that we place vertices labelling a 2-morphism is not relevant in the sense that placing the label at any height gives rise to the same 2-morphism in $\cal{K}$.  To see this note that the axiom describing the behaviour of identity 2-morphisms under vertical composition implies
\begin{alignat}{3}
 \xy
  (-8,0)*{y};
  (8,0)*{x};
  (0,10);(0,-10) **\dir{-};
  (0,-12)*{\scs f};
  (0,12)*{\scs g};
  (0,-5)*{\bullet};
  (3.5,-5)*{ \alpha};
  (0,5)*{};
  (3.5,5)*{};
  \endxy
  &\quad =\quad&
  \xy
  (-8,0)*{y};
  (8,0)*{x};
  (0,10);(0,-10) **\dir{-};
  (0,0)*{\bullet}+(2.5,1)*{ \alpha};
  (0,-12)*{\scs f};
  (0,12)*{\scs g};
  \endxy
  &\quad = \quad&
     \xy
  (-8,0)*{y};
  (8,0)*{x};
  (0,10);(0,-10) **\dir{-};
  (0,-12)*{\scs f};
  (0,12)*{\scs g};
  (0,-5)*{};
  (3.5,-5)*{\scs };
  (0,5)*{\bullet};
  (3.5,5)*{ \alpha};
  \endxy \nn\\ \nn\\\nn
\xy (8,0)*+{\scs x}="4"; (-8,0)*+{\scs y}="6"; {\ar^<<<<g "4";"6"};
{\ar@/^1.75pc/^{f} "4";"6"}; {\ar@/_1.75pc/_{g} "4";"6"};
{\ar@{=>}_<<{\scriptstyle \alpha} (0,-6)*{};(0,-1)*{}} ;
{\ar@{=>}_<<{\scriptstyle 1_g} (0,1)*{};(0,6)*{}} ;
\endxy
&\quad=\quad&
\xy (8,0)*+{\scs x}="4"; (-8,0)*+{\scs y}="6"; {\ar@/^1.65pc/^{f}
"4";"6"}; {\ar@/_1.65pc/_{g} "4";"6"}; {\ar@{=>}_<<<{\scriptstyle \alpha}
(.5,-3)*{};(.5,3)*{}} ;
\endxy
 &\quad=\quad&
 \xy (8,0)*+{\scs x}="4"; (-8,0)*+{\scs y}="6"; {\ar^<<<<f
"4";"6"}; {\ar@/^1.75pc/^{f} "4";"6"}; {\ar@/_1.75pc/_{g} "4";"6"};
{\ar@{=>}^<<{\scriptstyle 1_f} (0,-6)*{};(0,-1)*{}} ; {\ar@{=>}^<<{\scriptstyle
\alpha} (0,1)*{};(0,6)*{}} ;
\endxy
\end{alignat}
and more generally, using the interchange law the relative positions of these vertices is also not relevant.
\begin{alignat}{3}
   \xy
  (-18,0)*{z};
  (-5,0)*{y};
  (8,0)*{x};
  (0,10);(0,-10) **\dir{-};
  (0,-12)*{\scs f};
  (0,12)*{\scs g};
  (-10,10);(-10,-10) **\dir{-};
  (-10,-12)*{\scs f'};
  (-10,12)*{\scs g'};
  (0,-5)*{\bullet};
  (-10,5)*{\bullet};
  (2.5,-5)*{\scs \alpha};
  (-7.5,5)*{\scs \beta};
  \endxy
 & \quad = \quad &
   \xy
  (-18,6)*{z};
  (-5,6)*{y};
  (8,6)*{x};
  (0,10);(0,-10) **\dir{-};
  (0,-12)*{\scs f};
  (0,12)*{\scs g};
  (-10,10);(-10,-10) **\dir{-};
  (-10,-12)*{\scs f'};
  (-10,12)*{\scs g'};
  (0,0)*{\bullet};
  (-10,0)*{\bullet};
  (2.5,0)*{\scs \alpha};
  (-7.5,0)*{\scs \beta};
  \endxy
 & \quad = \quad &
   \xy
  (-18,0)*{z};
  (-5,0)*{y};
  (8,0)*{x};
  (0,10);(0,-10) **\dir{-};
  (0,-12)*{\scs f};
  (0,12)*{\scs g};
  (-10,10);(-10,-10) **\dir{-};
  (-10,-12)*{\scs f'};
  (-10,12)*{\scs g'};
  (0,5)*{\bullet};
  (-10,-5)*{\bullet};
  (2.5,5)*{\scs \alpha};
  (-7.5,-5)*{\scs \beta};
  \endxy \nn\\
   \xy
    (-16,0)*+{\scs z}="4";
    (0,0)*+{\scs y}="6";
    (16,0)*+{\scs x}="8";
    {\ar "6";"4"};
    {\ar "8";"6"};
    {\ar@/^1.8pc/^{f'} "6";"4"};
    {\ar@/_1.8pc/_{g'} "6";"4"};
    {\ar@/^1.8pc/^{f} "8";"6"};
    {\ar@/_1.8pc/_{g} "8";"6"};
    {\ar@{=>}^<<{\scriptstyle 1_{f'}} (-8,-6)*{};(-8,-1)*{}} ;
    {\ar@{=>}^<<{\scriptstyle \beta} (-8,1)*{};(-8,6)*{}} ;
    {\ar@{=>}^<<{\scriptstyle \alpha} (8,-6)*{};(8,-1)*{}} ;
    {\ar@{=>}^<<{\scriptstyle 1_g} (8,1)*{};(8,6)*{}} ;
\endxy
 & \quad = \quad &
 \xy (16,0)*+{\scs x}="4"; (0,0)*+{\scs y}="6"; {\ar@/^1.65pc/^{f} "4";"6"};
{\ar@/_1.65pc/_{g} "4";"6"}; {\ar@{=>}_<<<{\scriptstyle \alpha}
(8,-3)*{};(8,3)*{}} ; (0,0)*+{\scs y}="4"; (-16,0)*+{\scs z}="6";
{\ar@/^1.65pc/^{f'} "4";"6"}; {\ar@/_1.65pc/_{g'} "4";"6"};
{\ar@{=>}_<<<{\scriptstyle \beta} (-8,-3)*{};(-8,3)*{}} ;
\endxy
 & \quad = \quad &
   \xy
    (-16,0)*+{\scs x}="4";
    (0,0)*+{\scs y}="6";
    (16,0)*+{\scs z}="8";
    {\ar "6";"4"};
    {\ar "8";"6"};
    {\ar@/^1.8pc/^{f'} "6";"4"};
    {\ar@/_1.8pc/_{g'} "6";"4"};
    {\ar@/^1.8pc/^{f} "8";"6"};
    {\ar@/_1.8pc/_{g} "8";"6"};
    {\ar@{=>}^<<{\scriptstyle \beta} (-8,-6)*{};(-8,-1)*{}} ;
    {\ar@{=>}^<<{\scriptstyle 1_{g'}} (-8,1)*{};(-8,6)*{}} ;
    {\ar@{=>}^<<{\scriptstyle 1_{f}} (8,-6)*{};(8,-1)*{}} ;
    {\ar@{=>}^<<{\scriptstyle \alpha} (8,1)*{};(8,6)*{}} ;
\endxy \nn
\end{alignat}
In this way, the respective heights of the labels for 2-morphisms can be regarded as corresponding to different composites with identity 2-morphisms.  However, the identity axioms together with the interchange ensure that all such choices give rise to string diagrams representing the same 2-morphism in $\cal{K}$.

%
\subsection{Graphical calculus for biadjoints} \label{subsec_biadjoint}
%

In the previous section we saw that the axioms of a 2-category give string diagrams a topological flavour, allowing us to exchange the heights of various morphisms and slide the endpoints right and left.  String diagrams take on a more dramatic  topological flavour when we consider adjoints in a 2-category.  An adjunction in a 2-category is a 2-categorical notion generalizing adjoint functors in the 2-category $\cat{Cat}$.   Recall that a 1-morphism $f \maps x \to y$ is a left adjoint to a 1-morphism $u\maps y \to x$ in a 2-category $\cal{K}$ if there exists 2-morphisms
\begin{equation} \label{eq_glob_unit-counit}
 \vcenter{\xy
 (0,0)*+{\scs y}="4";
 (-12,0)*+{\scs x}="m";
 (-24,0)*+{\scs y}="6";
 {\ar^{u} "4";"m"}; {\ar^{f} "m";"6"};
 {\ar@/_1.95pc/_{1_y} "4";"6"};
 {\ar@{=>}_<<{\scriptstyle \varepsilon} (-12,2)*{};(-12,7)*{}};
\endxy}
\qquad \qquad
 \vcenter{\xy
 (12,0)*+{\scs x}="r";
 (0,0)*+{\scs y}="4";
 (-12,0)*+{\scs x}="m";
 {\ar^{u} "4";"m"};
 {\ar@/^1.95pc/^{1_x} "r";"m"};
 {\ar@{=>}_<<{\scriptstyle \eta}(0,-7)*{};(0,-2)*{}};
 {\ar^{f} "r";"4"};
\endxy},
\end{equation}
called the counit and unit of the adjunction, such that the equalities
\begin{align} \nn
 \xy
 (12,0)*+{\scs x}="r";
 (0,0)*+{\scs y}="4";
 (-12,0)*+{\scs x}="m";
 (-24,0)*+{\scs y}="6";
 {\ar^{u} "4";"m"}; {\ar^{f} "m";"6"};  {\ar^{f} "r";"4"};
 {\ar@/^1.95pc/^{1_x} "r";"m"};
 {\ar@/_1.95pc/_{1_y} "4";"6"};
 {\ar@{=>}_<<{\scriptstyle \eta}(0,-7)*{};(0,-2)*{}};
 {\ar@{=>}_<<{\scriptstyle \varepsilon} (-12,2)*{};(-12,7)*{}};
\endxy
\quad &= \quad
\xy (8,0)*+{\scs x}="4"; (-8,0)*+{\scs y}="6"; {\ar@/^1.65pc/^{f}
"4";"6"}; {\ar@/_1.65pc/_{f} "4";"6"}; {\ar@{=>}_<<<{ 1_f} (-.5,-3)*{};(-.5,3)*{}} ;
\endxy
\\ \label{eq_glob-zigzag}
 \xy
 (12,0)*+{\scs y}="r";
 (0,0)*+{\scs x}="4";
 (-12,0)*+{\scs y}="m";
 (-24,0)*+{\scs x}="6";
 {\ar^{f} "4";"m"}; {\ar^{u} "m";"6"};  {\ar^{u} "r";"4"};
 {\ar@/^1.95pc/^{1_y} "4";"6"};
 {\ar@/_1.95pc/_{1_x} "r";"m"};
 {\ar@{=>}_<<{\scriptstyle \eta}(-12,-7)*{};(-12,-2)*{}};
 {\ar@{=>}_<<{\scriptstyle \varepsilon} (0,2)*{};(0,7)*{}};
\endxy
\quad &= \quad
\xy (8,0)*+{\scs y}="4"; (-8,0)*+{\scs x}="6";
 {\ar@/^1.65pc/^{u} "4";"6"};
 {\ar@/_1.65pc/_{u} "4";"6"};
 {\ar@{=>}_<<<{ 1_u} (-.5,-3)*{};(-.5,3)*{}} ;
\endxy
\end{align}
hold. See \cite{Bor} for the relationship of this definition to the usual definition of adjoint functors defined in terms of a natural bijection between Hom sets.

In string notation the counit and unit for the adjunction take the form:
\[
 \vcenter{\xy
 (0,0)*+{\scs y}="4";
 (-12,0)*+{\scs x}="m";
 (-24,0)*+{\scs y}="6";
 {\ar^{u} "4";"m"}; {\ar^{f} "m";"6"};
 {\ar@/_1.95pc/_{1_y} "4";"6"};
 {\ar@{=>}_<<{\scriptstyle \varepsilon} (-12,2)*{};(-12,7)*{}};
\endxy} \;\;\rightsquigarrow \;\;
\xy
  (0,0)*{\bullet}="X";
   (-4,-10)*{}="BL"; (4,-10)*{}="BR";
  "X"; "BL" **\crv{(-4,-4)};
  "X"; "BR" **\crv{(4,-4)} ;
  (3,0)*{\varepsilon};
  (10,4)*{y};
  (0,-6)*{x};
  (4,-12)*{\scs u};
  (-4,-12)*{\scs  f};
\endxy
\]
\[
 \vcenter{\xy
 (12,0)*+{\scs x}="r";
 (0,0)*+{\scs y}="4";
 (-12,0)*+{\scs x}="m";
 {\ar_{u} "4";"m"};
 {\ar@/^1.95pc/^{1_x} "r";"m"};
 {\ar@{=>}_<<{\scriptstyle \eta}(0,-7)*{};(0,-2)*{}};
 {\ar_{f} "r";"4"};
\endxy} \;\; \rightsquigarrow \;\;
\xy
  (0,0)*{\bullet}="X";
   (-4,10)*{}="BL"; (4,10)*{}="BR";
  "X"; "BL" **\crv{(-4,4)};
  "X"; "BR" **\crv{(4,4)} ;
  (3,0)*{\eta};
  (10,-4)*{x};
  (0,6)*{y};
  (4,12)*{\scs f};
  (-4,12)*{\scs  u};
\endxy
\]
It is common  to simplify the presentation of string diagrams for the unit and counit of an adjunction.  Adding an orientation to the diagram by writing
\[
  1_{f} \quad:= \quad
   \xy
  (-8,0)*{y};
  (8,0)*{x};
  (0,10);(0,-10) **\dir{-} ?(.5)*\dir{<};
  (0,-12)*{\scs f};
  \endxy
\qquad \qquad
1_{u}  \quad := \quad
   \xy
  (-8,0)*{x};
  (8,0)*{y};
  (0,10);(0,-10) **\dir{-} ?(.5)*\dir{>};
  (0,-12)*{\scs u};
  \endxy
\]
the unit and counit can be expressed
as
\begin{equation}
 \xy
 {\ar@{=>}^{\varepsilon} (0,-10)*+{\scs fu}; (0,6)*+{\scs 1_y }};
 \endxy
\quad \rightsquigarrow \quad
 \xy
  (-5,-10)*+{\scs f};(5,-10)*+{\scs u} **\crv{(-5,4)&(5,4)} ?(.2)*\dir{>} ?(.85)*\dir{>}; (0,0.5)*{\bullet}+(0,2.5)*{\scs \varepsilon};
  (5,6)*{y};(0,-4)*{x};
 \endxy
 \qquad
 \qquad
  \xy
 {\ar@{=>}^{\eta} (0,-10)*+{\scs 1_x}; (0,6)*+{\scs uf}};
 \endxy
 \quad \rightsquigarrow \quad
 \xy
  (-5,6)*+{\scs u};(5,6)*+{\scs f} **\crv{(-5,-8)&(5,-8)} ?(.17)*\dir{>} ?(.85)*\dir{>};;
  (0,-4.5)*{\bullet}+(0,-2.5)*{\scs \eta};
  (-5,-10)*{x};(0,0)*{y};
 \endxy
\end{equation}
For convenience we often omit the vertices from these diagrams and write these 2-morphisms in a simplified notation
\begin{equation}
 \xy
 {\ar@{=>}^{\varepsilon} (0,-10)*+{\scs fu}; (0,6)*+{\scs 1_y }};
 \endxy
\quad \rightsquigarrow \quad
 \xy
  (-5,-10)*+{\scs f};(5,-10)*+{\scs u} **\crv{(-5,4)&(5,4)} ?(.2)*\dir{>} ?(.85)*\dir{>};
  (0,6)*{y};(0,-4)*{x};
 \endxy
 \qquad
 \qquad
  \xy
 {\ar@{=>}^{\eta} (0,-10)*+{\scs 1_x}; (0,6)*+{\scs uf}};
 \endxy
 \quad \rightsquigarrow \quad
 \xy
  (-5,6)*+{\scs u};(5,6)*+{\scs f} **\crv{(-5,-8)&(5,-8)} ?(.17)*\dir{>} ?(.85)*\dir{>};;
  (0,-10)*{x};(0,0)*{y};
 \endxy
\end{equation}
With these orientations we have a less cluttered notation where it is no longer necessary to label the unit and counit.  We could even remove the labels for objects and morphisms from the diagrams since an upward oriented arrow corresponds to the 1-morphism $f$ and a downward oriented line corresponds to the 1-morphism $u$.

This simplification to the presentation of the unit and counit of an adjunction is best understood by considering the axioms of an adjunction in \eqref{eq_glob-zigzag}. In string notation, with the above simplifications, these equations give diagrammatic identities:
\begin{equation} \label{eq_string-zigzag1}
  \xy
    (-8,0)*{}="1";
    (0,0)*{}="2";
    (8,0)*{}="3";
    (-8,-10);"1" **\dir{-};
    "1";"2" **\crv{(-8,8) & (0,8)} ?(0)*\dir{>} ?(1)*\dir{>};
    "2";"3" **\crv{(0,-8) & (8,-8)}?(1)*\dir{>};
    "3"; (8,10) **\dir{-};
    (12,-9)*{x};
    (-6,9)*{y};
    \endxy
    \;\; =  \;\;
\xy
    (-8,0)*{}="1";
    (0,0)*{}="2";
    (8,0)*{}="3";
    (0,-10);(0,10)**\dir{-} ?(.5)*\dir{>};
    (5,8)*{x};
    (-9,8)*{y};
    \endxy
\qquad \qquad \xy
    (8,0)*{}="1";
    (0,0)*{}="2";
    (-8,0)*{}="3";
    (8,-10);"1" **\dir{-};
    "1";"2" **\crv{(8,8) & (0,8)} ?(0)*\dir{<} ?(1)*\dir{<};
    "2";"3" **\crv{(0,-8) & (-8,-8)}?(1)*\dir{<};
    "3"; (-8,10) **\dir{-};
    (12,9)*{y};
    (-6,-9)*{x};
    \endxy
    \;\; =  \;\;
\xy
    (8,0)*{}="1";
    (0,0)*{}="2";
    (-8,0)*{}="3";
    (0,-10);(0,10)**\dir{-} ?(.5)*\dir{<};
    (9,-8)*{y};
    (-6,-8)*{x};
    \endxy
\end{equation}
that we call the zig-zag identities, because they say that we can straighten out a zig-zag when it occurs in a string diagram.

When the 1-morphism $f \maps x \to y$ is both left and right adjoint to $u \maps y \to x$ we say that $f$ and $u$ are biadjoint 1-morphisms in $\cal{K}$.  Then there are cap and cup string diagrams with all possible orientations
\begin{equation}
 \xy
  (-5,-10)*+{\scs f};(5,-10)*+{\scs u} **\crv{(-5,4)&(5,4)} ?(.2)*\dir{>} ?(.85)*\dir{>};
  (0,6)*{y};(0,-4)*{x};
 \endxy
 \qquad \quad
 \xy
  (-5,6)*+{\scs u};(5,6)*+{\scs f} **\crv{(-5,-8)&(5,-8)} ?(.17)*\dir{>} ?(.85)*\dir{>};;
  (0,-10)*{x};(0,0)*{y};
 \endxy
 \qquad \quad
  \xy
  (-5,-10)*+{\scs u};(5,-10)*+{\scs f} **\crv{(-5,4)&(5,4)} ?(.15)*\dir{<} ?(.8)*\dir{<};
  (0,6)*{x};(0,-4)*{y};
 \endxy
 \qquad \quad
 \xy
  (-5,6)*+{\scs f};(5,6)*+{\scs u} **\crv{(-5,-8)&(5,-8)} ?(.17)*\dir{<} ?(.8)*\dir{<};;
  (0,-10)*{y};(0,0)*{x};
 \endxy
\end{equation}
satisfying the zig-zag law \eqref{eq_string-zigzag1} and the new zig-zag equations
\begin{equation}
 \xy
    (8,0)*{}="1";
    (0,0)*{}="2";
    (-8,0)*{}="3";
    (8,-10);"1" **\dir{-};
    "1";"2" **\crv{(8,8) & (0,8)} ?(0)*\dir{>} ?(1)*\dir{>};
    "2";"3" **\crv{(0,-8) & (-8,-8)}?(1)*\dir{>};
    "3"; (-8,10) **\dir{-};
    (12,9)*{x};
    (-5,-9)*{y};
    \endxy
    \; =
    \;
      \xy
    (8,0)*{}="1";
    (0,0)*{}="2";
    (-8,0)*{}="3";
    (0,-10);(0,10)**\dir{-} ?(.5)*\dir{>};
    (5,-8)*{x};
    (-9,-8)*{y};
    \endxy
\qquad \quad  \xy
    (-8,0)*{}="1";
    (0,0)*{}="2";
    (8,0)*{}="3";
    (-8,-10);"1" **\dir{-};
    "1";"2" **\crv{(-8,8) & (0,8)} ?(0)*\dir{<} ?(1)*\dir{<};
    "2";"3" **\crv{(0,-8) & (8,-8)}?(1)*\dir{<};
    "3"; (8,10) **\dir{-};
    (12,-9)*{y};
    (-6,9)*{x};
    \endxy
    \; =
    \;
\xy
    (-8,0)*{}="1";
    (0,0)*{}="2";
    (8,0)*{}="3";
    (0,-10);(0,10)**\dir{-} ?(.5)*\dir{<};
   (9,8)*{y};
    (-6,8)*{x};
    \endxy
\end{equation}
saying that $f$ is right adjoint to $u$.

In general between any two objects $x$ and $y$ there may be many 1-morphisms
between them with biadjoints.  However, the biadjoint $u$ of a given
1-morphism $f$ is unique up to isomorphism \footnote{It is a fun exercise for the reader to prove this fact using string diagrams.  }.  In the presence of biadjoints the string diagrams can get quite involved.   An example of a typical diagram
representing a 2-morphisms consisting of composites of units and counits for
various biadjoints is given below:
\[
    \xy
    (-8,-0)*{}="1"; (0,0)*{}="2"; (0,-10)*{}="2'"; (8,0)*{}="3";
    (-8,-10)*{}="0'"; (8,0)*{}="A"; (16,0)="B";
    "1";"0'" **\dir{-} ?(0)*\dir{<};
    "B";"1" **\crv{(16,26)& (-8,26)}?(0)*\dir{<};
        "A";"B" **\crv{(8,-6)&(16,-6)};
        "2";"A" **\crv{(0,6)&(8,6)};
        "2'";"2" **\dir{-} ?(.5)*\dir{<};
        (0,12);(6,12) **\crv{(0,18)&(6,18)}?(0)*\dir{<};
        (0,12);(6,12) **\crv{(0,6)&(6,6)} ;
        (22,24);(30,24) **\crv{(22,12)&(30,12)}?(.1)*\dir{>};
        (24,2);(30,2) **\crv{(24,8)&(30,8)}?(0)*\dir{<};
        (24,2);(30,2) **\crv{(24,-4)&(30,-4)};
        (20,2);(34,2) **\crv{(20,16)&(34,16)}?(1)*\dir{>};
        (20,2);(34,2) **\crv{(20,-14)&(34,-14)};
        (-16,-10);(-16,24) **\dir{-}?(.5)*\dir{>};
    (-4,-6)*{x};(6,-6)*{y};(3,12)*{z};
    (26,20)*{z};(27,2)*{x};(27,-6)*{w};
    (-10,20)*{y};(-20,20)*{w};
        (-16,26)*{\scs f_1};(-16,-12)*{\scs f_1};
        (0,-12)*{\scs u_2};(-8,-12)*{\scs f_2};
        (22,26)*{\scs u_3};(30,26)*{\scs f_3};
        (8,11)*{\scs f_4}; (21,-10)*{\scs f_5}; (31,-2)*{\scs f_6};
\endxy
\]
where we have only labelled the upward oriented lines in the bubble like diagrams. The downward oriented lines carry the label of the 1-morphism biadjoint to the 1-morphism labelling the upward oriented line.

The string diagrams above are beginning to look like diagrams drawn in knot theory and low-dimensional topology, and this is no coincidence.  In knot theory one is typically working with string diagrams for monoidal categories (like the monoidal category of representations of the Hopf algebra $\mathbf{U}_q(\mathfrak{g})$ obtained by $q$-deforming the universal enveloping algebra of a Lie algebra $\mathfrak{g}$).  In this context we are viewing the monoidal category as a 2-category with one object using the trick described in Example~\ref{examp-monoidal}.  Since there is only one object there is no need to label the regions of a string diagram because they all carry the label of the single object.  For more on string diagrams in a 2-category see~\cite{js2,Street,Street2} or the YouTube videos by the Catsters~\cite{Catsters}. Additional material on biadjoints can be found in~\cite{Bart,Kh2,Kh4,Lau0,Muger}.

%
\section{Categorifying Lusztig's $\U$} \label{sec_catU}
%

%
\subsection{What to expect from categorified quantum groups} \label{subsec_what-structure}
%
The $\Q(q)$-algebra $\U$ can be regarded as a category whose Hom spaces $1_m \U 1_n$ mapping $n$ to $m$ have the additional structure of a $\Q(q)$-module.  This observation applies to any nonunital ring with a collection of mutually orthogonal idempotents. When thinking of $\U$ as a category we have
\begin{itemize}
  \item objects: $n \in \Z$,
  \item morphisms $n \to m$: the $\Q(q)$-module $1_m \U 1_n$
  \begin{itemize}
    \item identities: $1_n$
    \item composition: $1_{m'} \U 1_m \otimes 1_{n'} \U 1_n \to \delta_{n',m} 1_{m'}\U1_n$
    given by multiplication.
  \end{itemize}
\end{itemize}
What would it mean to categorify $\U$?  Since $\U$ is already a category, we would expect its categorification to have the structure of a 2-category.

The additional structure on the Hom sets $1_m \U 1_n$ of $\U$ suggests that the Homs in our 2-category will have the additional structure of a grading allowing us to lift the action of $q$.  But this raises a problem as the Hom sets $1_m \U 1_n$ are modules over the ring of rational functions in $q$, while we can only lift a $\Z[q,q^{-1}]$-module structure using gradings.  Luckily, there is an integral version of $\U$, denoted $\UA$ for $\cal{A}=\Z[q,q^{-1}]$.  This $\Z[q,q^{-1}]$-subalgebra of $\U$ is spanned by products of divided powers
\begin{equation}
  E^{(a)}1_n := \frac{E^a}{[a]!}1_n, \qquad F^{(b)}1_n := \frac{F^b}{[b]!}1_n. \nn
\end{equation}
When thinking of the algebra $\UA$ as a category the Hom sets have the structure of a $\Z[q,q^{-1}]$-module, rather than a $\Q(q)$-module.  Hence, the $\Z[q,q^{-1}]$-algebra $\UA$ is the algebra that we should expect to have a categorification.

The first step is to realize the $\Z[q,q^{-1}]$-modules $1_m\UA1_n$ as the shadows of some higher categorical structure.   One such categorification is to identify $1_m\UA1_n$ with the split Grothendieck group $K_0(\UDnm)$ of an additive category $_{m}\UcatD_{n}$. The split Grothendieck group of $_{m}\UcatD_{n}$ is the  $\Z[q,q^{-1}]$-module is generated by symbols
$[x]$ for each isomorphism class of object $x$ in $\UDnm$ modulo the relations
\begin{equation}
  [x]=[x_1]+[x_2] \qquad  \text{if $x=x_1\oplus x_2$.} \nn
\end{equation}
We further require that the objects of the categories $\UDnm$  carry a $\Z$-grading allowing us to shift the grading $x\{t\}$ up by $t$ for each object $x$ and $t \in \Z$; we restrict to morphisms that are degree preserving  with respect to this grading.  In this way, we lift the $\Z[q,q^{-1}]$-module structure on $\Unm$ by requiring
\begin{equation}
 [x\{t\}]=q^t[x], \nn
\end{equation}
so that multiplication by $q$ lifts to the invertible functor $\{1\}$ of shifting the grading by $1$.

Next we need to piece the various categories $\UDnm$ together so that they induce the algebra structure on $\UA$.  This is done by gluing the categories $\UDnm$ into a 2-category $\UcatD=\UcatD(\mathfrak{sl}_2)$ whose objects are indexed by weights $n \in \Z$ and whose Hom categories from $n$ to $m$ are the additive categories $\UDnm$. For any $n,m\in \Z$ we call the objects of $\UDnm$ 1-morphisms in $\UcatD$ and the morphisms in $\UDnm$ 2-morphisms in $\UcatD$.  A 2-category $\UcatD$ whose Hom categories are additive categories and whose composition functor
\[
  {}_{n''}\UcatD_{n'} \times {}_{n'}\UcatD_n \to {}_{n''}\UcatD_n
\]
preserves this additive structure is called an {\em additive 2-category}.

Define the split Grothendieck group of the additive 2-category $\UcatD$ as follows:
\begin{equation}
  K_0(\UcatD) = \bigoplus_{n,m} K_0(\UDnm) \nn
\end{equation}
and require that
\begin{equation}
  [x] = [x_1][x_2] \qquad \text{if $x = x_1 \circ x_2$.} \nn
\end{equation}
Then the composition of 1-morphisms in the 2-category $\UcatD$ gives rise to the composition (or multiplication) in the category (algebra) $\UA$.

One might wonder why it is natural to consider additive categorifications $\UA=K_0(\UcatD)$ where we take the split Grothendieck ring.  The canonical basis $\B$ of $\UA$ is a big hint that a categorification of this form should be possible.  Recall that the canonical basis has the property that
\begin{equation}
  [b_x] [b_y] = \sum_{z} m_{x,y}^z [b_z] \qquad \text{for $[b_x]$, $[b_y]$, $[b_z] \in \B$,} \nn
\end{equation}
where the structure coefficients $m_{x,y}^z$ are elements of $\N[q,q^{-1}]$.  If the structure constants were in $\Z[q,q^{-1}]$ we would expect to work with a category of complexes over an additive category, or with more general triangulated categories where a more sophisticated notion of Grothendieck group is required to recover the negative integers.

The positivity and integrality of these structure coefficients suggests that it should be possible to define an additive 2-category $\UcatD$ whose indecomposable 1-morphisms correspond up to grading shift to elements in Lusztig's canonical basis $\B$. Recall that the isomorphism classes of indecomposable 1-morphisms in $\UcatD$, up to grading shift, give a basis in the split Grothendieck ring $K_0(\UcatD)$.  For indecomposable 1-morphisms $b_x , b_y$ in $\UcatD$, the structure coefficients $m_{x,y}^z$ for the product $[b_x][b_y]$ in $K_0(\UcatD)$ are obtained by decomposing the composite $b_x b_y$ into indecomposables
\[
\bigoplus_z  \left(\bigoplus_{m_{x,y}^z} b_z\right),
\]
where for any $f = \sum_{a} f_a q^a \in \N[q,q^{-1}]$ and a 1-morphism $x \in \UcatD$ we write $\bigoplus_f x$ for the direct sum over $a \in \Z$ of $f_a$ copies of $x\{a\}$. Then in $K_0(\UcatD)$ we have
\begin{equation} \nn
  [b_x][b_y] = \sum_{z} m_{x,y}^z [b_z].
\end{equation}
Because the structure coefficients are given by decomposing a graded 1-morphism into indecomposable graded 1-morphisms we must have that $m_{x,y}^z \in \N[q,q^{-1}]$.

Thus our goal is to define a 2-category $\UcatD$ whose indecomposable 1-morphisms correspond (up to grading shift) to Lusztig's canonical basis.  How do we control which 1-morphisms become indecomposable in our 2-category?  Decompositions of 1-morphisms into other 1-morphisms is governed by the 2-morphisms in $\UcatD$.   Furthermore, in a successful categorification these 2-morphisms will be responsible for giving rise to explicit isomorphisms of 1-morphisms lifting the defining relations in $\U$.  In particular, the relations for $\U$ should be consequences of the new higher structure present for 2-morphisms.  This higher structure of $\UcatD$ is what makes a categorification interesting.  This is where we expect to see new structure that could not be seen working with just quantum groups.  These observations are summarized in Figure~\ref{fig:idea}.
\begin{figure}[htp]
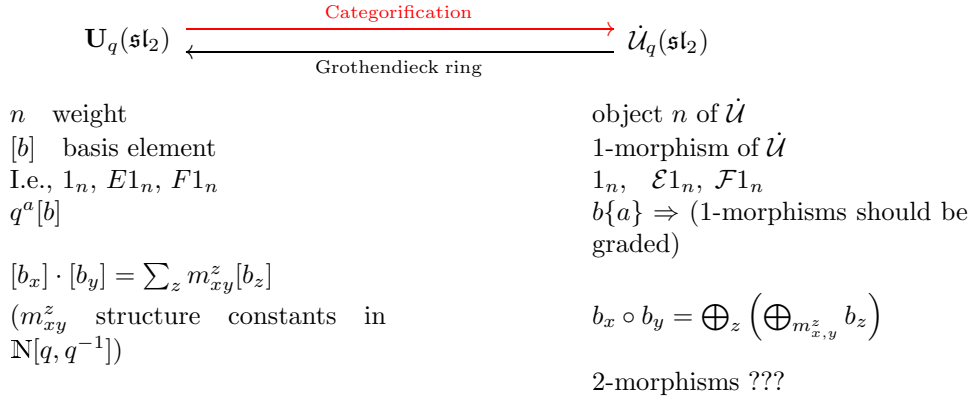

\[
   \;\;
\xy (-44,0)*++{\mathbf{U}_q(\mathfrak{sl}_2)}="1";
(28,0)*++{\dot{\cal{U}}_q(\mathfrak{sl}_2)}="2";
 (60,0)*{};
  {\ar@<1ex>^{\text{Grothendieck ring}} "2";"1"};
    \textcolor[rgb]{1.00,0.00,0.00}{{\ar@<1ex>^{\text{Categorification}}
  "1";"2"}}  \endxy  \;\;
\]
\begin{center}
 \begin{tabular}{p{5cm} l p{5cm} }
   $n \quad \text{weight}$ & \hspace{0.8in} & object $n$ of $\UcatD$ \\
   $[b] \quad \text{basis element}$ & \hspace{0.8in} & 1-morphism of $\UcatD$ \\
   I.e., $1_n$, $E1_n$, $F1_n$ & & $1_n$, \; $\cal{E}1_n, \; \cal{F}1_n$ \\
   $q^a [b]$ & & $b \{a\}$ $\To$ (1-morphisms should be graded) \\
   $[b_x] \cdot [b_y] = \sum_z m_{x y}^z [b_z]$ & & \\
  ($m_{x y}^z$ structure constants in $\N[q,q^{-1}]$) & & $b_x \circ b_y = \bigoplus_z  \left(\bigoplus_{m_{x,y}^z} b_z\right)$  \\ & &2-morphisms ???
 \end{tabular}
\end{center}
\caption{A schematic depiction of the lifting of structure under categorification.} \label{fig:idea}
\end{figure}

Since the 2-morphisms will play such a vital role in defining a categorification of $\U$, this raises the question of how one should go about defining them.  At this point it is useful to conjecture that a categorification $\UcatD$ of $\U$ exists and look for the remnants or shadows of this higher structure in the decategorified context.

In the theory of graded vector spaces the space of degree preserving linear maps $\Hom_{\Bbbk}(V,W)$ between graded vector spaces $V$, and $W$ forms a $\Bbbk$-vector space. However, there is also an {\em graded} Hom given by the {\em graded} vector space of degree homogeneous linear maps $\HOM_{\Bbbk}(V,W)$.
In the 2-category $\UcatD$ the space of 2-morphisms $\UcatD(x,y)$ between two 1-morphisms $x$ and $y$ should form some $\Bbbk$-vector space of degree preserving 2-morphisms.  But just as in the theory of graded vector spaces it is natural to consider the graded 2Hom  in $\UcatD$ given by taking all degree homogeneous 2-morphisms
\begin{equation} \nn
  \HOM_{\UcatD}(x,y):=\bigoplus_{t\in\Z}\UcatD(x\{t\},y).
\end{equation}
The graded 2Hom now associates a {\em graded} vector space $\HOM_{\UcatD}(x,y)$ to each pair of 1-morphisms $x$ and $y$ in $\UcatD$.  We sometimes write $\END_{\UcatD}(x)$ for $\HOM_{\UcatD}(x,x)$.

The graded 2Hom $\HOM_{\UcatD}(,)$ is a rigidly defined structure.  This is because the graded 2Hom on $\UcatD$ can be thought of as a map
\begin{equation}
 \xy
 (-30,10)*+{\HOM_{\UcatD}( , ) \maps {\rm 1morph}\left(\UcatD\right) \times
 {\rm 1morph}\left(\UcatD\right)}="1";
 (30,10)*+{\cat{GrVect}_{\Bbbk}}="2";
 (-25,4)*++{x \qquad \times \qquad y}="3";(30,4)*++{\HOM_{\UcatD}(x ,y)}="4";
 {\ar "1";"2"};{\ar@{|->} "3"+(26,0);"4"};
 \endxy \nn
\end{equation}
assigning the graded vector space of all 2-morphisms $x \To y$. If $\UcatD$ is a categorification of $\U$, so that the 1-morphisms in $\UcatD$ correspond to algebra elements in $\U$, then decategorifying the graded 2Hom  gives a pairing on $\U$:
\[
 \xy
 (-30,4)*+{\HOM_{\UcatD}( , ) \maps {\rm 1morph}\left(\UcatD\right) \times
 {\rm 1morph}\left(\UcatD\right)}="1";
 (30,4)*+{\cat{GrVect}_{\Bbbk}}="2";
 (-25,-20)*++{\U \qquad \times \qquad \U}="5";(30,-20)*++{\Z[[q,q^{-1}]]}="6";
 {\ar "1";"2"}; {\ar "5";"6"}; (-45,-20)*{\sla,\sra \maps};
 {\ar@{~>}^{K_0} (-38,0)*{}; (-38,-14)*{}};   {\ar@{~>}^{K_0} (-12,0)*{}; (-12,-14)*{}};
  {\ar@{~>}^{\gdim} (28,0)*{}; (28,-14)*{}}; (-75,-8)*{\textbf{Decategorification}};
 \endxy
\]
That is,
\[
\sla [x],[y] \sra := \gdim \HOM_{\UcatD}(x,y)=\sum_{t\in\Z}q^t \dim\Hom_{\UcatD}(x\{t\},y),
\]
where $\dim \Hom_{\UcatD}(x\{t\},y)$ is the usual dimension of the graded vector space $\UcatD(x\{t\},y)$ of degree zero 2-morphisms. Hence, any choice of 2-morphisms in $\HOM_{\UcatD}(x,y)$ gives rise to a pairing $\sla [x],[y]\sra$ on $\U$ given by taking the graded dimension $\gdim$ of the graded vector space $\HOM_{\UcatD}(x,y)$.

Notice the behaviour of this
pairing with respect to the shift functor in each variable. If a 1-morphism $f \maps x
\to y$ has degree $\alpha$, then the degree of the corresponding 1-morphism $f \maps
x\{t\} \to y$ will be $\alpha-t$. Similarly, the 1-morphism $f \maps x \to y\{t'\}$ will
have degree $\alpha+t'$, see Figure~\ref{fig:degree}.   Hence,
\begin{eqnarray*}
 \textcolor[rgb]{1.00,0.00,0.00}{\langle q^t x , y\rangle} = \gdim \left(\HOM_{\UcatD}(x\{t\},y)\right) = q^{-t} \gdim \left(\HOM_{\UcatD}(x,y)\right)
  = \textcolor[rgb]{1.00,0.00,0.00}{q^{-t}\langle x, y\rangle}\\   \\
  \textcolor[rgb]{1.00,0.00,0.00}{\langle x , q^{t'}y\rangle} = \gdim \left(\HOM_{\UcatD}(x,y\{t'\})\right) = q^{t'} \gdim \left(\HOM_{\UcatD}(x,y) \right) =
  \textcolor[rgb]{1.00,0.00,0.00}{q^{t'}\langle x, y\rangle}
\end{eqnarray*}
so that the pairing induced on $\U$ must be semilinear, i.e.
$\Z[q,q^{-1}]$-antilinear in the first slot, and $\Z[q,q^{-1}]$-linear in the
second.   Therefore we have found a clue hinting at the structure of $\UcatD$; the graded Hom on the 2-category $\UcatD$ must categorify a
semilinear form on $\U$.

\begin{figure}[h]
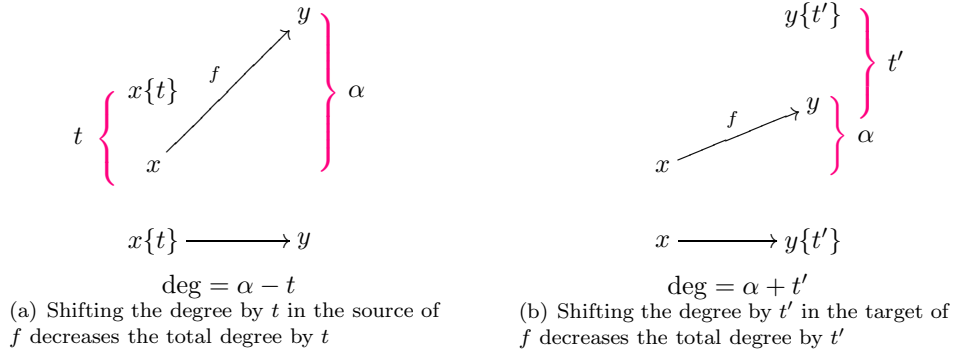

  \begin{center}
    \subfigure[Shifting the degree by $t$ in the source of $f$ decreases the total degree by $t$]{\label{fig:degree-a}
    \xy (-28,-14)*{}; (28,-14)*{};
 (-10,0)*+{x}="l"; (10,20)*+{y}="r1"; (-10,10)*+{x\{t\}}="l2";
 {\ar^f "l";"r1"};
 (13,10)*{\textcolor[rgb]{1.00,0.00,0.50}{\left. \xy (0,-10)*{};(0,11)*{};\endxy\right\}}};
 (17,10)*{\alpha};
 (-16,4)*{\textcolor[rgb]{1.00,0.00,0.50}{\left\{ \xy (0,-5)*{};(0,7)*{};\endxy\right.}};
 (-20,4)*{t};
 (-10,-10)*+{x\{t\}}="l"; (10,-10)*+{y}="r"; {\ar "l";"r"};
 (0,-16)*{\deg = \alpha - t};
\endxy
     }\qquad \quad
    \subfigure[Shifting the degree by $t'$ in the target of $f$ decreases the total degree by $t'$]{\label{fig:edge-b}
    \xy (-28,-14)*{}; (28,-14)*{};
 (-10,0)*+{x}="l"; (10,8)*+{y}="r1"; (10,20)*+{y\{t'\}}="r2";
 {\ar^f "l";"r1"};
 (13,4)*{\textcolor[rgb]{1.00,0.00,0.50}{\left. \xy (0,-5)*{};(0,5)*{};\endxy\right\}}};
 (17,4)*{\alpha};
 (17,14)*{\textcolor[rgb]{1.00,0.00,0.50}{\left. \xy (0,-5)*{};(0,8)*{};\endxy\right\}}};
 (21,14)*{t'};
 (-10,-10)*+{x}="l"; (10,-10)*+{y\{t'\}}="r"; {\ar "l";"r"};
 (0,-16)*{\deg = \alpha + t'};
\endxy}
  \end{center}
 \caption{This figure illustrates how shifting the source and target of a degree homogeneous map effects the total degree of the map. } \label{fig:degree}
\end{figure}

The prevalence of biadjointness in known examples of categorifications suggests that we should expect similar behaviours from a categorification $\UcatD$ of $\U$. This will have consequences for the semilinear form. Furthermore, we could  deduce properties of the form by arguing that this form should be related to the geometric constructions in \cite{BLM,GL1}, but it turns out this will not be necessary.  Such a form on $\U$ with all of these properties has already been defined on $\U$.  This form arises as the graded dimension of a certain Ext algebra between sheaves on Lusztig quiver varieties in Lusztig's geometric realization of $\U$.  For our purposes all that will be important is that the form $\sla ,\sra$ has the following defining properties:
\begin{enumerate}[(i)]
 \item $\sla ,\sra$ is semilinear,
 \item
$ \big\sla 1_{n_1}x1_{n_2}, 1_{n_1'}y1_{n_2'} \big\sra=0 \quad \text{for all $x,y
\in \U$ unless $n_1=n_1'$ and $n_2=n_2'$}$,
 \item $\sla ux,y\sra=\sla x,\tau(u)y\sra$ for $u\in {\bf U}$ and $x,y \in
 \U$,
 \item $
 \sla E^{(a)}1_n,E^{(a)}1_n\sra =
 \sla F^{(a)}1_n,F^{(a)}1_n\sra = \prod_{j=1}^a\frac{1}{ (1-q^{2j})},
$
\end{enumerate}
where $\tau$ is an antilinear antiautomorphism of $\U$ given by
\begin{equation}
\tau \maps
\begin{array}{ccc}
   q^s 1_m E^{(a)} F^{(b)}1_n
    &\mapsto&
 q^{-s-(a-b)(a-b+n)}1_{n}E^{(b)} F^{(a)}1_m, \\
   q^s1_m  F^{(b)}E^{(a)} 1_n
    &\mapsto &
 q^{-s-(a-b)(a-b+n)} 1_{n}  E^{(b)}F^{(a)}1_{m}.
\end{array}
\end{equation}

\begin{example} Some examples of the value of the semilinear form are given below.
\begin{itemize}
  \item $\langle E1_n, E1_n \rangle = \frac{1}{1-q^2} = 1+q^2+q^4+q^6+\dots$
      \medskip
  \item $\langle E^2 1_n, E^2 1_n \rangle = [2][2]\frac{1}{1-q^2}\frac{1}{1-q^4} = (1+q^{-2})(\frac{1}{1-q^2})^2$.
\end{itemize}
\end{example}

%
\subsection{The basic set up}
%

In this section we begin to construct the categorification $\UcatD$ of $\U$. We proceed in several steps.  First we define a 2-category $\Ucat$.  This 2-category arises from a choice of parameters $\chi$ defining a family of 2-categories $\Ucatt$ associated to the algebra $\U$.  In Section~\ref{subsubsec_rescaling} we show that all of these 2-categories are isomorphic to a single 2-category that we denote by $\Ucat$. Then in Section~\ref{subsec_karoubi} we add additional 1-morphisms to $\Ucat$ corresponding to divided powers $E^{(a)}1_n$ and $F^{(b)}1_n$. This is achieved by taking a certain completion of the 2-category $\Ucat$ called the Karoubi envelope, or idempotent completion.

Let $\ep=\epsilon_1\dots\epsilon_m$ with $\epsilon_1,\dots,\epsilon_m \in \{+,-\}$ and write $E_+=E$, $E_-=F$. Write
\begin{equation}
  1_{n'} E_{\ep} 1_n = 1_{n'} E_{\epsilon_1} E_{\epsilon_2} \dots E_{\epsilon_m} 1_n \nn
\end{equation}
with $n'-n=2\sum_{i=1}^m \epsilon_i 1$.

The category $\U$ has objects $n \in \Z$ indexed by weights of $\mathfrak{sl}_2$, so in our 2-category $\Ucat$ we will also have objects $n\in \Z$.  To lift the monomial $1_{n'}E_{\ep}1_n$ we add 1-morphisms $\onenn{n'}\cal{E}_{\ep}\onen = \cal{E}_{\epsilon_1}\dots \cal{E}_{\epsilon_m}\onen$ where $\cal{E}_{+}=\cal{E}$, $\cal{E_{-}}=\cal{F}$. In string notation we can think of the monomial $\onenn{n'}\cal{E}_{\ep}\onen$ as a sequence of labelled dots on a line
From the discussion in the previous section we allow grading shifts $\onenn{n'}\cal{E}_{\ep}\onen\{t\}$ and formal direct sums of 1-morphisms $\onenn{n'}\cal{E}_{\ep}\onen\{t\} \oplus \onenn{n'}\cal{E}_{\ep'}\onen\{t'\}$.

Since there are only two types of 1-morphisms $\cal{E}$ and $\cal{F}$, we  make a less cluttered notation by adding an orientation so that the identity 2-morphisms for $\cal{E}\onen\{t\}$ and $\cal{F}\onen\{t\}$ take the form
\[
\begin{array}{ccc}
  1_{\cal{E} \onen \{t\}} &\quad  & 1_{\cal{F} \onen\{t\}} \\ \\
    \xy
 (0,8);(0,-8); **\dir{-} ?(.5)*\dir{>}+(2.3,0)*{\scriptstyle{}};
 (6,2)*{ n};
 (-8,2)*{ n +2};
 (-10,0)*{};(10,0)*{};
 \endxy
 & &
 \;\;   \xy
 (0,8);(0,-8); **\dir{-} ?(.5)*\dir{<}+(2.3,0)*{\scriptstyle{}};
 (6,2)*{ n};
 (-8,2)*{ n -2};
 (-12,0)*{};(12,0)*{};
 \endxy
\end{array}
\]
in string notation. Notice that the grading shift is omitted from the sting diagram so that the same diagrams corresponds to the identity 2-morphisms of $\cal{E}\onen\{t\}$ and $\cal{F}\onen\{t\}$ for any shift $\{t\}$.  It turns out that this orientation will be consistent with the convention for biadjoints in Section~\ref{subsec_biadjoint}.

Regions in the plane correspond to objects $n \in \Z$ of $\Ucat$. Notice that passing from right to left through an upward oriented line increases the weight labelling a region by two, while passing from right to left through a downward oriented line decreases the weight by two.  In this way, it will only be necessary to label a single region of a diagram, with the labels of all other regions deduced from this rule.

%
\subsection{Generators and relations from the semilinear form} \label{subsec_generators}
%

The semilinear form $\sla , \sra \maps \U \times \U \to \Z[q,q^{-1}]$ holds the
key to understanding the generating 2-morphisms and the relations on these
2-morphisms.  Note that property (ii) of the semilinear form is consistent with the form arising from the graded 2Hom   of a 2-category $\Ucat$.  It ensures that there will always be compatible source and targets for 2-morphisms, so that the space $\HOM_{\Ucat}(\onenn{n_2}\cal{E}_{\ep}\onenn{n_1}, \onenn{n_2'}\cal{E}_{\ep'}\onenn{n_1'})$ is only nonzero when the source and targets of the respective 1-morphisms agree, i.e. $n_1=n_1'$ and $n_2=n_2'$.

Our guiding principle will be to construct the 2-morphisms in $\Ucat$ so that
\begin{equation}
 \gdim \HOM_{\Ucat}(\onenn{m}\cal{E}_{\ep}\onenn{n},\onenn{m}\cal{E}_{\ep'}\onenn{n}) = \sla 1_mE_{\ep}1_n, 1_mE_{\ep'}1_n \sra.
\end{equation}
This means that each term $a q^{t}$ appearing in $\sla 1_mE_{\ep}1_n, 1_mE_{\ep'}1_n \sra$ will be interpreted as the dimension of the $a$-dimensional homogeneous space of 2-morphisms in degree $t$.  This of course requires that the coefficients be natural numbers rather than integers.  If the coefficient $a$ is zero for a term $a q^{t}$ we expect to have no 2-morphisms in degree $t$.  When $a$ is nonzero we add new graded 2-morphisms to act as basis vectors for the space of 2-morphisms in that degree.

The 2-morphisms in $\Ucat$ will be specified using string notation described in Section~\ref{sec_diagrammatics}.  In general, a 2-morphism is a $\Bbbk$-linear combinations of string diagrams, giving the space of 2-morphisms between a pair 1-morphisms the structure of a $\Bbbk$-vector space.

%
\subsubsection{The graded vector space $\HOM_{\Ucat}(\cal{E}\onen,\cal{E}\onen)$}
%

Starting with 2-morphisms between the simplest monomials $\cal{E}_{\ep}\onen$ we can build up the 2-morphisms for more complicated monomials.  The first computation gives the graded dimension of the space of 2-morphisms $\HOM_{\Ucat}(\cal{E}\onen,\cal{E}\onen)$:
\begin{equation} \label{eq_EtoE}
\gdim\left(\HOM_{\Ucat}(\cal{E}\onen,\cal{E}\onen) \right) = \sla E1_n,E1_n \sra =
\frac{1}{1-q^2} = 1+q^2+q^4+q^6+\dots
\end{equation}
The coefficient of $q^{t}$ for $t<0$ is always zero implying that the 2Homs $\Hom_{\Ucat}(\cal{E}\onen\{t\}, \cal{E}\onen)$ are zero for $t<0$. Hence, $\cal{E}\onen$ has no negative degree endomorphisms.

The identity 2-morphisms of $\cal{E}\onen$ must be degree zero.  We interpret
the ``$1=q^0$" appearing in the above sum as the dimension of the 1-dimensional $\Bbbk$-vector space spanned by linear combinations of the identity 2-morphism on $\cal{E}1_n$
\begin{equation} \label{eq_Eident}
\deg\left(\;
    \xy   0;/r.18pc/:
    (-8,0)*{}="1";
    (0,0)*{}="2";
    (8,0)*{}="3";
    (0,-10);(0,10)**\dir{-} ?(.5)*\dir{>};
    (5,9)*{n};
    (-8,9)*{n+2};
    \endxy \right) = 0.
\end{equation}
Because the coefficient of $q^0$ is $1$ all degree zero endomorphisms of $\cal{E}\onen$ should be equal to multiple of the identity 2-morphism.

Our plan to use the semilinear form to guess the space of 2-morphisms would be hopeless if the space of 2-morphisms $\Ucat(\cal{E}\onen, \cal{E}\onen)$ in degree zero was empty.  Our 2-category $\Ucat$ must have identity 2-morphisms for every 1-morphism. It can be shown (see for example \cite[Theorem 2.7]{KL3}) that the coefficient of $q^0$ in $\sla 1_{n'}E_{\ep}1_n,1_{n'}E_{\ep}1_n \sra$ is always greater than or equal to 1 so that the semilinear form will always allow for identity 2-morphisms on 1-morphisms $\onenn{n'}\cal{E}_{\ep}\onen\{t\}$.

The $q^2$ in \eqref{eq_EtoE} suggests that there should be an additional
2-morphism from $\cal{E}\onen$ to itself
\begin{equation} \label{eq_new_dot}
\deg\left(\;
    \xy
 (0,7);(0,-7); **\dir{-} ?(.75)*\dir{>}+(2.3,0)*{\scriptstyle{}}
 ?(.1)*\dir{ }+(2,0)*{\scs };
 (0,-2)*{\txt\large{$\bullet$}};
 (6,4)*{n};
 (-8,4)*{n+2};
 (-10,0)*{};(10,0)*{};
 \endxy\right) \quad = 2.
\end{equation}
Here we are using a simplification of the string notation introduced in Section~\ref{sec_string}, representing this new 2-morphisms by a dot with no label.  If the space of 2-morphisms $\Hom(\cal{E}\onen\{2\}, \cal{E}\onen)$ is spanned by $\Bbbk$-linear combinations of this new 2-morphisms above, then $\Hom(\cal{E}\onen\{2\}, \cal{E}\onen)$ will contribute a factor of $q^2$ to the graded dimension.

It is tempting to think that we must add a new generator in all positive even degrees to account for the remaining terms in \eqref{eq_EtoE}. However, this would fail to fully utilize the 2-categorical structure at our disposal. Using vertical composition in the 2-category we can compose the new degree two 2-morphisms
with itself to get
\begin{equation}
\deg\left(\;
    \xy
 (0,7);(0,-7); **\dir{-} ?(.55)*\dir{>}+(2.3,0)*{\scriptstyle{}}
 ?(.1)*\dir{ }+(2,0)*{\scs };
 (0,-4)*{\txt\large{$\bullet$}};(0,4)*{\txt\large{$\bullet$}};
 (6,4)*{ n};
 (-8,4)*{ n+2};
 (-10,0)*{};(10,0)*{};
 \endxy\right) \quad = 4. \nn
\end{equation}
Iterated vertical composites of the 2-morphism
\eqref{eq_new_dot} can be written as
\begin{equation}
    \xy
 (0,7);(0,-7); **\dir{-} ?(.75)*\dir{>}+(2.3,0)*{\scriptstyle{}}
 ?(.1)*\dir{ }+(2,0)*{\scs };
 (0,-2)*{\txt\large{$\bullet$}}; (-3,-2)*{\scs \alpha};
 (6,4)*{n};
 (-8,4)*{n+2};
 (-10,0)*{};(10,0)*{};
 \endxy \;\; :=\;\;
 \left(
    \xy
 (0,7);(0,-7); **\dir{-} ?(.75)*\dir{>}+(2.3,0)*{\scriptstyle{}}
 ?(.1)*\dir{ }+(2,0)*{\scs };
 (0,-2)*{\txt\large{$\bullet$}};
 (6,4)*{n};
 (-8,4)*{n+2};
 (-10,0)*{};(10,0)*{};
 \endxy\right)^{\alpha} \nn
\end{equation}
for $\alpha \in \N$, so that
\begin{equation}
\deg\left(\;
    \xy
 (0,7);(0,-7); **\dir{-} ?(.75)*\dir{>}+(2.3,0)*{\scriptstyle{}}
 ?(.1)*\dir{ }+(2,0)*{\scs };
 (0,-2)*{\txt\large{$\bullet$}}; (-3,-2)*{\scs \alpha};
 (6,4)*{n};
 (-8,4)*{n+2};
 (-10,0)*{};(10,0)*{};
 \endxy\right)  \;\; = {2\alpha}. \nn
\end{equation}
Imposing no relations on these vertical composites, all terms in \eqref{eq_EtoE}
can be accounted for by the $\Bbbk$-span of these diagrams,
\begin{equation}
  \Hom(\cal{E}\onen\{2\alpha\},\cal{E}\onen) = \left\langle
   \xy
 (0,7);(0,-7); **\dir{-} ?(.75)*\dir{>}+(2.3,0)*{\scriptstyle{}}
 ?(.1)*\dir{ }+(2,0)*{\scs };
 (0,-2)*{\txt\large{$\bullet$}}; (-3,-2)*{\scs \alpha};
 (6,4)*{n};
 (-8,4)*{n+2};
 (-10,0)*{};(10,0)*{};
 \endxy
  \right\rangle_{\Bbbk}.
\end{equation}
Then the graded dimension is given by
\begin{alignat}{5}
\gdim\left(\HOM_{\Ucat}(\cal{E}\onen,\cal{E}\onen) \right)
 &= q^{\deg\left(\;
    \xy
 (0,7);(0,-7); **\dir{-} ?(.75)*\dir{>}+(2.3,0)*{\scriptstyle{}}
 ?(.1)*\dir{ }+(2,0)*{\scs };
 (6,4)*{n};
 (-4,0)*{};(10,0)*{};
 \endxy\right)}
 &+& q^{\deg\left(\;
     \xy
 (0,7);(0,-7); **\dir{-} ?(.75)*\dir{>}+(2.3,0)*{\scriptstyle{}}
 ?(.1)*\dir{ }+(2,0)*{\scs };
 (0,-2)*{\txt\large{$\bullet$}}; (-3,-2)*{\scs };
 (6,4)*{n};
 (-6,0)*{};(10,0)*{};
 \endxy\right)}
 &+& q^{\deg\left(\;
     \xy
 (0,7);(0,-7); **\dir{-} ?(.75)*\dir{>}+(2.3,0)*{\scriptstyle{}}
 ?(.1)*\dir{ }+(2,0)*{\scs };
 (0,-2)*{\txt\large{$\bullet$}}; (-3,-2)*{\scs 2};
 (6,4)*{n};
 (-6,0)*{};(10,0)*{};
 \endxy\right)}
 &+& \cdots \nn \\ \\
 &=  1 &+& q^2 &+& q^4 &+& \cdots \nn
\end{alignat}
The value of the semilinear form \eqref{eq_EtoE} is independent of $n$.  Thus, we
have the degree 2 map \eqref{eq_new_dot} for each $n$ which we denote by
$\Uup_{n}$.

Arguing similarly, we also have a degree two endomorphism $\cal{F}\onen$ which we represent as a dot on a downwards oriented line
\begin{equation}
\deg\left(\;
    \xy
 (0,7);(0,-7); **\dir{-} ?(.75)*\dir{<}+(2.3,0)*{\scriptstyle{}}
 ?(.1)*\dir{ }+(2,0)*{\scs };
 (0,-2)*{\txt\large{$\bullet$}};
 (6,4)*{n};
 (-8,4)*{n-2};
 (-10,0)*{};(10,0)*{};
 \endxy\right) \quad = 2. \nn
\end{equation}
The graded $\Bbbk$-vector space $\HOM_{\Ucat}(\cal{F}\onen,\cal{F}\onen)$ is spanned by the identity 2-morphism $1_{\cal{F}\onen}$ and vertical composites of this new 2-morphism.

We can think of each power of $\frac{1}{1-q^2} = 1+q^2+q^4+q^6+\dots$ that appears in the semilinear form as representing one strand on which we can place arbitrarily many dots.

%
\subsubsection{The graded vector space $\HOM_{\Ucat}(\cal{E}\cal{E}\onen,\cal{E}\cal{E}\onen)$}
%

Horizontally composing $\left(\Uup_{n}\right)^{\alpha_1}$ with
$\left(\Uup_{n+2}\right)^{\alpha_2}$ gives a 2-morphism of degree
$2\alpha_1+2\alpha_2$ from $\cal{E}\cal{E}\onen$ to itself.  If no relations are imposed on these composites we find
\begin{eqnarray}
  \sum_{0\leq \alpha_1, 0\leq \alpha_2}
  q^{\deg\left(\xy
 (3,7);(3,-7); **\dir{-} ?(.75)*\dir{>}+(2.3,0)*{\scriptstyle{}}
 ?(.1)*\dir{ }+(2,0)*{\scs };
 (3,-2)*{\txt\large{$\bullet$}}; (6,-2)*{\scs \alpha_1};
  (-3,7);(-3,-7); **\dir{-} ?(.75)*\dir{>}+(2.3,0)*{\scriptstyle{}}
 ?(.1)*\dir{ }+(2,0)*{\scs };
 (-3,-2)*{\txt\large{$\bullet$}}; (-6,-2)*{\scs \alpha_2};
 (9,4)*{n};(-12,4)*{n+4};
 (-14,0)*{};(12,0)*{};
 \endxy \right)} \quad &=& (1+q^2+q^4+q^6+\dots)(1+q^2+q^4+q^6+\dots) \nn \\
 &=& \left(\frac{1}{1-q^2}\right)^2. \nn
\end{eqnarray}
However, the semilinear form gives
\begin{equation}\label{eq_semi_EE}
\gdim \left(\HOM_{\Ucat}(\cal{E}\cal{E}\onen,\cal{E}\cal{E}\onen) \right) = \sla
EE1_n,EE1_n \sra = [2][2]\sla \cal{E}^{(2)}1_n,\cal{E}^{(2)}1_n \sra = (1+q^{-2})
\left(\frac{1}{1-q^2}\right)^2 \nn
\end{equation}
suggesting that there should be an additional generating 2-morphism of degree -2,
\begin{equation}
\deg\left( \;  \xy 0;/r.18pc/:(0,12)*{};
  (0,0)*{\bullet}="X";
  (-4,8)*{}="TL"; (4,8)*{}="TR";(-4,-8)*{}="BL"; (4,-8)*{}="BR";
  "TL"; "X" **\crv{(-4,3)}?(.25)*\dir{<};
  "TR"; "X" **\crv{(4,3)} ?(.25)*\dir{<};
  "X"; "BL" **\crv{(-4,-3)}?(.65)*\dir{<};
  "X"; "BR" **\crv{(4,-3)} ?(.65)*\dir{<};
  (-12,4)*{n+4};
  (10,4)*{n};
  (0,-6)*{};
\endxy\; \right) = -2. \nn
\end{equation}
To simplify notation write
\begin{equation}
   \xy
  (0,0)*{\xybox{
    (-4,-4)*{};(4,4)*{} **\crv{(-4,-1) & (4,1)}?(1)*\dir{>} ;
    (4,-4)*{};(-4,4)*{} **\crv{(4,-1) & (-4,1)}?(1)*\dir{>};
    (-5,-3)*{\scs };
     (5.1,-3)*{\scs };
     (8,1)*{n};
     (-12,0)*{};(12,0)*{};
     }};
  \endxy \quad  := \quad   \xy 0;/r.18pc/:(0,12)*{};
  (0,0)*{\bullet}="X";
  (-4,8)*{}="TL"; (4,8)*{}="TR";(-4,-8)*{}="BL"; (4,-8)*{}="BR";
  "TL"; "X" **\crv{(-4,3)}?(.25)*\dir{<};
  "TR"; "X" **\crv{(4,3)} ?(.25)*\dir{<};
  "X"; "BL" **\crv{(-4,-3)}?(.65)*\dir{<};
  "X"; "BR" **\crv{(4,-3)} ?(.65)*\dir{<};
  (-12,4)*{n+4};
  (10,4)*{n};
  (0,-6)*{};
\endxy \nn
\end{equation}
and use the shorthand $\Ucross_{n}$ to denote this 2-morphism.

The value of the semilinear form in \eqref{eq_semi_EE} suggests that we can account for all the terms that appear using dots on the identity 2-morphism $1_{\cal{E}\cal{E}\onen}$ and the 2-morphism $\Ucross_{n}$.  However, it is clear that matching the graded dimension for the space $\HOM_{\Ucat}(\cal{E}\cal{E}\onen,\cal{E}\cal{E}\onen)$ with the semilinear form will require us to introduce relations on some of the possible composites.
For example, the vertical composite of $\Ucross_{n}$ with itself has degree $-4$. Since
$q^{-4}$ does not appear in the expansion $(1+q^{-2})
\left(\frac{1}{1-q^2}\right)^2$ we see that the semilinear form requires us to impose the relation
\begin{equation} \label{eq_semiEEm4}
 \vcenter{\xy 0;/r.18pc/:
    (-4,-4)*{};(4,4)*{} **\crv{(-4,-1) & (4,1)}?(1)*\dir{>};
    (4,-4)*{};(-4,4)*{} **\crv{(4,-1) & (-4,1)}?(1)*\dir{>};
    (-4,4)*{};(4,12)*{} **\crv{(-4,7) & (4,9)}?(1)*\dir{>};
    (4,4)*{};(-4,12)*{} **\crv{(4,7) & (-4,9)}?(1)*\dir{>};
  (8,8)*{n};(-5,-3)*{\scs };
     (5.1,-3)*{\scs };
 \endxy}
 =0
\end{equation}
if we want the graded dimension of the 2Homs in $\Ucat$ to match the semilinear form.

The 2-morphism $\Ucross_{n}$  has four upward oriented strands providing four
possible places to add dots\footnote{In case the reader finds string notation confusing, remember that $\xy 0;/r.15pc/:
  (0,0)*{\xybox{
    (-4,-4)*{};(4,4)*{} **\crv{(-4,-1) & (4,1)}?(1)*\dir{>};
    (4,-4)*{};(-4,4)*{} **\crv{(4,-1) & (-4,1)}?(1)*\dir{>}?(.25)*{\bullet};
     (8,1)*{ n};
     (-10,0)*{};(10,0)*{};
     }};
  \endxy$ represents the composite
\[
 \xymatrix{
 \cal{E}\cal{E}\onen
 \ar[rr]^-{\xy 0;/r.15pc/:
 (3,4);(3,-7); **\dir{-} ?(.75)*\dir{>}+(2.3,0)*{\scriptstyle{}}
 ?(.1)*\dir{ }+(2,0)*{\scs };
 (3,-2)*{\bullet};
  (-3,4);(-3,-7); **\dir{-} ?(.75)*\dir{>};
 (9,1)*{n};
 (-14,0)*{};(12,0)*{};
 \endxy}
  & & \cal{E}\cal{E}\onen
   \ar[rr]^-{\xy 0;/r.15pc/:
  (0,0)*{\xybox{
    (-4,-4)*{};(4,4)*{} **\crv{(-4,-1) & (4,1)}?(1)*\dir{>};
    (4,-4)*{};(-4,4)*{} **\crv{(4,-1) & (-4,1)}?(1)*\dir{>};
     (8,1)*{ n};
     (-10,0)*{};(10,0)*{};
     }};
  \endxy
  } & & \cal{E}\cal{E}\onen}
\]
where the dot is represents a degree 2-morphism $\cal{E}\onen \to \cal{E}\onen$ and the crossing is a string diagram representing a 2-morphism of degree $-2$. }
\begin{equation}
\xy
  (0,0)*{\xybox{
    (-4,-4)*{};(4,4)*{} **\crv{(-4,-1) & (4,1)}?(1)*\dir{>}?(.25)*{\bullet};
    (4,-4)*{};(-4,4)*{} **\crv{(4,-1) & (-4,1)}?(1)*\dir{>};
     (8,1)*{ n};
     (-10,0)*{};(10,0)*{};
     }};
  \endxy
\qquad
 \xy
  (0,0)*{\xybox{
    (-4,-4)*{};(4,4)*{} **\crv{(-4,-1) & (4,1)}?(1)*\dir{>}?(.75)*{\bullet};
    (4,-4)*{};(-4,4)*{} **\crv{(4,-1) & (-4,1)}?(1)*\dir{>};
     (8,1)*{ n};
     (-10,0)*{};(10,0)*{};
     }};
  \endxy
\qquad
\xy
  (0,0)*{\xybox{
    (-4,-4)*{};(4,4)*{} **\crv{(-4,-1) & (4,1)}?(1)*\dir{>};
    (4,-4)*{};(-4,4)*{} **\crv{(4,-1) & (-4,1)}?(1)*\dir{>}?(.75)*{\bullet};
     (8,1)*{ n};
     (-10,0)*{};(10,0)*{};
     }};
  \endxy
\qquad
  \xy
  (0,0)*{\xybox{
    (-4,-4)*{};(4,4)*{} **\crv{(-4,-1) & (4,1)}?(1)*\dir{>} ;
    (4,-4)*{};(-4,4)*{} **\crv{(4,-1) & (-4,1)}?(1)*\dir{>}?(.25)*{\bullet};
     (8,1)*{ n};
     (-10,0)*{};(10,0)*{};
     }};
  \endxy
\end{equation}
A quick degree check shows that each of these maps have degree $2-2=0$.  Together with the identity 2-morphism $1_{\cal{E}\cal{E}\onen}$, this gives five different 2-morphisms in degree zero.  However, the coefficient of $q^0$ in \eqref{eq_semi_EE} is only $3$.  Therefore, all five of the 2-morphisms can not be linearly independent if the dimension of our 2Homs is going to match the semilinear form.  We will address how to find the precise form of these relations in Section~\ref{subsec_flag}.

Similar computations for $\HOM_{\Ucat}(\cal{F}\cal{F}\onen, \cal{F}\cal{F}\onen)$ suggest we add a new generating 2-morphism
\begin{equation}
    \xy
  (0,0)*{\xybox{
    (-4,-4)*{};(4,4)*{} **\crv{(-4,-1) & (4,1)}?(0)*\dir{<} ;
    (4,-4)*{};(-4,4)*{} **\crv{(4,-1) & (-4,1)}?(0)*\dir{<};
     (8,1)*{ n};
     (-10,0)*{};(10,0)*{};
     }};
  \endxy \nn
\end{equation}
in degree $-2$ with the same restrictions on the placement of dots.

%
\subsubsection{Generators for biadjunctions}
%

Other generating 2-morphisms are suggested by the computations:
\begin{eqnarray}
  \gdim\left(\HOM_{\Ucat}(\cal{F}\cal{E}\onen,\onen) \right) = \sla FE1_n, 1_n \sra
  = \sla E 1_n , \tau(F)1_n \sra  =  q^{1+n} \sla E1_n , E1_n \sra  =
  \frac{q^{1+n}}{1-q^2},\nn \\
  \gdim\left(\HOM_{\Ucat}(\cal{E}\cal{F}\onen,\onen) \right) = \sla EF1_n, 1_n \sra
  = \sla F 1_n , \tau(E)1_n \sra  =  q^{1-n} \sla F1_n , F1_n \sra  = \frac{q^{1-n}}{1-q^2},\nn \\ \label{eq_semi_EF}
\end{eqnarray}
suggesting just a single generator for each of these 2Homs together with one strand on which we can place arbitrary many dots.  Again, we simplify notation (following our convention for units and counits of an adjunction) and write these generators
in the form
\[
\begin{tabular}{|l|c|c|}
  \hline
  generator &     \xy
    (0,0)*{\bbcef{ }};
    (8,5)*{ n};
    (-12,8)*{};(12,0)*{};
    \endxy &       \xy
    (0,0)*{\bbcfe{ }};
    (8,5)*{n};
    (-12,8)*{};(12,0)*{};
 \endxy \\ & &\\ \hline
  degree & 1+n & 1-n \\
  \hline
\end{tabular} \]
Notice that there are two natural places to add dots on these new generating 2-morphisms using either a dot on the upward oriented strand or a dot on the downward oriented strand.  But only one factor of $\frac{1}{1-q^2}$ appears in the semilinear form \eqref{eq_semi_EF}, so these cannot be linearly independent.  We must have relations
\begin{equation}
 \lambda_1 \;   \xy
    (8,5)*{}="1";
    (0,5)*{}="2";
    (0,-5)*{}="2'";
    (8,-5);"1" **\dir{-}?(.2)*\dir{<};
    "2";"2'" **\dir{-} ?(.6)*\dir{<};
    "1";"2" **\crv{(8,12) & (0,12)} ;
    (15,9)*{n};
    (0,4)*{\bullet};
    \endxy
 \;\; + \;\; \lambda_2 \;\;
     \xy
    (8,5)*{}="1";
    (0,5)*{}="2";
    (0,-5)*{}="2'";
    (8,-5);"1" **\dir{-}?(.2)*\dir{<};
    "2";"2'" **\dir{-} ?(.6)*\dir{<};
    "1";"2" **\crv{(8,12) & (0,12)} ;
    (15,9)*{n};
    (8,4)*{\bullet};
    \endxy
    \;\; = 0 \nn
\end{equation}
\begin{equation}
 \lambda_3 \;   \xy
    (8,5)*{}="1";
    (0,5)*{}="2";
    (0,-5)*{}="2'";
    (8,-5);"1" **\dir{-}?(.3)*\dir{>};
    "2";"2'" **\dir{-} ?(.7)*\dir{>};
    "1";"2" **\crv{(8,12) & (0,12)} ;
    (15,9)*{n};
    (0,4)*{\bullet};
    \endxy
 \;\; + \;\; \lambda_4 \;\;
     \xy
    (8,5)*{}="1";
    (0,5)*{}="2";
    (0,-5)*{}="2'";
    (8,-5);"1" **\dir{-}?(.3)*\dir{>};
    "2";"2'" **\dir{-} ?(.7)*\dir{>};
    "1";"2" **\crv{(8,12) & (0,12)} ;
    (15,9)*{n};
    (8,4)*{\bullet};
    \endxy
    \;\; = 0 \nn
\end{equation}
for some $\lambda_i \in \Bbbk$.

By matching graded dimensions of various 2Homs with values of the semilinear form, one quickly finds that it is most natural to impose the relations
\begin{equation}
 \xy
    (8,5)*{}="1";
    (0,5)*{}="2";
    (0,-5)*{}="2'";
    (8,-5);"1" **\dir{-}?(.2)*\dir{<};
    "2";"2'" **\dir{-} ?(.6)*\dir{<};
    "1";"2" **\crv{(8,12) & (0,12)} ;
    (15,9)*{n};
    (0,4)*{\bullet};
    \endxy
 \;\;= \qquad
     \xy
    (8,5)*{}="1";
    (0,5)*{}="2";
    (0,-5)*{}="2'";
    (8,-5);"1" **\dir{-}?(.2)*\dir{<};
    "2";"2'" **\dir{-} ?(.6)*\dir{<};
    "1";"2" **\crv{(8,12) & (0,12)} ;
    (15,9)*{n};
    (8,4)*{\bullet};
    \endxy \qquad \quad
    \xy
    (8,5)*{}="1";
    (0,5)*{}="2";
    (0,-5)*{}="2'";
    (8,-5);"1" **\dir{-}?(.3)*\dir{>};
    "2";"2'" **\dir{-} ?(.7)*\dir{>};
    "1";"2" **\crv{(8,12) & (0,12)} ;
    (15,9)*{n};
    (0,4)*{\bullet};
    \endxy
 \;\; = \qquad
     \xy
    (8,5)*{}="1";
    (0,5)*{}="2";
    (0,-5)*{}="2'";
    (8,-5);"1" **\dir{-}?(.3)*\dir{>};
    "2";"2'" **\dir{-} ?(.7)*\dir{>};
    "1";"2" **\crv{(8,12) & (0,12)} ;
    (15,9)*{n};
    (8,4)*{\bullet};
    \endxy \nn
\end{equation}
so that we can draw the dot anywhere on the curve without ambiguity.  Then the 2Homs have graded dimension
\begin{eqnarray}
 \sum_{\alpha=0}^{\infty} q^{\deg\left(  \xy
    (0,-3)*{\xy
        (4,0)*{}="t1";(-4,0)*{}="t2";
        "t1";"t2" **\crv{(4,6) & (-4,6)}; ?(.15)*\dir{>} ?(.9)*\dir{>}
   ?(.5)*\dir{}+(0,0)*{\bullet}+(0,2.5)*{\scs \alpha};\endxy };
    (8,5)*{ n};
    (-8,0)*{};(10,0)*{};
    \endxy  \right)} = q^{1+n}(1+q^2+q^4+ \dots) = \frac{q^{1+n}}{1-q^2},
    \nn \\
    \sum_{\alpha=0}^{\infty} q^{\deg\left(  \xy
    (0,-3)*{\xy
        (-4,0)*{}="t1";(4,0)*{}="t2";
        "t1";"t2" **\crv{(-4,6) & (4,6)}; ?(.15)*\dir{>} ?(.9)*\dir{>}
   ?(.5)*\dir{}+(0,0)*{\bullet}+(0,2.5)*{\scs \alpha};\endxy };
    (8,5)*{ n};
    (-8,0)*{};(10,0)*{};
    \endxy  \right)} = q^{1-n}(1+q^2+q^4+ \dots) = \frac{q^{1-n}}{1-q^2}, \nn \\
    \nn
\end{eqnarray}
so that the space of graded 2Homs again agrees with the semilinear form.

Similar calculations for $\HOM_{\Ucat}(\onen,\cal{F}\cal{E}\onen)$ and
$\HOM_{\Ucat}(\onen,\cal{E}\cal{F}\onen)$ suggest generators
\[
\begin{tabular}{|l|c|c|}
  \hline
  generator &  \xy
    (0,-3)*{\bbpef{ }};
    (8,-5)*{ n};
    (-12,0)*{};(12,0)*{};
    \endxy &    \xy
    (0,-3)*{\bbpfe{ }};
    (8,-5)*{ n};
    (-12,0)*{};(12,0)*{};
    \endxy \\ & &\\ \hline
  degree & 1+n & 1-n \\
  \hline
\end{tabular} \]
of degrees $1+n$ and $1-n$, respectively.

We are defining the 2-category $\Ucat$ by generators and relations.  This means that we can take arbitrary horizontal and vertical composites of the generators and produce new 2-morphisms.  We must ensure that with these new generators we have not added unwanted 2-morphisms to the spaces $\HOM_{\Ucat}(\cal{E}\onen,\cal{E}\onen)$ and $\HOM_{\Ucat}(\cal{E}\cal{E}\onen,\cal{E}\cal{E}\onen)$.  Taking the horizontal composite of a cap with a vertical line and vertically
composing this with the horizontal composite of a cup and a vertical line
produces 2-morphisms $\cal{E}\onen \To \cal{E}\onen$ of the form
\begin{equation}
  \xy   0;/r.18pc/:
    (-8,0)*{}="1";
    (0,0)*{}="2";
    (8,0)*{}="3";
    (-8,-10);"1" **\dir{-};
    "1";"2" **\crv{(-8,8) & (0,8)} ?(0)*\dir{>} ?(1)*\dir{>};
    "2";"3" **\crv{(0,-8) & (8,-8)}?(1)*\dir{>};
    "3"; (8,10) **\dir{-};
    (12,-9)*{n};
    (-6,9)*{n+2};
    \endxy
 \qquad  \qquad \xy   0;/r.18pc/:
    (8,0)*{}="1";
    (0,0)*{}="2";
    (-8,0)*{}="3";
    (8,-10);"1" **\dir{-};
    "1";"2" **\crv{(8,8) & (0,8)} ?(0)*\dir{>} ?(1)*\dir{>};
    "2";"3" **\crv{(0,-8) & (-8,-8)}?(1)*\dir{>};
    "3"; (-8,10) **\dir{-};
    (12,9)*{n};
    (-5,-9)*{n+2};
    \endxy \nn
\end{equation}
But
\begin{align}
    \deg
\left(\;\;
  \text{ $\xy   0;/r.18pc/:
    (-8,0)*{}="1";
    (0,0)*{}="2";
    (8,0)*{}="3";
    (-8,-10);"1" **\dir{-};
    "1";"2" **\crv{(-8,8) & (0,8)} ?(0)*\dir{>} ?(1)*\dir{>};
    "2";"3" **\crv{(0,-8) & (8,-8)}?(1)*\dir{>};
    "3"; (8,10) **\dir{-};
    (12,-9)*{n};
    (-6,9)*{n+2};\endxy$}
    \;\; \right)
    \;\; &= \;\;
    \deg \left(\;\;
  \vcenter{ \xy  0;/r.18pc/:
    (-8,0)*{}="1";
    (0,0)*{}="2";
    (8,0)*{}="3";
    (-8,-10);"1" **\dir{-};
    "2";"3" **\crv{(0,-8) & (8,-8)}?(1)*\dir{>};
    (12,-9)*{n};
    \endxy}
    \;\; \right) \;\; + \;\;
        \deg \left(\;\;
  \vcenter{ \xy   0;/r.18pc/:
    (-8,0)*{}="1";
    (0,0)*{}="2";
    (8,0)*{}="3";
    "1";"2" **\crv{(-8,8) & (0,8)}  ?(1)*\dir{>};
    "3"; (8,10) **\dir{-};
    (-6,9)*{n+2};
    \endxy}
    \;\; \right)
\nn \\
 &= 1+n  \;\; + 1-(n+2) \;\; = \;\; 0, \nn
\end{align}
and similarly
\[
\deg \left(\;\;\xy   0;/r.18pc/:
    (8,0)*{}="1";
    (0,0)*{}="2";
    (-8,0)*{}="3";
    (8,-10);"1" **\dir{-};
    "1";"2" **\crv{(8,8) & (0,8)} ?(0)*\dir{>} ?(1)*\dir{>};
    "2";"3" **\crv{(0,-8) & (-8,-8)}?(1)*\dir{>};
    "3"; (-8,10) **\dir{-};
    (12,9)*{n};
    (-5,-9)*{n+2};
    \endxy \;\; \right) \quad = \quad 0.
\]

If we impose no relations on these 2-morphisms they will contribute additional terms to the degree zero term of the graded dimension of $\HOM_{\Ucat}(\cal{E}\onen,\cal{E}\onen)$.  Recall that we already accounted for the degree zero term of this 2Hom   using $\Bbbk$-multiples of the identity 2-morphism $1_{\cal{E}\onen}$.  Hence, we must impose relations relating these new zig-zag composites to the identity 2-morphism.

A similar argument applies to maps $\cal{F}\onen \To \cal{F}\onen$.  For convenience and consistency of the graphical calculus we take all of the multiples to be 1 and impose relations:
\begin{equation} \label{eq_biadjointness1}
 \xy   0;/r.18pc/:
    (8,0)*{}="1";
    (0,0)*{}="2";
    (-8,0)*{}="3";
    (8,-10);"1" **\dir{-};
    "1";"2" **\crv{(8,8) & (0,8)} ?(0)*\dir{>} ?(1)*\dir{>};
    "2";"3" **\crv{(0,-8) & (-8,-8)}?(1)*\dir{>};
    "3"; (-8,10) **\dir{-};
    (12,9)*{n};
    (-5,-9)*{n+2};
    \endxy
    \; =
    \;
      \xy 0;/r.18pc/:
    (8,0)*{}="1";
    (0,0)*{}="2";
    (-8,0)*{}="3";
    (0,-10);(0,10)**\dir{-} ?(.5)*\dir{>};
    (5,-8)*{n};
    (-9,-8)*{n+2};
    \endxy
\qquad \quad \xy  0;/r.18pc/:
    (8,0)*{}="1";
    (0,0)*{}="2";
    (-8,0)*{}="3";
    (8,-10);"1" **\dir{-};
    "1";"2" **\crv{(8,8) & (0,8)} ?(0)*\dir{<} ?(1)*\dir{<};
    "2";"3" **\crv{(0,-8) & (-8,-8)}?(1)*\dir{<};
    "3"; (-8,10) **\dir{-};
    (12,9)*{n+2};
    (-6,-9)*{n};
    \endxy
    \; =
    \;
\xy  0;/r.18pc/:
    (8,0)*{}="1";
    (0,0)*{}="2";
    (-8,0)*{}="3";
    (0,-10);(0,10)**\dir{-} ?(.5)*\dir{<};
    (9,-8)*{n+2};
    (-6,-8)*{n};
    \endxy
\end{equation}
\begin{equation} \label{eq_biadjointness2}
  \xy   0;/r.18pc/:
    (-8,0)*{}="1";
    (0,0)*{}="2";
    (8,0)*{}="3";
    (-8,-10);"1" **\dir{-};
    "1";"2" **\crv{(-8,8) & (0,8)} ?(0)*\dir{>} ?(1)*\dir{>};
    "2";"3" **\crv{(0,-8) & (8,-8)}?(1)*\dir{>};
    "3"; (8,10) **\dir{-};
    (12,-9)*{n};
    (-6,9)*{n+2};
    \endxy
    \; =
    \;
\xy   0;/r.18pc/:
    (-8,0)*{}="1";
    (0,0)*{}="2";
    (8,0)*{}="3";
    (0,-10);(0,10)**\dir{-} ?(.5)*\dir{>};
    (5,8)*{n};
    (-9,8)*{n+2};
    \endxy
\qquad \quad  \xy   0;/r.18pc/:
    (-8,0)*{}="1";
    (0,0)*{}="2";
    (8,0)*{}="3";
    (-8,-10);"1" **\dir{-};
    "1";"2" **\crv{(-8,8) & (0,8)} ?(0)*\dir{<} ?(1)*\dir{<};
    "2";"3" **\crv{(0,-8) & (8,-8)}?(1)*\dir{<};
    "3"; (8,10) **\dir{-};
    (12,-9)*{n+2};
    (-6,9)*{n};
    \endxy
    \; =
    \;
\xy   0;/r.18pc/:
    (-8,0)*{}="1";
    (0,0)*{}="2";
    (8,0)*{}="3";
    (0,-10);(0,10)**\dir{-} ?(.5)*\dir{<};
   (9,8)*{n+2};
    (-6,8)*{n};
    \endxy
\end{equation}
so that biadjointness arises naturally from consideration of the semilinear form!

%
\subsubsection{The graded vector spaces $\HOM_{\Ucat}(\cal{E}\cal{F}\onen, \cal{F}\cal{E}\onen)$ and $\HOM_{\Ucat}(\cal{F}\cal{E}\onen, \cal{E}\cal{F}\onen)$}
%

The graded dimension of the space of 2Homs $\HOM_{\Ucat}(\cal{E}\cal{F}\onen, \cal{F}\cal{E}\onen)$
is determined from the semilinear form as follows:
\begin{eqnarray}
 \sla EF1_n,FE1_n\sra &=& \sla F1_n,\tau(E)FE1_n \sra =
 \sla F1_n,q^{1-n}FFE1_n \sra =
  q^{1-n}\sla F1_n,FFE1_n \sra  \nn
\end{eqnarray}
but $F^2E1_n = EF^21_n - ([n]+[n-2])F1_n$ so that
\begin{eqnarray}
 \sla EF1_n,FE1_n\sra
&=& q^{1-n}\left(\sla F1_n,EFF1_n \sra -([n]+[n-2])\sla F1_n, F1_n\sra \right) \nn \\
&=&
 q^{1-n}\sla \tau^{-1}(E1_{n-4})F1_n,FF1_n\sra -([n]+[n-2]) q^{1-n}\frac{q^{1-n}}{1-q^2}   \nn
 \\
 &=&
 q^{1-n}\sla q^{-3+n}FF1_n,FF1_n\sra -([n]+[n-2]) q^{1-n}\frac{q^{1-n}}{1-q^2}   \nn
\end{eqnarray}
which after simplifying becomes
\begin{equation} \label{eq_semi_EFFE}
  \gdim\left(
\HOM_{\Ucat}(\cal{E}\cal{F}\onen, \cal{F}\cal{E}\onen)\right) = (EF1_n,FE1_n) =
(1+q^2)\left( \frac{1}{1-q^2}\right)^2.
\end{equation}

In this case we don't have an identity 2-morphism to account for the degree zero term in \eqref{eq_semi_EFFE}.  However, it is not necessary to add an additional generating 2-morphism since one can check that
\begin{eqnarray}
  \deg\left( \;\xy 0;/r.19pc/:
  (0,0)*{\xybox{
    (-4,-4)*{};(4,4)*{} **\crv{(-4,-1) & (4,1)}?(1)*\dir{>};
    (4,-4)*{};(-4,4)*{} **\crv{(4,-1) & (-4,1)};
     (4,4);(4,12) **\dir{-};
     (12,-4);(12,12) **\dir{-};
     (-4,-4);(-4,-12) **\dir{-};(-12,4);(-12,-12) **\dir{-};
     (-16,1)*{n};
     (10,0)*{};(-10,0)*{};
     (4,-4)*{};(12,-4)*{} **\crv{(4,-10) & (12,-10)}?(1)*\dir{<}?(0)*\dir{<};
      (-4,4)*{};(-12,4)*{} **\crv{(-4,10) & (-12,10)}?(1)*\dir{>}?(0)*\dir{>};
     }};
  \endxy  \; \right)
  \quad = \quad
  \deg\left( \; \xy 0;/r.19pc/:
  (0,0)*{\xybox{
    (4,-4)*{};(-4,4)*{} **\crv{(4,-1) & (-4,1)}?(1)*\dir{<};
    (-4,-4)*{};(4,4)*{} **\crv{(-4,-1) & (4,1)};
     (-4,4);(-4,12) **\dir{-};
     (-12,-4);(-12,12) **\dir{-};
     (4,-4);(4,-12) **\dir{-};(12,4);(12,-12) **\dir{-};
     (-16,1)*{n};
     (-10,0)*{};(10,0)*{};
     (-4,-4)*{};(-12,-4)*{} **\crv{(-4,-10) & (-12,-10)}?(1)*\dir{>}?(0)*\dir{>};
      (4,4)*{};(12,4)*{} **\crv{(4,10) & (12,10)}?(1)*\dir{<}?(0)*\dir{<};
     }};
  \endxy  \; \right) \quad = 0. \nn
\end{eqnarray}
Since there are two different 2-morphisms in degree zero, these two diagrams must be linearly dependent.  Setting them equal to each other, we
begin to see planar isotopy invariance of the graphical calculus emerge.

A similar calculation for $\UcatD(\cal{F}\cal{E}\onen, \cal{E}\cal{F}\onen)$
suggests setting the diagrams below equal to each other:
 \begin{eqnarray}
 \deg\left( \; \xy 0;/r.19pc/:
  (0,0)*{\xybox{
    (4,-4)*{};(-4,4)*{} **\crv{(4,-1) & (-4,1)}?(1)*\dir{>};
    (-4,-4)*{};(4,4)*{} **\crv{(-4,-1) & (4,1)};
     (-4,4);(-4,12) **\dir{-};
     (-12,-4);(-12,12) **\dir{-};
     (4,-4);(4,-12) **\dir{-};(12,4);(12,-12) **\dir{-};
     (16,1)*{n};
     (-10,0)*{};(10,0)*{};
     (-4,-4)*{};(-12,-4)*{} **\crv{(-4,-10) & (-12,-10)}?(1)*\dir{<}?(0)*\dir{<};
      (4,4)*{};(12,4)*{} **\crv{(4,10) & (12,10)}?(1)*\dir{>}?(0)*\dir{>};
     }};
  \endxy \right)
  \quad \;= \quad
 \deg\left( \;  \xy 0;/r.19pc/:
  (0,0)*{\xybox{
    (-4,-4)*{};(4,4)*{} **\crv{(-4,-1) & (4,1)}?(1)*\dir{<};
    (4,-4)*{};(-4,4)*{} **\crv{(4,-1) & (-4,1)};
     (4,4);(4,12) **\dir{-};
     (12,-4);(12,12) **\dir{-};
     (-4,-4);(-4,-12) **\dir{-};(-12,4);(-12,-12) **\dir{-};
     (16,1)*{n};
     (10,0)*{};(-10,0)*{};
     (4,-4)*{};(12,-4)*{} **\crv{(4,-10) & (12,-10)}?(1)*\dir{>}?(0)*\dir{>};
      (-4,4)*{};(-12,4)*{} **\crv{(-4,10) & (-12,10)}?(1)*\dir{<}?(0)*\dir{<};
     }};
  \endxy \; \right) \quad = 0.
 \nn \end{eqnarray}

It will be convenient to introduce a special notation for these degree zero 2-morphisms:
\begin{equation} \label{eq_crossl-gen}
  \xy
  (0,0)*{\xybox{
    (-4,-4)*{};(4,4)*{} **\crv{(-4,-1) & (4,1)}?(1)*\dir{>} ;
    (4,-4)*{};(-4,4)*{} **\crv{(4,-1) & (-4,1)}?(0)*\dir{<};
     (8,2)*{ n};
     (-12,0)*{};(12,0)*{};
     }};
  \endxy
:=
 \xy 0;/r.19pc/:
  (0,0)*{\xybox{
    (4,-4)*{};(-4,4)*{} **\crv{(4,-1) & (-4,1)}?(1)*\dir{>};
    (-4,-4)*{};(4,4)*{} **\crv{(-4,-1) & (4,1)};
     (-4,4);(-4,12) **\dir{-};
     (-12,-4);(-12,12) **\dir{-};
     (4,-4);(4,-12) **\dir{-};(12,4);(12,-12) **\dir{-};
     (16,1)*{n};
     (-10,0)*{};(10,0)*{};
     (-4,-4)*{};(-12,-4)*{} **\crv{(-4,-10) & (-12,-10)}?(1)*\dir{<}?(0)*\dir{<};
      (4,4)*{};(12,4)*{} **\crv{(4,10) & (12,10)}?(1)*\dir{>}?(0)*\dir{>};
     }};
  \endxy
  \quad = \quad
  \xy 0;/r.19pc/:
  (0,0)*{\xybox{
    (-4,-4)*{};(4,4)*{} **\crv{(-4,-1) & (4,1)}?(1)*\dir{<};
    (4,-4)*{};(-4,4)*{} **\crv{(4,-1) & (-4,1)};
     (4,4);(4,12) **\dir{-};
     (12,-4);(12,12) **\dir{-};
     (-4,-4);(-4,-12) **\dir{-};(-12,4);(-12,-12) **\dir{-};
     (16,1)*{n};
     (10,0)*{};(-10,0)*{};
     (4,-4)*{};(12,-4)*{} **\crv{(4,-10) & (12,-10)}?(1)*\dir{>}?(0)*\dir{>};
      (-4,4)*{};(-12,4)*{} **\crv{(-4,10) & (-12,10)}?(1)*\dir{<}?(0)*\dir{<};
     }};
  \endxy
\end{equation}
\begin{equation} \label{eq_crossr-gen}
  \xy
  (0,0)*{\xybox{
    (-4,-4)*{};(4,4)*{} **\crv{(-4,-1) & (4,1)}?(0)*\dir{<} ;
    (4,-4)*{};(-4,4)*{} **\crv{(4,-1) & (-4,1)}?(1)*\dir{>};
     (-8,2)*{ n};
     (-12,0)*{};(12,0)*{};
     }};
  \endxy
:=
 \xy 0;/r.19pc/:
  (0,0)*{\xybox{
    (-4,-4)*{};(4,4)*{} **\crv{(-4,-1) & (4,1)}?(1)*\dir{>};
    (4,-4)*{};(-4,4)*{} **\crv{(4,-1) & (-4,1)};
     (4,4);(4,12) **\dir{-};
     (12,-4);(12,12) **\dir{-};
     (-4,-4);(-4,-12) **\dir{-};(-12,4);(-12,-12) **\dir{-};
     (-16,1)*{n};
     (10,0)*{};(-10,0)*{};
     (4,-4)*{};(12,-4)*{} **\crv{(4,-10) & (12,-10)}?(1)*\dir{<}?(0)*\dir{<};
      (-4,4)*{};(-12,4)*{} **\crv{(-4,10) & (-12,10)}?(1)*\dir{>}?(0)*\dir{>};
     }};
  \endxy
  \quad = \quad
  \xy 0;/r.19pc/:
  (0,0)*{\xybox{
    (4,-4)*{};(-4,4)*{} **\crv{(4,-1) & (-4,1)}?(1)*\dir{<};
    (-4,-4)*{};(4,4)*{} **\crv{(-4,-1) & (4,1)};
     (-4,4);(-4,12) **\dir{-};
     (-12,-4);(-12,12) **\dir{-};
     (4,-4);(4,-12) **\dir{-};(12,4);(12,-12) **\dir{-};
     (-16,1)*{n};
     (-10,0)*{};(10,0)*{};
     (-4,-4)*{};(-12,-4)*{} **\crv{(-4,-10) & (-12,-10)}?(1)*\dir{>}?(0)*\dir{>};
      (4,4)*{};(12,4)*{} **\crv{(4,10) & (12,10)}?(1)*\dir{<}?(0)*\dir{<};
     }};
  \endxy
\end{equation}

%
\subsubsection{The graded vector spaces $\HOM_{\Ucat}(\cal{E}\cal{F}\onen, \cal{E}\cal{F}\onen)$ and $\HOM_{\Ucat}(\cal{F}\cal{E}\onen, \cal{F}\cal{E}\onen)$}
%

From the generating 2-morphisms defined above, we can define two obvious
generators for the space of 2Homs from $\cal{E}\cal{F}\onen$ to itself.  Namely,
 the diagrams
\begin{equation}
\xy
 (3,7);(3,-7); **\dir{-} ?(.75)*\dir{<}+(2.3,0)*{\scriptstyle{}}
 ?(.1)*\dir{ }+(2,0)*{\scs };
  (-3,7);(-3,-7); **\dir{-} ?(.75)*\dir{>}+(2.3,0)*{\scriptstyle{}}
 ?(.1)*\dir{ }+(2,0)*{\scs };
 (9,4)*{n};(-9,4)*{n};
 (-14,0)*{};(12,0)*{};
 \endxy
\qquad \quad  \text{and} \qquad \quad\xy
    (-3,7)*{}="1";
    (3,7)*{}="2";
    "1";"2" **\crv{(-3,1) & (3,1)} ?(.1)*\dir{<} ?(.9)*\dir{<};
    (-3,-7)*{}="1";
    (3,-7)*{}="2";
    "1";"2" **\crv{(-3,-1) & (3,-1)} ?(.1)*\dir{>} ?(.9)*\dir{>};
    (9,4)*{n};
    \endxy
\end{equation}
together with arbitrarily many dots on each strand.  These 2-morphisms contribute
\begin{equation}
  \sum_{0\leq \alpha_1, 0\leq \alpha_2}
  q^{\deg\left(\xy
 (3,7);(3,-7); **\dir{-} ?(.75)*\dir{<}+(2.3,0)*{\scriptstyle{}}
 ?(.1)*\dir{ }+(2,0)*{\scs };
 (3,-2)*{\txt\large{$\bullet$}}; (6,-2)*{\scs \alpha_1};
  (-3,7);(-3,-7); **\dir{-} ?(.75)*\dir{>}+(2.3,0)*{\scriptstyle{}}
 ?(.1)*\dir{ }+(2,0)*{\scs };
 (-3,-2)*{\txt\large{$\bullet$}}; (-6,-2)*{\scs \alpha_2};
 (9,4)*{n};
 (-9,0)*{};(12,0)*{};
 \endxy \right)}
\;+\; \sum_{0\leq \beta_1, 0\leq \beta_2}
  q^{\deg\left(\xy
    (-3,7)*{}="1";
    (3,7)*{}="2";
    "1";"2" **\crv{(-3,1) & (3,1)} ?(.1)*\dir{<} ?(.9)*\dir{<};
    (0.3,2.4)*{\txt\large{$\bullet$}}; (-2,1)*{\scs \beta_1};
    (0.3,-2.6)*{\txt\large{$\bullet$}}; (3,-1)*{\scs \beta_2};
    (-3,-7)*{}="1";
    (3,-7)*{}="2";
    "1";"2" **\crv{(-3,-1) & (3,-1)} ?(.1)*\dir{>} ?(.9)*\dir{>};
    (9,4)*{n}; (-9,0)*{};
    \endxy \right)} = (1+q^{2(1-n)})\left(\frac{1}{1-q^{2}}\right)^2.
\end{equation}
to the graded dimension.  Comparing with the semilinear form, we see that there is perfect agreement $\sla
EF1_n,EF1_n \sra = (1+q^{2(1-n)})\left(\frac{1}{1-q^{2}}\right)^2$.  Hence, there
is no need to add any additional generators.  However,
\begin{equation}
\deg \left(  \;
 \vcenter{   \xy 0;/r.18pc/:
    (-4,-4)*{};(4,4)*{} **\crv{(-4,-1) & (4,1)}?(1)*\dir{>};
    (4,-4)*{};(-4,4)*{} **\crv{(4,-1) & (-4,1)}?(1)*\dir{<};?(0)*\dir{<};
    (-4,4)*{};(4,12)*{} **\crv{(-4,7) & (4,9)};
    (4,4)*{};(-4,12)*{} **\crv{(4,7) & (-4,9)}?(1)*\dir{>};
  (8,8)*{n};
 \endxy} \;
\right) = 0,
\end{equation}
so we will need an additional relation relating this diagram to those in the basis above (see Definition~\ref{defn-Ucatt}).

%
\subsubsection{Other graded Homs}
%

Continuing in this way it becomes clear that we can account for all the terms appearing in the semilinear form $\sla E_{\ep}1_n, E_{\ep'}1_n\sra$ using the generating 2-morphisms we have already added to the 2-category $\Ucat$.  However, we will still need to introduce some additional relations.

For example, one can check that
\begin{eqnarray}
 \sla E^31_n,E^31_n \sra &=& (1+2q^{-2}+2q^{-4}+q^{-6})\left(\frac{1}{1-q^2} \right)^3
  \nn\\
 &=&\left( q^{\deg\left(\;\;\vcenter{
 \xy 0;/r.15pc/:
    (-3,-4)*{};(-3,4)*{} **\dir{-}?(1)*\dir{>};
    (3,-4)*{};(3,4)*{} **\dir{-}?(1)*\dir{>};
    (9,-4)*{}; (9,4) **\dir{-}?(1)*\dir{>};
\endxy}\;\;\right)} \;\; +\;\;
  q^{\deg\left(\;\;\vcenter{
 \xy 0;/r.15pc/:
    (-4,-4)*{};(4,4)*{} **\crv{(-4,-1) & (4,1)}?(1)*\dir{>};
    (4,-4)*{};(-4,4)*{} **\crv{(4,-1) & (-4,1)}?(1)*\dir{>};
    (12,-4)*{}; (12,4) **\dir{-}?(1)*\dir{>};
\endxy}\;\;\right)} \;\; +\;\;
  q^{\deg\left(\;\;\vcenter{
 \xy 0;/r.15pc/:
    (-4,12)*{};(4,20)*{} **\crv{(-4,15) & (4,17)}?(1)*\dir{>};
    (4,12)*{};(-4,20)*{} **\crv{(4,15) & (-4,17)}?(1)*\dir{>};
    (12,12)*{}; (12,20) **\dir{-}?(1)*\dir{>};
\endxy}\;\;\right)}
 \right.\nn \\
 & &   \left.\;\; +\;\;
  q^{\deg\left(\;\;\vcenter{
 \xy 0;/r.15pc/:
    (4,4)*{};(12,12)*{} **\crv{(4,7) & (12,9)}?(1)*\dir{>};
    (12,4)*{};(4,12)*{} **\crv{(12,7) & (4,9)}?(1)*\dir{>};
    (-4,12)*{};(4,20)*{} **\crv{(-4,15) & (4,17)}?(1)*\dir{>};
    (4,12)*{};(-4,20)*{} **\crv{(4,15) & (-4,17)}?(1)*\dir{>};
    (-4,4)*{}; (-4,12) **\dir{-};
    (12,12)*{}; (12,20) **\dir{-};
\endxy}\;\;\right)}  \;\; +\;\;
  q^{\deg\left(\;\;\vcenter{
 \xy 0;/r.15pc/:
    (-4,-4)*{};(4,4)*{} **\crv{(-4,-1) & (4,1)}?(1)*\dir{>};
    (4,-4)*{};(-4,4)*{} **\crv{(4,-1) & (-4,1)}?(1)*\dir{>};
    (4,4)*{};(12,12)*{} **\crv{(4,7) & (12,9)}?(1)*\dir{>};
    (12,4)*{};(4,12)*{} **\crv{(12,7) & (4,9)}?(1)*\dir{>};
    (-4,4)*{}; (-4,12) **\dir{-};
    (12,-4)*{}; (12,4) **\dir{-};
\endxy}\;\;\right)} \;\; +\;\;
  q^{\deg\left(\;\;\vcenter{
 \xy 0;/r.15pc/:
    (-4,-4)*{};(4,4)*{} **\crv{(-4,-1) & (4,1)}?(1)*\dir{>};
    (4,-4)*{};(-4,4)*{} **\crv{(4,-1) & (-4,1)}?(1)*\dir{>};
    (4,4)*{};(12,12)*{} **\crv{(4,7) & (12,9)}?(1)*\dir{>};
    (12,4)*{};(4,12)*{} **\crv{(12,7) & (4,9)}?(1)*\dir{>};
    (-4,12)*{};(4,20)*{} **\crv{(-4,15) & (4,17)}?(1)*\dir{>};
    (4,12)*{};(-4,20)*{} **\crv{(4,15) & (-4,17)}?(1)*\dir{>};
    (-4,4)*{}; (-4,12) **\dir{-};
    (12,-4)*{}; (12,4) **\dir{-};
    (12,12)*{}; (12,20) **\dir{-};
\endxy}\;\;\right)} \right) \left(\frac{1}{1-q^2}\right)^3\nn
\end{eqnarray}
so that all generators can be accounted for by crossings and dots on three strands. But we must impose a relation of the form
\begin{equation}
  \vcenter{
 \xy 0;/r.18pc/:
    (-4,-4)*{};(4,4)*{} **\crv{(-4,-1) & (4,1)}?(1)*\dir{>};
    (4,-4)*{};(-4,4)*{} **\crv{(4,-1) & (-4,1)}?(1)*\dir{>};
    (4,4)*{};(12,12)*{} **\crv{(4,7) & (12,9)}?(1)*\dir{>};
    (12,4)*{};(4,12)*{} **\crv{(12,7) & (4,9)}?(1)*\dir{>};
    (-4,12)*{};(4,20)*{} **\crv{(-4,15) & (4,17)}?(1)*\dir{>};
    (4,12)*{};(-4,20)*{} **\crv{(4,15) & (-4,17)}?(1)*\dir{>};
    (-4,4)*{}; (-4,12) **\dir{-};
    (12,-4)*{}; (12,4) **\dir{-};
    (12,12)*{}; (12,20) **\dir{-};
  (18,8)*{n};
\endxy}
 \;\; =\;\;
 \vcenter{
 \xy 0;/r.18pc/:
    (4,-4)*{};(-4,4)*{} **\crv{(4,-1) & (-4,1)}?(1)*\dir{>};
    (-4,-4)*{};(4,4)*{} **\crv{(-4,-1) & (4,1)}?(1)*\dir{>};
    (-4,4)*{};(-12,12)*{} **\crv{(-4,7) & (-12,9)}?(1)*\dir{>};
    (-12,4)*{};(-4,12)*{} **\crv{(-12,7) & (-4,9)}?(1)*\dir{>};
    (4,12)*{};(-4,20)*{} **\crv{(4,15) & (-4,17)}?(1)*\dir{>};
    (-4,12)*{};(4,20)*{} **\crv{(-4,15) & (4,17)}?(1)*\dir{>};
    (4,4)*{}; (4,12) **\dir{-};
    (-12,-4)*{}; (-12,4) **\dir{-};
    (-12,12)*{}; (-12,20) **\dir{-};
  (10,8)*{n};
\endxy}
\end{equation}
where again the exact form of this relation was determined by consistency of the graphical calculus and the action on the cohomology rings of partial flag varieties defined in the next section.

To see that we do not need any additional generating 2-morphisms we could either show that with the appropriate relations the indecomposable 1-morphisms of our 2-category bijectively correspond up to a shift with Lusztig canonical basis elements as was done in \cite{Lau1}, or we could give a purely diagrammatic interpretation of the semilinear form as in \cite[Section 2.2]{KL3} and  argue that our generators can account for all the diagrams appearing in this formula.  In either case, the result is that we will not require any more generating 2-morphisms in the 2-category $\Ucat$.

What remains is to find the exact form of the relations in $\Ucat$ and show that we can lift the $\mathfrak{sl}_2$ relations to explicit isomorphisms in the 2-category $\Ucat$.

\begin{rem} \label{rem_other_inv}
It is natural to wonder how general this approach to categorification is and whether or not it can be applied to categorify other algebras of interest.  The use of a semilinear form for categorification seems to be most useful in situations where an algebra already has a known geometric interpretation where a semilinear form arises as the graded dimension of some graded Ext algebras of sheaves. Such methods were used in \cite{EK,VV2,SVV,Web}.

This is not to say that categorification is impossible without the use of a semilinear forms.  There has been a great deal of success categorifying other algebras by different methods. The positive half of the quantum super algebra $\mathfrak{gl}(1\mid 2)$ was categorified by Khovanov~\cite{Kh3} using inspiration from Heegaard-Floer theory and the Lipshitz-Ozsv\'{a}th-Thurston dg algebra~\cite{LOT2,LOT}.
Khovanov also defined a conjectural categorification of the Heisenberg algebra by guessing diagrammatic relations directly from a presentation of the algebra by generators and relations~\cite{Kh4}, see also a related construction~\cite{LS}.  The categorification of a different version of the Heisenberg algebra in \cite{CL} uses a hybrid of geometric methods and diagrammatic lifting of defining relations. A diagrammatic approach to categorification of the polynomial ring was given in \cite{KR} and the q-Schur algebra in~\cite{MSV}.
\end{rem}

%
\subsection{An actions on partial flag varieties} \label{subsec_flag}
%

To find the precise form of the relations in $\Ucat$ we can look at examples of categorical $\U$ actions and look for natural transformations of functors corresponding to our generating 2-morphisms.  The relations that hold in these examples give insights into the relations that should hold in $\Ucat$.

%
\subsubsection{A model example of a categorical $\U$-action}
%

There is a very nice example of a categorification of the irreducible  representation $V^N$ of $\U$ that utilizes some geometry, while at the same time it is completely algebraic and computable. This example was first considered in unpublished work of Khovanov and has since appeared several places in the literature~\cite{CR,FKS}. This categorification of the irreducible $N+1$-dimensional representation $V^N$ of $\U$ is constructed using  categories of graded modules over cohomology rings of Grassmannians.  For $0 \leq k \leq N$, let $Gr(k,N)$ denote the variety of complex k-planes in $\C^N$. The cohomology ring of $Gr(k,N)$ has a natural structure of a graded $\Q$-algebra,
\[
 H^*(Gr(k,N),\Q) = \oplus_{0 \leq i \leq k(N-k)}H^i(G_k,\Q) ~.
\]
For simplicity we sometimes write $H_k:=H^*(Gr(k,N),\Q)$ omitting the explicit $N$ dependence.

An explicit description of $H_k$ can be given using Chern classes.  This description is reviewed in~\cite[Section 6]{Lau1} and \cite{Hiller}.  In particular, the graded ring $H_k$ is given by a quotient
\begin{equation}
H_k = \Q[c_1,\ldots,c_k,\bar{c}_1,\ldots,\bar{c}_{N-k}]/I_{k,N}
\end{equation}
where $\deg c_{i} = 2i$, $\deg{\bar{c}_j}=2j$, and $I_{k,N}$ is the ideal, referred to here as the Grassmannian ideal, generated by equating the terms in the equation
\begin{equation} \label{eq_chern_rel}
 \left(1 +c_1t + c_2 t^2\cdots +c_k t^k\right)
  \left(1+\bar{c}_1 t +\bar{c}_2 t^2 +\cdots + \bar{c}_{N-k}t^{N-k}\right) =1
\end{equation}
that are homogeneous in $t$.

\begin{example} \label{ex_cohomologyring}
Let $N=5$ and consider the ring $H_2=H^*(Gr(2,5))$.  The ring $H_2$ is the quotient of the polynomial ring $\Q[c_1,c_2,\bar{c}_1,\bar{c}_2,\bar{c}_3]$ by the relations given by the homogeneous terms in
\begin{equation}
  (1+c_1t + c_2 t^2)(1+\bar{c}_1 t + \bar{c}_2 t^2+ \bar{c}_3 t^3) =1.
\end{equation}
In particular, the equations
\begin{align}
  c_1+\bar{c}_1=0, \quad
  c_2 + c_1 \bar{c}_1 + \bar{c}_2 = 0, \quad
  c_2\bar{c}_1+c_1\bar{c}_2+\bar{c}_3 =0, \quad
  c_2\bar{c}_2+c_1\bar{c}_3=0, \quad
  c_2 \bar{c}_3 =0,
\end{align}
generate the ideal $I_{2,5}$.
Solving these equations we find that $\bar{c}_1=-c_1$, $\bar{c}_2=c_1^2-c_2$ and $\bar{c}_3=2c_1c_2-c_1^3$. The remaining relations on $c_1$ and $c_2$ are given by solving the remaining equations so that
\begin{equation}
 H_3=\Q[c_1,c_2]/(c_1^4-3c_1^2c_2+c_2^2,\; 2c_1c_2^2-c_1^3c_2).
\end{equation}
\end{example}

In general, using the relations coming from \eqref{eq_chern_rel} we can solve for the variables $\bar{c}_j$ in terms of the variables $c_i$ with $1\leq i \leq k$ when $k\leq N-k$, and solve for the variables $c_i$ in terms of the variables $\bar{c}_j$ with $1\leq j\leq N-k$ when
$N-k \geq k$.  Without loss of generality assume $k\leq N-k$ and write all the
$\bar{c}_{j}$ in terms of $c_{i}$ using the relations from \eqref{eq_chern_rel}. Using elementary linear algebra it can be shown that the remaining $k$ relations on these generators of $H_k$ are given by the first column of the matrix product
\begin{equation} \label{eq_grass_rels}
  \left(\begin{array}{ccccc}
    c_{1} & 1 & 0 & 0 & { 0} \\
    -c_{2} & 0 & 1 & \ddots & 0 \\
    c_{3} & 0 & \ddots & \ddots & 0 \\
    \vdots & \vdots & \ddots&  & 1 \\
   (-1)^{k+1}c_{k} & 0 &  &  &  0\\
  \end{array}
\right)^{N-k+1}
\end{equation}
(see for example \cite[page 107]{Hiller}).

\begin{example}
Taking $N=5$ and $k=2$, the ring $H_2=H^*(Gr(2,5))$ is the quotient of the polynomial ring $\Q[c_1,c_2]$ by the ideal generated by the terms in the first column of the matrix product
\begin{equation}
  \left(
    \begin{array}{ccc}
      c_1 & 1  \\
      -c_2 & 0  \\
    \end{array}
  \right)^{4} =
\left(
    \begin{array}{ccc}
      c_1^4-3c_1^2c_2+c_2^2 & c_1^3-2c_1c_2  \\
      -c_1^3c_2+2c_1c_2^2 & -c_1^2c_2+c_2^2  \\
    \end{array} \right),
\end{equation}
which agrees with Example~\ref{ex_cohomologyring}.
\end{example}

To categorify finite dimensional irreducible representations $V^N$ of $\U$ we need to identify some graded additive categories $\cal{V}_n^N$ whose split Grothendieck groups are 1-dimensional for $0 \leq k \leq N$ and $n=2k-N$.  Let $H_k\pmod$ denote the category of graded finitely generated projective $H_k$-modules.  Set $\cal{V}_n^N = H_k\pmod$.  The  rings $H_k$ being graded local rings implies that their split Grothendieck group is are free $\Z[q,q^{-1}]$-modules generated by a unique indecomposable projective module.  Hence,
\begin{equation}
  K_0(\cal{V}_n^N) \otimes_{\Z[q,q^{-1}]} \Q(q) = \Q(q)
\end{equation}
so that the category $\cal{V}^N=\bigoplus_{k=0}^N \cal{V}_{n}^{N}$ (with $n=2k-N$) categorifies the irreducible representation $V^N$ in the sense that
\begin{equation}
 K_0(\cal{V}^N) := \bigoplus_{k=0}^N K_0(\cal{V}_{n}^{N}) \otimes_{\Z[q,q^{-1}]} \Q(q) \cong V^N
\end{equation}
as $\Q(q)$-vector spaces.

The action of $E1_n$ and $F1_n$ in $\U$ can also be interpreted in this geometric setting.  Consider the one step flag variety
\[
 Fl(k,k+1,N) = \left\{ (W_k,W_{k+1}) |
 \dim_{\C} W_k = k, \; \dim_{\C} W_{k+1} =(k+1), \; 0
 \subset W_k \subset W_{k+1} \subset \C^N  \right\}.
\]
We write $H_{k,k+1}:=H^*(Fl(k,k+1,N)$ for the cohomology ring of this variety. Again, this ring is simple to describe explicitly, see \eqref{eq_Hkkp1} below.   This variety has natural forgetful maps
\[
 \xy
  (0,20)*++{Fl(k,k+1,N)}="m"; (0,15)*{\scs\{ 0 \subset \C^k \subset \C^{k+1} \subset \C^N\}};
  (-35,14)*+{Gr(k,N)}="l";(35,14)*{Gr(k+1,N)}="r";
  (-35,8)*{\scs\{ 0 \subset \C^k  \subset \C^N\}};
  (35,8)*{\scs\{ 0  \subset \C^{k+1} \subset \C^N\}};
   {\ar@<-1ex> "m";"l"};{\ar@<1ex> "m";"r"};
 \endxy
\]
which give rise to inclusions
\[
\xy
 (0,0)*++{H_{k,k+1}}="m";(-35,-10)*+{H_k:=H^*(Gr(k,N))}="l";
 (35,-10)*+{H_{k+1}:=H^*(Gr(k+1,N))}="r";
  {\ar"l";"m"};{\ar "r";"m"};
\endxy
\]
on cohomology.   These inclusions make $H_{k,k+1}$ an $(H_{k+1},H_k)$-bimodule.  Since these rings are commutative we can also think of $H_{k,k+1}$ as an $(H_{k},H_{k+1})$-bimodule which we will denote by $H_{k+1,k}$. We get functors between categories of modules by tensoring with a bimodule.  We compose these functors by tensoring the corresponding bimodules.

The action of $E1_n$ and $1_nF$ for $n=2k-N$ is given by tensoring with $H_{k+1,k}$ and $H_{k,k+1}$, respectively. More precisely, we have functors
\begin{eqnarray}
 \onen &:=& H_k \otimes_{H_k} (-) \maps H_{k} {\rm -pmod} \to H_{k}{\rm -pmod} \nn \\
 \cal{E}\onen &:=& H_{k+1,k} \otimes_{H_k} (-) \{1-N+k\} \maps
 H_{k}{\rm -pmod} \to H_{k+1}{\rm -pmod}  \nn\\
 \cal{F}\onenn{n+2} &:=& H_{k,k+1} \otimes_{H_{k+1}} (-) \{-k\}\maps
H_{k+1} {\rm -pmod} \to H_{k} {\rm -pmod}. \nn
\end{eqnarray}
The grading shifts in the definition of $\cal{E}\onen$ and $\cal{F}\onen$ are necessary to ensure that these functors satisfy the quantum $\mathfrak{sl}_2$-relations
\begin{eqnarray}
  \cal{E}\cal{F}\onen \cong \cal{F}\cal{E}\onen \oplus \onen^{\oplus_{[n]}}  & \qquad & \text{for $n \geq 0$}, \nn\\
  \cal{F}\cal{E}\onen  \cong \cal{E}\cal{F}\onen\oplus\onen^{\oplus_{[-n]}} & \qquad & \text{for $n \leq 0$,} \nn
\end{eqnarray}
where we recall that
\begin{eqnarray}
  \onen^{\oplus_{[n]}}\; :=\; \onen\{n-1\} \oplus \onen\{n-3\} \oplus \cdots \oplus
  \onen\{1-n\}. \nn
\end{eqnarray}

Figure~\ref{fig:irreps} summarizes this categorification.
One can show that these functors have both left and right adjoints and commute with the grading shift functor on graded modules. This implies that they induce maps on the split Grothendieck rings giving the actions of $1_n$, $E1_n$, and $F1_n$.

\begin{figure}[h]
\[
 \xy
  (-55,0)*+{\cal{V}^N_{-N}}="1";
  (-34,0)*++{\;}="12";
  (-20,0)*+{\cal{V}^N_{n-2}}="2";
  (34,0)*++{\;}="23";
  (0,0)*+{\cal{V}^N_{n}}="3";
  (20,0)*+{\cal{V}^N_{n+2}}="4";
  (55,0)*+{\cal{V}^N_{N}}="5";
    {\ar@/^0.7pc/^{\cal{E}\onenn{-N}} "1";"12"};
    {\ar@/^0.7pc/^{\cal{E}\onenn{n-2}} "2";"3"};
    {\ar@/^0.7pc/^{\cal{E}\onenn{n}} "3";"4"};
    {\ar@/^0.7pc/^{\cal{E}\onenn{N}} "23"; "5"};
    {\ar@/^0.7pc/^{\cal{F}\onenn{N}} "12"; "1"};
    {\ar@/^0.7pc/^{\cal{F}\onenn{n}} "3";"2"};
    {\ar@/^0.7pc/^{\cal{F}\onenn{n+2}} "4";"3"};
    {\ar@/^0.7pc/^{\cal{F}\onenn{N}} "5";"23"};
  (27,0)*{\cdots};
  (-27,0)*{\cdots};
 \endxy
\]
\[
 \xy
  (-55,0)*+{H_0}="1";
  (-34,0)*++{\;}="12";
  (-20,0)*+{H_{k-1}}="2";
  (34,0)*++{\;}="23";
  (0,0)*+{H_k}="3";
  (20,0)*+{H_{k+1}}="4";
  (55,0)*+{H_N}="5";
    {\ar@/^0.7pc/^{H_{1,0}} "1";"12"};
    {\ar@/^0.7pc/^{H_{k,k-1}} "2";"3"};
    {\ar@/^0.7pc/^{H_{k+1,k}} "3";"4"};
    {\ar@/^0.7pc/^{H_{N,N-1}} "23"; "5"};
    {\ar@/^0.7pc/^{H_{0,1}} "12"; "1"};
    {\ar@/^0.7pc/^{H_{k-1,k}} "3";"2"};
    {\ar@/^0.7pc/^{H_{k,k+1}} "4";"3"};
    {\ar@/^0.7pc/^{H_{N-1,N}} "5";"23"};
  (27,0)*{\cdots};
  (-27,0)*{\cdots};
 \endxy
\]
\caption{The top diagram illustrates the weight space decomposition of categories and functors appearing in the categorification of $V^N$.  The lower diagram illustrates the graded rings whose module categories give weight space categories and the bimodules giving rise to functors on the module categories.} \label{fig:irreps}
\end{figure}

%
\subsubsection{Natural transformations}
%

The rings $H_{k,k+1}$ can also be explicitly defined as quotients of a graded polynomial ring:
\begin{equation} \label{eq_Hkkp1}
  H_{k,k+1}:= \Q[c_1,c_2, \ldots c_k;\xi;\bar{c}_1,\bar{c}_2,\ldots,
 \bar{c}_{N-k-1}]/ I_{k,k+1,N},
\end{equation}
where $I_{k,k+1,N}$ is the ideal generated by equating the homogeneous terms in the equation
\[
 \left(1+c_{1} +c_{2}t^2+ \ldots +c_{k}t^k \right)\left(1+\xi t
 \right)
  \left(1+ \bar{c}_{1}t+\bar{c}_{2}t^2+\ldots+
 \bar{c}_{N-k-1}t^{N-k-1}\right) \quad = \quad 1.
\]
Here the generator $\xi$ has degree 2 and corresponds to the Chern class of the natural line bundle over $Fl(k,k+1,N)$ whose fibre over a point $0 \subset W_k \subset W_{k+1} \subset \C^N$ in $Fl(k,k+1,N)$ is the line $W_{k+1}/W_k$.

This presentation of $H_{k,k+1}$ makes it easy to explicitly construct bimodule homomorphisms between the various cohomology rings and determine relations between them. In particular, there is a natural degree two bimodule map
\begin{equation}
  \Uup_{n} \maps H_{k+1,k} \otimes_{H_k} - \To H_{k+1,k} \otimes_{H_{k}} -
\end{equation}
given by multiplication by the generator $\xi$ corresponding to the first Chern class of the line bundle $W^{k+1}/W^k$.

The crossing 2-morphism $\Ucross_{n}$ also turns out to have a natural geometric interpretation.  Note that the bimodule corresponding to the action of $\cal{E}\cal{E}\onen$ is $H_{k+2,k+1} \otimes_{H_{k+1}} H_{k+1,k}$ which can be shown to be isomorphic to the cohomology ring
\begin{equation}
  H_{k,k+1,k+2} := H^*( Fl(k,k+1,k+2,N)) \nn
\end{equation}
where $ Fl(k,k+1,k+2,N)$ is the iterated partial flag variety
\[
 Fl(k,k+1,k+2,N) = \left\{ (W_k,W_{k+1},W_{k+2}) |
 \dim_{\C} W_j = j, \; \; 0
 \subset W_k \subset W_{k+1} \subset W_{k+2} \subset \C^N  \right\}.
\]
The natural forgetful map from $Fl(k,k+1,k+2,N)$ to
\[
 Fl(k,k+2,N) = \left\{ (W_k,W_{k+2}) |
 \dim_{\C} W_j = j, \; \; 0
 \subset W_k  \subset W_{k+2} \subset \C^N  \right\},
\]
given by forgetting the subspace $W_{k+1}$ of dimension $k+1$ is a proper map, so that it gives rise to a degree -2 map
\begin{equation}
  \Ucross_{n} \maps H_{k,k+1,k+2}\otimes_{H_k}- \To H_{k,k+1,k+2}\otimes_{H_k}-
\end{equation}
given by pushing forward a class to $H_{k,k+2}$ and pulling back to $H_{k,k+1,k+2}$. But this action can be expressed more explicitly on generators of the ring $H_{k,k+1,k+2}$.   In particular,
\begin{equation}
  H_{k,k+1,k+2} = \Z[c_1,\dots,c_k, \xi_1,\xi_2, \bar{c}_1,\dots,\bar{c}_{N-k-2}]/I_{k,k+1,k+2} \nn
\end{equation}
where $I_{k,k+1,k+2}$ is the ideal generated by the homogeneous terms in the equation
\begin{equation}
  (1+x_1 t +x_2 t^2 + \dots + x_k t^k) (1+\xi_1 t) (1+\xi_2) (1+ y_1 t + \dots + y_{N-k-2}t^{N-k-2}). \nn
\end{equation}
The degree two generators $\xi_1$ and $\xi_2$ arise from the Chern classes of the line bundles $W_{k+1}/W_k$ and $W_{k+2}/W_{k+1}$, respectively.

The action of  $\Ucross_{n}$ is given by fixing generators $c_i$ and $\bar{c}_j$ for $0 \leq i \leq k$ and $0 \leq j \leq N-k-2$ and acting by the divided difference operator
\[
 \partial_1 (f) :=  \frac{f-s_1f}{\xi_1-\xi_{2}}
\]
for any polynomial $f \in \Z[\xi_1,\xi_2]$ with the action of the symmetric group generator $s_1 \in S_2$ given by permuting the variables $\xi_1$ and $\xi_2$.

More generally, the composite $\cal{E}^a\onen$ acts on the category $H_{k}\pmod$ by the functor of tensoring with the $(H_{k+a},H_k)$-bimodule
\begin{equation}
H_{k+a,k+a-1, \dots, k+1,k} := H^*(Fl(k,k+1,\dots,k+a,N)) \nn
\end{equation}
where $Fl(k,k+1,\dots,k+a,N)$ is the $a$-step iterated partial flag variety. One can show that
\[
 H_{k+a,k+a-1, \dots, k+1,k} \cong \Z[x_1,\dots, x_k,\xi_1,\dots,\xi_a,y_1,\dots, y_{N-k-a}] /I_{k,k+1,\dots,k+a}
\]
with $I_{k,k+1,\dots,k+a}$ defined analogously as $I_{k,k+1,k+2}$.   There are divided difference operators $\partial_i$ for $1 \leq i \leq a-1$ acting on $f\in \Z[\xi_1, \dots, \xi_a]$ given by
\begin{equation}
 \partial_i := \frac{1-s_i}{\xi_i-\xi_{i+1}}.
\end{equation}
From this definition one can show that both the image and kernel of
the operator $\partial_i$ consists of polynomials which
are symmetric in $\xi_i$ and $\xi_{i+1}$.  Hence, these operators satisfy the relations $\partial_i^2=0$.

The divided difference operators are a natural choice for the action of $\Ucross_n$ since they are degree $-2$ and satisfy the relation $\partial_i^2=0$, see \eqref{eq_semiEEm4}.  The algebra of endomorphisms of $\Z[\xi_1,\dots,\xi_a]$ generated by multiplication by variables $\xi_i$ and the action of divided difference operators $\partial_i$ is called the nilHecke algebra $\BNC_a$. The nilHecke algebra has a well documented connection to cohomology rings of flag varieties~\cite{BGG,Dem} and is further related to the theory
of Schubert varieties, see ~\cite{KK,BilLak,Kum,Man}.

This makes the axioms for the nilHecke algebra a natural choice for resolving the relations on upward pointing strands.  The algebra $\BNC_a$ has generators $\xi_j$ for $1 \leq j \leq a$ and $\partial_i$ with $1 \leq j \leq a-1$ with relations
\[
 \begin{array}{ll}
 \xi_i \xi_j =   \xi_j \xi_i , &  \\
   \partial_i \xi_j = \xi_j\partial_i \quad \text{if $|i-j|>1$}, &
   \partial_i\partial_j = \partial_j\partial_i \quad \text{if $|i-j|>1$}, \\
  \partial_i^2 = 0,  &
   \partial_i\partial_{i+1}\partial_i = \partial_{i+1}\partial_i\partial_{i+1},  \\
   \xi_i \partial_i - \partial_i \xi_{i+1}=1,  &   \partial_i \xi_i - \xi_{i+1} \partial_i =1.
  \end{array}
\]
If we reinterpret these relations in the graphic calculus for $\Ucat$, the first two lines amount to isotopies of diagrams
\begin{equation}
    \dots \;\; \xy
    (0,-8)*{};(0,8)*{} **\dir{-}?(1)*\dir{>}?(.25)*{\bullet};
  \endxy \;\; \dots \;\;
  \xy
    (0,-8)*{};(0,8)*{} **\dir{-}?(1)*\dir{>}?(.7)*{\bullet};
  \endxy \;\; \dots =
      \dots \;\; \xy
    (0,-8)*{};(0,8)*{} **\dir{-}?(1)*\dir{>}?(.7)*{\bullet};
  \endxy \;\; \dots \;\;
  \xy
    (0,-8)*{};(0,8)*{} **\dir{-}?(1)*\dir{>}?(.25)*{\bullet};
  \endxy \;\; \dots
\end{equation}
\begin{equation}
      \dots \;\; \xy
    (-4,-8)*{};(4,0)*{} **\crv{(-4,-5) & (4,-3)};
    (4,-8)*{};(-4,0)*{} **\crv{(4,-5) & (-4,-3)};
    (-4,0)*{}; (-4,8)*{} **\dir{-}?(1)*\dir{>};
    (4,0)*{}; (4,8)*{} **\dir{-}?(1)*\dir{>};
  \endxy \;\; \dots \;\;
  \xy
    (0,-8)*{};(0,8)*{} **\dir{-}?(1)*\dir{>}?(.7)*{\bullet};
  \endxy \;\; \dots =
      \dots \;\; \xy
    (-4,0)*{};(4,8)*{} **\crv{(-4,3) & (4,5)}?(1)*\dir{>};
    (4,0)*{};(-4,8)*{} **\crv{(4,3) & (-4,5)}?(1)*\dir{>};
    (-4,-8)*{}; (-4,0)*{} **\dir{-};
    (4,-8)*{}; (4,0)*{} **\dir{-};
  \endxy \;\; \dots \;\;
  \xy
    (0,-8)*{};(0,8)*{} **\dir{-}?(1)*\dir{>}?(.25)*{\bullet};
  \endxy \;\; \dots
\end{equation}
while the second two lines imply
\begin{equation}
  \vcenter{\xy 0;/r.19pc/:
    (-4,-4)*{};(4,4)*{} **\crv{(-4,-1) & (4,1)}?(1)*\dir{>};
    (4,-4)*{};(-4,4)*{} **\crv{(4,-1) & (-4,1)}?(1)*\dir{>};
    (-4,4)*{};(4,12)*{} **\crv{(-4,7) & (4,9)}?(1)*\dir{>};
    (4,4)*{};(-4,12)*{} **\crv{(4,7) & (-4,9)}?(1)*\dir{>};
 \endxy}
 \;\;=\;\;0, \qquad \quad
 \vcenter{
 \xy 0;/r.18pc/:
    (-4,-4)*{};(4,4)*{} **\crv{(-4,-1) & (4,1)}?(1)*\dir{>};
    (4,-4)*{};(-4,4)*{} **\crv{(4,-1) & (-4,1)}?(1)*\dir{>};
    (4,4)*{};(12,12)*{} **\crv{(4,7) & (12,9)}?(1)*\dir{>};
    (12,4)*{};(4,12)*{} **\crv{(12,7) & (4,9)}?(1)*\dir{>};
    (-4,12)*{};(4,20)*{} **\crv{(-4,15) & (4,17)}?(1)*\dir{>};
    (4,12)*{};(-4,20)*{} **\crv{(4,15) & (-4,17)}?(1)*\dir{>};
    (-4,4)*{}; (-4,12) **\dir{-};
    (12,-4)*{}; (12,4) **\dir{-};
    (12,12)*{}; (12,20) **\dir{-};
\endxy}
 \;\; =\;\;
 \vcenter{
 \xy 0;/r.18pc/:
    (4,-4)*{};(-4,4)*{} **\crv{(4,-1) & (-4,1)}?(1)*\dir{>};
    (-4,-4)*{};(4,4)*{} **\crv{(-4,-1) & (4,1)}?(1)*\dir{>};
    (-4,4)*{};(-12,12)*{} **\crv{(-4,7) & (-12,9)}?(1)*\dir{>};
    (-12,4)*{};(-4,12)*{} **\crv{(-12,7) & (-4,9)}?(1)*\dir{>};
    (4,12)*{};(-4,20)*{} **\crv{(4,15) & (-4,17)}?(1)*\dir{>};
    (-4,12)*{};(4,20)*{} **\crv{(-4,15) & (4,17)}?(1)*\dir{>};
    (4,4)*{}; (4,12) **\dir{-};
    (-12,-4)*{}; (-12,4) **\dir{-};
    (-12,12)*{}; (-12,20) **\dir{-};
\endxy} \label{eq_nil_rels_grass}
  \end{equation}
\begin{eqnarray}
  \xy
  (4,4);(4,-4) **\dir{-}?(0)*\dir{<}+(2.3,0)*{};
  (-4,4);(-4,-4) **\dir{-}?(0)*\dir{<}+(2.3,0)*{};
  (9,2)*{n};
 \endxy
 \quad =
\xy
  (0,0)*{\xybox{
    (-4,-4)*{};(4,4)*{} **\crv{(-4,-1) & (4,1)}?(1)*\dir{>}?(.25)*{\bullet};
    (4,-4)*{};(-4,4)*{} **\crv{(4,-1) & (-4,1)}?(1)*\dir{>};
     (8,1)*{ n};
     (-10,0)*{};(10,0)*{};
     }};
  \endxy
 \;\; -
 \xy
  (0,0)*{\xybox{
    (-4,-4)*{};(4,4)*{} **\crv{(-4,-1) & (4,1)}?(1)*\dir{>}?(.75)*{\bullet};
    (4,-4)*{};(-4,4)*{} **\crv{(4,-1) & (-4,1)}?(1)*\dir{>};
     (8,1)*{ n};
     (-10,0)*{};(10,0)*{};
     }};
  \endxy
 \;\; =
\xy
  (0,0)*{\xybox{
    (-4,-4)*{};(4,4)*{} **\crv{(-4,-1) & (4,1)}?(1)*\dir{>};
    (4,-4)*{};(-4,4)*{} **\crv{(4,-1) & (-4,1)}?(1)*\dir{>}?(.75)*{\bullet};
     (8,1)*{ n};
     (-10,0)*{};(10,0)*{};
     }};
  \endxy
 \;\; -
  \xy
  (0,0)*{\xybox{
    (-4,-4)*{};(4,4)*{} **\crv{(-4,-1) & (4,1)}?(1)*\dir{>} ;
    (4,-4)*{};(-4,4)*{} **\crv{(4,-1) & (-4,1)}?(1)*\dir{>}?(.25)*{\bullet};
     (8,1)*{ n};
     (-10,0)*{};(10,0)*{};
     }};
  \endxy \nn \\ \label{eq_nil_dotslide_grass}
\end{eqnarray}
By repeatedly applying \eqref{eq_nil_dotslide_grass} one can show that the equation
\begin{equation} \label{eq_ind_dotslide}
\xy
  (0,0)*{\xybox{
    (-4,-4)*{};(4,4)*{} **\crv{(-4,-1) & (4,1)}?(1)*\dir{>}?(.25)*{\bullet}+(-2.5,1)*{\alpha};
    (4,-4)*{};(-4,4)*{} **\crv{(4,-1) & (-4,1)}?(1)*\dir{>};
     (8,-4)*{n};
     (-10,0)*{};(10,0)*{};
     }};
  \endxy
 \;\; -
 \xy
  (0,0)*{\xybox{
    (-4,-4)*{};(4,4)*{} **\crv{(-4,-1) & (4,1)}?(1)*\dir{>}?(.75)*{\bullet}+(2.5,-1)*{\alpha};
    (4,-4)*{};(-4,4)*{} **\crv{(4,-1) & (-4,1)}?(1)*\dir{>};
     (8,-4)*{n};
     (-10,0)*{};(10,0)*{};
     }};
  \endxy
 \;\; =
\xy
  (0,0)*{\xybox{
    (-4,-4)*{};(4,4)*{} **\crv{(-4,-1) & (4,1)}?(1)*\dir{>};
    (4,-4)*{};(-4,4)*{} **\crv{(4,-1) & (-4,1)}?(1)*\dir{>}?(.75)*{\bullet}+(-2.5,-1)*{\alpha};
     (8,3)*{ n};
     (-10,0)*{};(10,0)*{};
     }};
  \endxy
 \;\; -
  \xy
  (0,0)*{\xybox{
    (-4,-4)*{};(4,4)*{} **\crv{(-4,-1) & (4,1)}?(1)*\dir{>} ;
    (4,-4)*{};(-4,4)*{} **\crv{(4,-1) & (-4,1)}?(1)*\dir{>}?(.25)*{\bullet}+(2.5,1)*{\alpha};
     (8,3)*{n};
     (-10,0)*{};(10,0)*{};
     }};
  \endxy
  \;\; = \;\;
  \sum_{f_1 + f_2 = \alpha-1}
  \xy
  (3,4);(3,-4) **\dir{-}?(0)*\dir{<} ?(.5)*\dir{}+(0,0)*{\bullet}+(2.5,1)*{\scs f_2};
  (-3,4);(-3,-4) **\dir{-}?(0)*\dir{<}?(.5)*\dir{}+(0,0)*{\bullet}+(-2.5,1)*{\scs f_1};;
  (9,-4)*{n};
 \endxy
\end{equation}
holds.

In this way we see that requiring an action of $\Ucat$ on the cohomology rings of iterated partial flag varieties clarifies the precise form of the relations that should hold on upward oriented strands in $\Ucat$. Using the adjoint structure we get similar relations on downward oriented strands. Bimodule homomorphisms of the appropriate degree can also be found for the cap and cup 2-morphisms in $\Ucat$~\cite[Section 7]{Lau1}.  These maps turn out to be unique up to a scalar.

%
\subsubsection{What about bubbles?} \label{subsubsec_bubbles}
%

We have carefully avoided the simplest of all graded 2Homs, namely $\Hom_{\Ucat}(\onen, \onen)$.  Using the cap and cup generators we can construct closed diagrams that define 2-morphisms from $\onen$ to itself. The simplest of these diagrams are the dotted bubbles:
\begin{equation}
      \xy
      (0,0)*{\cbub{\alpha}{}};
      (8,8)*{n}; (-22,8)*{}; (22,-10)*{}; \endxy\qquad
    \xy  (0,0)*{\ccbub{\beta}{}}; (8,8)*{n};  (-22,8)*{}; (22,-10)*{}; \endxy \nn
\end{equation}
formed using a cap and a cup together with some number of dots.  The relations in $\Ucat$ ensure that where we place the dots on a dotted bubble is irrelevant, any placement gives rise to the same 2-morphism.

The degrees of these dotted bubbles are easily computed
\begin{eqnarray}
  \deg\left( \;\;\xy
      (0,0)*{\cbub{\alpha}{}};
      (8,8)*{n};  \endxy \;\;\right)
      = 2(-n+1+\alpha ), \qquad
  \deg\left( \;\;\xy  (0,0)*{\ccbub{\beta}{}}; (8,8)*{n};  \endxy \;\;\right)
      = 2(n+1+\beta ). \nn
\end{eqnarray}
However, the semilinear form indicates $\sla 1_n,1_n\sra=1$. If we interpret this as saying that all of the 2Homs $\Hom_{\Ucat}(\onen\{t\},\onen)$ are zero for $t \neq 0$, then this implies that all dotted bubbles of nonzero degree are zero.

This is the first place where we must deviate from a naive interpretation of the semilinear form.  Computing the action of the positive degree bubbles on weight spaces of the 2-representation $\cal{V}^N$ we see that for large $N$ these bubbles {\em do not} act by zero.  Furthermore, one can show that combining the assumption that positive degree bubbles are zero with other relations that arise when lifting the $\mathfrak{sl}_2$-relations to $\Ucat$ gives rise to an inconsistent calculus, see Remark~\ref{rem_otherinv} below.  Therefore, we do not impose the condition that dotted bubbles of nonzero degree are always zero.

We do however impose the condition that negative degree bubbles are zero.  This is to ensure that the space of 2Homs in each degree is finite dimensional.  This is further justified by the fact that negative degree bubbles act by zero in the action of $\Ucat$ on cohomology rings of flag varieties.

The relations in $\Ucat$ that result from our analysis will show that any complicated closed diagram can always be reduced to a product of non-nested dotted bubbles that all have the same orientation, and that any dotted bubble appearing in a diagram in $\HOM_{\Ucat}(\cal{E}_{\ep}\onen, \cal{E}_{\ep'}\onen)$ can be rewritten as a linear combination of diagrams where all dotted bubbles are confined to a single region of the diagram.  For concreteness, we can rewrite all 2-morphisms as linear combinations of diagrams where all dotted bubbles appear at the far right of the diagrams.  Hence our graded dimensions can be renormalized to account for the contributions coming from these dotted bubbles and all graded dimensions are interpreted modulo the contribution from $\HOM_{\Ucat}(\onen, \onen)$.

It will be convenient in what follows to write the dotted bubbles in the form
\begin{equation}
      \xy
      (0,0)*{\cbub{n-1+\alpha}{}};
      (8,8)*{n}; (-22,8)*{}; (22,-10)*{}; \endxy\qquad
    \xy  (0,0)*{\ccbub{-n-1+\beta}{}}; (8,8)*{n};  (-22,8)*{}; (22,-10)*{}; \endxy
\end{equation}
so that their degree is easy to read off from the diagram
\begin{eqnarray}
  \deg\left( \;\;\xy
      (0,0)*{\cbub{n-1+\alpha}{}};
      (8,8)*{n};  \endxy \;\;\right)
      = 2\alpha, \qquad
  \deg\left( \;\;\xy  (0,0)*{\ccbub{-n-1+\beta}{}}; (8,8)*{n};  \endxy \;\;\right)
      = 2\beta. \nn
\end{eqnarray}
Notice that the number of dots on these bubbles depends on $n$.  One potential drawback of this notation is that if we are not careful we could inadvertently  label the bubbles by a negative number of dots.  For example, if $\alpha \geq 0$ and $n<1-\alpha$, then the number of dots in the diagram
\[
\xy
      (0,0)*{\cbub{n-1+\alpha}{}};
      (8,8)*{n}; (-22,8)*{}; (22,-10)*{}; \endxy
\]
is negative.   But the vertical composite of a 2-morphism with itself a negative number of times does not make sense. In Section~\ref{subsec_fake} we explain that it is helpful to treat these negative labels as formal symbols used to represent more complicated diagrams.

\begin{rem} \label{rem_bub_linear_ind}
By explicitly computing the $(H_k,H_k)$-bimodule homomorphisms corresponding to the degree $\alpha$ counter-clockwise oriented dotted bubbles in a region $n$ one finds that this bimodule homomorphisms acts on $H_k$ by multiplication by the element
\begin{equation}
  \sum_{\ell=0}^{\min(\alpha,k)}c_{\ell}c_{\alpha-\ell}
\end{equation}
of $H_k$.  But by \eqref{eq_grass_rels} there are no relations on the variables $c_k$ below degree $N-k+1$.  This implies that there are no relations on products of non-nested clockwise-oriented dotted bubbles below degree $N-k+1$.  Since one can choose $N$ arbitrarily, this implies that products of non-nested dotted bubbles with a fixed orientation should be linearly independent.
\end{rem}

%
\subsubsection{Symmetric functions} \label{subsec_symm}
%

The linear independence of products of non-nested dotted bubbles derived from requiring an action of $\Ucat$ on the cohomology rings of Grassmannians suggests a relationship between closed diagrams in $\Ucat$ and the algebra of symmetric functions. Recall the ring of symmetric functions $\Lambda(\underline{x})$ in variables $\underline{x}=x_1,x_2,\dots, x_n$ is algebraically generated by
\begin{itemize}
  \item elementary symmetric functions: $e_r(x_1,\dots,x_n) = \sum_{j_1<j_2<\dots<j_r} x_{j_1}x_{j_2}\dots x_{j_rb}$, or
  \item complete symmetric functions: $h_r(x_1,\dots,x_n) = \sum_{m_1+m_2+\dots+m_n=r} x_{1}^{m_1}x_{2}^{m_2}\dots x_{n}^{m_n}$.
\end{itemize}

\begin{example}[n=3]
\begin{align*}
  e_1 & =x_1 + x_2 + x_3  & h_1 &=x_1 + x_2 + x_3 \\
  e_2 &= x_1x_2+ x_1x_3 +x_2x_3 & h_2 &= x_1^2 + x_2^2 +x_3^2+ x_1x_2+ x_1x_3 +x_2x_3 \\
  e_3 &= x_1x_2x_3
  & h_3 &=
\begin{array}{l}
  x_1^3 +x_2^3+x_3^3+x_1^2x_2 +x_1^2x_3+x_1x_2^2\\
  +x_2^2x_3 + x_1x_3^2+  x_2 x_3^2   + x_1x_2x_3 .
\end{array}
\end{align*}
\end{example}
Elementary and complete symmetric functions can also be defined for $\underline{x}$ an infinite set of variables. These two bases are related by the family of equations
\begin{equation} \label{eq_eh_rel}
  \sum_{\ell \geq 0} (-1)^{\ell}e_{\ell} h_{\alpha-\ell}=\delta_{\alpha,0}
\end{equation}
for each $\alpha\geq 0$, where we define $h_j=e_j=0$ for $j<0$ and $e_1=h_1=1$.

The ring $\symm(\underline{x})$ can be identified with the polynomial ring $\Z[e_1,e_2,\dots]$ and the polynomial ring $\Z[h_1,h_2,\dots]$. As a graded vector space
\begin{equation}
   \symm(\underline{x}) = \bigoplus_{\ell=0}^{\infty} \symm(\underline{x})^{\ell} \nn
\end{equation}
where $\symm(\underline{x})^{\ell}$ denote the homogeneous symmetric functions of degree $\ell$.  Given a partition $\lambda = (\lambda_1,\lambda_2, \dots, \lambda_m)$ let $e_{\lambda} = e_{\lambda_1}e_{\lambda_2} \dots e_{\lambda_m}$, so that $\symm(\underline{x})^{\ell}$ is spanned by those $e_{\lambda}$ with $|\lambda|=\sum_{j=1}^m\lambda_j=\ell$.

For $n \neq 0$ define a grading preserving map
\begin{eqnarray} \label{eq_symm_iso}
 \phi^n \maps \symm(\underline{x}) &
 \longrightarrow& \END_{\Ucat}(\onen) \\ \nn
  e_{\lambda}=e_{\lambda_1}\dots e_{\lambda_m} &\mapsto&
  \left\{
  \begin{array}{ccc}
    \xy 0;/r.18pc/:
 (-12,0)*{\cbub{n-1+\lambda_1}};
  (8,0)*{\cbub{n-1+\lambda_2}};
  (36,0)*{\cbub{n-1+\lambda_m}};
  (20,0)*{\cdots};
 (-8,8)*{n};
 \endxy & \quad & \text{for $n > 0$} \\ & &\\
    \xy 0;/r.18pc/:
 (-12,0)*{\ccbub{-n-1+\lambda_1}};
  (8,0)*{\ccbub{-n-1+\lambda_2}};
  (36,0)*{\ccbub{-n-1+\lambda_m}};
  (20,0)*{\cdots};
 (-8,8)*{n};
 \endxy & \quad & \text{if $n<0$.}
  \end{array}
  \right.
\end{eqnarray}
and write
\begin{equation} \label{eq_def_eln}
  e_{\lambda,n} := \phi^n(e_{\lambda}).
\end{equation}
When $\lambda$ is the empty partition we
write $e_{\lambda,n}=e_{\emptyset}=1$.

By Remark~\ref{rem_bub_linear_ind} it follows that this map is injective.

%
\subsection{Lifting equations to isomorphisms}
%

Using the semilinear form it was possible to determine the generating 2-morphisms and some relations for the 2-category $\Ucat$.  Requiring an action of the 2-category $\Ucat$ on the cohomology of iterated flag varieties clarified the precise form of some of these relations.  It remains to be shown that the 2-category lifts the defining relations in $\U$ to explicit isomorphisms:
\begin{eqnarray}
  \cal{E}\cal{F}\onen \cong \cal{F}\cal{E}\onen \oplus \onen^{\oplus_{[n]}}  & \qquad & \text{for $n \geq 0$}, \nn\\
  \cal{F}\cal{E}\onen  \cong \cal{E}\cal{F}\onen\oplus\onen^{\oplus_{[-n]}} & \qquad & \text{for $n \leq 0$.} \nn
\end{eqnarray}

For $n \geq 0$ there is a natural map $\zeta_+$ from the direct sum $\cal{F}\cal{E}\onen \oplus\onen^{\oplus_{[n]}}$ to the 1-morphism $ \cal{E}\cal{F}\onen$ given by the direct sum of maps
\begin{equation}
\xy
 (0,25)*+{\cal{E}\cal{F}\onen}="T";
 (-65,0)*+{\cal{F}\cal{E}\onen}="m1";
 (-50,0)*{\oplus};
 (-35,0)*+{\onen\{n-1\}}="m2";
 (-12,0)*{\oplus \;\; \cdots \;\; \oplus};
 (38,0)*{\oplus\;\; \cdots\;\; \oplus};
 (10,0)*+{\onen\{n-1-2\ell\}}="m3";
 (65,0)*+{ \onen\{1-n\}}="m4";
    {\ar@/^1.4pc/^{  \xy 0;/r.15pc/:
  (0,0)*{\xybox{
    (-4,-4)*{};(4,4)*{} **\crv{(-4,-1) & (4,1)}?(0)*\dir{<} ;
    (4,-4)*{};(-4,4)*{} **\crv{(4,-1) & (-4,1)}?(1)*\dir{>};
    (-7,-3)*{\scs };     (6,-3)*{\scs };     (10,2)*{n};
     }};
  \endxy} "m1";"T"};
    {\ar^{   \vcenter{\xy 0;/r.15pc/:
            (-8,0)*{};
           (-4,2)*{}="t1";
            (4,2)*{}="t2";
            "t2";"t1" **\crv{(4,-5) & (-4,-5)}; ?(.1)*\dir{>} ?(.95)*\dir{>}
            ?(.2)*\dir{}+(0,-.1)*{\bullet}+(4,-2)*{\scs \quad n-1};;
        \endxy}} "m2";"T"};
    {\ar^{ \vcenter{\xy 0;/r.15pc/:
            (-8,0)*{};
           (-4,2)*{}="t1";
            (4,2)*{}="t2";
            "t2";"t1" **\crv{(4,-5) & (-4,-5)}; ?(.1)*\dir{>} ?(.95)*\dir{>}
            ?(.2)*\dir{}+(0,-.1)*{\bullet}+(4,-2)*{\scs \quad n-1-\ell};;
        \endxy} } "m3";"T"};
    {\ar@/_1.4pc/_{ \vcenter{\xy 0;/r.15pc/:
            (-8,0)*{};
            (-4,2)*{}="t1";
            (4,2)*{}="t2";
            "t2";"t1" **\crv{(4,-5) & (-4,-5)}; ?(.1)*\dir{>} ?(.95)*\dir{>};
        \endxy} } "m4";"T"};
 \endxy
\end{equation}
Likewise, for $n \leq 0$ there is a natural map $\zeta_-$
\begin{equation} \label{eq_FEiso}
\xy
 (0,25)*+{\cal{F}\cal{E}\onen}="T";
 (-60,0)*+{\cal{E}\cal{F}\onen}="m1";
  (-50,0)*+{\oplus};
 (-36,0)*+{\onen\{-n-1\}}="m2";
 (-15,0)*{\oplus \;\; \cdots \;\; \oplus};
 (36,0)*{\oplus\;\; \cdots\;\; \oplus};
 (10,0)*+{\onen\{-n-1-2\ell\}}="m3";
 (60,0)*+{\oplus \;\;\onen\{1+n\}}="m4";
    {\ar@/^1.4pc/^{  \xy 0;/r.15pc/:
        (0,0)*{\xybox{
        (-4,-4)*{};(4,4)*{} **\crv{(-4,-1) & (4,1)}?(1)*\dir{>} ;
        (4,-4)*{};(-4,4)*{} **\crv{(4,-1) & (-4,1)}?(0)*\dir{<};
        (-5,-3)*{\scs };
        (7,-3)*{\scs };
        (14,2)*{n};
        }};
    \endxy} "m1";"T"};
    {\ar^{\vcenter{\xy 0;/r.15pc/:
            (-8,0)*{};
           (-4,2)*{}="t1";
            (4,2)*{}="t2";
            "t2";"t1" **\crv{(4,-5) & (-4,-5)}; ?(.05)*\dir{<} ?(.88)*\dir{<}
            ?(.3)*\dir{}+(0,-.1)*{\bullet}+(4,-2)*{\scs \quad -n-1};;
        \endxy}} "m2";"T"};
    {\ar^{\vcenter{\xy 0;/r.15pc/:
            (-8,0)*{};
           (-4,2)*{}="t1";
            (4,2)*{}="t2";
            "t2";"t1" **\crv{(4,-5) & (-4,-5)}; ?(.05)*\dir{<} ?(.88)*\dir{<}
            ?(.3)*\dir{}+(0,-.1)*{\bullet}+(4,-2)*{\scs \quad -n-1-\ell};;
        \endxy}} "m3";"T"};
    {\ar@/_1.4pc/_{\vcenter{\xy 0;/r.15pc/:
            (-8,0)*{};
           (-4,2)*{}="t1";
            (4,2)*{}="t2";
            "t2";"t1" **\crv{(4,-5) & (-4,-5)}; ?(.05)*\dir{<} ?(.88)*\dir{<};
        \endxy}} "m4";"T"};
 \endxy
\end{equation}

We look for the most general inverse of the maps $\zeta_+$ and $\zeta_-$ in $\Ucat$ making the following assumptions.
\begin{enumerate}
  \item All 1-morphisms have specified cyclic biadjoints, see\eqref{eq_biadjointness1} and \eqref{eq_biadjointness2}.

 \item The nilHecke algebra axioms \eqref{eq_nil_rels_grass} and \eqref{eq_nil_dotslide_grass}.

 \item There are no negative degree endomorphisms in $\End_{\Ucat}(\onen)$.
\item  The graded 2Homs in $\Ucat$ categorify the semilinear form on $\U$ modulo the contribution from dotted bubbles. That is,
 \begin{equation}
  \gdim \HOM_{\Ucat}(\cal{E}_{\ep}\onen, \cal{E}_{\ep'}\onen)/ \gdim \END_{\Ucat}(\onen)  =
  \langle E_{\ep} 1_n , E_{\ep'} 1_n \rangle.
 \end{equation}
\end{enumerate}

%
\subsubsection{A general form of the $\mathfrak{sl}_2$-isomorphisms}
%

Because $\zeta_{+}$ is a map from a direct sum, its inverse $\overline{\zeta_{+}}$ will have a component for each summand:
\begin{equation} \label{eq_phi-minus-invL}
\xy
 (0,-25)*+{\cal{E}\cal{F}\onen}="B";
 (-65,0)*+{\cal{F}\cal{E}\onen}="m1";
 (-50,0)*{\oplus};
 (-35,0)*+{\onen\{n-1\}}="m2";
 (-12,0)*{\oplus \;\; \cdots \;\; \oplus};
 (38,0)*{\oplus\;\; \cdots\;\; \oplus};
 (10,0)*+{\onen\{n-1-2\ell\}}="m3";
 (65,0)*+{ \onen\{1-n\}}="m4";
  {\ar@/^1.6pc/^(.7){\xy 0;/r.15pc/:
    (-4,-8)*{};(4,8)*{} **\crv{(-4,-1) & (4,1)}?(.95)*\dir{>} ?(.1)*\dir{>};
    (4,-8)*{};(-4,8)*{} **\crv{(4,-1) & (-4,1)}?(.05)*\dir{<} ?(.9)*\dir{<};
    (0,0)*{\chern{\overline{\zeta_{+}}^{n}}}; (-12,2)*{n};
  \endxy\;\;} "B";"m1"};
  {\ar@/^0.6pc/^(.7){
  \xy 0;/r.15pc/:
     (-4,-2)*{}="t1";  (4,-2)*{}="t2";
     "t2";"t1" **\crv{(4,5) & (-4,5)};  ?(.9)*\dir{<} ?(.05)*\dir{<};
      (0,6)*{\chern{\overline{\zeta_{+}}^{0}}};  (-10,10)*{n}; \endxy } "B";"m2"};
  {\ar_(.5){ \xy 0;/r.15pc/:
     (-4,-2)*{}="t1";  (4,-2)*{}="t2";
     "t2";"t1" **\crv{(4,5) & (-4,5)};  ?(.9)*\dir{<} ?(.05)*\dir{<};
      (0,6)*{\chern{\overline{\zeta_{+}}^{\ell}}}; (10,10)*{n};\endxy } "B";"m3"};
    {\ar@/_1.6pc/_(.8){ \xy 0;/r.15pc/:
     (-4,-2)*{}="t1";  (4,-2)*{}="t2";
     "t2";"t1" **\crv{(4,5) & (-4,5)};  ?(.93)*\dir{<} ?(.03)*\dir{<};
      (0,6)*{\chern{\overline{\zeta_{+}}^{(n-1)}}};  (6,16)*{n};\endxy } "B";"m4"};
 \endxy
\end{equation}
Likewise,  the inverse $\overline{\zeta}_-$ of $\zeta_{-}$ can be written
\begin{equation} \label{eq_phi-plus-invL}
\xy
 (0,-25)*+{\cal{F}\cal{E}\onen}="B";
(-60,0)*+{\cal{E}\cal{F}\onen}="m1";
  (-50,0)*+{\oplus};
 (-36,0)*+{\onen\{-n-1\}}="m2";
 (-15,0)*{\oplus \;\; \cdots \;\; \oplus};
 (36,0)*{\oplus\;\; \cdots\;\; \oplus};
 (10,0)*+{\onen\{-n-1-2\ell\}}="m3";
 (60,0)*+{\oplus \;\;\onen\{1+n\}}="m4";
  {\ar@/^1.6pc/^(.7){\xy 0;/r.15pc/:
    (-4,-8)*{};(4,8)*{} **\crv{(-4,-1) & (4,1)}?(.9)*\dir{<} ?(.05)*\dir{<};
    (4,-8)*{};(-4,8)*{} **\crv{(4,-1) & (-4,1)}?(.1)*\dir{>} ?(.95)*\dir{>};
    (0,0)*{\chern{\overline{\zeta_{-}}^{-n}}}; (-12,2)*{n};
  \endxy\;\;} "B";"m1"};
  {\ar^(.7){
  \xy 0;/r.15pc/:
     (-4,-2)*{}="t1";  (4,-2)*{}="t2";
     "t2";"t1" **\crv{(4,5) & (-4,5)};  ?(.98)*\dir{>} ?(.1)*\dir{>};
      (0,6)*{\chern{\overline{\zeta_{-}}^{0}}};  (-10,10)*{n}; \endxy } "B";"m2"};
  {\ar_(.5){ \xy 0;/r.15pc/:
     (-4,-2)*{}="t1";  (4,-2)*{}="t2";
     "t2";"t1" **\crv{(4,5) & (-4,5)}; ?(.98)*\dir{>} ?(.1)*\dir{>};
      (0,6)*{\chern{\overline{\zeta_{-}}^{\ell}}}; (10,10)*{n};\endxy } "B";"m3"};
    {\ar@/_1.6pc/_(.8){ \xy 0;/r.15pc/:
     (-4,-2)*{}="t1";  (4,-2)*{}="t2";
     "t2";"t1" **\crv{(4,5) & (-4,5)};  ?(.98)*\dir{>} ?(.1)*\dir{>};
      (0,6)*{\chern{\overline{\zeta_{-}}^{(-n-1)}}};  (6,16)*{n};\endxy } "B";"m4"};
 \endxy
\end{equation}

From the assumption that the graded 2Homs categorify the semilinear form, it can be shown that the summands of any inverse built from $\Bbbk$-linear combinations of composites of generating 2-morphisms of $\Ucat$ must have the form
\begin{equation} \label{eq_phi-inverses-U}
\vcenter{\xy 0;/r.15pc/:
     (-4,-2)*{}="t1";  (4,-2)*{}="t2";
     "t2";"t1" **\crv{(4,5) & (-4,5)};  ?(.9)*\dir{<} ?(.05)*\dir{<};
     (0,6)*{\chern{\overline{\zeta_{+}}^{\ell}}};
      (10,10)*{n};\endxy}
\;\; = \;\;
\sum_{|\lambda|+j=\ell} \;\; \alpha_{\lambda}^{\ell}(n)e_{\lambda,n}\;
\vcenter{\xy 0;/r.15pc/:
     (-4,-2)*{}="t1";  (4,-2)*{}="t2";
     "t2";"t1" **\crv{(4,5) & (-4,5)};  ?(.9)*\dir{<} ?(.05)*\dir{<}
     ?(.3)*\dir{}+(0,0)*{\bullet}+(3,1)*{\scs j};
    (12,10)*{n};\endxy}
\qquad  \qquad
\vcenter{\xy 0;/r.15pc/:
     (4,-2)*{}="t1";  (-4,-2)*{}="t2";
     "t2";"t1" **\crv{(-4,5) & (4,5)};  ?(.9)*\dir{<} ?(.05)*\dir{<};
      (0,6)*{\chern{\overline{\zeta_{-}}^{\ell}}}; (10,10)*{n};\endxy}
\;\; = \;\;
\sum_{|\lambda|+j=\ell} \;\; \alpha_{\lambda}^{\ell}(n)e_{\lambda,n} \;
\vcenter{\xy 0;/r.15pc/:
     (4,-2)*{}="t1";  (-4,-2)*{}="t2";
     "t2";"t1" **\crv{(-4,5) & (4,5)};  ?(.9)*\dir{<} ?(.05)*\dir{<}
     ?(.7)*\dir{}+(0,0)*{\bullet}+(3,1)*{\scs j};
      (12,10)*{n};\endxy}
\end{equation}
for some coefficients $\alpha_{i}^{\ell}(n) \in \Bbbk$ and $e_{\lambda,n}$ defined in \eqref{eq_def_eln}.
Furthermore,
\begin{equation} \label{eq_EFcrossing}
  \vcenter{\xy 0;/r.15pc/:
    (-4,-8)*{};(4,8)*{} **\crv{(-4,-1) & (4,1)}?(.95)*\dir{>} ?(.1)*\dir{>};
    (4,-8)*{};(-4,8)*{} **\crv{(4,-1) & (-4,1)}?(.05)*\dir{<} ?(.9)*\dir{<};
    (0,0)*{\chern{\overline{\zeta_{+}}^{n}}}; (-12,2)*{n};
  \endxy} \;\;=\;\; \beta_n
  \;
  \vcenter{\xy 0;/r.15pc/:
    (-4,-8)*{};(4,8)*{} **\crv{(-4,-1) & (4,1)}?(.95)*\dir{>} ?(.1)*\dir{>};
    (4,-8)*{};(-4,8)*{} **\crv{(4,-1) & (-4,1)}?(.05)*\dir{<} ?(.9)*\dir{<};
    (12,2)*{n};
  \endxy\;\;}
  \qquad \qquad
\vcenter{\xy 0;/r.15pc/:
    (-4,-8)*{};(4,8)*{} **\crv{(-4,-1) & (4,1)}?(.9)*\dir{<} ?(.05)*\dir{<};
    (4,-8)*{};(-4,8)*{} **\crv{(4,-1) & (-4,1)}?(.1)*\dir{>} ?(.95)*\dir{>};
    (0,0)*{\chern{\overline{\zeta_{-}}^{-n}}}; (-12,2)*{n};
  \endxy}
 \;\; = \;\; \beta_n
\vcenter{\xy 0;/r.15pc/:
    (-4,-8)*{};(4,8)*{} **\crv{(-4,-1) & (4,1)}?(.9)*\dir{<} ?(.05)*\dir{<};
    (4,-8)*{};(-4,8)*{} **\crv{(4,-1) & (-4,1)}?(.1)*\dir{>} ?(.95)*\dir{>};
    (12,2)*{n};
  \endxy}
\end{equation}
for some coefficients $\beta_n \in \Bbbk$.

For notational simplicity we sometimes suppressed the $n$ dependence
of the coefficients $\alpha_{\lambda}^{\ell}(n)$ and write
$\alpha_{\lambda}^{\ell}$. We write $\alpha_{\lambda}^{\ell}(n)$ for all
$\ell$ and $\lambda$ and $n$, but set $\alpha_{\lambda}^{\ell}(n)=0$
unless $|\lambda|\leq \ell \leq n-1$.

%
\subsubsection{Relations from invertibility}
%

Requiring that the maps $\overline{\zeta_{+}}$ and $\overline{\zeta_{-}}$ defined in \eqref{eq_phi-minus-invL} and \eqref{eq_phi-plus-invL} are inverses of the maps $\zeta_{+}$ and $\zeta_{-}$, respectively, gives rise to relations in the 2-category $\Ucat$.

\begin{center}
\begin{tabular}{|l|c|}
\hline \multicolumn{2}{|c|}{\bf Relations for $n \geq 0$} \\ \hline & \\
{\bf (A1)} &
$   \delta_{b,0} = \xsum{\lambda: |\lambda| \leq b}{}\; \alpha_{\lambda}^{\ell}(n)\; e_{\lambda,n}
\xy 0;/r.16pc/:
 (0,0)*{\cbub{n-1+b-|\lambda|}};
  (4,8)*{n};
 \endxy
$
   \\  & \\
   \hline
   & \\
{\bf (A2)} &
$ \vcenter{\xy 0;/r.16pc/:
  (-8,0)*{};(-6,-8)*{\scs };(6,-8)*{\scs };
  (8,0)*{};
  (-4,10)*{}="t1";
  (4,10)*{}="t2";
  (-4,-10)*{}="b1";
  (4,-10)*{}="b2";
  "t1";"b1" **\dir{-} ?(.5)*\dir{>};
  "t2";"b2" **\dir{-} ?(.5)*\dir{<};
  (10,2)*{n};
  (-10,2)*{n};
  \endxy}
\;\; = \;\;
 \beta_{n}\;\;
   \vcenter{\xy 0;/r.16pc/:
    (-4,-6)*{};(4,4)*{} **\crv{(-4,-1) & (4,1)}?(1)*\dir{<};?(0)*\dir{<};
    (4,-6)*{};(-4,4)*{} **\crv{(4,-1) & (-4,1)}?(1)*\dir{>};
    (-4,4)*{};(4,14)*{} **\crv{(-4,7) & (4,9)}?(1)*\dir{>};
    (4,4)*{};(-4,14)*{} **\crv{(4,7) & (-4,9)};
  (8,8)*{n};(-6,-3)*{\scs };  (6,-3)*{\scs };
 \endxy}$
   \\  & \\
   \hline
   & \\
{\bf (A3)} &
$ \beta_{n} \;\; \vcenter{\xy 0;/r.16pc/:
    (-4,6)*{};(4,-4)*{} **\crv{(-4,1) & (4,-1)}?(1)*\dir{>};
    (4,6)*{};(-4,-4)*{} **\crv{(4,1) & (-4,-1)}?(0)*\dir{<};
    (-4,-4)*{};(4,-4)*{} **\crv{(-4,-8) & (4,-8)}?(.5)*\dir{}+(0,0)*{\bullet}+(-1,-3)*{\scs n-1-\ell};
  (12,-2)*{n};(-6,-3)*{\scs };  (6,-3)*{\scs };
 \endxy} \;\; = \;\; 0$
     \\  & \\
   \hline
   & \\
{\bf (A4)} &
$
    \sum_{\lambda} \;\alpha_{\lambda}^{\ell}(n) \;e_{\lambda,n} \vcenter{\xy 0;/r.16pc/:
    (-4,-6)*{};(4,4)*{} **\crv{(-4,-1) & (4,1)}?(1)*\dir{};?(0)*\dir{<};
    (4,-6)*{};(-4,4)*{} **\crv{(4,-1) & (-4,1)}?(1)*\dir{>};
    (-4,4)*{};(4,4)*{} **\crv{(-4,8) & (4,8)}?(.5)*\dir{}+(0,0)*{\bullet}+(-1,3)*{\scs n-1-\ell};
  (12,2)*{n};(-6,-3)*{\scs };  (6,-3)*{\scs };
 \endxy} \;\; = \;\; 0$
    \\  & \\
   \hline
   & \\
{\bf (A5)} &
$ \vcenter{\xy 0;/r.16pc/:
  (-8,0)*{};
  (8,0)*{};
  (-4,10)*{}="t1";
  (4,10)*{}="t2";
  (-4,-10)*{}="b1";
  (4,-10)*{}="b2";(-6,-8)*{\scs };(6,-8)*{\scs };
  "t1";"b1" **\dir{-} ?(.5)*\dir{<};
  "t2";"b2" **\dir{-} ?(.5)*\dir{>};
  (10,2)*{n};
  (-10,2)*{n};
  \endxy}
\;\; = \;\;
 \beta_{n}\;\;
 \vcenter{   \xy 0;/r.16pc/:
    (-4,-6)*{};(4,4)*{} **\crv{(-4,-1) & (4,1)}?(1)*\dir{>};
    (4,-6)*{};(-4,4)*{} **\crv{(4,-1) & (-4,1)}?(1)*\dir{<};?(0)*\dir{<};
    (-4,4)*{};(4,14)*{} **\crv{(-4,7) & (4,9)};
    (4,4)*{};(-4,14)*{} **\crv{(4,7) & (-4,9)}?(1)*\dir{>};
  (8,8)*{n};(-6,-3)*{\scs };
     (6.5,-3)*{\scs };
 \endxy}
  \;\; + \;\;
   \xsum{ \xy  (0,8)*{}; (0,3)*{\scs f_1+f_2+|\lambda|}; (0,0)*{\scs =n-1};\endxy}{}\;
    \alpha_{\lambda}^{|\lambda|+f_2}(n) e_{\lambda,n}
    \vcenter{\xy 0;/r.16pc/:
    (-10,10)*{n};
    (-8,0)*{};
  (8,0)*{};
  (-4,-10)*{}="b1";
  (4,-10)*{}="b2";
  "b2";"b1" **\crv{(5,-3) & (-5,-3)}; ?(.05)*\dir{<} ?(.93)*\dir{<}
  ?(.8)*\dir{}+(0,-.1)*{\bullet}+(-3,2)*{\scs f_2};
  (-4,10)*{}="t1";
  (4,10)*{}="t2";
  "t2";"t1" **\crv{(5,3) & (-5,3)}; ?(.15)*\dir{>} ?(.95)*\dir{>}
  ?(.4)*\dir{}+(0,-.2)*{\bullet}+(3,-2)*{\scs \; f_1};
  \endxy}$
   \\ & \\
  \hline
\end{tabular}
\end{center}
\medskip
Note that relations {\bf A1}, {\bf A3}, and {\bf A4} are only valid for $n>0$.

\begin{center}
\begin{tabular}{|l|c|}
\hline \multicolumn{2}{|c|}{\bf Relations for $n \geq 0$} \\ \hline & \\
{\bf (B1)} &
$  \delta_{b,0} = \xsum{\lambda: |\lambda| \leq b}{}\; \alpha_{\lambda}^{\ell}(n)\;
  e_{\lambda,n}
  \xy 0;/r.16pc/:
 (0,0)*{\ccbub{-n-1+b-|\lambda|}};
  (4,8)*{n};
 \endxy
$
   \\  & \\
   \hline
   & \\
{\bf (B2)} &
$ \vcenter{\xy 0;/r.16pc/:
  (-8,0)*{};
  (8,0)*{};
  (-4,10)*{}="t1";
  (4,10)*{}="t2";
  (-4,-10)*{}="b1";
  (4,-10)*{}="b2";(-6,-8)*{\scs };(6,-8)*{\scs };
  "t1";"b1" **\dir{-} ?(.5)*\dir{<};
  "t2";"b2" **\dir{-} ?(.5)*\dir{>};
  (10,2)*{n};
  (-10,2)*{n};
  \endxy}
\;\; = \;\;
 \beta_{n}\;\;
 \vcenter{   \xy 0;/r.16pc/:
    (-4,-6)*{};(4,4)*{} **\crv{(-4,-1) & (4,1)}?(1)*\dir{>};
    (4,-6)*{};(-4,4)*{} **\crv{(4,-1) & (-4,1)}?(1)*\dir{<};?(0)*\dir{<};
    (-4,4)*{};(4,14)*{} **\crv{(-4,7) & (4,9)};
    (4,4)*{};(-4,14)*{} **\crv{(4,7) & (-4,9)}?(1)*\dir{>};
  (8,8)*{n};(-6,-3)*{\scs };
     (6.5,-3)*{\scs };
 \endxy}$
   \\  & \\
   \hline
   & \\
{\bf (B3)} &
$\beta_{n}\;\; \vcenter{\xy 0;/r.16pc/:
    (-4,6)*{};(4,-4)*{} **\crv{(-4,1) & (4,-1)}?(1)*\dir{};?(0)*\dir{<};
    (4,6)*{};(-4,-4)*{} **\crv{(4,1) & (-4,-1)}?(1)*\dir{>};
    (-4,-4)*{};(4,-4)*{} **\crv{(-4,-8) & (4,-8)}?(.5)*\dir{}+(0,0)*{\bullet}+(-1,-3)*{\scs -n-1-\ell};
  (12,-2)*{n};(-6,-3)*{\scs };  (6,-3)*{\scs };
 \endxy} \;\; = \;\; 0$
     \\  & \\
   \hline
   & \\
{\bf (B4)} &
$\xsum{\lambda}{}\;\alpha_{\lambda}^{\ell}(n)\; e_{\lambda,n}
    \vcenter{\xy 0;/r.16pc/:
    (-4,-6)*{};(4,4)*{} **\crv{(-4,-1) & (4,1)}?(1)*\dir{>};
    (4,-6)*{};(-4,4)*{} **\crv{(4,-1) & (-4,1)}?(0)*\dir{<};
    (-4,4)*{};(4,4)*{} **\crv{(-4,8) & (4,8)}?(.5)*\dir{}+(0,0)*{\bullet}+(-1,3)*{\scs -n-1-\ell};
  (12,2)*{n};(-6,-3)*{\scs };  (6,-3)*{\scs };
 \endxy} \;\; = \;\; 0$
    \\  & \\
   \hline
   & \\
{\bf (B5)} &
$ \vcenter{\xy 0;/r.16pc/:
  (-8,0)*{};(-6,-8)*{\scs };(6,-8)*{\scs };
  (8,0)*{};
  (-4,10)*{}="t1";
  (4,10)*{}="t2";
  (-4,-10)*{}="b1";
  (4,-10)*{}="b2";
  "t1";"b1" **\dir{-} ?(.5)*\dir{>};
  "t2";"b2" **\dir{-} ?(.5)*\dir{<};
  (10,2)*{n};
  (-10,2)*{n};
  \endxy}
\;\; = \;\;
 \beta_{n}\;\;
   \vcenter{\xy 0;/r.16pc/:
    (-4,-4)*{};(4,4)*{} **\crv{(-4,-1) & (4,1)}?(1)*\dir{<};?(0)*\dir{<};
    (4,-4)*{};(-4,4)*{} **\crv{(4,-1) & (-4,1)}?(1)*\dir{>};
    (-4,4)*{};(4,12)*{} **\crv{(-4,7) & (4,9)}?(1)*\dir{>};
    (4,4)*{};(-4,12)*{} **\crv{(4,7) & (-4,9)};
  (8,8)*{n};(-6,-3)*{\scs };  (6,-3)*{\scs };
 \endxy}
  \;\; + \;\;
    \xsum{ \xy  (0,8)*{}; (0,3)*{\scs g_1+g_2+|\lambda|}; (0,0)*{\scs =-n-1};\endxy}{} \; \alpha_{\lambda}^{|\lambda|+g_2}(n) \; e_{\lambda,n}
    \vcenter{\xy 0;/r.16pc/:
    (-8,0)*{};
  (8,0)*{};
  (-4,-10)*{}="b1";
  (4,-10)*{}="b2";
  "b2";"b1" **\crv{(5,-3) & (-5,-3)}; ?(.1)*\dir{>} ?(.95)*\dir{>}
  ?(.8)*\dir{}+(0,-.1)*{\bullet}+(-3,2)*{\scs g_2};
  (-4,10)*{}="t1";
  (4,10)*{}="t2";
  "t2";"t1" **\crv{(5,3) & (-5,3)}; ?(.15)*\dir{<} ?(.9)*\dir{<}
  ?(.4)*\dir{}+(0,-.2)*{\bullet}+(3,-2)*{\scs g_1};
  (-10,10)*{n};
  \endxy}$
   \\ & \\
  \hline
\end{tabular}
\end{center}
Note that relations {\bf B1}, {\bf B3}, and {\bf B4} are only valid for $n<0$.

With a little work, one can show that these relations suffice to reduce any closed diagram to a linear combination of products of non-nested dotted bubbles with the same orientation (see Section~\ref{subsubsec_infinitegrass} for a discussion on changing the orientation of dotted bubbles).  Furthermore, it follows from {\bf (A1)} and {\bf (B1)} that the degree zero dotted bubble must be a multiple of the identity.  We use this fact extensively in what follows.

%
\subsubsection{Reducing coefficients}
%

All dotted bubbles of negative degree in $\Ucat$ are zero by the assumption about $\END_{\Ucat}(\onen)$. (Note that when $n=0$ it is impossible for a dotted bubble to have negative degree.) Thus we can write,
\begin{equation}
 \xy 0;/r.18pc/:
 (-12,0)*{\cbub{\alpha}};
 (-8,8)*{n};
 \endxy
  = 0
 \qquad
  \text{if $\alpha<n-1$,} \qquad
 \xy 0;/r.18pc/:
 (-12,0)*{\ccbub{\alpha}};
 (-8,8)*{n};
 \endxy = 0\quad
  \text{if $\alpha< -n-1$} \nn
\end{equation}
for all $\alpha \in \Z_+$. Since the space of endomorphisms $\End(\onen)$ in degree zero one dimensional, a dotted bubble of degree zero must be a
multiple of the identity map. We fix these constants as follows
\begin{equation}\label{eq_cpm}
\xy 0;/r.18pc/:
 (0,0)*{\cbub{n-1}};
  (4,8)*{n};
 \endxy
  = c_{n}^+ \cdot \Id_{\onen} \quad \text{for $n \geq  1$,}
  \qquad \quad
  \xy 0;/r.18pc/:
 (0,0)*{\ccbub{-n-1}};
  (4,8)*{n};
 \endxy
  = c_{n}^- \cdot \Id_{\onen} \quad \text{for $n \leq -1$.}
\end{equation}
For the maps $\zeta_+$ and $\zeta_-$ to be invertible the coefficients $c_n^+$ and $c_n^-$ must be nonzero invertible scalars.

Using the fact that the map $\phi^n$ from Equation~\eqref{eq_symm_iso} is injective the coefficients $\alpha_{\lambda}^{\ell}(n)$ are completely
determined for all $n \in \Z$ from Conditions (A1) and (B1), which can be written as
\begin{equation} \label{eq_A1B1}
\delta_{b,0} = \sum_{\lambda: |\lambda| \leq b}
\alpha_{\lambda}^{\ell}(n)e_{\lambda,n} e_{b-|\lambda|,n}.
\end{equation}
Recall that we have set $e_{\emptyset,n}=1$, but degree zero bubbles $e_{0,n}$ are multiplication by nonzero free parameters $c_n^{\pm}$, see \eqref{eq_cpm}.

Write $\overline{\alpha_{\lambda}^{\ell}(n)}
:=(\alpha_{\lambda}^{\ell}(n))|_{c_n^{\pm}=1}$ to denote the coefficients $\alpha_{\lambda}^{\ell}(n)$ corresponding to setting the degree zero bubbles equal to $1$.  These coefficients
$\overline{\alpha_{\lambda}^{\ell}(n)}$ are well known.  They are
related to the expression for complete symmetric functions in terms
of elementary symmetric functions.  In particular, if for any $0\leq
s\leq \ell$ we define
\begin{equation}
  h_{s,n} :=
  \begin{array}{ccc}
      (-1)^s\xsum{\lambda : |\lambda|=s}{}\overline{\alpha_{\lambda}^{\ell}(n)}e_{\lambda,n} & \quad &
 \text{for $n \in \Z$},
  \end{array}
\end{equation}
then \eqref{eq_A1B1} becomes the statement that
\begin{equation}
  \sum_{s+r=b}(-1)^s h_{s,n} e_{r,n} = \delta_{b,0}.
\end{equation}
Hence these coefficients are the same as those arising when expressing a complete symmetric function in the basis of elementary symmetric functions, see for example \cite{Mac}.

It is straightforward to check that if $\lambda=(\lambda_1,\dots,
\lambda_m)$, then
\begin{eqnarray} \label{eq_alpha}
  \alpha_{\lambda}^{\ell}(n) &=& \left(\frac{1}{c_{n}^{\pm}}\right)^{m+1} \overline{\alpha_{\lambda}^{\ell}(n)} .
\end{eqnarray}
In particular, $\alpha_{\emptyset}^{\ell}(n)=\frac{1}{c_n^{\pm}}$ for all $n$
and $0 \leq \ell \leq |n|-1$.

Furthermore since the $\alpha_{\lambda}^{\ell}(n)$ for appropriate
$\ell$ are given by the expansion of complete symmetric functions
in terms of elementary symmetric functions, whenever $|\lambda|\leq
\ell,\ell' \leq |n|-1$ it follows that
\begin{equation}
  \alpha_{\lambda}^{|\lambda|}(n) = \alpha_{\lambda}^{\ell}(n) = \alpha_{\lambda}^{\ell'}(n). \nn
\end{equation}

%
\subsection{Fake bubbles} \label{subsec_fake}
%

In this section we augment the graphical calculus by adding formal symbols called {\em fake bubbles}.  Fake bubbles are dotted bubbles of positive degree that carry a label of a negative number of dots, see Figure~\ref{fig_fakebub}.
\begin{figure}[htp]
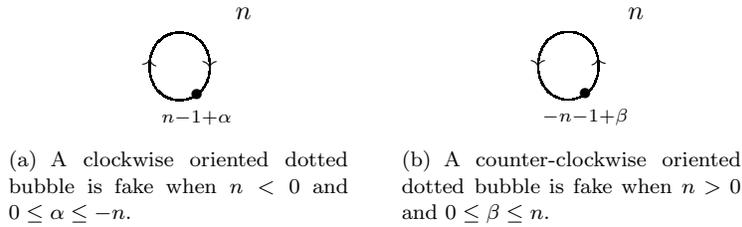

  \begin{center}
    \subfigure[A clockwise oriented dotted bubble is fake when $n<0$ and $0\leq \alpha \leq -n$.]{
      \xy
      (0,0)*{\cbub{n-1+\alpha}{}}; (8,8)*{n}; (-22,8)*{}; (22,-10)*{}; \endxy}\qquad
    \subfigure[A counter-clockwise oriented dotted bubble is fake when $n>0$ and $0 \leq \beta \leq n$.]{
    \xy  (0,0)*{\ccbub{-n-1+\beta}{}}; (8,8)*{n};  (-22,8)*{}; (22,-10)*{}; \endxy}
  \end{center}
 \caption{Fake bubbles. }
  \label{fig_fakebub}
\end{figure}
A dot with a negative label is not defined in the graphical calculus we have studied so far and we do not want to add an inverse to the dot 2-morphisms.  This would cause a number of serious problems including a deviation from the semilinear form.  Instead, these fake bubbles are interpreted as formal symbols that are defined in terms of linear combinations of products of real dotted bubbles.

Since fake bubbles are perhaps the most confusing aspect of the graphical calculus,  we take a moment to explain their motivation and uses.

%
\subsubsection{A motivating example}
%
%
As a motivating example, consider the diagram
\begin{equation}
\vcenter{\xy 0;/r.18pc/:
    (-4,6)*{};(4,-4)*{} **\crv{(-4,1) & (4,-1)}?(1)*\dir{>};
    (4,6)*{};(-4,-4)*{} **\crv{(4,1) & (-4,-1)}?(0)*\dir{<};
    (-4,-4)*{};(4,-4)*{} **\crv{(-4,-8) & (4,-8)}?(.5)*\dir{}+(0,0)*{\bullet}+(-1,-3)*{\scs j};
  (12,-2)*{n};(-6,-3)*{\scs };  (6,-3)*{\scs };
 \endxy}
\end{equation}
in $\Ucat$. If possible we would like to simplify this diagram for all $n$ and
$j \geq 0$.  Relation (A3) says that when $n> 0$ this diagram is
zero unless $j\geq n$.  In fact, we can simplify the above diagram
for $n>0$  and $j \geq n$ using the nilHecke relation to slide the
dots outside of the curl,
\begin{equation}
\vcenter{\xy 0;/r.18pc/:
    (-4,6)*{};(4,-4)*{} **\crv{(-4,1) & (4,-1)}?(1)*\dir{>};
    (4,6)*{};(-4,-4)*{} **\crv{(4,1) & (-4,-1)}?(0)*\dir{<};
    (-4,-4)*{};(4,-4)*{} **\crv{(-4,-9) & (4,-9)}?(.5)*\dir{}+(0,0)*{\bullet}+(-1,-3)*{\scs j};
  (12,-2)*{n};(-6,-3)*{\scs };  (6,-3)*{\scs };
 \endxy} \;\; \refequal{\eqref{eq_ind_dotslide}} \;\;  \vcenter{\xy 0;/r.18pc/:
    (-4,6)*{};(4,-4)*{} **\crv{(-4,1) & (4,-1)}?(1)*\dir{>};
    (4,6)*{};(-4,-4)*{} **\crv{(4,1) & (-4,-1)}?(0)*\dir{<} ?(.25)*\dir{}
     +(0,0)*{\bullet}+(3,0)*{\scs j};
    (-4,-4)*{};(4,-4)*{} **\crv{(-4,-9) & (4,-9)};
  (12,-2)*{n};(-6,-3)*{\scs };  (6,-3)*{\scs };
 \endxy} \; -\sum_{f_1+f'_2=j}
  \vcenter{\xy 0;/r.18pc/:
    (-4,6)*{};(4,6)*{} **\crv{(-4,-2) & (4,-2)}?(1)*\dir{>} ?(.25)*\dir{}+(0,0)*{\bullet}+(-3,0)*{\scs f_1};
    (2,-9)*{\cbub{f'_2}{}};
  (12,-2)*{n};(-6,-3)*{\scs };  (6,-3)*{\scs };
 \endxy} = \;\; -\sum_{f_1+f_2=j-n}
  \vcenter{\xy 0;/r.18pc/:
    (-4,6)*{};(4,6)*{} **\crv{(-4,-2) & (4,-2)}?(1)*\dir{>} ?(.25)*\dir{}+(0,0)*{\bullet}+(-3,0)*{\scs f_1};
    (2,-9)*{\cbub{n-1+f_2}{}};
  (12,-2)*{n};(-6,-3)*{\scs };  (6,-3)*{\scs };
 \endxy} \nn
\end{equation}
where we used that negative degree bubbles are zero for the last
equality.
To reduce this right twist curl when $n<0$ we can cap the bottom of
equation (B5) with $-n$ dots to get
\begin{equation}
  \vcenter{\xy 0;/r.18pc/:
 (-4,6)*{};(4,6)*{} **\crv{(-4,-2) & (4,-2)}?(1)*\dir{>} ?(.25)*\dir{}+(0,0)*{\bullet}+(-5,-1)*{\scs -n};
  (10,2)*{n};
  \endxy}
\quad = \quad
 \beta_{n}\;\;
   \vcenter{\xy 0;/r.18pc/:
    (-4,-4)*{};(4,4)*{} **\crv{(-4,-1) & (4,1)}?(1)*\dir{<};?(0)*\dir{<};
    (4,-4)*{};(-4,4)*{} **\crv{(4,-1) & (-4,1)}?(1)*\dir{>};
    (-4,4)*{};(4,12)*{} **\crv{(-4,7) & (4,9)}?(1)*\dir{>};
    (4,4)*{};(-4,12)*{} **\crv{(4,7) & (-4,9)};
    (4,-4)*{};(-4,-4)*{} **\crv{(4,-8) & (-4,-8)} ?(.5)*\dir{}+(0,0)*{\bullet}+(1,-3)*{\scs -n};
  (8,8)*{n};(-6,-3)*{\scs };  (6,-3)*{\scs };
 \endxy}
  \quad + \quad
    \sum_{ \xy  (0,3)*{\scs g_1+g_2+|\lambda|};
    (0,0)*{\scs =-n-1};\endxy}
    \alpha_{\lambda}^{|\lambda|+g_2} e_{\lambda,n}
    \vcenter{\xy 0;/r.18pc/:
    (-4,6)*{};(4,6)*{} **\crv{(-4,-2) & (4,-2)}?(1)*\dir{>} ?(.25)*\dir{}+(0,0)*{\bullet}+(-3,0)*{\scs g_1};
    (2,-9)*{\ccbub{-n-1+g_2+1}{}};
  (12,-2)*{n};(-6,-3)*{\scs };  (6,-3)*{\scs };
 \endxy} \nn
  \end{equation}
But
\begin{equation}
   \vcenter{\xy 0;/r.18pc/:
    (-4,-4)*{};(4,4)*{} **\crv{(-4,-1) & (4,1)}?(1)*\dir{<};?(0)*\dir{<};
    (4,-4)*{};(-4,4)*{} **\crv{(4,-1) & (-4,1)}?(1)*\dir{>};
    (-4,4)*{};(4,12)*{} **\crv{(-4,7) & (4,9)}?(1)*\dir{>};
    (4,4)*{};(-4,12)*{} **\crv{(4,7) & (-4,9)};
    (4,-4)*{};(-4,-4)*{} **\crv{(4,-8) & (-4,-8)} ?(.5)*\dir{}+(0,0)*{\bullet}+(1,-3)*{\scs -n};
  (8,8)*{n};(-6,-3)*{\scs };  (6,-3)*{\scs };
 \endxy}
\;\;=\;\;
   \vcenter{\xy 0;/r.18pc/:
    (-4,-4)*{};(4,4)*{} **\crv{(-4,-1) & (4,1)}?(1)*\dir{<};?(0)*\dir{<};
    (4,-4)*{};(-4,4)*{} **\crv{(4,-1) & (-4,1)}?(1)*\dir{>};
    (-4,4)*{};(4,12)*{} **\crv{(-4,7) & (4,9)}?(1)*\dir{>};
    (4,4)*{};(-4,12)*{} **\crv{(4,7) & (-4,9)};
    (4,-4)*{};(-4,-4)*{} **\crv{(4,-8) & (-4,-8)} ?(.5)*\dir{}+(0,0)*{\bullet}+(1,-3)*{\scs -n-1};
    (3,2)*{\bullet};
  (8,8)*{n};(-6,-3)*{\scs };  (6,-3)*{\scs };
 \endxy}
\;\;+\;\; \;\; \vcenter{\xy 0;/r.18pc/:
    (-4,6)*{};(4,-4)*{} **\crv{(-4,1) & (4,-1)}?(1)*\dir{>};
    (4,6)*{};(-4,-4)*{} **\crv{(4,1) & (-4,-1)}?(0)*\dir{<};
    (-4,-4)*{};(4,-4)*{} **\crv{(-4,-9) & (4,-9)};
       (2,-16)*{\ccbub{-n-1}{}};
  (12,-2)*{n};(-6,-3)*{\scs };  (6,-3)*{\scs };
 \endxy}
\;\; \refequal{\scs {\bf (B3)}}\;\; c_n^-
 \vcenter{\xy 0;/r.18pc/:
    (-4,6)*{};(4,-4)*{} **\crv{(-4,1) & (4,-1)}?(1)*\dir{>};
    (4,6)*{};(-4,-4)*{} **\crv{(4,1) & (-4,-1)}?(0)*\dir{<};
    (-4,-4)*{};(4,-4)*{} **\crv{(-4,-9) & (4,-9)};
  (12,-2)*{n};(-6,-3)*{\scs };  (6,-3)*{\scs };
 \endxy} \nn
  \end{equation}
and
\begin{eqnarray}
 \sum_{ \xy  (0,3)*{\scs g_1+g_2+|\lambda|};
    (0,0)*{\scs =-n-1};\endxy}
    \alpha_{\lambda}^{|\lambda|+g_2} e_{\lambda,n}
    \vcenter{\xy 0;/r.18pc/:
    (-4,6)*{};(4,6)*{} **\crv{(-4,-2) & (4,-2)}?(1)*\dir{>} ?(.25)*\dir{}+(0,0)*{\bullet}+(-3,0)*{\scs g_1};
    (2,-9)*{\ccbub{-n-1+g_2+1}{}};
  (12,-2)*{n};(-6,-3)*{\scs };  (6,-3)*{\scs };
 \endxy} &=&
 \sum_{g=1}^{-n}\sum_{\lambda:|\lambda|\leq g-1}\alpha_{\lambda}^{g-1}e_{\lambda,n}e_{g-|\lambda|,n}
     \vcenter{\xy 0;/r.18pc/:
    (-4,6)*{};(4,6)*{} **\crv{(-4,-2) & (4,-2)}?(1)*\dir{>} ?(.5)*\dir{}+(0,0)*{\bullet}+(1,-3)*{\scs -n-g};
  (12,-2)*{n};(-6,-3)*{\scs };  (6,-3)*{\scs };
 \endxy} \\
 & \refequal{(B1)} & -
  \sum_{g=1}^{-n}\sum_{\lambda:|\lambda|=g}
  \alpha_{\lambda}^{g}e_{\lambda,n}c_n^-
     \vcenter{\xy 0;/r.18pc/:
    (-4,6)*{};(4,6)*{} **\crv{(-4,-2) & (4,-2)}?(1)*\dir{>} ?(.5)*\dir{}+(0,0)*{\bullet}+(1,-3)*{\scs -n-g};
  (12,-2)*{n};(-6,-3)*{\scs };  (6,-3)*{\scs };
 \endxy} \nn
\end{eqnarray}
Therefore for $n<0$ we have
\begin{equation}
   \vcenter{\xy 0;/r.18pc/:
    (-4,6)*{};(4,-4)*{} **\crv{(-4,1) & (4,-1)}?(1)*\dir{>};
    (4,6)*{};(-4,-4)*{} **\crv{(4,1) & (-4,-1)}?(0)*\dir{<};
    (-4,-4)*{};(4,-4)*{} **\crv{(-4,-9) & (4,-9)};
  (12,-5)*{n};(-6,-3)*{\scs };  (6,-3)*{\scs };
 \endxy}
 =
  \frac{1}{c_n^-\beta_{n}}     \vcenter{\xy 0;/r.18pc/:
    (-4,6)*{};(4,6)*{} **\crv{(-4,-2) & (4,-2)}?(1)*\dir{>} ?(.5)*\dir{}+(0,0)*{\bullet}+(1,-3)*{\scs -n-g};
  (12,-2)*{n};(-6,-3)*{\scs };  (6,-3)*{\scs };
 \endxy}
 \;\; + \;\; \frac{1}{\beta_{n}}
   \sum_{g=1}^{-n}\sum_{\lambda:|\lambda|=g}
  \alpha_{\lambda}^{g-1}e_{\lambda,n}
     \vcenter{\xy 0;/r.18pc/:
    (-4,6)*{};(4,6)*{} **\crv{(-4,-2) & (4,-2)}?(1)*\dir{>} ?(.5)*\dir{}+(0,0)*{\bullet}+(1,-3)*{\scs -n-g};
  (12,-2)*{n};(-6,-3)*{\scs };  (6,-3)*{\scs };
 \endxy} \nn
\end{equation}
using this fact, together with repeated application of the nilHecke
dot slide formula, we have
\begin{equation}
  \vcenter{\xy 0;/r.18pc/:
    (-4,6)*{};(4,-4)*{} **\crv{(-4,1) & (4,-1)}?(1)*\dir{>};
    (4,6)*{};(-4,-4)*{} **\crv{(4,1) & (-4,-1)}?(0)*\dir{<};
    (-4,-4)*{};(4,-4)*{} **\crv{(-4,-8) & (4,-8)}?(.5)*\dir{}+(0,0)*{\bullet}+(-1,-3)*{\scs j};
  (12,-2)*{n};(-6,-3)*{\scs };  (6,-3)*{\scs };
 \endxy}
 =
   \vcenter{\xy 0;/r.18pc/:
    (-4,6)*{};(4,-4)*{} **\crv{(-4,1) & (4,-1)}?(1)*\dir{>};
    (4,6)*{};(-4,-4)*{} **\crv{(4,1) & (-4,-1)}?(0)*\dir{<} ?(.25)*\dir{}+(0,0)*{\bullet}+(3,-1)*{\scs j};
    (-4,-4)*{};(4,-4)*{} **\crv{(-4,-8) & (4,-8)};
  (12,-2)*{n};(-6,-3)*{\scs };  (6,-3)*{\scs };
 \endxy} \;\; - \;\;
 \sum_{g_1+g_2=j}\vcenter{\xy 0;/r.18pc/:
    (-4,6)*{};(4,6)*{} **\crv{(-4,-2) & (4,-2)}?(1)*\dir{>} ?(.25)*\dir{}+(0,0)*{\bullet}+(-3,0)*{\scs g_1};
    (2,-9)*{\cbub{g_2}{}};
  (12,-2)*{n};(-6,-3)*{\scs };  (6,-3)*{\scs };
 \endxy} \nn
\end{equation}
If we rewrite the last summation to emphasize the degree of the
bubble we can summarize our findings by the equation
\begin{equation} \label{eq_ugly_curldot}
  \vcenter{\xy 0;/r.18pc/:
    (-4,6)*{};(4,-4)*{} **\crv{(-4,1) & (4,-1)}?(1)*\dir{>};
    (4,6)*{};(-4,-4)*{} **\crv{(4,1) & (-4,-1)}?(0)*\dir{<};
    (-4,-4)*{};(4,-4)*{} **\crv{(-4,-8) & (4,-8)}?(.5)*\dir{}+(0,0)*{\bullet}+(-1,-3)*{\scs j};
  (8,-8)*{n};(-6,-3)*{\scs };  (6,-3)*{\scs };
 \endxy}
 \; = \;
\left\{ \begin{array}{ccl}
  0 & \quad & \text{for $n>0$ and $j<n$} \\ & &\\
  -  \xy
  (0,.4)*{\sum};
  (0,-3.3)*{\scs f_1+f_2=j-n};
  \endxy
  \vcenter{\xy 0;/r.18pc/:
    (-4,6)*{};(4,6)*{} **\crv{(-4,-2) & (4,-2)}?(1)*\dir{>} ?(.25)*\dir{}+(0,0)*{\bullet}+(-3,0)*{\scs f_1};
    (2,-9)*{\cbub{n-1+f_2}{}};
  (12,-2)*{n};(-6,-3)*{\scs };  (6,-3)*{\scs };
 \endxy}  & \quad & \text{for $n > 0$ and $j\geq n$}
 \\ & &\\
   \frac{1}{c_n^-\beta_{n}}
  \vcenter{\xy 0;/r.18pc/:
    (-4,6)*{};(4,6)*{} **\crv{(-4,-2) & (4,-2)}?(1)*\dir{>} ?(.5)*\dir{}+(0,0)*{\bullet}+(1,-3)*{\scs j-n-g};
  (12,-2)*{n};(-6,-3)*{\scs };  (6,-3)*{\scs };
 \endxy}
 \;\; + \;\; \frac{1}{\beta_{n}}
      \xy
  (0,.4)*{\sum};
  (0,-3.3)*{\scs g=1};(0,3.7)*{\scs -n};
  \endxy\;
   \xy
  (0,.4)*{\sum};
  (2,-3.3)*{\scs \lambda:|\lambda|=g};
  \endxy
  \alpha_{\lambda}^{g}e_{\lambda,n}
     \vcenter{\xy 0;/r.18pc/:
    (-4,6)*{};(4,6)*{} **\crv{(-4,-2) & (4,-2)}?(1)*\dir{>} ?(.5)*\dir{}+(0,0)*{\bullet}+(1,-3)*{\scs j-n-g};
  (12,-2)*{n};(-6,-3)*{\scs };  (6,-3)*{\scs };
 \endxy}
 & \quad &  \\ & &\text{for $n<0$}\\
  -
    \xy
  (0,.4)*{\sum};
  (2,-3.3)*{\scs g=-n+1}; (0,3.8)*{\scs j-n};
  \endxy
 \vcenter{\xy 0;/r.18pc/:
    (-4,6)*{};(4,6)*{} **\crv{(-4,-2) & (4,-2)}?(1)*\dir{>} ?(.5)*\dir{}+(0,0)*{\bullet}+(1,-3)*{\scs j-n-g};
    (2,-11)*{\cbub{n-1+g}{}};
  (12,-2)*{n};(-6,-3)*{\scs };  (6,-3)*{\scs };
 \endxy} & &
 \end{array}\right.
\end{equation}
The case when $n=0$ must be treated separately as this case is not covered by the equations we have seen so far.

These formulas seem rather chaotic with an intricate dependence on
$n$ and $j$.  However, they can be presented in a uniform way by defining
formal symbols as follows:

For $n=0$ we have that
\begin{equation}
  \deg\left(  \text{$\xy 0;/r.18pc/:
  (14,8)*{n};
  (-3,-8)*{};(3,8)*{} **\crv{(-3,-1) & (3,1)}?(1)*\dir{>};?(0)*\dir{>};
    (3,-8)*{};(-3,8)*{} **\crv{(3,-1) & (-3,1)}?(1)*\dir{>};
  (-3,-12)*{\bbsid};  (-3,8)*{\bbsid};
  (3,8)*{}="t1";  (9,8)*{}="t2";
  (3,-8)*{}="t1'";  (9,-8)*{}="t2'";
   "t1";"t2" **\crv{(3,14) & (9, 14)};
   "t1'";"t2'" **\crv{(3,-14) & (9, -14)};
   "t2'";"t2" **\dir{-} ?(.5)*\dir{<};
   (9,0)*{}; (-6,-8)*{\scs };
 \endxy$} \right) = 0, \qquad \qquad
   \deg\left(\text{$ \xy 0;/r.18pc/:
  (-12,8)*{n};
   (-3,-8)*{};(3,8)*{} **\crv{(-3,-1) & (3,1)}?(1)*\dir{>};?(0)*\dir{>};
    (3,-8)*{};(-3,8)*{} **\crv{(3,-1) & (-3,1)}?(1)*\dir{>};
  (3,-12)*{\bbsid};
  (3,8)*{\bbsid}; (6,-8)*{\scs };
  (-9,8)*{}="t1";
  (-3,8)*{}="t2";
  (-9,-8)*{}="t1'";
  (-3,-8)*{}="t2'";
   "t1";"t2" **\crv{(-9,14) & (-3, 14)};
   "t1'";"t2'" **\crv{(-9,-14) & (-3, -14)};
  "t1'";"t1" **\dir{-} ?(.5)*\dir{<};
 \endxy$}\right) = 0. \nn
\end{equation}
Since the space of degree zero endomorphisms $\cal{E}\onen$ is one
dimensional in $\Ucat$, we must have that both of the above morphisms are
multiples of the identity 2-morphism $\Id_{\cal{E}\onen}$.  We fix
the coefficients as follows:
\begin{equation} \label{eq_def_czeropm}
    \text{$\xy 0;/r.18pc/:
  (14,8)*{n};
  (-3,-8)*{};(3,8)*{} **\crv{(-3,-1) & (3,1)}?(1)*\dir{>};?(0)*\dir{>};
    (3,-8)*{};(-3,8)*{} **\crv{(3,-1) & (-3,1)}?(1)*\dir{>};
  (-3,-12)*{\bbsid};  (-3,8)*{\bbsid};
  (3,8)*{}="t1";  (9,8)*{}="t2";
  (3,-8)*{}="t1'";  (9,-8)*{}="t2'";
   "t1";"t2" **\crv{(3,14) & (9, 14)};
   "t1'";"t2'" **\crv{(3,-14) & (9, -14)};
   "t2'";"t2" **\dir{-} ?(.5)*\dir{<};
   (9,0)*{}; (-6,-8)*{\scs };
 \endxy$}  = - c_0^+\;
    \xy
  (5,4)*{n};
  (0,0)*{\bbe{}};(-2,-8)*{\scs };
 \endxy, \qquad \qquad
   \text{$ \xy 0;/r.18pc/:
  (-12,8)*{n};
   (-3,-8)*{};(3,8)*{} **\crv{(-3,-1) & (3,1)}?(1)*\dir{>};?(0)*\dir{>};
    (3,-8)*{};(-3,8)*{} **\crv{(3,-1) & (-3,1)}?(1)*\dir{>};
  (3,-12)*{\bbsid};
  (3,8)*{\bbsid}; (6,-8)*{\scs };
  (-9,8)*{}="t1";
  (-3,8)*{}="t2";
  (-9,-8)*{}="t1'";
  (-3,-8)*{}="t2'";
   "t1";"t2" **\crv{(-9,14) & (-3, 14)};
   "t1'";"t2'" **\crv{(-9,-14) & (-3, -14)};
  "t1'";"t1" **\dir{-} ?(.5)*\dir{<};
 \endxy$} =   c_0^- \xy
  (5,4)*{n};
  (0,0)*{\bbe{}};(-2,-8)*{\scs };
 \endxy
\end{equation}

\begin{defn}[Fake Bubbles] \hfill \label{def_fake}
Define the fake bubbles as
\begin{itemize}
 \item For $n=0$ set
  \begin{equation}
\vcenter{\xy 0;/r.18pc/:
    (2,-11)*{\cbub{-1}{}};
  (12,-2)*{n};
 \endxy} \;\; = c_0^+ \Id_{\mathbf{1}_0}, \qquad \quad
 \vcenter{\xy 0;/r.18pc/:
    (2,-11)*{\ccbub{-1}{}};
  (12,-2)*{n};
 \endxy} \;\; = c_0^- \Id_{\mathbf{1}_0},
 \end{equation}
  \item For $n<0$  define
  \begin{equation} \label{eq_fake_nleqz}
  \vcenter{\xy 0;/r.18pc/:
    (2,-11)*{\cbub{n-1+j}{}};
  (12,-2)*{n};
 \endxy} \;\; = \;\;
\left\{
 \begin{array}{ccl}
   - \frac{1}{\beta_{n}}\xsum{\lambda: |\lambda|=j}{}
 \alpha_{\lambda}^{j}(n)   \vcenter{\xy 0;/r.18pc/:
    (2,-3)*{\ccbub{-n-1+\lambda_1}{}};
    (13,-2)*{\cdots};
    (28,-3)*{\ccbub{-n-1+\lambda_m}{}};
  (12,8)*{n};
 \endxy}  & \quad & \text{if $0 \leq j < -n+1$}
 \\ \\
   0 & \quad & \text{$j<0$.}
 \end{array}
 \right.
 \end{equation}

  \item For $n>0$ define
  \begin{equation}
  \vcenter{\xy 0;/r.18pc/:
    (2,-11)*{\ccbub{-n-1+j}{}};
  (12,-2)*{n};
 \endxy} \;\; = \;\;
\left\{
 \begin{array}{ccl}
  - \frac{1}{\beta_{n}}\xsum{\lambda: |\lambda|=j}{}
 \alpha_{\lambda}^{j}(n)   \vcenter{\xy 0;/r.18pc/:
    (2,-3)*{\cbub{n-1+\lambda_1}{}};
    (13,-2)*{\cdots};
    (28,-3)*{\cbub{n-1+\lambda_m}{}};
  (12,8)*{n};
 \endxy} & \quad & \text{if $0 \leq j < n+1$}
 \\ \\
   0 & \quad & \text{$j<0$.}
 \end{array}
 \right.
 \end{equation}
\end{itemize}
\end{defn}

Although the labels are negative for fake bubbles, one can check that their
degree are still positive if we take the degree of a negative labelled dot to be $2$ times the label.  Also note the clockwise oriented degree zero
fake bubbles, with $j=0$ is equal to
$-\frac{\alpha_{\emptyset}^0}{\beta_{n}}= -\frac{1}{c_n^-\beta_{n}}$ and the
degree zero counter-clockwise oriented bubble is equal to
$-\frac{\alpha_{\emptyset}^0}{\beta_{n}}= -\frac{1}{c_n^+\beta_{n}}$.

With these definitions \eqref{eq_ugly_curldot} can be written in a
compact way as
\begin{equation}
  \vcenter{\xy 0;/r.18pc/:
    (-4,6)*{};(4,-4)*{} **\crv{(-4,1) & (4,-1)}?(1)*\dir{>};
    (4,6)*{};(-4,-4)*{} **\crv{(4,1) & (-4,-1)}?(0)*\dir{<};
    (-4,-4)*{};(4,-4)*{} **\crv{(-4,-8) & (4,-8)}?(.5)*\dir{}+(0,0)*{\bullet}+(-1,-3)*{\scs j};
  (8,-8)*{n};(-6,-3)*{\scs };  (6,-3)*{\scs };
 \endxy}
 \; = \;
   -  \xy
  (0,.4)*{\sum};
  (0,-3.3)*{\scs f_1+f_2=j-n};
  \endxy
  \vcenter{\xy 0;/r.18pc/:
    (-4,6)*{};(4,6)*{} **\crv{(-4,-2) & (4,-2)}?(1)*\dir{>} ?(.25)*\dir{}+(0,0)*{\bullet}+(-3,0)*{\scs f_1};
    (2,-9)*{\cbub{n-1+f_2}{}};
  (12,-2)*{n};(-6,-3)*{\scs };  (6,-3)*{\scs };
 \endxy} \nn
\end{equation}
where the above formula is now valid for all $n \in \Z$ and $j \geq
0$.  Our convention is that the summation is zero when $j-n<0$.

%
\subsubsection{Another interpretation of fake bubbles} \label{subsubsec_infinitegrass}
%

One can check that (for $n \neq 0$) the definition of fake bubbles in Definition~\ref{def_fake} is equivalent to requiring that the fake bubbles are inductively defined by the equations resulting from comparing the homogeneous terms of degree less than or equal to $|n|$ on both sides of equation
\begin{center}
\begin{eqnarray}
 \makebox[0pt]{ $
\left( \xy 0;/r.15pc/:
 (0,0)*{\ccbub{-n-1}{}};
  (4,8)*{n};
 \endxy
 +
 \xy 0;/r.15pc/:
 (0,0)*{\ccbub{-n-1+1}{}};
  (4,8)*{n};
 \endxy t
 + \cdots +
 \xy 0;/r.15pc/:
 (0,0)*{\ccbub{-n-1+\beta}{}};
  (4,8)*{n};
 \endxy t^{\beta}
 + \cdots
\right)
%
\left( \xy 0;/r.15pc/:
 (0,0)*{\cbub{n-1}{}};
  (4,8)*{n};
 \endxy
 + \xy 0;/r.15pc/:
 (0,0)*{\cbub{n-1+1}{}};
  (4,8)*{n};
 \endxy t
 +\cdots +
 \xy 0;/r.15pc/:
 (0,0)*{\cbub{n-1+\alpha}{}};
 (4,8)*{n};
 \endxy t^{\alpha}
 + \cdots
\right) =-\frac{1}{\beta_n} .$ } \nn \\ \label{eq_infinite_Grass}
\end{eqnarray}
\end{center}
If $\beta_n=-1$, then Equation~\eqref{eq_infinite_Grass} can be thought of as a limit of equation \eqref{eq_chern_rel} defining the cohomology ring of $Gr(k,N)$. We refer to the above equation the {\em infinite Grassmannian equation}.

We will show in Section~\ref{subsubsec_rescaling} that it is always possible to choose the coefficients $c_n^{\pm}=1$ and $\beta_n=-1$.  In that case, the infinite Grassmannian equation takes on a simplified form.

\begin{example} For $n>0$, $c_n^+=1$, and $\beta_n=-1$, the first few fake bubbles are collected below
\begin{enumerate}[a)]
  \item $ \xy 0;/r.18pc/:
 (-12,0)*{\ccbub{-n-1+0}{}};
 (-8,8)*{n};
 \endxy :=1$

 \item
 $
  \xy 0;/r.18pc/:
 (-12,0)*{\ccbub{-n-1+1}{}};
 (-8,8)*{n};(-8,13)*{};
 \endxy := -
 \xy 0;/r.19pc/:
 (0,0)*{\cbub{n-1+1}{}};
 (8,8)*{n}; (-8,13)*{};
 \endxy
 $

 \item
 $\xy 0;/r.18pc/:
 (-12,0)*{\ccbub{-n-1+2}{}};
 (-8,8)*{n};(0,10)*{}; (-8,13)*{};
 \endxy := \;\;
-
 \xy 0;/r.19pc/:
 (0,0)*{\cbub{n-1+2}{}};
 (8,8)*{n};(0,10)*{}; (-8,13)*{};
 \endxy
 -
 \xy 0;/r.19pc/:
 (0,0)*{\cbub{n-1+1}{}};
 (20,0)*{\ccbub{-n-1+1}{}};
 (8,8)*{n};(0,10)*{}; (-8,13)*{};
 \endxy
 \;\; = \;\;
-
 \xy 0;/r.19pc/:
 (0,0)*{\cbub{n-1+2}{}};
 (8,8)*{n};(0,10)*{}; (-8,13)*{};
 \endxy
 +
 \xy 0;/r.19pc/:
 (0,0)*{\cbub{n-1+1}{}};
 (20,0)*{\cbub{n-1+1}{}};
 (8,8)*{n};(0,10)*{}; (-8,13)*{};
 \endxy  $
\end{enumerate}
Note that in general, the fake bubbles for $n> 0$ are inductively defined as
\begin{equation} \label{eq_homo_inf_grass}
\xy 0;/r.18pc/:
 (-12,0)*{\ccbub{-n-1+j}{}};
 (-8,8)*{n};(0,10)*{};
 \endxy
 \;\; = \;\;
-\sum_{
 \xy (0,3)*{\scs \ell_1+\ell_2=j}; (0,-1)*{\scs  \ell_1 \geq 1}; \endxy}
 \xy 0;/r.19pc/:
 (0,0)*{\cbub{n-1+\ell_1}{}};
 (20,0)*{\ccbub{-n-1+\ell_2}{}};(0,10)*{};
 (8,8)*{n};
 \endxy
\end{equation}
The clockwise oriented fake bubbles for $n<0$ are defined similarly.
\end{example}

While we have used the equations resulting from the homogeneous terms with degree less than or equal to $|n|$ in the infinite Grassmannian equation to define fake bubbles, one can show that for all $\alpha \geq 0$ the equation
\begin{equation} \label{eq_homogeneous_term}
\sum_{\ell_1+\ell_2=\alpha}
 \xy 0;/r.19pc/:
 (0,0)*{\cbub{n-1+\ell_1}{}};
 (20,0)*{\ccbub{-n-1+\ell_2}{}};(0,10)*{};
 (8,8)*{n};
 \endxy \quad = \quad 0.
\end{equation}
must also hold in the 2-category $\Ucat$. To see this, consider the diagram
\begin{equation}
   \vcenter{   \xy 0;/r.18pc/:
    (-4,-4)*{};(4,4)*{} **\crv{(-3,-1) & (3,1)}?(1)*\dir{>};
    (4,-4)*{};(-4,4)*{} **\crv{(3,-1) & (-3,1)}?(1)*\dir{>};;
    (-4,4)*{};(-12,4)*{} **\crv{(-4,8) & (-12,8)};
    (-4,-4)*{};(-12,-4)*{} **\crv{(-4,-8) & (-12,-8)};
    (4,4)*{};(12,4)*{} **\crv{(4,8) & (12,8)};
    (4,-4)*{};(12,-4)*{} **\crv{(4,-8) & (12,-8)};
    (12,-4)*{};(12,4)*{} **\dir{-};
    (-12,-4)*{};(-12,4)*{} **\dir{-};
    (-12,0)*{\bullet}+(-9,1)*{\scs -n-1+\alpha_1};
    (12,0)*{\bullet}+(9,1)*{\scs  n-1+\alpha_2};
  (2,9)*{n};
 \endxy}
\end{equation}
where $\min(0,n+1) \leq \alpha_1$ and $\min(0,-n+1) \leq \alpha_2$.  Simplifying this diagram in two possible ways using equation \eqref{eq_ugly_curldot}, and a similar equation derived for the other dotted curl, implies \eqref{eq_homogeneous_term} with $\alpha=\alpha_1+\alpha_2-1$.  Thus, the definition of fake bubbles, and the relations in $\Ucat$ imply that all homogeneous terms in the infinite Grassmannian relation hold in $\Ucat$.

%
\subsubsection{Fake bubbles and symmetric functions} \label{subsubsec_fake_symm}
%

The infinite Grassmannian equation described above suggests a natural interpretation of fake bubbles.  To simplify the exposition we continue to assume that $c_n^{\pm}=1$ and $\beta_n=-1$.

The map $\phi^n$ from equation \eqref{eq_symm_iso} provides an  identification between elementary symmetric functions and dotted bubbles of a given orientation.  For example, when $n>0$ we have a bijection between the elementary symmetric function of degree $r$ and the clockwise oriented dotted bubble of degree $r$.  It is natural to wonder if there are any diagrams corresponding to complete symmetric functions.
The image of the Equation~\eqref{eq_eh_rel} relating the elementary and complete symmetric functions is given by
\[\sum_{\ell=0}^{\alpha} (-1)^r e_r h_{\alpha-r} = \delta_{\alpha,0}
  \quad \xymatrix{
 \ar[r]^{\phi_n} &
 } \quad
 \sum_{\ell=0}^{\alpha}
\xy 0;/r.15pc/:
 (-10,1)*{\cbub{(n-1)+\alpha-\ell}{}};
  (14,-1)*{\ccbub{(-n-1)+\ell}{}};
 (3,10)*{n};
 \endxy
 = \delta_{\alpha,0}.
\]
In particular, $\phi^n$ maps Equation~\eqref{eq_eh_rel} to the corresponding equation given by equating homogeneous terms of degree $\alpha$ in the infinite Grassmannian equation
\[
\left( \xy 0;/r.14pc/:
 (0,0)*{\ccbub{-n-1}{}};
  (4,8)*{n};
 \endxy
 + \cdots +
 \xy 0;/r.15pc/:
 (0,0)*{\ccbub{-n-1+\alpha}{}};
  (4,8)*{n};
 \endxy t^{\alpha}
 + \cdots
\right)
\left( \xy 0;/r.15pc/:
 (0,0)*{\cbub{n-1}{}};
  (4,8)*{n};
 \endxy
 + \cdots +
 \xy 0;/r.14pc/:
 (0,0)*{\cbub{n-1+\alpha}{}};
 (4,8)*{n};
 \endxy t^{\alpha}
 + \cdots
\right) =1.\]
Therefore it is natural to conclude that for $n>0$
\begin{equation}
  \phi^n(h_r)\;\; = \;\;  (-1)^r
    \xy 0;/r.18pc/:
 (-12,0)*{\ccbub{(-n-1)+r}{}};
 (-8,8)*{n};(0,10)*{};
 \endxy
\end{equation}
For $0 \leq r \leq n$ we can take this as the definition of fake bubbles.  A similar interpretation applies to fake bubbles for $n<0$.
\begin{quote}
{\em Up to a sign, the coefficients expressing a fake bubble in terms of real bubbles are the same as the coefficients expressing a complete symmetric function in terms of elementary symmetric functions.}
\end{quote}

%
\subsection{The 2-categories $\Ucatt$}
%
Let $\chi$ be a set of invertible elements
\begin{equation}
  \beta_n \quad \text{for $n\in \Z$}, \qquad c_n^+ \quad \text{for $n \geq 0$},
  \qquad c_n^- \quad \text{for $n \leq 0$}, \nn
\end{equation}
in $\Bbbk$.  Introduce coefficients $\alpha_{\lambda}^{\ell}(n)$ for
$|\lambda|\leq \ell \leq |n|-1$ satisfying
\begin{equation}
\delta_{b,0} = \sum_{\lambda: |\lambda| \leq b}
\alpha_{\lambda}^{\ell}(n)e_{\lambda,n} e_{b-|\lambda|,n}. \nn
\end{equation}
We write $\alpha_{\lambda}^{\ell}(n)$ for all $\ell$ and $\lambda$ and $n$, but set $\alpha_{\lambda}^{\ell}(n)=0$ unless $|\lambda|\leq \ell \leq n-1$.

Thus far, from our study of invertibility of the maps $\zeta_+$ and $\zeta_-$ we have defined a family of categories $\Ucatt$. Here we use fake bubbles to compactly define this 2-category using a small set of relations that imply relations {\bf (A1)}--{\bf (A5)} and {\bf (B1)}--{\bf (B5)}.

\begin{defn} \label{defn-Ucatt}
$\Ucatt$ is the additive 2-category consisting
of
\begin{itemize}
  \item objects: $n$ for $n \in \Z$.
\end{itemize}
The Hom $\Ucatt(n,n')$ between two objects $n$, $n'$ is an additive $\Bbbk$-linear category:
\begin{itemize}
  \item objects of $\Ucatt(n,n')$: for a signed sequence $\ep = (\epsilon_1,\epsilon_2, \dots, \epsilon_m)$, where $\epsilon_1, \dots, \epsilon_m \in \{ +,-\}$, define
 $$\cal{E}_{\ep} := \cal{E}_{\epsilon_1} \cal{E}_{\epsilon_2}\dots \cal{E}_{\epsilon_m}$$
where $\cal{E}_{+}:= \cal{E}$ and $\cal{E}_{-}:= \cal{F}$.  An object of $\Ucatt(n,n')$, called a 1-morphism
in $\Ucatt$, is a formal finite direct sum of 1-morphisms
  \[
 \cal{E}_{\ep} \onen\{t\} =\onenn{n'} \cal{E}_{\ep} \onen\{t\}
  \]
for any $t\in \Z$ and signed sequence $\ep$ such that $n'=n+\sum_{j=1}^m \epsilon_j2 $.

  \item morphisms of $\Ucatt(n,n')$: given objects $\cal{E}_{\ep} \onen\{t\}
  ,\cal{E}_{\ep'} \onen\{t'\} \in \Ucatt(n,n')$, the Hom sets
$$\Hom_{\Ucatt}(\cal{E}_{\ep} \onen\{t\},\cal{E}_{\ep'} \onen\{t'\})$$ of $\Ucatt(n,n')$ are $\Bbbk$-vector spaces given by linear combinations of diagrams with degree $t-t'$, modulo certain relations, built from composites of identity 2-morphisms and for each $n\in \Z$ generating 2-morphisms\footnote{See example~\ref{ex_generators} for additional examples of the generating 2-morphisms and their gradings. }
\begin{align}
  \xy 0;/r.17pc/:
 (0,7);(0,-7); **\dir{-} ?(.75)*\dir{>}+(2.3,0)*{\scriptstyle{}};
 (0,0)*{\bullet};
 (7,-3)*{ \scs n};
 (-9,-3)*{\scs  n+2};
 (-10,0)*{};(10,0)*{};
 \endxy &\maps \cal{E}\onen\{t+2\} \to \cal{E}\onen\{t\}  & \quad
 &
    \xy 0;/r.17pc/:
 (0,7);(0,-7); **\dir{-} ?(.75)*\dir{<}+(2.3,0)*{\scriptstyle{}};
 (0,0)*{\bullet};
 (7,-3)*{ \scs n};
 (-9,-3)*{\scs  n-2};
 (-10,0)*{};(10,0)*{};
 \endxy\maps \cal{F}\onen\{t+2\} \to \cal{F}\onen\{t\}  \nn \\
   & & & \nn \\
 \xy 0;/r.17pc/:
  (0,0)*{\xybox{
    (-4,-4)*{};(4,4)*{} **\crv{(-4,-1) & (4,1)}?(1)*\dir{>} ;
    (4,-4)*{};(-4,4)*{} **\crv{(4,-1) & (-4,1)}?(1)*\dir{>};
     (8,1)*{\scs n};
     (-12,0)*{};(12,0)*{};
     }};
  \endxy &\maps \cal{E}\cal{E}\onen \{t-2\} \to \cal{E}\cal{E}\onen\{t\}  &
  &
 \xy 0;/r.17pc/:
  (0,0)*{\xybox{
    (-4,4)*{};(4,-4)*{} **\crv{(-4,1) & (4,-1)}?(1)*\dir{>} ;
    (4,4)*{};(-4,-4)*{} **\crv{(4,1) & (-4,-1)}?(1)*\dir{>};
     (8,1)*{\scs n};
     (-12,0)*{};(12,0)*{};
     }};
  \endxy \maps \cal{F}\cal{F}\onen \{t-2\} \to \cal{F}\cal{F}\onen\{t\}  \nn \\
  & & & \nn \\
     \xy 0;/r.17pc/:
    (0,-3)*{\bbpef{}};
    (8,-5)*{\scs  n};
    (-12,0)*{};(12,0)*{};
    \endxy &\maps \onen\{t+1+n\} \to \cal{F}\cal{E}\onen\{t\}  &
    &
   \xy 0;/r.17pc/:
    (0,-3)*{\bbpfe{}};
    (8,-5)*{\scs n};
    (-12,0)*{};(12,0)*{};
    \endxy \maps \onen\{t+1-n\} \to\cal{E}\cal{F}\onen\{t\}  \nn \\
      & & & \nn \\
  \xy 0;/r.17pc/:
    (0,0)*{\bbcef{}};
    (8,4)*{\scs  n};
    (-12,0)*{};(12,0)*{};
    \endxy & \maps \cal{F}\cal{E}\onen\{t+1+n\} \to\onen\{t\}  &
    &
 \xy 0;/r.17pc/:
    (0,0)*{\bbcfe{}};
    (8,4)*{\scs  n};
    (-12,0)*{};(12,0)*{};
    \endxy \maps\cal{E}\cal{F}\onen\{t+1-n\} \to\onen\{t\} \nn
\end{align}
\end{itemize}
such that the following identities hold:
\begin{enumerate}
\item  cups and caps are biadjointness morphisms up to grading shifts. That is, equations \eqref{eq_biadjointness1} and \eqref{eq_biadjointness2} hold.

  \item All 2-morphisms are cyclic with respect to the above biadjoint structure.  This is ensured by imposing the relations:
\begin{equation} \label{eq_cyclic_dot}
    \xy 0;/r.17pc/:
    (-8,5)*{}="1";
    (0,5)*{}="2";
    (0,-5)*{}="2'";
    (8,-5)*{}="3";
    (-8,-10);"1" **\dir{-};
    "2";"2'" **\dir{-} ?(.5)*\dir{<};
    "1";"2" **\crv{(-8,12) & (0,12)} ?(0)*\dir{<};
    "2'";"3" **\crv{(0,-12) & (8,-12)}?(1)*\dir{<};
    "3"; (8,10) **\dir{-};
    (15,-9)*{ n+2};
    (-12,9)*{n};
    (0,4)*{\txt\large{$\bullet$}};
    (10,8)*{\scs };
    (-10,-8)*{\scs };
    \endxy
    \quad = \quad
      \xy 0;/r.17pc/:
 (0,10);(0,-10); **\dir{-} ?(.75)*\dir{<}+(2.3,0)*{\scriptstyle{}}
 ?(.1)*\dir{ }+(2,0)*{\scs };
 (0,0)*{\txt\large{$\bullet$}};
 (-6,5)*{ n};
 (8,5)*{ n +2};
 (-10,0)*{};(10,0)*{};(-2,-8)*{\scs };
 \endxy
    \quad = \quad
    \xy 0;/r.17pc/:
    (8,5)*{}="1";
    (0,5)*{}="2";
    (0,-5)*{}="2'";
    (-8,-5)*{}="3";
    (8,-10);"1" **\dir{-};
    "2";"2'" **\dir{-} ?(.5)*\dir{<};
    "1";"2" **\crv{(8,12) & (0,12)} ?(0)*\dir{<};
    "2'";"3" **\crv{(0,-12) & (-8,-12)}?(1)*\dir{<};
    "3"; (-8,10) **\dir{-};
    (15,9)*{n+2};
    (-12,-9)*{n};
    (0,4)*{\txt\large{$\bullet$}};
    (-10,8)*{\scs };
    (10,-8)*{\scs };
    \endxy
\end{equation}
\begin{equation} \label{eq_cyclic_cross-gen}
\xy 0;/r.19pc/:
  (0,0)*{\xybox{
    (-4,-4)*{};(4,4)*{} **\crv{(-4,-1) & (4,1)}?(1)*\dir{>};
    (4,-4)*{};(-4,4)*{} **\crv{(4,-1) & (-4,1)};
     (4,4)*{};(-18,4)*{} **\crv{(4,16) & (-18,16)} ?(1)*\dir{>};
     (-4,-4)*{};(18,-4)*{} **\crv{(-4,-16) & (18,-16)} ?(1)*\dir{<}?(0)*\dir{<};
     (18,-4);(18,12) **\dir{-};(12,-4);(12,12) **\dir{-};
     (-18,4);(-18,-12) **\dir{-};(-12,4);(-12,-12) **\dir{-};
     (8,1)*{ n};
     (-10,0)*{};(10,0)*{};
      (4,-4)*{};(12,-4)*{} **\crv{(4,-10) & (12,-10)}?(1)*\dir{<}?(0)*\dir{<};
      (-4,4)*{};(-12,4)*{} **\crv{(-4,10) & (-12,10)}?(1)*\dir{>}?(0)*\dir{>};
     }};
  \endxy
\quad =  \quad \xy
  (0,0)*{\xybox{
    (-4,-4)*{};(4,4)*{} **\crv{(-4,-1) & (4,1)}?(0)*\dir{<} ;
    (4,-4)*{};(-4,4)*{} **\crv{(4,-1) & (-4,1)}?(0)*\dir{<};
     (-8,0)*{ n};
     (-12,0)*{};(12,0)*{};
     }};
  \endxy \quad =  \quad
 \xy 0;/r.19pc/:
  (0,0)*{\xybox{
    (4,-4)*{};(-4,4)*{} **\crv{(4,-1) & (-4,1)}?(1)*\dir{>};
    (-4,-4)*{};(4,4)*{} **\crv{(-4,-1) & (4,1)};
     (-4,4)*{};(18,4)*{} **\crv{(-4,16) & (18,16)} ?(1)*\dir{>};
     (4,-4)*{};(-18,-4)*{} **\crv{(4,-16) & (-18,-16)} ?(1)*\dir{<}?(0)*\dir{<};
     (-18,-4);(-18,12) **\dir{-};(-12,-4);(-12,12) **\dir{-};
     (18,4);(18,-12) **\dir{-};(12,4);(12,-12) **\dir{-};
     (8,1)*{ n};
     (-10,0)*{};(10,0)*{};
     (-4,-4)*{};(-12,-4)*{} **\crv{(-4,-10) & (-12,-10)}?(1)*\dir{<}?(0)*\dir{<};
      (4,4)*{};(12,4)*{} **\crv{(4,10) & (12,10)}?(1)*\dir{>}?(0)*\dir{>};
     }};
  \endxy
\end{equation}
This condition ensures that isotopic diagrams represent
the same 2-morphism in $\Ucatt$.

\item The nilHecke relations hold:
 \begin{equation} \label{eq_nilHecke1}
  \vcenter{\xy 0;/r.18pc/:
    (-4,-4)*{};(4,4)*{} **\crv{(-4,-1) & (4,1)}?(1)*\dir{>};
    (4,-4)*{};(-4,4)*{} **\crv{(4,-1) & (-4,1)}?(1)*\dir{>};
    (-4,4)*{};(4,12)*{} **\crv{(-4,7) & (4,9)}?(1)*\dir{>};
    (4,4)*{};(-4,12)*{} **\crv{(4,7) & (-4,9)}?(1)*\dir{>};
 \endxy}
 =0, \qquad \quad
 \vcenter{
 \xy 0;/r.18pc/:
    (-4,-4)*{};(4,4)*{} **\crv{(-4,-1) & (4,1)}?(1)*\dir{>};
    (4,-4)*{};(-4,4)*{} **\crv{(4,-1) & (-4,1)}?(1)*\dir{>};
    (4,4)*{};(12,12)*{} **\crv{(4,7) & (12,9)}?(1)*\dir{>};
    (12,4)*{};(4,12)*{} **\crv{(12,7) & (4,9)}?(1)*\dir{>};
    (-4,12)*{};(4,20)*{} **\crv{(-4,15) & (4,17)}?(1)*\dir{>};
    (4,12)*{};(-4,20)*{} **\crv{(4,15) & (-4,17)}?(1)*\dir{>};
    (-4,4)*{}; (-4,12) **\dir{-};
    (12,-4)*{}; (12,4) **\dir{-};
    (12,12)*{}; (12,20) **\dir{-};
  (18,8)*{n};
\endxy}
 \;\; =\;\;
 \vcenter{
 \xy 0;/r.18pc/:
    (4,-4)*{};(-4,4)*{} **\crv{(4,-1) & (-4,1)}?(1)*\dir{>};
    (-4,-4)*{};(4,4)*{} **\crv{(-4,-1) & (4,1)}?(1)*\dir{>};
    (-4,4)*{};(-12,12)*{} **\crv{(-4,7) & (-12,9)}?(1)*\dir{>};
    (-12,4)*{};(-4,12)*{} **\crv{(-12,7) & (-4,9)}?(1)*\dir{>};
    (4,12)*{};(-4,20)*{} **\crv{(4,15) & (-4,17)}?(1)*\dir{>};
    (-4,12)*{};(4,20)*{} **\crv{(-4,15) & (4,17)}?(1)*\dir{>};
    (4,4)*{}; (4,12) **\dir{-};
    (-12,-4)*{}; (-12,4) **\dir{-};
    (-12,12)*{}; (-12,20) **\dir{-};
  (10,8)*{n};
\endxy}
  \end{equation}
\begin{eqnarray} \label{eq_nilHecke2}
  \xy
  (3,4);(3,-4) **\dir{-}?(0)*\dir{<}+(2.3,0)*{};
  (-3,4);(-3,-4) **\dir{-}?(0)*\dir{<}+(2.3,0)*{};
  (8,2)*{n};
 \endxy
 \quad =
\xy
  (0,0)*{\xybox{
    (-4,-4)*{};(4,4)*{} **\crv{(-4,-1) & (4,1)}?(1)*\dir{>}?(.25)*{\bullet};
    (4,-4)*{};(-4,4)*{} **\crv{(4,-1) & (-4,1)}?(1)*\dir{>};
     (8,1)*{ n};
     (-10,0)*{};(10,0)*{};
     }};
  \endxy
 \;\; -
 \xy
  (0,0)*{\xybox{
    (-4,-4)*{};(4,4)*{} **\crv{(-4,-1) & (4,1)}?(1)*\dir{>}?(.75)*{\bullet};
    (4,-4)*{};(-4,4)*{} **\crv{(4,-1) & (-4,1)}?(1)*\dir{>};
     (8,1)*{ n};
     (-10,0)*{};(10,0)*{};
     }};
  \endxy
 \;\; =
\xy
  (0,0)*{\xybox{
    (-4,-4)*{};(4,4)*{} **\crv{(-4,-1) & (4,1)}?(1)*\dir{>};
    (4,-4)*{};(-4,4)*{} **\crv{(4,-1) & (-4,1)}?(1)*\dir{>}?(.75)*{\bullet};
     (8,1)*{ n};
     (-10,0)*{};(10,0)*{};
     }};
  \endxy
 \;\; -
  \xy
  (0,0)*{\xybox{
    (-4,-4)*{};(4,4)*{} **\crv{(-4,-1) & (4,1)}?(1)*\dir{>} ;
    (4,-4)*{};(-4,4)*{} **\crv{(4,-1) & (-4,1)}?(1)*\dir{>}?(.25)*{\bullet};
     (8,1)*{ n};
     (-10,0)*{};(10,0)*{};
     }};
  \endxy \nn \\
\end{eqnarray}

\item  Negative degree bubbles are zero, but a dotted bubble of degree zero must be a multiple of the identity map.  By (A1) and (B1) above this multiple must be nonzero.  We fix these constants as follows
\[
\xy 0;/r.18pc/:
 (0,0)*{\cbub{n-1}};
  (4,8)*{n};
 \endxy
  = c_{n}^+ \cdot \Id_{\onen} \quad \text{for $n \geq 1$,}
  \qquad \quad
  \xy 0;/r.18pc/:
 (0,0)*{\ccbub{-n-1}};
  (4,8)*{n};
 \endxy
  = c_{n}^- \cdot \Id_{\onen} \quad \text{for $n \leq -1$.}
\]

\item For the fake bubbles defined as in Definition~\ref{def_fake}
the following relations hold
\begin{equation} \label{eq_reductionT}
  \xy 0;/r.16pc/:
  (14,8)*{n};
  (-3,-8)*{};(3,8)*{} **\crv{(-3,-1) & (3,1)}?(1)*\dir{>};?(0)*\dir{>};
    (3,-8)*{};(-3,8)*{} **\crv{(3,-1) & (-3,1)}?(1)*\dir{>};
  (-3,-12)*{\bbsid};  (-3,8)*{\bbsid};
  (3,8)*{}="t1";  (9,8)*{}="t2";
  (3,-8)*{}="t1'";  (9,-8)*{}="t2'";
   "t1";"t2" **\crv{(3,14) & (9, 14)};
   "t1'";"t2'" **\crv{(3,-14) & (9, -14)};
   "t2'";"t2" **\dir{-} ?(.5)*\dir{<};
   (9,0)*{}; (-6,-8)*{\scs };
 \endxy \;\; = \;\; -\sum_{f_1+f_2=-n}
   \xy 0;/r.18pc/:
  (19,4)*{n};
  (0,0)*{\bbe{}};(-2,-8)*{\scs };
  (12,-2)*{\cbub{(n-1)+f_2}};
  (0,6)*{\bullet}+(3,-1)*{\scs f_1};
 \endxy
\qquad \qquad
\xy 0;/r.16pc/:
  (-12,8)*{n};
   (-3,-8)*{};(3,8)*{} **\crv{(-3,-1) & (3,1)}?(1)*\dir{>};?(0)*\dir{>};
    (3,-8)*{};(-3,8)*{} **\crv{(3,-1) & (-3,1)}?(1)*\dir{>};
  (3,-12)*{\bbsid};
  (3,8)*{\bbsid}; (6,-8)*{\scs };
  (-9,8)*{}="t1";
  (-3,8)*{}="t2";
  (-9,-8)*{}="t1'";
  (-3,-8)*{}="t2'";
   "t1";"t2" **\crv{(-9,14) & (-3, 14)};
   "t1'";"t2'" **\crv{(-9,-14) & (-3, -14)};
  "t1'";"t1" **\dir{-} ?(.5)*\dir{<};
 \endxy \;\; = \;\;
 \sum_{g_1+g_2=n}^{}
   \xy 0;/r.18pc/:
  (-12,8)*{n};
  (0,0)*{\bbe{}};(2,-8)*{\scs};
  (-12,-2)*{\ccbub{(-n-1)+g_2}};
  (0,6)*{\bullet}+(3,-1)*{\scs g_1};
 \endxy
\end{equation}
\begin{eqnarray}
 \vcenter{\xy 0;/r.18pc/:
  (-8,0)*{};
  (8,0)*{};
  (-4,10)*{}="t1";
  (4,10)*{}="t2";
  (-4,-10)*{}="b1";
  (4,-10)*{}="b2";(-6,-8)*{\scs };(6,-8)*{\scs };
  "t1";"b1" **\dir{-} ?(.5)*\dir{<};
  "t2";"b2" **\dir{-} ?(.5)*\dir{>};
  (10,2)*{n};
  (-10,2)*{n};
  \endxy}
&\quad = \quad&
 \beta_{n}\;\;
 \vcenter{   \xy 0;/r.18pc/:
    (-4,-4)*{};(4,4)*{} **\crv{(-4,-1) & (4,1)}?(1)*\dir{>};
    (4,-4)*{};(-4,4)*{} **\crv{(4,-1) & (-4,1)}?(1)*\dir{<};?(0)*\dir{<};
    (-4,4)*{};(4,12)*{} **\crv{(-4,7) & (4,9)};
    (4,4)*{};(-4,12)*{} **\crv{(4,7) & (-4,9)}?(1)*\dir{>};
  (8,8)*{n};(-6,-3)*{\scs };
     (6.5,-3)*{\scs };
 \endxy}
  \quad - \beta_{n}\quad
   \sum_{ \xy  (0,3)*{\scs f_1+f_2+f_3}; (0,0)*{\scs =n-1};\endxy}
    \vcenter{\xy 0;/r.18pc/:
    (-10,10)*{n};
    (-8,0)*{};
  (8,0)*{};
  (-4,-15)*{}="b1";
  (4,-15)*{}="b2";
  "b2";"b1" **\crv{(5,-8) & (-5,-8)}; ?(.05)*\dir{<} ?(.93)*\dir{<}
  ?(.8)*\dir{}+(0,-.1)*{\bullet}+(-3,2)*{\scs f_3};
  (-4,15)*{}="t1";
  (4,15)*{}="t2";
  "t2";"t1" **\crv{(5,8) & (-5,8)}; ?(.15)*\dir{>} ?(.95)*\dir{>}
  ?(.4)*\dir{}+(0,-.2)*{\bullet}+(3,-2)*{\scs \; f_1};
  (0,0)*{\ccbub{\scs \quad -n-1+f_2}};
  \endxy} \nn
 \\  \; \nn \\
 \vcenter{\xy 0;/r.18pc/:
  (-8,0)*{};(-6,-8)*{\scs };(6,-8)*{\scs };
  (8,0)*{};
  (-4,10)*{}="t1";
  (4,10)*{}="t2";
  (-4,-10)*{}="b1";
  (4,-10)*{}="b2";
  "t1";"b1" **\dir{-} ?(.5)*\dir{>};
  "t2";"b2" **\dir{-} ?(.5)*\dir{<};
  (10,2)*{n};
  (-10,2)*{n};
  \endxy}
&\quad = \quad&
 \beta_{n}\;\;
   \vcenter{\xy 0;/r.18pc/:
    (-4,-4)*{};(4,4)*{} **\crv{(-4,-1) & (4,1)}?(1)*\dir{<};?(0)*\dir{<};
    (4,-4)*{};(-4,4)*{} **\crv{(4,-1) & (-4,1)}?(1)*\dir{>};
    (-4,4)*{};(4,12)*{} **\crv{(-4,7) & (4,9)}?(1)*\dir{>};
    (4,4)*{};(-4,12)*{} **\crv{(4,7) & (-4,9)};
  (8,8)*{n};(-6,-3)*{\scs };  (6,-3)*{\scs };
 \endxy}
  \quad - \beta_{n} \quad
    \sum_{ \xy  (0,3)*{\scs g_1+g_2+g_3}; (0,0)*{\scs =-n-1};\endxy}
    \vcenter{\xy 0;/r.18pc/:
    (-8,0)*{};
  (8,0)*{};
  (-4,-15)*{}="b1";
  (4,-15)*{}="b2";
  "b2";"b1" **\crv{(5,-8) & (-5,-8)}; ?(.1)*\dir{>} ?(.95)*\dir{>}
  ?(.8)*\dir{}+(0,-.1)*{\bullet}+(-3,2)*{\scs g_3};
  (-4,15)*{}="t1";
  (4,15)*{}="t2";
  "t2";"t1" **\crv{(5,8) & (-5,8)}; ?(.15)*\dir{<} ?(.9)*\dir{<}
  ?(.4)*\dir{}+(0,-.2)*{\bullet}+(3,-2)*{\scs g_1};
  (0,0)*{\cbub{\scs \quad\; n-1 + g_2}};
  (-10,10)*{n};
  \endxy} \label{eq_ident_decompT}
\end{eqnarray}
for all $n\in \Z$.  In equations \eqref{eq_reductionT} and
\eqref{eq_ident_decompT} whenever the summations are nonzero they
utilize fake bubbles.
\end{enumerate}

The additive monoidal composition functor $\Ucat(n,n')
 \times  \Ucat(n',n'')\to\Ucat(n,n'') $ is given on
 1-morphisms of $\Ucat$ by
\begin{equation}
  \cal{E}_{\ep'}\mathbf{1}_{n'}\{t'\} \times \cal{E}_{\ep}\onen\{t\} \mapsto
  \cal{E}_{\ep'\ep}\onen\{t+t'\} \nn
\end{equation}
for $n'=n+\sum_{j=1}^m \epsilon_i 2$, and on 2-morphisms of $\Ucat$ by juxtaposition of diagrams
\[
\left(\;\;\vcenter{\xy 0;/r.16pc/:
 (-4,-15)*{}; (-20,25) **\crv{(-3,-6) & (-20,4)}?(0)*\dir{<}?(.6)*\dir{}+(0,0)*{\bullet};
 (-12,-15)*{}; (-4,25) **\crv{(-12,-6) & (-4,0)}?(0)*\dir{<}?(.6)*\dir{}+(.2,0)*{\bullet};
 ?(0)*\dir{<}?(.75)*\dir{}+(.2,0)*{\bullet};?(0)*\dir{<}?(.9)*\dir{}+(0,0)*{\bullet};
 (-28,25)*{}; (-12,25) **\crv{(-28,10) & (-12,10)}?(0)*\dir{<};
  ?(.2)*\dir{}+(0,0)*{\bullet}?(.35)*\dir{}+(0,0)*{\bullet};
 (-36,-15)*{}; (-36,25) **\crv{(-34,-6) & (-35,4)}?(1)*\dir{>};
 (-28,-15)*{}; (-42,25) **\crv{(-28,-6) & (-42,4)}?(1)*\dir{>};
 (-42,-15)*{}; (-20,-15) **\crv{(-42,-5) & (-20,-5)}?(1)*\dir{>};
 (6,10)*{\cbub{}{}};
 (-23,0)*{\cbub{}{}};
 (8,-4)*{n'};(-44,-4)*{n''};
 \endxy}\;\;\right) \;\; \times \;\;
\left(\;\;\vcenter{ \xy 0;/r.18pc/: (-14,8)*{\xybox{
 (0,-10)*{}; (-16,10)*{} **\crv{(0,-6) & (-16,6)}?(.5)*\dir{};
 (-16,-10)*{}; (-8,10)*{} **\crv{(-16,-6) & (-8,6)}?(1)*\dir{}+(.1,0)*{\bullet};
  (-8,-10)*{}; (0,10)*{} **\crv{(-8,-6) & (-0,6)}?(.6)*\dir{}+(.2,0)*{\bullet}?
  (1)*\dir{}+(.1,0)*{\bullet};
  (0,10)*{}; (-16,30)*{} **\crv{(0,14) & (-16,26)}?(1)*\dir{>};
 (-16,10)*{}; (-8,30)*{} **\crv{(-16,14) & (-8,26)}?(1)*\dir{>};
  (-8,10)*{}; (0,30)*{} **\crv{(-8,14) & (-0,26)}?(1)*\dir{>}?(.6)*\dir{}+(.25,0)*{\bullet};
   }};
 (-2,-4)*{n}; (-26,-4)*{n'};
 \endxy} \;\;\right)
 \;\;\mapsto \;\;
\vcenter{\xy 0;/r.16pc/:
 (-4,-15)*{}; (-20,25) **\crv{(-3,-6) & (-20,4)}?(0)*\dir{<}?(.6)*\dir{}+(0,0)*{\bullet};
 (-12,-15)*{}; (-4,25) **\crv{(-12,-6) & (-4,0)}?(0)*\dir{<}?(.6)*\dir{}+(.2,0)*{\bullet};
 ?(0)*\dir{<}?(.75)*\dir{}+(.2,0)*{\bullet};?(0)*\dir{<}?(.9)*\dir{}+(0,0)*{\bullet};
 (-28,25)*{}; (-12,25) **\crv{(-28,10) & (-12,10)}?(0)*\dir{<};
  ?(.2)*\dir{}+(0,0)*{\bullet}?(.35)*\dir{}+(0,0)*{\bullet};
 (-36,-15)*{}; (-36,25) **\crv{(-34,-6) & (-35,4)}?(1)*\dir{>};
 (-28,-15)*{}; (-42,25) **\crv{(-28,-6) & (-42,4)}?(1)*\dir{>};
 (-42,-15)*{}; (-20,-15) **\crv{(-42,-5) & (-20,-5)}?(1)*\dir{>};
 (6,10)*{\cbub{}};
 (-23,0)*{\cbub{}};
 \endxy}
 \vcenter{ \xy 0;/r.16pc/: (-14,8)*{\xybox{
 (0,-10)*{}; (-16,10)*{} **\crv{(0,-6) & (-16,6)}?(.5)*\dir{};
 (-16,-10)*{}; (-8,10)*{} **\crv{(-16,-6) & (-8,6)}?(1)*\dir{}+(.1,0)*{\bullet};
  (-8,-10)*{}; (0,10)*{} **\crv{(-8,-6) & (-0,6)}?(.6)*\dir{}+(.2,0)*{\bullet}?
  (1)*\dir{}+(.1,0)*{\bullet};
  (0,10)*{}; (-16,30)*{} **\crv{(0,14) & (-16,26)}?(1)*\dir{>};
 (-16,10)*{}; (-8,30)*{} **\crv{(-16,14) & (-8,26)}?(1)*\dir{>};
  (-8,10)*{}; (0,30)*{} **\crv{(-8,14) & (-0,26)}?(1)*\dir{>}?(.6)*\dir{}+(.25,0)*{\bullet};
   }};
 (0,-5)*{n};
 \endxy}
\]
\end{defn}

Several examples of the above relations are given in Section~\ref{subsec_Ucat}.

\begin{example}\label{ex_generators} Below we collect the generators of $\Ucatt$ and their corresponding degree:
\[
\begin{tabular}{|l|c|c|c|c|}
 \hline
 {\bf 2-morphism:} &   \xy 0;/r.17pc/:
 (0,7);(0,-7); **\dir{-} ?(.75)*\dir{>};
 (0,-2)*{\txt\large{$\bullet$}};
 (6,4)*{n}; (-8,4)*{n +2}; (-10,0)*{};(10,0)*{};
 \endxy
 &
     \xy 0;/r.17pc/:
 (0,7);(0,-7); **\dir{-} ?(.75)*\dir{<};
 (0,-2)*{\txt\large{$\bullet$}};
 (-6,4)*{n}; (8,4)*{n+2}; (-10,0)*{};(10,9)*{};
 \endxy
 &
   \xy 0;/r.17pc/:
  (0,0)*{\xybox{
    (-4,-4)*{};(4,4)*{} **\crv{(-4,-1) & (4,1)}?(1)*\dir{>} ;
    (4,-4)*{};(-4,4)*{} **\crv{(4,-1) & (-4,1)}?(1)*\dir{>};
     (8,1)*{n};     (-12,0)*{};(12,0)*{};     }};
  \endxy
 &
   \xy 0;/r.17pc/:
  (0,0)*{\xybox{
    (-4,4)*{};(4,-4)*{} **\crv{(-4,1) & (4,-1)}?(1)*\dir{>} ;
    (4,4)*{};(-4,-4)*{} **\crv{(4,1) & (-4,-1)}?(1)*\dir{>};
     (8,1)*{ n};     (-12,0)*{};(12,0)*{};     }};
  \endxy
\\ & & & &\\
\hline
 {\bf Degree:} & \;\;\text{  2 }\;\;
 &\;\;\text{  2}\;\;& \;\;\text{ -2}\;\;
 & \;\;\text{  -2}\;\; \\
 \hline
\end{tabular}
\]
\[
\begin{tabular}{|l|c|c|c|c|}
 \hline
  {\bf 2-morphism:} &  \xy 0;/r.17pc/:
    (0,-3)*{\bbpef{}};
    (8,-5)*{n};    (-12,0)*{};(12,0)*{};
    \endxy
  & \xy 0;/r.17pc/:
    (0,-3)*{\bbpfe{}};
    (8,-5)*{n};    (-12,0)*{};(12,0)*{};
    \endxy
  & \xy 0;/r.17pc/:
    (0,-2)*{\bbcef{}};
    (8,0)*{n};     (-12,0)*{};(12,0)*{};
    \endxy
  & \xy 0;/r.17pc/:
    (0,-2)*{\bbcfe{}};
    (8,0)*{n};    (-12,0)*{};(12,0)*{};
    \endxy\\& & &  &\\ \hline
 {\bf Degree:} & \;\;\text{  $1+n$}\;\;
 & \;\;\text{ $1-n$}\;\;
 & \;\;\text{ $1+n$}\;\;
 & \;\;\text{  $1-n$}\;\;
 \\
 \hline
\end{tabular}
\]
The diagrammatic relation
\[
\vcenter{
 \xy 0;/r.18pc/:
    (-4,-4)*{};(4,4)*{} **\crv{(-4,-1) & (4,1)}?(1)*\dir{>};
    (4,-4)*{};(-4,4)*{} **\crv{(4,-1) & (-4,1)}?(1)*\dir{>};
    (4,4)*{};(12,12)*{} **\crv{(4,7) & (12,9)}?(1)*\dir{>};
    (12,4)*{};(4,12)*{} **\crv{(12,7) & (4,9)}?(1)*\dir{>};
    (-4,12)*{};(4,20)*{} **\crv{(-4,15) & (4,17)}?(1)*\dir{>};
    (4,12)*{};(-4,20)*{} **\crv{(4,15) & (-4,17)}?(1)*\dir{>};
    (-4,4)*{}; (-4,12) **\dir{-};
    (12,-4)*{}; (12,4) **\dir{-};
    (12,12)*{}; (12,20) **\dir{-};
  (18,8)*{n};
\endxy}
 \;\; =\;\;
 \vcenter{
 \xy 0;/r.18pc/:
    (4,-4)*{};(-4,4)*{} **\crv{(4,-1) & (-4,1)}?(1)*\dir{>};
    (-4,-4)*{};(4,4)*{} **\crv{(-4,-1) & (4,1)}?(1)*\dir{>};
    (-4,4)*{};(-12,12)*{} **\crv{(-4,7) & (-12,9)}?(1)*\dir{>};
    (-12,4)*{};(-4,12)*{} **\crv{(-12,7) & (-4,9)}?(1)*\dir{>};
    (4,12)*{};(-4,20)*{} **\crv{(4,15) & (-4,17)}?(1)*\dir{>};
    (-4,12)*{};(4,20)*{} **\crv{(-4,15) & (4,17)}?(1)*\dir{>};
    (4,4)*{}; (4,12) **\dir{-};
    (-12,-4)*{}; (-12,4) **\dir{-};
    (-12,12)*{}; (-12,20) **\dir{-};
  (10,8)*{n};
\endxy}
\]
gives rise to relations in
$\Ucat\big(\cal{E}_{+++}\onen\{t\},\cal{E}_{+++}\onen\{t+6\}\big)=
\Ucat\big(\cal{E}^3\onen\{t\},\cal{E}^3\onen\{t+6\}\big)$ for
all $t\in \Z$.
\end{example}

\begin{example}
One can check that
\[
\deg\left( \;\vcenter{\xy 0;/r.15pc/:
 (-4,-15)*{}; (-20,25) **\crv{(-3,-6) & (-20,4)}?(0)*\dir{<}?(.6)*\dir{}+(0,0)*{\bullet};
 (-12,-15)*{}; (-4,25) **\crv{(-12,-6) & (-4,0)}?(0)*\dir{<}?(.6)*\dir{}+(.2,0)*{\bullet};
 ?(0)*\dir{<}?(.75)*\dir{}+(.2,0)*{\bullet};?(0)*\dir{<}?(.9)*\dir{}+(0,0)*{\bullet};
 (-28,25)*{}; (-12,25) **\crv{(-28,10) & (-12,10)}?(0)*\dir{<};
  ?(.2)*\dir{}+(0,0)*{\bullet}?(.35)*\dir{}+(0,0)*{\bullet};
 (-36,-15)*{}; (-36,25) **\crv{(-34,-6) & (-35,4)}?(1)*\dir{>};
 (-28,-15)*{}; (-42,25) **\crv{(-28,-6) & (-42,4)}?(1)*\dir{>};
 (-42,-15)*{}; (-20,-15) **\crv{(-42,-5) & (-20,-5)}?(1)*\dir{>};
 (6,10)*{\cbub{}{}};
 (-23,0)*{\cbub{}{}};
 (8,-4)*{n};(-44,-4)*{n};
 \endxy} \; \right) = 28-6n,
\]
so that this diagram represents a 2-morphism $\cal{E}_{+++---}\{t\} \to \cal{E}_{+++---}\{t-28+6n\}$, where $\cal{E}_{+++---}=\cal{E}^3\cal{F}^3\onen$.
\end{example}

One can check that \eqref{eq_ind_dotslide} together with the relations above imply:
\begin{equation} \label{eq_reduction-dots}
   \xy 0;/r.16pc/:
  (14,8)*{n};
  (-3,-8)*{};(3,8)*{} **\crv{(-3,-1) & (3,1)}?(1)*\dir{>};?(0)*\dir{>};
    (3,-8)*{};(-3,8)*{} **\crv{(3,-1) & (-3,1)}?(1)*\dir{>};
  (-3,-12)*{\bbsid};  (-3,8)*{\bbsid};
  (3,8)*{}="t1";  (9,8)*{}="t2";
  (3,-8)*{}="t1'";  (9,-8)*{}="t2'";
   "t1";"t2" **\crv{(3,14) & (9, 14)};
   "t1'";"t2'" **\crv{(3,-14) & (9, -14)};
   "t2'";"t2" **\dir{-} ?(.5)*\dir{<};
   (9,0)*{}; (-6,-8)*{\scs };
   (9,-4)*{\bullet}+(3,-1)*{\scs x};
 \endxy\;\; = \;\; -\sum_{f_1+f_2=x-n}
   \xy 0;/r.18pc/:
  (19,4)*{n};
  (0,0)*{\bbe{}};(-2,-8)*{\scs };
  (12,-2)*{\cbub{(n-1)+f_2}};
  (0,6)*{\bullet}+(3,-1)*{\scs f_1};
 \endxy
\qquad \qquad
 \xy 0;/r.16pc/:
  (-12,8)*{n};
   (-3,-8)*{};(3,8)*{} **\crv{(-3,-1) & (3,1)}?(1)*\dir{>};?(0)*\dir{>};
    (3,-8)*{};(-3,8)*{} **\crv{(3,-1) & (-3,1)}?(1)*\dir{>};
  (3,-12)*{\bbsid};
  (3,8)*{\bbsid}; (6,-8)*{\scs };
  (-9,8)*{}="t1";
  (-3,8)*{}="t2";
  (-9,-8)*{}="t1'";(-9,-4)*{\bullet}+(-3,-1)*{\scs x};
  (-3,-8)*{}="t2'";
   "t1";"t2" **\crv{(-9,14) & (-3, 14)};
   "t1'";"t2'" **\crv{(-9,-14) & (-3, -14)};
  "t1'";"t1" **\dir{-} ?(.5)*\dir{<};
 \endxy \;\; = \;\;
 \sum_{g_1+g_2=x+n}^{}
   \xy 0;/r.18pc/:
  (-12,8)*{n};
  (0,0)*{\bbe{}};(2,-8)*{\scs};
  (-12,-2)*{\ccbub{(-n-1)+g_2}};
  (0,6)*{\bullet}+(3,-1)*{\scs g_1};
 \endxy
\end{equation}

\begin{rem} \label{rem_otherinv}
It is natural to wonder what goes wrong if we chose the simplest possible maps for the inverse to $\zeta_+$, namely,
\begin{equation}
 (\overline{\zeta_{+}})^n \;\; := \;\; \beta_n\;\xy 0;/r.15pc/:
    (-4,-4)*{};(4,4)*{} **\crv{(-4,-1) & (4,1)}?(1)*\dir{>} ;
    (4,-4)*{};(-4,4)*{} **\crv{(4,-1) & (-4,1)}?(0)*\dir{<};
  \endxy
  \qquad \quad
 (\overline{\zeta_{+}})^{\ell}  \;\; := \;\;
\vcenter{\xy 0;/r.15pc/:
  (-4,-2)*{}="t1";  (4,-2)*{}="t2";
  "t2";"t1" **\crv{(4,5) & (-4,5)};  ?(.85)*\dir{<} 
   ?(.2)*\dir{}+(0,-.1)*{\bullet}+(5,1)*{\scs \ell};
 \endxy}
\end{equation}
If this map were the correct form of the inverse of $\zeta_+$, then the implications of the relations $\overline{\zeta_{+}}\zeta_+=1_{\cal{F}\cal{E}\onen \oplus\onen^{\oplus_{[n]}}}$ and $\zeta_+\overline{\zeta_{+}}=1_{\cal{E}\cal{F}\onen}$ would include the requirement that
\begin{equation} \label{eq_contra_bubbles}
  \xy
      (0,0)*{\cbub{n-1+\ell}{}};
      (8,8)*{n}; \endxy =\delta_{\ell,0} \qquad \text{for $-n+1 \leq \ell \leq n-1$.}
\end{equation}
The relation {\bf (A2)} would still hold with this form of the inverse $\overline{\zeta_{+}}$.  Closing off the right strand of relation {\bf (A2)} with $n+2$ dots results in the equation
\begin{equation}  \label{eq_contradiction}
 \vcenter{\xy 0;/r.17pc/:
  (-8,0)*{};(-6,-8)*{ };(6,-8)*{ };
  (8,0)*{};
  (-4,10)*{}="t1";
  (4,6)*{}="t2";
  (-4,-10)*{}="b1";
  (4,-6)*{}="b2";
  "t1";"b1" **\dir{-} ?(.5)*\dir{>};
  "t2";"b2" **\dir{-} ?(.5)*\dir{<};
    (4,6)*{};(12,6)*{} **\crv{(4,10) & (12,10)};
    (4,-6)*{};(12,-6)*{} **\crv{(4,-10) & (12,-10)};
    (12,-6)*{}; (12,6)*{} **\dir{-}
    ?(.5)*\dir{}+(0,-.1)*{\bullet}+(6,1)*{\scs n+2};
  (24,9)*{n+2};
  \endxy}
\;\; = \;\;
 \beta_n\;\;
   \vcenter{\xy 0;/r.19pc/:
    (-4,-6)*{};(4,4)*{} **\crv{(-4,-1) & (4,1)}?(1)*\dir{<};?(0)*\dir{<};
    (4,-4)*{};(-4,4)*{} **\crv{(4,-1) & (-4,1)}?(1)*\dir{>};
    (-4,4)*{};(4,12)*{} **\crv{(-4,7) & (4,9)}?(1)*\dir{>};
    (4,4)*{};(-4,14)*{} **\crv{(4,7) & (-4,9)};
    (4,12)*{};(12,12)*{} **\crv{(4,16) & (12,16)};
    (4,-4)*{};(12,-4)*{} **\crv{(4,-8) & (12,-8)};
    (12,-4)*{}; (12,12)*{} **\dir{-}
    ?(.5)*\dir{}+(0,-.1)*{\bullet}+(5,1)*{\scs n+2};
  (8,8)*{n};(-6,-3)*{ };
     (6,-3)*{ };
 \endxy}\;\; =\;\;  \beta_n\;\;
   \vcenter{\xy 0;/r.19pc/:
    (-4,-6)*{};(4,4)*{} **\crv{(-4,-1) & (4,1)}?(1)*\dir{<};?(0)*\dir{<};
    (4,-4)*{};(-4,4)*{} **\crv{(4,-1) & (-4,1)}?(1)*\dir{>};
    (-4,4)*{};(4,12)*{} **\crv{(-4,7) & (4,9)} ;
    (4,4)*{};(-4,14)*{} **\crv{(4,7) & (-4,9)};
    (4,12)*{};(12,12)*{} **\crv{(4,16) & (12,16)}; ?(.05)*\dir{}+(0,-.1)*{\bullet}+(-2,4)*{\scs n+2} ;
    (4,-4)*{};(12,-4)*{} **\crv{(4,-8) & (12,-8)};
    (12,-4)*{}; (12,12)*{} **\dir{-} ?(.5)*\dir{<};
  (8,8)*{n};(-6,-3)*{ };
     (6,-3)*{ };
 \endxy}
\end{equation}
so that sliding the $n+2$ dots implies
\begin{equation} \label{eq_XY}
\beta_n\;\;
   \vcenter{\xy 0;/r.19pc/:
    (-4,-6)*{};(4,4)*{} **\crv{(-4,-1) & (4,1)}?(1)*\dir{<};?(0)*\dir{<};
    (4,-4)*{};(-4,4)*{} **\crv{(4,-1) & (-4,1)}?(1)*\dir{>};
    (-4,4)*{};(4,12)*{} **\crv{(-4,7) & (4,9)} ;
    (4,4)*{};(-4,14)*{} **\crv{(4,7) & (-4,9)};
    (4,12)*{};(12,12)*{} **\crv{(4,16) & (12,16)}; ?(.05)*\dir{}+(0,-.1)*{\bullet}+(-2,4)*{\scs n+2} ;
    (4,-4)*{};(12,-4)*{} **\crv{(4,-8) & (12,-8)};
    (12,-4)*{}; (12,12)*{} **\dir{-} ?(.5)*\dir{<};
  (8,8)*{n};(-6,-3)*{ };
     (6,-3)*{ };
 \endxy} \;\; = \;\; \beta_n\;\;
   \vcenter{\xy 0;/r.19pc/:
    (-4,-6)*{};(4,4)*{} **\crv{(-4,-1) & (4,1)}?(1)*\dir{<};?(0)*\dir{<};
    (4,-4)*{};(-4,4)*{} **\crv{(4,-1) & (-4,1)} ?(1)*\dir{}+(0,-.1)*{\bullet}+(-2,4)*{\scs n+2} ;
    (-4,4)*{};(4,12)*{} **\crv{(-4,7) & (4,9)}?(1)*\dir{>};
    (4,4)*{};(-4,14)*{} **\crv{(4,7) & (-4,9)};
    (4,12)*{};(12,12)*{} **\crv{(4,16) & (12,16)};
    (4,-4)*{};(12,-4)*{} **\crv{(4,-8) & (12,-8)};
    (12,-4)*{}; (12,12)*{} **\dir{-}?(.5)*\dir{<};
  (8,8)*{n};(-6,-3)*{ };
     (6,-3)*{ };
 \endxy}
 \;\; +\;\; \beta_n \;\sum_{f_1+f_2=n+1}
   \vcenter{\xy 0;/r.19pc/:
    (-4,-6)*{};(4,4)*{} **\crv{(-4,-1) & (4,1)}?(1)*\dir{<};?(0)*\dir{<};
    (4,-4)*{};(-4,4)*{} **\crv{(4,-1) & (-4,1)}?(1)*\dir{>};
    (-4,4)*{};(4,4)*{} **\crv{(-4,8) & (4,8)}
    ?(.15)*\dir{}+(0,-.1)*{\bullet}+(-3,0)*{\scs f_1};
    (4,12)*{};(-4,16)*{} **\crv{(3,9) & (-5,9)}
    ?(.82)*\dir{}+(0,-.1)*{\bullet}+(-3,0)*{\scs f_2} ;
    (4,12)*{};(12,12)*{} **\crv{(4,16) & (12,16)};
    (4,-4)*{};(12,-4)*{} **\crv{(4,-8) & (12,-8)};
    (12,-4)*{}; (12,12)*{} **\dir{-}?(.5)*\dir{<};;
  (8,8)*{n};(-6,-3)*{ };
     (6,-3)*{ };
 \endxy}
\end{equation}
The first diagram is zero since we can write\footnote{Here we are assuming that the calculus is isotopy invariant which is not strictly necessary for this argument to hold. By the arguments in Section~\ref{subsec_generators}, diagrams must be isotopy invariant up to an invertible scalar.  Including this scalar one still derives the same contradiction.}
\begin{equation}
\vcenter{\xy 0;/r.19pc/:
    (-4,-6)*{};(4,4)*{} **\crv{(-4,-1) & (4,1)}?(1)*\dir{<};?(0)*\dir{<};
    (4,-4)*{};(-4,4)*{} **\crv{(4,-1) & (-4,1)} ?(1)*\dir{}+(0,-.1)*{\bullet}+(-2,4)*{\scs n+2} ;
    (-4,4)*{};(4,12)*{} **\crv{(-4,7) & (4,9)};
    (4,4)*{};(-4,14)*{} **\crv{(4,7) & (-4,9)};
    (4,12)*{};(12,12)*{} **\crv{(4,16) & (12,16)};
    (4,-4)*{};(12,-4)*{} **\crv{(4,-8) & (12,-8)};
    (12,-4)*{}; (12,12)*{} **\dir{-}?(.5)*\dir{<};
  (8,8)*{n};(-6,-3)*{ };
     (6,-3)*{ };
 \endxy} \;\; = \;\;
 \vcenter{\xy 0;/r.19pc/:
    (4,-6)*{};(-4,4)*{} **\crv{(4,-1) & (-4,1)}?(0)*\dir{<};
    (-4,-4)*{};(4,4)*{} **\crv{(-4,-1) & (4,1)}?(0)*\dir{<};
    (4,4)*{};(-4,12)*{} **\crv{(4,7) & (-4,9)};
    (-4,4)*{};(4,14)*{} **\crv{(-4,7) & (4,9)};
    (-4,12)*{};(-12,12)*{} **\crv{(-4,16) & (-12,16)};
    (-4,-4)*{};(-12,-4)*{} **\crv{(-4,-8) & (-12,-8)};
    (-12,-4)*{}; (-12,12)*{} **\dir{-}?(.5)*\dir{>} ?(.75)*\dir{}+(0,-.1)*{\bullet}+(-5,1)*{\scs n+2};
  (8,8)*{n};
 \endxy} = 0
\end{equation}
by the nilHecke axiom \eqref{eq_nilHecke1}.  Therefore if we remove the summands in the second term of \eqref{eq_XY} that are zero by {\bf (A4)} (or the adjoint of {\bf (A3)}) we have
\begin{equation}
\beta_n\;\;
   \vcenter{\xy 0;/r.19pc/:
    (-4,-6)*{};(4,4)*{} **\crv{(-4,-1) & (4,1)}?(1)*\dir{<};?(0)*\dir{<};
    (4,-4)*{};(-4,4)*{} **\crv{(4,-1) & (-4,1)}?(1)*\dir{>};
    (-4,4)*{};(4,12)*{} **\crv{(-4,7) & (4,9)} ?(1)*\dir{>};
    (4,4)*{};(-4,14)*{} **\crv{(4,7) & (-4,9)};
    (4,12)*{};(12,12)*{} **\crv{(4,16) & (12,16)}; ?(.05)*\dir{}+(0,-.1)*{\bullet}+(-2,4)*{\scs n+2} ;
    (4,-4)*{};(12,-4)*{} **\crv{(4,-8) & (12,-8)};
    (12,-4)*{}; (12,12)*{} **\dir{-} ?(.5)*\dir{<};
  (8,8)*{n};(-6,-3)*{ };
     (6,-3)*{ };
 \endxy} \;\; =
 \;\; \beta_n\;\;
   \vcenter{\xy 0;/r.19pc/:
    (-4,-6)*{};(4,4)*{} **\crv{(-4,-1) & (4,1)}?(1)*\dir{<};?(0)*\dir{<};
    (4,-4)*{};(-4,4)*{} **\crv{(4,-1) & (-4,1)}?(1)*\dir{>};
    (-4,4)*{};(4,4)*{} **\crv{(-4,8) & (4,8)}
    ?(.15)*\dir{}+(0,-.1)*{\bullet}+(-3,0)*{\scs n};
    (4,12)*{};(-4,16)*{} **\crv{(3,9) & (-5,9)}
    ?(.82)*\dir{}+(0,-.1)*{\bullet}+(-3,0)*{\scs } ;
    (4,12)*{};(12,12)*{} **\crv{(4,16) & (12,16)};
    (4,-4)*{};(12,-4)*{} **\crv{(4,-8) & (12,-8)};
    (12,-4)*{}; (12,12)*{} **\dir{-}?(.5)*\dir{<};;
  (8,8)*{n};(-6,-3)*{ };
     (6,-3)*{ };
 \endxy}
  \;\; + \;\;\beta_n \;
   \vcenter{\xy 0;/r.19pc/:
    (-4,-6)*{};(4,4)*{} **\crv{(-4,-1) & (4,1)}?(1)*\dir{<};?(0)*\dir{<};
    (4,-4)*{};(-4,4)*{} **\crv{(4,-1) & (-4,1)}?(1)*\dir{>};
    (-4,4)*{};(4,4)*{} **\crv{(-4,8) & (4,8)}
    ?(.15)*\dir{}+(0,-.1)*{\bullet}+(-5,1)*{\scs n+1};
    (4,12)*{};(-4,16)*{} **\crv{(3,9) & (-5,9)} ;
    (4,12)*{};(12,12)*{} **\crv{(4,16) & (12,16)};
    (4,-4)*{};(12,-4)*{} **\crv{(4,-8) & (12,-8)};
    (12,-4)*{}; (12,12)*{} **\dir{-}?(.5)*\dir{<};;
  (8,8)*{n};(-6,-3)*{ };
     (6,-3)*{ };
 \endxy} \;\; = \;\;
 2\beta_n  \;\; \vcenter{\xy 0;/r.19pc/:
    (-4,-6)*{};(4,-4)*{} **\crv{(-5,-1) & (3,-1)}?(0)*\dir{<};
    (-8,5)*{\cbub{n-1}{}};
    (4,12)*{};(-4,16)*{} **\crv{(3,9) & (-5,9)} ;
    (4,12)*{};(12,12)*{} **\crv{(4,16) & (12,16)};
    (4,-4)*{};(12,-4)*{} **\crv{(4,-8) & (12,-8)};
    (12,-4)*{}; (12,12)*{} **\dir{-}?(.5)*\dir{<} ?(.25)*\dir{}+(0,-.1)*{\bullet};
  (8,8)*{n};(-6,-3)*{ };
     (6,-3)*{ };
 \endxy} \;\; = \;\; 2\beta_n c_n^+\;\; \vcenter{\xy 0;/r.19pc/:
    (12,-6)*{}; (12,16)*{} **\dir{-}?(.5)*\dir{<} ?(.25)*\dir{}+(0,-.1)*{\bullet};
  (8,8)*{n};
     (6,-3)*{ };
 \endxy}
\end{equation}
where we used the nilHecke dot slide relation \eqref{eq_ind_dotslide}, the adjoint of {\bf (A3)}, and equation \eqref{eq_contra_bubbles} to simplify the dotted curls in the second to last equality.  However, the left hand side in \eqref{eq_contradiction} is zero by \eqref{eq_contra_bubbles}. Since $\beta_n$ is nonzero by assumption, and  \eqref{eq_contra_bubbles} implies $c_n^+=1$,  consistency of the graphical calculus implies that a dot on a downward line is zero!  This is a disaster from the point of view of matching our graded Homs with the semilinear form and this explains why a more complicated form of the inverse $\overline{\zeta_{+}}$ should be expected.
\end{rem}

%
\subsubsection{Free parameters}
%

In this section we show that the coefficients are completely determined by the coefficients $\beta_n$ for $n \in \Z$, $c_1^+$ and $c_0^+$.  Alternatively, one can choose the value of the degree zero bubbles $c_n^{\pm}$ and the coefficients $\beta_n$ are determined.

Recall that we have a set $\chi$ of elements
\begin{equation}
  \beta_n \quad \text{for $n\in \Z$}, \qquad c_n^+ \quad \text{for $n \geq 0$},
  \qquad c_n^- \quad \text{for $n \leq 0$},
\end{equation}
in $\Bbbk$.  The elements $c_n^{\pm}$ and $\beta_n$ must be invertible.

\begin{prop} \label{prop_coeff}
Consistency of the diagrammatic calculus for the 2-category
$\Ucatt$ requires the following relationship among the
invertible elements of $\chi$.
\begin{itemize}
  \item $-c_0^+c_0^-\beta_0 =1$
  \item $c_{-1}^-= \frac{1}{c_1^+}$
    \item  $c_{n-2}^- = - \beta_nc_n^- \qquad \text{for all $n \leq
  0$}$
  \item $c_{n+2}^+ = - \beta_nc_n^+ \qquad \text{for all $n \geq
  0$}$
\end{itemize}
\end{prop}

\begin{proof}
The first equality follows by capping off \eqref{eq_ident_decompT} for $n=0$
and simplifying the resulting diagram.  The third equality follow
from the calculation below:
\begin{equation} \label{eq_coeffreduction}
\frac{1}{\beta_n}
   \xy 0;/r.18pc/:
  (14,8)*{n-2};
  (0,0)*{\bbe{}};(-2,-8)*{\scs };
  (12,-2)*{\ccbub{-(n-2)-1}};
 \endxy
\;\; =\;\; \vcenter{\xy 0;/r.18pc/:
    (-4,-10)*{};(0,0)*{} **\crv{(-4,-5) & (0,-5)};
    (0,0)*{};(-4,10)*{} **\crv{(0,5) & (-4,5)}?(1)*\dir{>};
  (8,8)*{n-2};(-5,-3)*{\scs {\bf }};
  (0,-1)*{\sccbub{}}; (3,-4)*{\bullet}+(3,-2)*{\scs -n+1};
 \endxy}
 \;\; = \;\;
  \vcenter{\xy 0;/r.18pc/:
    (4,-10)*{};(0,0)*{} **\crv{(4,-5) & (0,-5)};
    (0,0)*{};(4,10)*{} **\crv{(0,5) & (4,5)}?(1)*\dir{>};
  (-6,8)*{n};(-5,-3)*{\scs {\bf }};
  (1,-1)*{\sccbub{}}; (4,-4)*{\bullet}+(3,-2)*{\scs -n+1};
 \endxy}
 \;\; =\;\;
 -\sum_{\ell_1+\ell_2=-n}
 \xy 0;/r.18pc/:
  (-12,8)*{n};
   (-3,-4)*{};(3,4)*{} **\crv{(-3,-1) & (3,1)};
    (3,-4)*{};(-3,4)*{} **\crv{(3,-1) & (-3,1)};
    (3,4);(3,10) **\dir{-} ?(1)*\dir{>};
    (3,-10);(3,-4) **\dir{-};
  (-9,4)*{}="t1";
  (-3,4)*{}="t2";
  (-9,-4)*{}="t1'";
  (-3,-4)*{}="t2'";
   "t1";"t2" **\crv{(-9,9) & (-3, 9)};
   "t1'";"t2'" **\crv{(-9,-9) & (-3, -9)};
  "t1'";"t1" **\dir{-} ?(.5)*\dir{<};
  (3,5)*{\bullet}+(4,1)*{\scs \ell_2};
  (-3,5)*{\bullet}+(2,3)*{\scs \ell_1};
 \endxy
\end{equation}
where $n\leq 0$.  The left-hand side is just
$\frac{c_{n-2}^-}{\beta_n}$ while the right-hand side is $-c_n^-$
by \eqref{eq_def_czeropm} and \eqref{eq_reduction-dots}. The last
equality follows from a similar calculation using a clockwise
oriented bubble with $n+1$ dots.

The second equality follows from the equalities
\begin{equation}
0
 \;\; \refequal{\eqref{eq_nilHecke1}}\;\;
  \vcenter{\xy 0;/r.18pc/:
    (4,-10)*{};(0,0)*{} **\crv{(4,-5) & (0,-5)};
    (0,0)*{};(4,10)*{} **\crv{(0,5) & (4,5)}?(1)*\dir{>};
  (-6,8)*{+1};
  (1,-1)*{\sccbub{}};
 \endxy}
 \;\; = \;\;
\vcenter{\xy 0;/r.18pc/:
    (-4,-10)*{};(0,0)*{} **\crv{(-4,-5) & (0,-5)};
    (0,0)*{};(-4,10)*{} **\crv{(0,5) & (-4,5)}?(1)*\dir{>};
  (8,8)*{-1};(-5,-3)*{\scs {\bf }};
  (0,-1)*{\sccbub{}};
 \endxy}
\end{equation}
by applying \eqref{eq_ident_decompT} to the rightmost term and using the definition of fake bubbles.
\end{proof}

%
\subsection{The 2-category $\Ucat$} \label{subsec_Ucat}
%

We make a preferred choice of coefficients for the 2-category $\Ucatt$.
\begin{defn}
Let $\Ucat$ denote the 2-category $\Ucatt$ with coefficients chosen so that $c_n^+=c_n^-=1$ and $\beta_n=-1$ for all $n \in \Z$.
\end{defn}

Since the 2-category $\Ucat$ is a key to our construction of the 2-category $\UcatD$ we pause to summarize the structure of this 2-category.  In  $\Ucat$ the inverse of the map $\zeta_+$ is given by
\begin{equation}
\xy
 (0,-25)*+{\cal{E}\cal{F}\onen}="B";
 (-65,0)*+{\cal{F}\cal{E}\onen}="m1";
  (-50,0)*{\oplus};
 (-35,0)*+{\onen\{n-1\}}="m2";
 (-12,0)*{\oplus \;\; \cdots \;\; \oplus};
 (38,0)*{\oplus\;\; \cdots\;\; \oplus};
 (10,0)*+{\onen\{n-1-2\ell\}}="m3";
 (65,0)*+{ \onen\{1-n\}}="m4";
    {\ar@/^1.6pc/^{-\;\xy 0;/r.15pc/:
    (-4,-4)*{};(4,4)*{} **\crv{(-4,-1) & (4,1)}?(1)*\dir{>} ;
    (4,-4)*{};(-4,4)*{} **\crv{(4,-1) & (-4,1)}?(0)*\dir{<};
  \endxy\;\;} "B";"m1"};
    {\ar@/^0.8pc/^{\vcenter{\xy 0;/r.15pc/:
  (-4,-2)*{}="t1";  (4,-2)*{}="t2";
  "t2";"t1" **\crv{(4,5) & (-4,5)};  ?(.85)*\dir{<}; \endxy} } "B";"m2"};
  {\ar_(.7){  \xsum{j_1+j_2=\ell}{}
    \vcenter{\xy 0;/r.15pc/:
    (-4,-2)*{}="t1";  (4,-2)*{}="t2";
     "t2";"t1" **\crv{(4,5) & (-4,5)};  ?(.85)*\dir{<} 
     ?(.2)*\dir{}+(0,-.1)*{\bullet}+(5,1)*{\scs j_1};
 (2,13)*{\ccbub{\;\;-n-1+j_2}{}};  \endxy} } "B";"m3"};
    {\ar@/_1.6pc/_(.8){  \xsum{j_1+j_2=n-1}{}
    \vcenter{\xy 0;/r.15pc/:
    (-4,-2)*{}="t1";  (4,-2)*{}="t2";
     "t2";"t1" **\crv{(4,5) & (-4,5)};  ?(.85)*\dir{<} 
     ?(.2)*\dir{}+(0,-.1)*{\bullet}+(5,1)*{\scs j_1};
 (2,13)*{\ccbub{\;\;-n-1+j_2}{}};  \endxy} } "B";"m4"};
 \endxy
\end{equation}
and the inverse of the map $\zeta_{-}$ is given by
\begin{equation}
\xy
 (0,-25)*+{\cal{F}\cal{E}\onen}="B";
(-60,0)*+{\cal{E}\cal{F}\onen}="m1";
  (-50,0)*+{\oplus};
 (-36,0)*+{\onen\{-n-1\}}="m2";
 (-15,0)*{\oplus \;\; \cdots \;\; \oplus};
 (36,0)*{\oplus\;\; \cdots\;\; \oplus};
 (10,0)*+{\onen\{-n-1-2\ell\}}="m3";
 (60,0)*+{\oplus \;\;\onen\{1+n\}}="m4";
    {\ar@/^1.6pc/^{-\;\xy 0;/r.15pc/:
  (0,0)*{\xybox{
    (-4,-4)*{};(4,4)*{} **\crv{(-4,-1) & (4,1)}?(0)*\dir{<} ;
    (4,-4)*{};(-4,4)*{} **\crv{(4,-1) & (-4,1)}?(1)*\dir{>};
    (-7,-3)*{\scs };     (6,-3)*{\scs };     (10,2)*{n};
     }};
  \endxy\;\;} "B";"m1"};
{\ar@/^0.8pc/^{\vcenter{\xy 0;/r.15pc/:
     (4,-2)*{}="t1";  (-4,-2)*{}="t2";  "t2";"t1" **\crv{(-4,5) & (4,5)};  ?(.85)*\dir{<}; \endxy} } "B";"m2"};
{\ar_(.7){  \xsum{j_1+j_2=\ell}{}
    \vcenter{\xy 0;/r.15pc/:
    (-4,-2)*{}="t1";  (4,-2)*{}="t2";
     "t2";"t1" **\crv{(4,5) & (-4,5)};  ?(.9)*\dir{>} 
     ?(.2)*\dir{}+(0,-.1)*{\bullet}+(5,1)*{\scs j_1};
     (2,13)*{\cbub{\;\;n-1+j_2}{}};  \endxy} } "B";"m3"};
{\ar@/_1.6pc/_(.8){  \xsum{j_1+j_2=-n-1}{}
    \vcenter{\xy 0;/r.15pc/:
    (-4,-2)*{}="t1";  (4,-2)*{}="t2";
     "t2";"t1" **\crv{(4,5) & (-4,5)};  ?(.9)*\dir{<} 
     ?(.2)*\dir{}+(0,-.1)*{\bullet}+(5,1)*{\scs j_1};
     (2,13)*{\cbub{\;\;n-1+j_2}{}};  \endxy} } "B";"m4"};
 \endxy
\end{equation}
Note that all bubbles that appear in $\overline{\zeta_{+}}$ and $\overline{\zeta_{-}}$ above are fake bubbles.  Specifying these inverse maps in a compact form was one of the main motivations for introducing fake bubbles in Section~\ref{subsec_fake}.

The requirement that the above maps are inverses of $\zeta_{+}$ and $\zeta_{-}$ in the 2-category $\Ucat$ is equivalent to adding the following more succinct relations.
\begin{enumerate}
\item  All dotted bubbles of negative degree are zero. That is,
\begin{equation} 
 \xy 0;/r.18pc/:
 (-12,0)*{\cbub{\alpha}};
 (-8,8)*{n};
 \endxy
  = 0
 \qquad
  \text{if $\alpha<n-1$,} \qquad
 \xy 0;/r.18pc/:
 (-12,0)*{\ccbub{\alpha}};
 (-8,8)*{n};
 \endxy = 0\quad
  \text{if $\alpha< -n-1$}
\end{equation}
for all $\alpha \in \Z_+$.  A dotted bubble of degree zero equals 1:
\[
\xy 0;/r.18pc/:
 (0,0)*{\cbub{n-1}};
  (4,8)*{n};
 \endxy
  = 1 \quad \text{for $n \geq 1$,}
  \qquad \quad
  \xy 0;/r.18pc/:
 (0,0)*{\ccbub{-n-1}};
  (4,8)*{n};
 \endxy
  = 1 \quad \text{for $n \leq -1$.}
\]

  \item For the following relations we employ the convention that all summations
are increasing, so that $\sum_{f=0}^{\alpha}$ is zero if $\alpha < 0$.
\begin{equation} \label{eq_reduction}
  \xy 0;/r.16pc/:
  (14,8)*{n};
  (-3,-8)*{};(3,8)*{} **\crv{(-3,-1) & (3,1)}?(1)*\dir{>};?(0)*\dir{>};
    (3,-8)*{};(-3,8)*{} **\crv{(3,-1) & (-3,1)}?(1)*\dir{>};
  (-3,-12)*{\bbsid};  (-3,8)*{\bbsid};
  (3,8)*{}="t1";  (9,8)*{}="t2";
  (3,-8)*{}="t1'";  (9,-8)*{}="t2'";
   "t1";"t2" **\crv{(3,14) & (9, 14)};
   "t1'";"t2'" **\crv{(3,-14) & (9, -14)};
   "t2'";"t2" **\dir{-} ?(.5)*\dir{<};
   (9,0)*{}; (-6,-8)*{\scs };
 \endxy \;\; = \;\; -\sum_{f_1+f_2=-n}
   \xy 0;/r.18pc/:
  (19,4)*{n};
  (0,0)*{\bbe{}};(-2,-8)*{\scs };
  (12,-2)*{\cbub{(n-1)+f_2}};
  (0,6)*{\bullet}+(3,-1)*{\scs f_1};
 \endxy
\qquad \qquad
\xy 0;/r.16pc/:
  (-12,8)*{n};
   (-3,-8)*{};(3,8)*{} **\crv{(-3,-1) & (3,1)}?(1)*\dir{>};?(0)*\dir{>};
    (3,-8)*{};(-3,8)*{} **\crv{(3,-1) & (-3,1)}?(1)*\dir{>};
  (3,-12)*{\bbsid};
  (3,8)*{\bbsid}; (6,-8)*{\scs };
  (-9,8)*{}="t1";
  (-3,8)*{}="t2";
  (-9,-8)*{}="t1'";
  (-3,-8)*{}="t2'";
   "t1";"t2" **\crv{(-9,14) & (-3, 14)};
   "t1'";"t2'" **\crv{(-9,-14) & (-3, -14)};
  "t1'";"t1" **\dir{-} ?(.5)*\dir{<};
 \endxy \;\; = \;\;
 \sum_{g_1+g_2=n}^{}
   \xy 0;/r.18pc/:
  (-12,8)*{n};
  (0,0)*{\bbe{}};(2,-8)*{\scs};
  (-12,-2)*{\ccbub{(-n-1)+g_2}};
  (0,6)*{\bullet}+(3,-1)*{\scs g_1};
 \endxy
\end{equation}
\begin{eqnarray}
 \vcenter{\xy 0;/r.18pc/:
  (-8,0)*{};
  (8,0)*{};
  (-4,10)*{}="t1";
  (4,10)*{}="t2";
  (-4,-10)*{}="b1";
  (4,-10)*{}="b2";(-6,-8)*{\scs };(6,-8)*{\scs };
  "t1";"b1" **\dir{-} ?(.5)*\dir{<};
  "t2";"b2" **\dir{-} ?(.5)*\dir{>};
  (10,2)*{n};
  (-10,2)*{n};
  \endxy}
&\quad = \quad&
 -\;\;
 \vcenter{   \xy 0;/r.18pc/:
    (-4,-4)*{};(4,4)*{} **\crv{(-4,-1) & (4,1)}?(1)*\dir{>};
    (4,-4)*{};(-4,4)*{} **\crv{(4,-1) & (-4,1)}?(1)*\dir{<};?(0)*\dir{<};
    (-4,4)*{};(4,12)*{} **\crv{(-4,7) & (4,9)};
    (4,4)*{};(-4,12)*{} **\crv{(4,7) & (-4,9)}?(1)*\dir{>};
  (8,8)*{n};(-6,-3)*{\scs };
     (6.5,-3)*{\scs };
 \endxy}
  \quad + \quad
   \sum_{ \xy  (0,3)*{\scs f_1+f_2+f_3}; (0,0)*{\scs =n-1};\endxy}
    \vcenter{\xy 0;/r.18pc/:
    (-10,10)*{n};
    (-8,0)*{};
  (8,0)*{};
  (-4,-15)*{}="b1";
  (4,-15)*{}="b2";
  "b2";"b1" **\crv{(5,-8) & (-5,-8)}; ?(.05)*\dir{<} ?(.93)*\dir{<}
  ?(.8)*\dir{}+(0,-.1)*{\bullet}+(-3,2)*{\scs f_3};
  (-4,15)*{}="t1";
  (4,15)*{}="t2";
  "t2";"t1" **\crv{(5,8) & (-5,8)}; ?(.15)*\dir{>} ?(.95)*\dir{>}
  ?(.4)*\dir{}+(0,-.2)*{\bullet}+(3,-2)*{\scs \; f_1};
  (0,0)*{\ccbub{\scs \quad (-n-1)+f_2}};
  \endxy} \nn
 \\  \; \nn \\
 \vcenter{\xy 0;/r.18pc/:
  (-8,0)*{};(-6,-8)*{\scs };(6,-8)*{\scs };
  (8,0)*{};
  (-4,10)*{}="t1";
  (4,10)*{}="t2";
  (-4,-10)*{}="b1";
  (4,-10)*{}="b2";
  "t1";"b1" **\dir{-} ?(.5)*\dir{>};
  "t2";"b2" **\dir{-} ?(.5)*\dir{<};
  (10,2)*{n};
  (-10,2)*{n};
  \endxy}
&\quad = \quad&
 -\;\;
   \vcenter{\xy 0;/r.18pc/:
    (-4,-4)*{};(4,4)*{} **\crv{(-4,-1) & (4,1)}?(1)*\dir{<};?(0)*\dir{<};
    (4,-4)*{};(-4,4)*{} **\crv{(4,-1) & (-4,1)}?(1)*\dir{>};
    (-4,4)*{};(4,12)*{} **\crv{(-4,7) & (4,9)}?(1)*\dir{>};
    (4,4)*{};(-4,12)*{} **\crv{(4,7) & (-4,9)};
  (8,8)*{n};(-6,-3)*{\scs };  (6,-3)*{\scs };
 \endxy}
  \quad + \quad
    \sum_{ \xy  (0,3)*{\scs g_1+g_2+g_3}; (0,0)*{\scs =-n-1};\endxy}
    \vcenter{\xy 0;/r.18pc/:
    (-8,0)*{};
  (8,0)*{};
  (-4,-15)*{}="b1";
  (4,-15)*{}="b2";
  "b2";"b1" **\crv{(5,-8) & (-5,-8)}; ?(.1)*\dir{>} ?(.95)*\dir{>}
  ?(.8)*\dir{}+(0,-.1)*{\bullet}+(-3,2)*{\scs g_3};
  (-4,15)*{}="t1";
  (4,15)*{}="t2";
  "t2";"t1" **\crv{(5,8) & (-5,8)}; ?(.15)*\dir{<} ?(.9)*\dir{<}
  ?(.4)*\dir{}+(0,-.2)*{\bullet}+(3,-2)*{\scs g_1};
  (0,0)*{\cbub{\scs \quad\; (n-1) + g_2}};
  (-10,10)*{n};
  \endxy} \label{eq_ident_decomp}
\end{eqnarray}
for all $n\in \Z$.  In equations \eqref{eq_reduction} and \eqref{eq_ident_decomp} whenever the summations are nonzero they utilize fake bubbles.
\end{enumerate}

\begin{example} Below we give several examples of the curl relation for different values of $n$.
  \begin{enumerate}[a)]
 \item If $n>0$,  then $\xy 0;/r.13pc/:
  (14,8)*{n};
  (-3,-8)*{};(3,8)*{} **\crv{(-3,-1) & (3,1)}?(1)*\dir{>};?(0)*\dir{>};
    (3,-8)*{};(-3,8)*{} **\crv{(3,-1) & (-3,1)}?(1)*\dir{>};
  (-3,-12)*{\bbsid};
  (-3,8)*{\bbsid};
  (3,8)*{}="t1";
  (9,8)*{}="t2";
  (3,-8)*{}="t1'";
  (9,-8)*{}="t2'";
   "t1";"t2" **\crv{(3,14) & (9, 14)};
   "t1'";"t2'" **\crv{(3,-14) & (9, -14)};
   "t2'";"t2" **\dir{-} ?(.5)*\dir{<};
   (9,0)*{};
 \endxy \quad = \quad 0.$
 \item If $n=0$, then $$\xy 0;/r.14pc/:
  (14,8)*{n};
  (-3,-8)*{};(3,8)*{} **\crv{(-3,-1) & (3,1)}?(1)*\dir{>};?(0)*\dir{>};
    (3,-8)*{};(-3,8)*{} **\crv{(3,-1) & (-3,1)}?(1)*\dir{>};
  (-3,-12)*{\bbsid};
  (-3,8)*{\bbsid};
  (3,8)*{}="t1";
  (9,8)*{}="t2";
  (3,-8)*{}="t1'";
  (9,-8)*{}="t2'";
   "t1";"t2" **\crv{(3,14) & (9, 14)};
   "t1'";"t2'" **\crv{(3,-14) & (9, -14)};
   "t2'";"t2" **\dir{-} ?(.5)*\dir{<};
   (9,0)*{};
 \endxy \quad = \quad -\;
 \xy 0;/r.16pc/:
  (8,6)*{0};
  (0,0)*{\bbe{}};
  (12,-2)*{\cbub{-1}{}};
 \endxy \quad  = \quad -\;
 \xy 0;/r.16pc/:
  (5,4)*{0};
  (0,0)*{\bbe{}};
 \endxy$$
 since the fake bubble appearing above has $\quad \deg\left( \xy 0;/r.16pc/:
  (8,6)*{0};
  (0,-2)*{\cbub{-1}{}};
 \endxy \right) = 2(1-0)-2 =0 \quad $ and we have defined degree zero bubbles to be multiplication by $1$.

 \item If $n=-1$, then the curl relation takes the form
$$\xy 0;/r.14pc/:
  (14,8)*{-1};
  (-3,-8)*{};(3,8)*{} **\crv{(-3,-1) & (3,1)}?(1)*\dir{>};?(0)*\dir{>};
    (3,-8)*{};(-3,8)*{} **\crv{(3,-1) & (-3,1)}?(1)*\dir{>};
  (-3,-12)*{\bbsid};
  (-3,8)*{\bbsid};
  (3,8)*{}="t1";
  (9,8)*{}="t2";
  (3,-8)*{}="t1'";
  (9,-8)*{}="t2'";
   "t1";"t2" **\crv{(3,14) & (9, 14)};
   "t1'";"t2'" **\crv{(3,-14) & (9, -14)};
   "t2'";"t2" **\dir{-} ?(.5)*\dir{<};
   (9,0)*{};
 \endxy \quad = \quad -\;
 \xy 0;/r.16pc/:
  (8,7)*{-1};
  (0,0)*{\bbe{}};
  (12,-2)*{\cbub{-2}{}};
  (0,6)*{\bullet};
 \endxy \;\; - \;\;
 \xy 0;/r.16pc/:
  (8,7)*{-1};
  (0,0)*{\bbe{}};
  (12,-2)*{\cbub{-1}{}};
 \endxy$$
Computing degrees we see that
$\deg\left( \xy 0;/r.16pc/:
  (8,6)*{-1};
  (0,-2)*{\cbub{-2}{}};
 \endxy \right)=0 $ and $ \deg\left( \xy 0;/r.16pc/:
  (8,6)*{-1};
  (0,-2)*{\cbub{-1}{}};
 \endxy\right) = 2$, so that using the infinite Grassmannian relations to rewrite the fake bubbles in terms of real bubbles gives
\begin{equation}
   \xy 0;/r.16pc/:
  (8,6)*{-1};
  (0,-2)*{\cbub{-2}{}};
 \endxy :=1\qquad  \qquad \xy 0;/r.16pc/:
  (8,6)*{-1};
  (0,-2)*{\cbub{-1}{}};
 \endxy = -\; \xy 0;/r.16pc/:
  (8,6)*{-1};
  (0,-2)*{\ccbub{1}{}};
 \endxy \nn
\end{equation}
so that the curl relation written in terms of real bubbles has the form
\begin{equation}
\xy 0;/r.14pc/:
  (18,8)*{-1};
  (-3,-8)*{};(3,8)*{} **\crv{(-3,-1) & (3,1)}?(1)*\dir{>};?(0)*\dir{>};
    (3,-8)*{};(-3,8)*{} **\crv{(3,-1) & (-3,1)}?(1)*\dir{>};
  (-3,-12)*{\bbsid};
  (-3,8)*{\bbsid};
  (3,8)*{}="t1";
  (9,8)*{}="t2";
  (3,-8)*{}="t1'";
  (9,-8)*{}="t2'";
   "t1";"t2" **\crv{(3,14) & (9, 14)};
   "t1'";"t2'" **\crv{(3,-14) & (9, -14)};
   "t2'";"t2" **\dir{-} ?(.5)*\dir{<};
   (9,0)*{};
 \endxy \quad = \quad -\;
 \xy 0;/r.16pc/:
  (8,6)*{-1};
  (0,0)*{\bbe{}};
  (0,6)*{\bullet};
 \endxy \;\; + \;\;
 \xy 0;/r.16pc/:
  (11,7)*{-1};
  (0,0)*{\bbe{}};
  (12,-2)*{\ccbub{1}{}};
 \endxy \nn
\end{equation}
for $n=-1$.

\end{enumerate}
\end{example}

\begin{example}[$n=0$, \quad  $\cal{E}\cal{F}\onenn{0}\cong \cal{F}\cal{E}\onenn{0}$] For $n=0$ the isomorphism $\cal{E}\cal{F}\onenn{0}\cong \cal{F}\cal{E}\onenn{0}$ takes the form
\[
 \xy
 (-20,0)*{\cal{F}\cal{E}\onenn{0}}="1"; (20,0)*{\cal{E}\cal{F}\onenn{0}}="2";
 {\ar@<1ex>^{\xy 0;/r.17pc/:
    (0,0)*{\xybox{
        (-4,-4)*{};(4,4)*{} **\crv{(-4,-1) & (4,1)}?(0)*\dir{<} ;
        (4,-4)*{};(-4,4)*{} **\crv{(4,-1) & (-4,1)}?(1)*\dir{>};
        (8,0)*{\scs 0};
        (-12,0)*{};(12,0)*{};   }};
  \endxy} "1";"2"};
 {\ar@<1ex>^{-\;\;\xy 0;/r.17pc/:
    (0,0)*{\xybox{
        (-4,-4)*{};(4,4)*{} **\crv{(-4,-1) & (4,1)}?(1)*\dir{>} ;
        (4,-4)*{};(-4,4)*{} **\crv{(4,-1) & (-4,1)}?(0)*\dir{<};
        (8,0)*{\scs 0};
        (12,0)*{};   }};
  \endxy}  "2";"1"};
 \endxy
\]
These maps are isomorphisms since for $n=0$ the relations
\[ -\;
   \vcenter{\xy 0;/r.17pc/:
    (-4,-4)*{};(4,4)*{} **\crv{(-4,-1) & (4,1)}?(1)*\dir{<};?(0)*\dir{<};
    (4,-4)*{};(-4,4)*{} **\crv{(4,-1) & (-4,1)}?(1)*\dir{>};
    (-4,4)*{};(4,12)*{} **\crv{(-4,7) & (4,9)}?(1)*\dir{>};
    (4,4)*{};(-4,12)*{} **\crv{(4,7) & (-4,9)};
  (8,8)*{0};(-6,-3)*{ };
     (6,-3)*{ };
 \endxy}
\quad = \quad
 \vcenter{\xy 0;/r.15pc/:
  (-8,0)*{};(-6,-8)*{ };(6,-8)*{ };
  (8,0)*{};
  (-4,10)*{}="t1";
  (4,10)*{}="t2";
  (-4,-10)*{}="b1";
  (4,-10)*{}="b2";
  "t1";"b1" **\dir{-} ?(.5)*\dir{>};
  "t2";"b2" **\dir{-} ?(.5)*\dir{<};
  (10,2)*{0};
  \endxy}
\qquad \quad
-\;\;
 \vcenter{   \xy 0;/r.19pc/:
    (-4,-4)*{};(4,4)*{} **\crv{(-4,-1) & (4,1)}?(1)*\dir{>};
    (4,-4)*{};(-4,4)*{} **\crv{(4,-1) & (-4,1)}?(1)*\dir{<};?(0)*\dir{<};
    (-4,4)*{};(4,12)*{} **\crv{(-4,7) & (4,9)};
    (4,4)*{};(-4,12)*{} **\crv{(4,7) & (-4,9)}?(1)*\dir{>};
  (8,8)*{0};(-6,-3)*{ };
     (6.5,-3)*{ };
 \endxy}
  \quad = \quad
   \vcenter{\xy 0;/r.16pc/:
  (-8,0)*{};
  (8,0)*{};
  (-4,10)*{}="t1";
  (4,10)*{}="t2";
  (-4,-10)*{}="b1";
  (4,-10)*{}="b2";(-6,-8)*{ };(6,-8)*{ };
  "t1";"b1" **\dir{-} ?(.5)*\dir{<};
  "t2";"b2" **\dir{-} ?(.5)*\dir{>};
  (10,2)*{0};
  \endxy}
\]
hold in $\Ucat$.
\end{example}

\begin{example}[$n=1\quad \cal{E}\cal{F} \onenn{1} \cong \cal{F}\cal{E} \onenn{1} \oplus \onenn{1}$ ]
Because $n$ is small there are no fake bubbles that appear in this example.
The isomorphism for $n=1$ takes the form
$$\xy
 (0,12)*+{\cal{E}\cal{F}\onenn{1}}="T";
 (0,-12)*+{\cal{E}\cal{F}\onenn{1}}="B";
 (-22,0)*+{\cal{F}\cal{E}\onenn{1}}="m1";
 (22,0)*+{\onenn{1}\{0\}}="m2";
  {\ar^{(\zeta_+)_1 =\xy 0;/r.17pc/:
    (0,0)*{\xybox{
        (-4,-4)*{};(4,4)*{} **\crv{(-4,-1) & (4,1)}?(0)*\dir{<} ;
        (4,-4)*{};(-4,4)*{} **\crv{(4,-1) & (-4,1)}?(1)*\dir{>};
        (8,0)*{\scs 1};
        (-12,0)*{};(12,0)*{};   }};
  \endxy} "m1";"T"};
    {\ar_{(\overline{\zeta_{+}})_2 =\xy 0;/r.17pc/:
        (0,0)*{\bbcfe{}};
        (8,4)*{\scs  1};
        (-12,0)*{};(12,0)*{}
    \endxy} "B";"m2"};
 {\ar^{(\overline{\zeta_{+}})_1=-\;  \xy 0;/r.17pc/:
    (0,0)*{\xybox{
        (-4,-4)*{};(4,4)*{} **\crv{(-4,-1) & (4,1)}?(1)*\dir{>} ;
        (4,-4)*{};(-4,4)*{} **\crv{(4,-1) & (-4,1)}?(0)*\dir{<};
        (8,0)*{\scs 1};
        (-8,0)*{};(12,0)*{};
     }}  \endxy} "B";"m1"};
{\ar_{(\zeta_+)_2= \xy 0;/r.17pc/:
    (0,0)*{\bbpfe{}};
    (6,-4)*{\scs 1};
    (-12,0)*{};(12,0)*{};
  \endxy} "m2";"T"};
 \endxy $$

Checking that these 2-morphisms define an isomorphism  $\cal{E}\cal{F} \onenn{1} \cong \cal{F}\cal{E} \onenn{1} \oplus \onenn{1}$ amounts to checking that
\[
 (\zeta_+)_1 \circ (\overline{\zeta_{+}})_1 + (\zeta_+)_2 \circ (\overline{\zeta_{+}})_2 =1_{\cal{E}\cal{F}\onenn{1}}, \quad
 (\overline{\zeta_{+}})_1 \circ (\zeta_+)_1 =1_{\cal{F}\cal{E}\onenn{1}},
 \quad
 (\overline{\zeta_{+}})_2 \circ (\zeta_+)_2 =1_{\onenn{1}},
\]
\[
(\overline{\zeta_{+}})_2 \circ (\zeta_+)_1 = 0, \qquad (\zeta_+)_2 \circ (\overline{\zeta_{+}})_1 =0.
\]
All of these equations follow from the relations in $\Ucat$:
$$
  (\overline{\zeta_{+}})_1 \circ (\zeta_+)_1 \;\; =\;\;
  -\;
   \vcenter{\xy 0;/r.17pc/:
    (-4,-4)*{};(4,4)*{} **\crv{(-4,-1) & (4,1)}?(1)*\dir{<};?(0)*\dir{<};
    (4,-4)*{};(-4,4)*{} **\crv{(4,-1) & (-4,1)}?(1)*\dir{>};
    (-4,4)*{};(4,12)*{} **\crv{(-4,7) & (4,9)}?(1)*\dir{>};
    (4,4)*{};(-4,12)*{} **\crv{(4,7) & (-4,9)};
  (8,8)*{1};(-6,-3)*{ };
     (6,-3)*{ };
 \endxy}
\quad = \quad
 \vcenter{\xy 0;/r.15pc/:
  (-8,0)*{};(-6,-8)*{ };(6,-8)*{ };
  (8,0)*{};
  (-4,10)*{}="t1";
  (4,10)*{}="t2";
  (-4,-10)*{}="b1";
  (4,-10)*{}="b2";
  "t1";"b1" **\dir{-} ?(.5)*\dir{>};
  "t2";"b2" **\dir{-} ?(.5)*\dir{<};
  (10,2)*{1};
  (-10,2)*{1};
  \endxy}\;\; =\;\;1_{\cal{F}\cal{E}\onenn{1}}$$
$$
    (\overline{\zeta_{+}})_2 \circ (\zeta_+)_2 \;\; =\;\;\xy 0;/r.17pc/:
        (0,0)*{\ncbub};
        (8,5)*{\scs  1};
        (-12,0)*{};(12,0)*{}
    \endxy =1  =1_{\onenn{1}}
$$
\[
(\overline{\zeta_{+}})_2 \circ (\zeta_+)_1  \;\; =\;\;\xy 0;/r.17pc/:
  (0,0)*{\xybox{
    (-4,-4)*{};(4,4)*{} **\crv{(-4,-1) & (4,1)}?(0)*\dir{<} ;
    (4,-4)*{};(-4,4)*{} **\crv{(4,-1) & (-4,1)}?(1)*\dir{>};
    (-4,4)*{}; (4,4)*{} **\crv{(-4,8) & (4,8)};
     (8,0)*{\scs 1};
     (-12,0)*{};(12,0)*{};
     }};
  \endxy \quad = \quad  0,
\]
\[
(\zeta_+)_1 \circ (\overline{\zeta_{+}})_1 + (\zeta_+)_2 \circ (\overline{\zeta_{+}})_2 \;=\;
-\;\;
 \vcenter{   \xy 0;/r.19pc/:
    (-4,-4)*{};(4,4)*{} **\crv{(-4,-1) & (4,1)}?(1)*\dir{>};
    (4,-4)*{};(-4,4)*{} **\crv{(4,-1) & (-4,1)}?(1)*\dir{<};?(0)*\dir{<};
    (-4,4)*{};(4,12)*{} **\crv{(-4,7) & (4,9)};
    (4,4)*{};(-4,12)*{} **\crv{(4,7) & (-4,9)}?(1)*\dir{>};
  (8,8)*{1};(-6,-3)*{ };
     (6.5,-3)*{ };
 \endxy}
  \quad + \quad
    \vcenter{\xy 0;/r.16pc/:
    (-10,4)*{1};
    (-8,0)*{};
  (8,0)*{};
  (-4,-10)*{}="b1";
  (4,-10)*{}="b2";
  "b2";"b1" **\crv{(5,-3) & (-5,-3)}; ?(.05)*\dir{<} ?(.93)*\dir{<};
  (-4,10)*{}="t1";
  (4,10)*{}="t2";
  "t2";"t1" **\crv{(5,3) & (-5,3)}; ?(.15)*\dir{>} ?(.95)*\dir{>}
  \endxy} \quad = \quad
   \vcenter{\xy 0;/r.16pc/:
  (-8,0)*{};
  (8,0)*{};
  (-4,10)*{}="t1";
  (4,10)*{}="t2";
  (-4,-10)*{}="b1";
  (4,-10)*{}="b2";(-6,-8)*{ };(6,-8)*{ };
  "t1";"b1" **\dir{-} ?(.5)*\dir{<};
  "t2";"b2" **\dir{-} ?(.5)*\dir{>};
  (10,2)*{1};
  (-10,2)*{1};
  \endxy} \;\; = \;\;1_{\cal{E}\cal{F}\onenn{1}}.
\]
\end{example}

%
\subsection{Rescalling 2-functors} \label{subsubsec_rescaling}
%

In this section we show that the 2-categories $\Ucatt$ are isomorphic to the 2-category $\Ucat$ for any choice of choice of parameters satisfying the conditions in Proposition~\ref{prop_coeff}.  Note that in $\Ucat$ the coefficients $\alpha_{\lambda}^{\ell}(n)$ are equal to the coefficients $\overline{\alpha_{\lambda}^{\ell}(n)}$ since $c_n^{\pm}=1$ in $\Ucat$.

\begin{thm}
Let $\chi$ be a set of parameters satisfying the conditions in Proposition~\ref{prop_coeff}.  There is an isomorphism of 2-categories
\begin{eqnarray}
 \cal{M} \maps \Ucat & \to & \Ucatt \\
  \cal{M}(n) &=& n \nn\\
  \cal{M}(\cal{E}_{\ep}\onen\{t\}) &=& \cal{E}_{\ep}\onen\{t\}\nn
\end{eqnarray}
\begin{equation}
 \cal{M}\left( \; \textcolor[rgb]{1.00,0.00,0.00}{\xy
 (0,7);(0,-7); **\dir{-} ?(.75)*\dir{>};
 (0,-2)*{\txt\large{$\bullet$}};
 (6,4)*{n}; (-8,4)*{n +2}; (-10,0)*{};(10,0)*{};
 \endxy} \; \right) \;\; = \;\; \xy
 (0,7);(0,-7); **\dir{-} ?(.75)*\dir{>};
 (0,-2)*{\txt\large{$\bullet$}};
 (6,4)*{n}; (-8,4)*{n +2}; (-10,0)*{};(10,0)*{};
 \endxy \qquad \quad
 \cal{M}\left(\; \textcolor[rgb]{1.00,0.00,0.00}{\xy
 (0,7);(0,-7); **\dir{-} ?(.75)*\dir{<};
 (0,-2)*{\txt\large{$\bullet$}};
 (-6,4)*{n}; (8,4)*{n+2}; (-10,0)*{};(10,9)*{};
 \endxy} \; \right) \;\; =\;\; \xy
 (0,7);(0,-7); **\dir{-} ?(.75)*\dir{<};
 (0,-2)*{\txt\large{$\bullet$}};
 (-6,4)*{n}; (8,4)*{n+2}; (-10,0)*{};(10,9)*{};
 \endxy \nn
\end{equation}

\begin{equation}
\cal{M}\left(\;\textcolor[rgb]{1.00,0.00,0.00}{\xy
  (0,0)*{\xybox{
    (-4,-4)*{};(4,4)*{} **\crv{(-4,-1) & (4,1)}?(1)*\dir{>} ;
    (4,-4)*{};(-4,4)*{} **\crv{(4,-1) & (-4,1)}?(1)*\dir{>};
     (8,1)*{n};     (-12,0)*{};(12,0)*{};     }};
  \endxy} \;\right) \;\; = \;\;
\xy
  (0,0)*{\xybox{
    (-4,-4)*{};(4,4)*{} **\crv{(-4,-1) & (4,1)}?(1)*\dir{>} ;
    (4,-4)*{};(-4,4)*{} **\crv{(4,-1) & (-4,1)}?(1)*\dir{>};
     (8,1)*{n};     (-12,0)*{};(12,0)*{};     }};
  \endxy
  \qquad \quad
\cal{M}\left(\;    \textcolor[rgb]{1.00,0.00,0.00}{\xy
  (0,0)*{\xybox{
    (-4,4)*{};(4,-4)*{} **\crv{(-4,1) & (4,-1)}?(1)*\dir{>} ;
    (4,4)*{};(-4,-4)*{} **\crv{(4,1) & (-4,-1)}?(1)*\dir{>};
     (8,1)*{ n};     (-12,0)*{};(12,0)*{};     }};
  \endxy}\;\right) \;=\;
     \xy
  (0,0)*{\xybox{
    (-4,4)*{};(4,-4)*{} **\crv{(-4,1) & (4,-1)}?(1)*\dir{>} ;
    (4,4)*{};(-4,-4)*{} **\crv{(4,1) & (-4,-1)}?(1)*\dir{>};
     (8,1)*{ n};     (-12,0)*{};(12,0)*{};     }};
  \endxy \nn
\end{equation}

\begin{eqnarray}
\cal{M}\left(\;
\textcolor[rgb]{1.00,0.00,0.00}{ \xy
    (0,0)*{\bbcfe{}};
    (8,5)*{n};
    \endxy}
\;\right) &=& \left\{
\begin{array}{ccl}
   \frac{1}{c_n^+}\;\xy
    (0,0)*{\bbcfe{}};
    (8,3)*{n};    (-12,6)*{};(8,0)*{};
    \endxy & \quad & \text{$n=2\ell$, or $n=2\ell+1$, for $\ell \in 2\Z_{\geq0}$} \\
  c_{n-2}^-  \xy
    (0,0)*{\bbcfe{}};
    (8,3)*{n};    (-12,6)*{};(8,0)*{};
    \endxy & \quad & \text{$n=-2\ell$, or $n=-(2\ell+3)$, for $\ell \in 2\Z_{\geq0}$} \\
   \xy
    (0,0)*{\bbcfe{}};
    (8,3)*{n};    (-12,6)*{};(8,0)*{};
    \endxy & \quad & \text{otherwise}
\end{array}
\right. \nn \\
\cal{M}\left(\;
\textcolor[rgb]{1.00,0.00,0.00}{ \xy
    (0,0)*{\bbcef{}};
    (8,5)*{n};
    \endxy}
\;\right) &=& \left\{
\begin{array}{ccl}
   c_{n+2}^+\;\xy
    (0,0)*{\bbcef{}};
    (8,3)*{n};    (-12,6)*{};(8,0)*{};
    \endxy & \quad & \text{$n=2\ell$, or $n=2\ell+1$, for $\ell \in 2\Z_{\geq0}$} \\
  \frac{1}{c_{n}^-}  \xy
    (0,0)*{\bbcef{}};
    (8,3)*{n};    (-12,6)*{};(8,0)*{};
    \endxy & \quad & \text{$n=-2\ell$, or $n=-(2\ell+3)$, for $\ell \in 2\Z_{\geq0}$} \\
   \xy
    (0,0)*{\bbcef{}};
    (8,3)*{n};    (-12,6)*{};(8,0)*{};
    \endxy & \quad & \text{otherwise}
\end{array}
\right. \nn
\end{eqnarray}
\begin{eqnarray}
  \cal{M}\left(\;
\textcolor[rgb]{1.00,0.00,0.00}{ \xy
    (0,0)*{\bbpfe{}};
    (8,5)*{n};
    \endxy}
\;\right) &=& \left\{
\begin{array}{ccl}
   \frac{1}{c_n^+}\;\xy
    (0,0)*{\bbpfe{}};
    (8,3)*{n};    (-12,6)*{};(8,0)*{};
    \endxy & \quad & \text{$n=2\ell+2$, or $n=2\ell+3$, for $\ell \in 2\Z_{\geq0}$} \\
  c_{n-2}^-  \xy
    (0,0)*{\bbpfe{}};
    (8,3)*{n};    (-12,6)*{};(8,0)*{};
    \endxy & \quad & \text{$n=-(2\ell+1)$, or $n=-(2\ell+2)$, for $\ell \in 2\Z_{\geq0}$} \\
   \xy
    (0,0)*{\bbpfe{}};
    (8,3)*{n};    (-12,6)*{};(8,0)*{};
    \endxy & \quad & \text{otherwise}
\end{array}
\right. \nn \\
  \cal{M}\left(\;
 \textcolor[rgb]{1.00,0.00,0.00}{\xy
    (0,0)*{\bbpef{}};
    (8,5)*{n};
    \endxy}
\;\right) &=& \left\{
\begin{array}{ccl}
   c_{n+2}^+\;\xy
    (0,0)*{\bbpef{}};
    (8,3)*{n};    (-12,6)*{};(8,0)*{};
    \endxy & \quad & \text{$n=2\ell+2$, or $n=2\ell+3$, for $\ell \in 2\Z_{\geq0}$} \\
  \frac{1}{c_n^-}  \xy
    (0,0)*{\bbpef{}};
    (8,3)*{n};    (-12,6)*{};(8,0)*{};
    \endxy & \quad & \text{$n=-(2\ell+1)$, or $n=-(2\ell+2)$, for $\ell \in 2\Z_{\geq0}$} \\
   \xy
    (0,0)*{\bbpef{}};
    (8,3)*{n};    (-12,6)*{};(8,0)*{};
    \endxy & \quad & \text{otherwise}
\end{array}
\right. \nn
\end{eqnarray}
\end{thm}

\begin{proof} For fixed $n$, it is easy to check using Proposition~\ref{prop_coeff} that the product of the rescaling factors of all caps and cups in the region $n$ is equal to $-\beta_n$.  Hence,
\begin{equation}
\cal{M} \left( \textcolor[rgb]{1.00,0.00,0.00}{\vcenter{   \xy 0;/r.19pc/:
    (-4,-4)*{};(4,4)*{} **\crv{(-4,-1) & (4,1)}?(1)*\dir{>};
    (4,-4)*{};(-4,4)*{} **\crv{(4,-1) & (-4,1)}?(1)*\dir{<};?(0)*\dir{<};
    (-4,4)*{};(4,12)*{} **\crv{(-4,7) & (4,9)};
    (4,4)*{};(-4,12)*{} **\crv{(4,7) & (-4,9)}?(1)*\dir{>};
  (8,8)*{n};(-6,-3)*{\scs };
     (6.5,-3)*{\scs };
 \endxy}}\right) \;\; = \;\; \cal{M}\left(  \textcolor[rgb]{1.00,0.00,0.00}{\xy   0;/r.16pc/:
    (-12,8)*{}="1";
    (-4,8)*{}="2";
    (8,8)*{}="3";
    (-12,-8);"1" **\dir{-};
    "1";"2" **\crv{(-12,16) & (-4,16)} ?(0)*\dir{<} ;
    "2";"3" **\crv{(-4,0) & (8,0)}?(1)*\dir{<};
    "3"; (8,16) **\dir{-};
    (-12,-8)*{}="1'";
    (-4,-8)*{}="2'";
    (8,-8)*{}="3'";
    "1'";"2'" **\crv{(-12,-16) & (-4,-16)} ?(0)*\dir{>} ;
    "2'";"3'" **\crv{(-4,0) & (8,0)};
    "3'"; (8,-16) **\dir{-};
    (14,0)*{n};
    (-15,0)*{ n};
    (0,16);(-6,0) **\crv{(0,8) & (-6,10)} ?(0)*\dir{<} ;;
    (0,-16);(-6,0) **\crv{(0,-8) & (-6,-10)};
    \endxy} \right) \;\; = \;\; -\beta_n
    \vcenter{   \xy 0;/r.19pc/:
    (-4,-4)*{};(4,4)*{} **\crv{(-4,-1) & (4,1)}?(1)*\dir{>};
    (4,-4)*{};(-4,4)*{} **\crv{(4,-1) & (-4,1)}?(1)*\dir{<};?(0)*\dir{<};
    (-4,4)*{};(4,12)*{} **\crv{(-4,7) & (4,9)};
    (4,4)*{};(-4,12)*{} **\crv{(4,7) & (-4,9)}?(1)*\dir{>};
  (8,8)*{n};(-6,-3)*{\scs };
     (6.5,-3)*{\scs };
 \endxy} \nn
\end{equation}
\begin{equation}
\cal{M} \left(    \textcolor[rgb]{1.00,0.00,0.00}{\vcenter{\xy 0;/r.18pc/:
    (-4,-4)*{};(4,4)*{} **\crv{(-4,-1) & (4,1)}?(1)*\dir{<};?(0)*\dir{<};
    (4,-4)*{};(-4,4)*{} **\crv{(4,-1) & (-4,1)}?(1)*\dir{>};
    (-4,4)*{};(4,12)*{} **\crv{(-4,7) & (4,9)}?(1)*\dir{>};
    (4,4)*{};(-4,12)*{} **\crv{(4,7) & (-4,9)};
  (8,8)*{n};(-6,-3)*{\scs };  (6,-3)*{\scs };
 \endxy}}\right) \;\; = \;\; \cal{M}\left( \textcolor[rgb]{1.00,0.00,0.00}{ \xy   0;/r.16pc/:
    (12,8)*{}="1";
    (4,8)*{}="2";
    (-8,8)*{}="3";
    (12,-8);"1" **\dir{-};
    "1";"2" **\crv{(12,16) & (4,16)} ?(0)*\dir{<} ;
    "2";"3" **\crv{(4,0) & (-8,0)}?(1)*\dir{<};
    "3"; (-8,16) **\dir{-};
    (12,-8)*{}="1'";
    (4,-8)*{}="2'";
    (-8,-8)*{}="3'";
    "1'";"2'" **\crv{(12,-16) & (4,-16)} ?(0)*\dir{>} ;
    "2'";"3'" **\crv{(4,0) & (-8,0)};
    "3'"; (-8,-16) **\dir{-};
    (-14,0)*{n};
    (15,0)*{ n};
    (0,16);(6,0) **\crv{(0,8) & (6,10)} ?(0)*\dir{<} ;;
    (0,-16);(6,0) **\crv{(0,-8) & (6,-10)};
    \endxy} \right) \;\; = \;\; -\beta_n
   \vcenter{\xy 0;/r.18pc/:
    (-4,-4)*{};(4,4)*{} **\crv{(-4,-1) & (4,1)}?(1)*\dir{<};?(0)*\dir{<};
    (4,-4)*{};(-4,4)*{} **\crv{(4,-1) & (-4,1)}?(1)*\dir{>};
    (-4,4)*{};(4,12)*{} **\crv{(-4,7) & (4,9)}?(1)*\dir{>};
    (4,4)*{};(-4,12)*{} **\crv{(4,7) & (-4,9)};
  (8,8)*{n};(-6,-3)*{\scs };  (6,-3)*{\scs };
 \endxy} \nn
\end{equation}
Furthermore, all real bubbles rescale as
\begin{equation} \label{eq_Mreal}
\cal{M}\left(
 \textcolor[rgb]{1.00,0.00,0.00}{ \vcenter{\xy 0;/r.18pc/:
    (2,-11)*{\cbub{n-1+j}{}};
  (12,-2)*{n};
 \endxy}}
\right) \;\;= \;\; \frac{1}{c_n^+}\;
  \vcenter{\xy 0;/r.18pc/:
    (2,-11)*{\cbub{n-1+j}{}};
  (12,-2)*{n};
 \endxy} \qquad \qquad
 \cal{M}\left(
  \textcolor[rgb]{1.00,0.00,0.00}{\vcenter{\xy 0;/r.18pc/:
    (2,-11)*{\ccbub{-n-1+j}{}};
  (12,-2)*{n};
 \endxy}}
\right) \;\;= \;\; \frac{1}{c_n^-}\;
  \vcenter{\xy 0;/r.18pc/:
    (2,-11)*{\ccbub{-n-1+j}{}};
  (12,-2)*{n};
 \endxy} .
\end{equation}
This implies that fake bubbles scale as follows: for $n<0$  and $0
\leq j < -n+1$
\begin{eqnarray}
\cal{M}\left(
\textcolor[rgb]{1.00,0.00,0.00}{  \vcenter{\xy 0;/r.18pc/:
    (2,-11)*{\cbub{n-1+j}{}};
  (12,-2)*{n};
 \endxy}}
\right) &:=& \xsum{\lambda: |\lambda|=j}{}
 \overline{\alpha_{\lambda,j}(n)}   \cal{M} \left( \textcolor[rgb]{1.00,0.00,0.00}{\vcenter{\xy 0;/r.18pc/:
    (2,-3)*{\ccbub{-n-1+\lambda_1}{}};
    (13,-2)*{\cdots};
    (28,-3)*{\ccbub{-n-1+\lambda_m}{}};
  (12,8)*{n};
 \endxy}} \right) \nn
 \\
 &\refequal{\eqref{eq_alpha}}& \xsum{\lambda: |\lambda|=j}{}
 \left( c_n^-\right)^{m+1}\alpha_{\lambda,j}(n)   \cal{M} \left( \textcolor[rgb]{1.00,0.00,0.00}{\vcenter{\xy 0;/r.18pc/:
    (2,-3)*{\ccbub{-n-1+\lambda_1}{}};
    (13,-2)*{\cdots};
    (28,-3)*{\ccbub{-n-1+\lambda_m}{}};
  (12,8)*{n};
 \endxy}} \right) \nn\\
  &\refequal{\eqref{eq_Mreal}}& c_{n}^-\xsum{\lambda: |\lambda|=j}{}
 \alpha_{\lambda,j}(n) \vcenter{\xy 0;/r.18pc/:
    (2,-3)*{\ccbub{-n-1+\lambda_1}{}};
    (13,-2)*{\cdots};
    (28,-3)*{\ccbub{-n-1+\lambda_m}{}};
  (12,8)*{n};
 \endxy} \nn \\
   &\refequal{\eqref{eq_fake_nleqz}}& - \beta_{n}c_{n}^-\;  \vcenter{\xy 0;/r.18pc/:
    (2,-11)*{\cbub{n-1+j}{}};
  (12,-2)*{n};
 \endxy} \nn
\end{eqnarray}
Similarly, for $n>0$ and $0 \leq j < n+1$
\begin{equation}
  \cal{M}\left(
\textcolor[rgb]{1.00,0.00,0.00}{  \vcenter{\xy 0;/r.18pc/:
    (2,-11)*{\ccbub{-n-1+j}{}};
  (12,-2)*{n};
 \endxy}}
  \right) \;\; = \;\;
   -\beta_nc_n^+\vcenter{\xy 0;/r.18pc/:
    (2,-11)*{\ccbub{-n-1+j}{}};
  (12,-2)*{n};
 \endxy} \nn
\end{equation}
Using these facts it is straightforward to verify that $\cal{M}$ is
a 2-functor.  Furthermore, since all of the coefficients used in its
definition are invertible, is also clear that $\cal{M}$ is a
2-isomorphism of 2-categories.
\end{proof}

%
\subsection{The Karoubi envelope and the 2-category $\UcatD$} \label{subsec_karoubi}
%

In order to categorify the $\Z[q,q^{-1}]$-module $\UA$ we must lift the divided powers $\cal{E}^{(a)}1_n=\frac{E^a}{[a]!}1_n$ and $F^{(b)}1_n=\frac{F^b}{[b]!}1_n$ and their products to the categorified setting.  In particular, we must introduce 1-morphisms $\cal{E}^{(a)}\onen$ and $\cal{F}^{(b)}\onen$ and give isomorphisms
\begin{equation}
  \cal{E}^a\onen \cong \bigoplus_{[a]!} \cal{E}^{(a)}\onen, \qquad \quad
  \cal{F}^b\onen \cong \bigoplus_{[b]!} \cal{F}^{(b)}\onen.
\end{equation}

The simplest example is the divided power $E^{(2)}1_n = \frac{E^2}{[2]}1_n$  which we can rewrite as $E^21_n =(q+q^{-1})E^{(2)}1_n$.  Lifting this equation to the categorical level requires isomorphisms
\[
\cal{E}^21_n \cong \cal{E}^{(2)}1_n\{1\} \oplus \cal{E}^{(2)}1_n\{-1\}.
\]
We will construct such a direct sum decomposition by finding an idempotent 2-morphism in $\Hom_{\Ucat}(\cal{E}\cal{E}\onen, \cal{E}\cal{E}\onen)$.

An idempotent $e \maps b\to b$ in a category $\cal{C}$ is a morphism such that
$e\circ e = e$.  The idempotent is said to split if there exist morphisms
\[
 \xymatrix{ b \ar[r]^g & b' \ar[r]^h &b}
\]
such that $e=h \circ g$ and $g\circ h = 1_{b'}$.  In an additive category we can write $b'=\im e$ so that the idempotent $e$ can be viewed as the projection onto a summand $b \cong \im e \oplus \im (1-e)$.

In the context of the 2-category $\Ucat$ we do not have splitting of idempotents in general. Splitting an idempotent would require us to take the image of an abstract diagram representing an idempotent 2-morphism in $\Ucat$.  Instead we enlarge $\Ucat$ to the 2-category $\UcatD$ by adding extra 1-morphisms so that all idempotent 2-morphisms split.

For a category $\cal{C}$ the Karoubi envelope $Kar(\cal{C})$ (also called the idempotent completion or Cauchy completion) is a minimal enlargement of the category $\cal{C}$ in which all idempotents split.  More precisely, the category $Kar(\cal{C})$ has
\begin{itemize}
  \item objects of $Kar(\cal{C})$:  pairs $(b,e)$
where $e \maps b \to b$ is an idempotent of $\cal{C}$.
\item morphisms: $(e,f,e') \maps (b,e) \to (b',e')$
where $f \maps b \to b'$ in $\cal{C}$ making the diagram
\begin{equation} \label{eq_Kar_morph}
 \xymatrix{
 b \ar[r]^f \ar[d]_e \ar[dr]^{f} & b' \ar[d]^{e'} \\ b \ar[r]_f & b'
 }
\end{equation}
commute, i.e. $ef=f=fe'$.
 \item identity 1-morphisms: $(e,e,e) \maps (b,e) \to (b,e)$.
\end{itemize}
When $\cal{C}$ is an additive category we write $(b,e)\in Kar(\cal{C})$ as $\im e$ and we have $b \cong \im e \oplus \im (1-e)$ in $Kar(\cal{C})$.  See \cite[Chapter 6.5]{Bor} for more discussion on the Karoubi envelope.

In order for all idempotent 2-morphisms to split in $\Ucat$ we must take the Karoubi envelope of all the additive categories $\Ucat(n,m)$ between objects $n,m\in \Z$. Define the 2-category $\UcatD$ to be the smallest 2-category which contains $\Ucat$ and has splitting of idempotent 2-morphisms.  More precisely, replace each Hom category in $\Ucat$ by its Karoubi envelope:
\[
 \UcatD(n,m) := Kar\big(\Ucat(n,m)\big)
\]
where $Kar\big(\Ucat(n,m)\big)$ is the usual Karoubi envelope of the additive category $\Ucat(n,m)$.  A 2-morphism in $\UcatD$ is a triple $(e,f,e')\maps
(\cal{E}_{\ep}\onen\{t\},e)
\to (\cal{E}_{\ep'}\onen\{t'\},e')$
where $e$ and $e'$ are idempotent 2-morphisms in $\Ucat$ and $f$ is
a degree $(t-t')$ 2-morphism in $\Ucat$ with the property that
\[
  \xy
 (0,0)*{\includegraphics[scale=0.4]{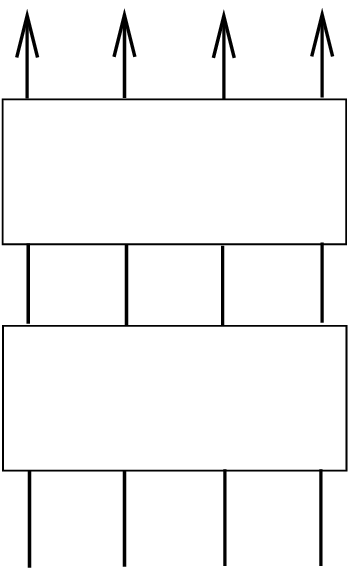}};
 (0,-4.5)*{e};(0,4.5)*{f};
  \endxy
 \quad
 =
 \quad
  \xy
 (0,0)*{\includegraphics[scale=0.4]{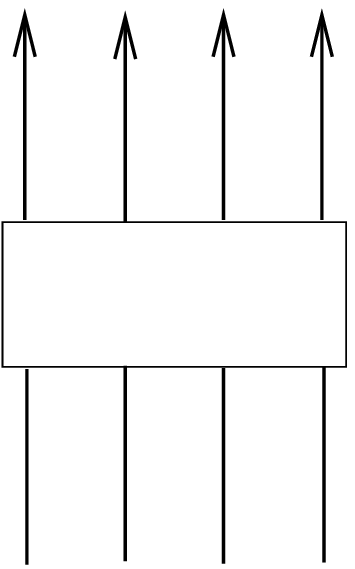}};
 (0,-0.5)*{f};
  \endxy
  \quad
 =
 \quad
   \xy
 (0,0)*{\includegraphics[scale=0.4]{c4-1.eps}};
 (0,-4.5)*{f};(0,4.5)*{e'};
  \endxy
\]

In the Karoubi envelope $\UcatD$ of $\Ucat$ any idempotent $e$ in $\Hom_{\Ucat}(\cal{E}\cal{E}\onen,\cal{E}\cal{E}\onen)$ gives a splitting
\[
\cal{E}^21_n \cong \cal{E}^21_n e \oplus \cal{E}^21_n(1-e).
\]
It is easy to check using the nilHecke relation \eqref{eq_nilHecke2} that the 2-morphism
\[
 e_2 \;\; = \;\;  \xy 0;/r.18pc/:
  (0,0)*{\xybox{
    (-4,-4)*{};(4,4)*{} **\crv{(-4,-1) & (4,1)}?(1)*\dir{>};
    (4,-4)*{};(-4,4)*{} **\crv{(4,-1) & (-4,1)}?(1)*\dir{>}?(.75)*{\bullet};
     (8,1)*{n};
     }};
  \endxy
\]
is an idempotent.  Indeed, precomposing the nilHecke relation with a crossing gives the identity
\[
  \vcenter{\xy 0;/r.16pc/:
    (-4,-4)*{};(4,4)*{} **\crv{(-4,-1) & (4,1)}?(1)*\dir{>};
    (4,-4)*{};(-4,4)*{} **\crv{(4,-1) & (-4,1)}?(1)*\dir{>};
    (8,1)*{n};
    (-.5,5.5)*{\textcolor[rgb]{0.00,0.00,1.00}{
    \xybox{(-4,5)*{\bullet};  (4,4)*{};(-4,12)*{} **\crv{(4,7) & (-4,9)}?(1)*\dir{>};
            (-4,4)*{};(4,12)*{} **\crv{(-4,7) & (4,9)}?(1)*\dir{>}}}};
 \endxy}
 \;\; - \;\;
  \vcenter{\xy 0;/r.16pc/:
    (-4,-4)*{};(4,4)*{} **\crv{(-4,-1) & (4,1)}?(1)*\dir{>};
    (4,-4)*{};(-4,4)*{} **\crv{(4,-1) & (-4,1)}?(1)*\dir{>};
    (8,1)*{n};
    (0,5.5)*{\textcolor[rgb]{0.00,0.00,1.00}{
    \xybox{(2,9)*{\bullet};  (4,4)*{};(-4,12)*{} **\crv{(4,7) & (-4,9)}?(1)*\dir{>};
            (-4,4)*{};(4,12)*{} **\crv{(-4,7) & (4,9)}?(1)*\dir{>}}}};
 \endxy} \;\; = \;\;
   \vcenter{\xy 0;/r.16pc/:
    (-4,-4)*{};(4,4)*{} **\crv{(-4,-1) & (4,1)}?(1)*\dir{>};
    (4,-4)*{};(-4,4)*{} **\crv{(4,-1) & (-4,1)}?(1)*\dir{>};
    (0,5.5)*{ \textcolor[rgb]{0.00,0.00,1.00}{\xybox{(-4,4)*{};(-4,12)*{} **\dir{-}?(1)*\dir{>};
    (4,4)*{};(4,12)*{} **\dir{-}?(1)*\dir{>};}}};
    (8,1)*{n};
 \endxy}
\qquad \To \qquad    \vcenter{\xy 0;/r.16pc/:
    (-4,-4)*{};(4,4)*{} **\crv{(-4,-1) & (4,1)}?(1)*\dir{>};
    (4,-4)*{};(-4,4)*{} **\crv{(4,-1) & (-4,1)};
    (-4,4)*{};(4,12)*{} **\crv{(-4,7) & (4,9)}?(1)*\dir{>};
    (4,4)*{};(-4,12)*{} **\crv{(4,7) & (-4,9)}?(1)*\dir{>};
    (8,1)*{n}; (-4,4)*{\bullet};
 \endxy} \;\; = \;\;
    \vcenter{\xy 0;/r.16pc/:
    (-4,-4)*{};(4,4)*{} **\crv{(-4,-1) & (4,1)}?(1)*\dir{>};
    (4,-4)*{};(-4,4)*{} **\crv{(4,-1) & (-4,1)}?(1)*\dir{>};
    (8,0)*{n};
 \endxy}
\]
which implies
\[
\qquad \qquad  e^2_2 \;\; = \;\;
  \vcenter{\xy 0;/r.16pc/:
    (-4,-4)*{};(4,4)*{} **\crv{(-4,-1) & (4,1)}?(1)*\dir{>};
    (4,-4)*{};(-4,4)*{} **\crv{(4,-1) & (-4,1)};
    (-4,4)*{};(4,12)*{} **\crv{(-4,7) & (4,9)}?(1)*\dir{>};
    (4,4)*{};(-4,12)*{} **\crv{(4,7) & (-4,9)}?(1)*\dir{>};
    (8,1)*{n}; (-3,2.2)*{\bullet}; (-2,9)*{\bullet};
 \endxy}
 \;\; = \;\;  \xy 0;/r.18pc/:
  (0,0)*{\xybox{
    (-4,-4)*{};(4,4)*{} **\crv{(-4,-1) & (4,1)}?(1)*\dir{>};
    (4,-4)*{};(-4,4)*{} **\crv{(4,-1) & (-4,1)}?(1)*\dir{>}?(.75)*{\bullet};
     (8,1)*{n};
     }};
  \endxy
 \;\; = \;\; e_2.
 \]
Using algebraic properties of the nilHecke algebra discussed below, one can show that $\im e_2 \cong \im (1-e_2)$ up to a shift, so that defining $\cal{E}^{(2)}1_n:=\im e_2\{1\}$ gives a categorification of the divided powers relation
\begin{equation}
  \cal{E}^21_n \cong \cal{E}^{(2)}1_n\{1\} \oplus \cal{E}^{(2)}1_n\{-1\}.
\end{equation}

The fact that the images of the idempotents $e_2$ and $(1-e_2)$ coincide are the same is no coincidence.  This is because the nilHecke algebra $\BNC_a$ is isomorphic to the algebra of $a!\times a!$ matrices with coefficients in the ring $\Lambda(x_1,\dots,x_a)$ of symmetric polynomials in the variables $x_1, \dots, x_a$, see \cite[Proposition 3.5]{Lau1} or more explicitly \cite{KLMS}.  This can be seen using the relationship between symmetric functions and Schubert polynomials, together with the well understood action of divided difference operators on the Schubert polynomials.

Recall that an idempotent idempotents $e'$, $e''$ are orthogonal when $e'e'' = e''e' =0$, and that an idempotent $e$ is said to be minimal (or primitive) if $e$ cannot be written as $e = e'+e''$ for some orthogonal idempotents $e'$, $e''$. The idempotent $e_2$ is the minimal idempotent for $\BNC_2$ that projects onto one of the matrix columns. Hence, it is clear that up to a grading shift the image of this idempotent will be isomorphic to the image of the idempotent $(1-e_2)$.  More generally, one can define a minimal idempotent $e_a \in \BNC_a$ as follows. Introduce a shorthand notation $D_a$ to denote the longest braid on $a$ strands, shown below for $a=4$
\[ 
    \xy
 (0,0)*{\includegraphics[scale=0.4]{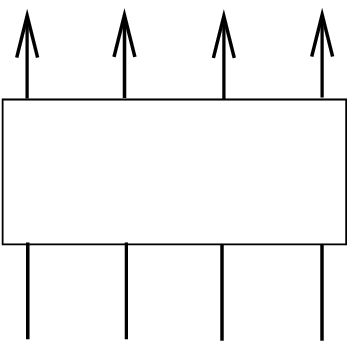}};
 (0,0)*{D_4};
  \endxy
 \quad := \;\;
   \xy
 (0,0)*{\includegraphics[scale=0.4]{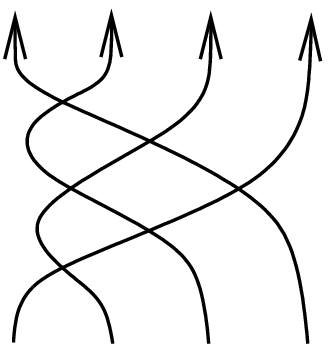}};
  \endxy  \nn
\]
This element is independent of the choice of reduced word by the nilHecke relations. We will also write
\[
\xy
 (0,0)*{\includegraphics[scale=0.4]{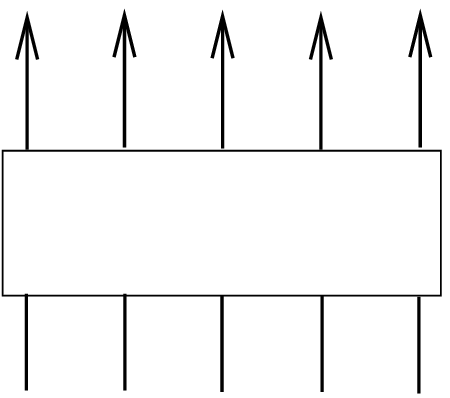}};
 (0,-1.5)*{\delta_{a}};
  \endxy
 \;\;:= \;\;
 \xy
 (-2.6,0)*{\includegraphics[scale=0.4]{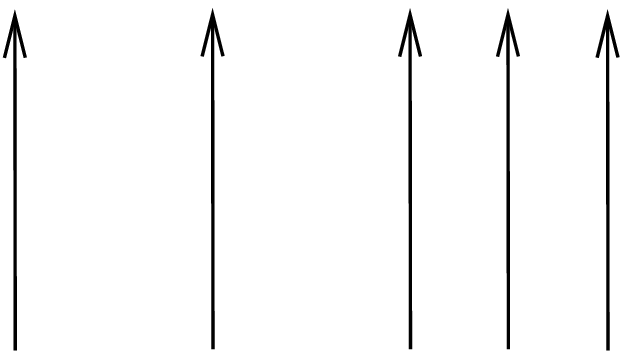}};
 (-15,0)*{\bullet}+(-4,1)*{\scs a-1};
 (-7,0)*{\bullet}+(-4,1)*{\scs a-2};
 (-3,-3)*{\cdots};
 (1,0)*{\bullet}+(-2,1)*{\scs 2};
 (5.3,0)*{\bullet};
  \endxy
\]
Introduce a 2-morphism $ e_a = \delta_a D_a$:
\begin{equation} \label{eq_def_ea}
  \xy
 (0,0)*{\includegraphics[scale=0.5]{c2-1.eps}};
 (0,-1.5)*{e_{a}};
  \endxy
\quad := \quad
 \xy
 (0,0)*{\includegraphics[scale=0.5]{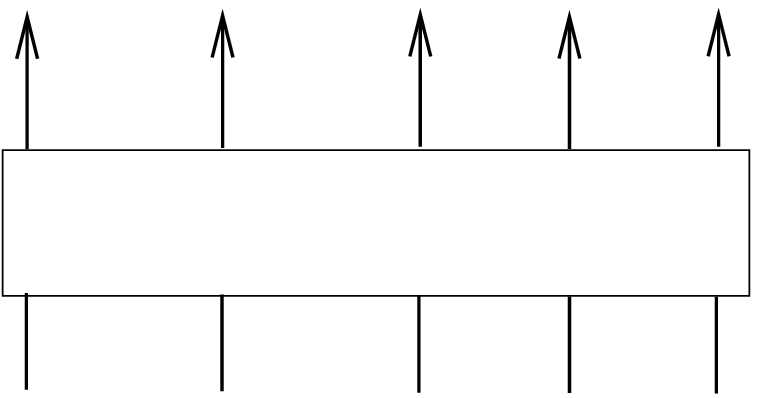}};
 (-17.7,5)*{\bullet}+(-4,1)*{\scs a-1};
 (-7.7,5)*{\bullet}+(-4,1)*{\scs a-2};
 (-3,4)*{\cdots};(-3,-7)*{\cdots};
 (2.3,5)*{\bullet}+(-2,1)*{\scs 2};
 (9.9,5)*{\bullet};(0,-1.5)*{D_a};
  \endxy
\end{equation}
One can show using properties of divided difference operators in $\BNC_a$ that $e_a$ is a minimal idempotent.  Define $\cal{E}^{(a)}\onen:=(\cal{E}^a\onen \{
\frac{-a(a-1)}{2}\}, e_a)= \im e_a\{ \frac{-a(a-1)}{2}\} $ in $\UcatD$, then one can show
\[
 \cal{E}^{a}\onen \cong \bigoplus_{[a]!}\cal{E}^{(a)}\onen.
\]

In passing to the Karoubi envelope $\UcatD$ we have lost the purely diagrammatic description of 2-morphisms since we have added new 1-morphisms corresponding to the images of idempotent 2-morphisms.  However, it is possible to enhance the graphical calculus by adding thickened strands.  This was used by the author and collaborators Khovanov, Mackaay, and Sto\v si\'c to give a diagrammatic description of the Karoubi envelope $\UcatD$ in \cite{KLMS}.

%
\subsection{The categorification theorem}
%

In the previous sections we have introduced a 2-category $\UcatD$, showed that the $\mathfrak{sl}_2$ relations hold in $\UcatD$, and identified 1-morphisms corresponding to divided powers and their products.  Let $(\ep)$ denote a divided powers signed sequence consisting of
\begin{equation}
  (\ep) = (\epsilon_1^{a_1}, \epsilon_2^{a_2}, \dots, \epsilon_k^{a_k}) \nn
\end{equation}
for $\ep_j \in \{ + , -\}$ and $a_j \in \N$.  Write
\begin{equation}
  \cal{E}_{(\ep)}\onen := \cal{E}_{\epsilon_1}^{(a_1)}\cal{E}_{\epsilon_2}^{(a_2)} \dots \cal{E}_{\epsilon_k}^{(a_k)} \onen. \nn
\end{equation}

Define a map
\begin{eqnarray}
  \gamma \maps \UA &\to&  K_0(\UcatD) \\
   E_{\ep} \onen &\mapsto& [\cal{E}_{\ep}\onen]. \nn
\end{eqnarray}
Let's look at what is left to be proven to show that this map is an isomorphism.
\begin{itemize}
 \item We need to know that the 2Homs in $\UcatD$ are not huge and unmanageable.   Using the higher relations of $\Ucat$ it can be shown that the space of 2Homs $\HOM_{\Ucat}(\cal{E}_{\ep}\onen,\cal{E}_{\ep'}\onen)$ have natural spanning sets given by diagrams with the following properties:
\begin{itemize}
 \item no strand intersects itself and no two strands intersect more than once,
 \item all dots are confined to a small interval on each strand,
 \item all closed diagrams are reduced to non-nested bubbles with the same orientation and are moved to the far right of each diagram.
\end{itemize}
Proving that such a set of diagrams spans the space of 2-morphisms amounts to showing that the relations in $\Ucat$ imply some analog of Reidemeister moves for all possible orientations and all possible placements of dots, so that any
complicated diagram can be reduced to a linear combination of diagrams of the above form.

\begin{example}
To find the spanning sets for  $\HOM_{\Ucat}(\cal{F}\cal{E}^2\onen,\cal{E}\onen)$ we chose a small interval on each strand of each possible diagram with no self intersections and no double intersections
\[
  \vcenter{\xy 0;/r.15pc/:
  (-4,-20)*{};(-6,10)*{} **\crv{(-4,-10) & (-6,5)}?(1)*\dir{>};
   ?(.55)*\dir{|}?(.7)*\dir{|};
   (4,-20)*{};(-12,-20)*{} **\crv{(4,-10) &
  (-12,-10)}?(1)*\dir{>};
  ?(.65)*\dir{|}?(.8)*\dir{|};
  (10,-5)*{n};(-6,13)*{\scs  \cal{E}};
  (-13,-23)*{\scs \cal{F}};(-4.5,-23)*{\scs  \cal{E}};(3.5,-23)*{\scs \cal{E}};
 \endxy}
 \qquad \qquad
  \vcenter{\xy 0;/r.15pc/:
  (4,-20)*{};(-6,10)*{} **\crv{(4,-10) & (-6,5)}?(1)*\dir{>};
   ?(.55)*\dir{|}?(.7)*\dir{|};
   (-4,-20)*{};(-12,-20)*{} **\crv{(-4,-10) &
  (-12,-10)}?(1)*\dir{>};
  ?(.65)*\dir{|}?(.8)*\dir{|};
  (10,-5)*{n};(-6,13)*{\scs  \cal{E}};
  (-13,-23)*{\scs \cal{F}};(-4.5,-23)*{\scs  \cal{E}};(3.5,-23)*{\scs \cal{E}};
 \endxy}
\]
The spanning sets are given by these diagrams together with arbitrary number of dots on these intervals and arbitrary products of nonnested dotted bubbles on the far right region.
\[
  \vcenter{\xy 0;/r.14pc/:
    (-4,-20)*{};(-6,10)*{} **\crv{(-4,-10) & (-6,5)}?(1)*\dir{>};
  ?(.65)*\dir{}+(0,0)*{\bullet}+(-4,1)*{\scs a_1};
   (4,-20)*{};(-12,-20)*{} **\crv{(4,-10) &
  (-12,-10)}?(1)*\dir{>};
   ?(.73)*\dir{}+(0,0)*{\bullet}+(-1,3)*{\scs a_2};;
 (-6,13)*{\scs  \cal{E}};
  (-13,-23)*{\scs \cal{F}};(-4.5,-23)*{\scs  \cal{E}};(3.5,-23)*{\scs \cal{E}};
 (40,0)*{\xybox{
 (14,-2)*{\cbub{n -1+\alpha_2}{}};
 (36,-2)*{\cbub{ \;\;\;n  -1+\alpha_4}{}};
 (14,13)*{\cbub{n -1+\alpha_1}{}};
 (37,13)*{\cbub{\;\;\;n  -1+\alpha_3}{}};
 (-2,0)*{n}}};
 \endxy}
 \qquad \qquad
   \vcenter{\xy 0;/r.14pc/:
     (4,-20)*{};(-6,10)*{} **\crv{(4,-10) & (-6,5)}?(1)*\dir{>};
       ?(.65)*\dir{}+(0,0)*{\bullet}+(-4,1)*{\scs a_1};
   (-4,-20)*{};(-12,-20)*{} **\crv{(-4,-10) &
  (-12,-10)}?(1)*\dir{>};
?(.73)*\dir{}+(0,0)*{\bullet}+(-3,2)*{\scs a_2};
(-6,13)*{\scs  \cal{E}};
  (-13,-23)*{\scs \cal{F}};(-4.5,-23)*{\scs  \cal{E}};(3.5,-23)*{\scs \cal{E}};
 (40,0)*{\xybox{
 (14,-2)*{\cbub{n -1+\alpha_2}{}};
 (14,13)*{\cbub{n -1+\alpha_1}{}};
 (37,13)*{\cbub{\;\;\;n  -1+\alpha_3}{}};
 (-2,0)*{n}}};
 \endxy}
\]
\end{example}

Showing that these spanning sets are in fact a basis is necessary to prove that our graded 2Homs lifts the semilinear form on $\U$.

\item In addition to showing that the 2Homs are not too large, we also need to check that the relations in our 2-category are not so strong that they actually force all the Hom spaces to be trivial. This can be accomplished by constructing an action of the 2-category $\Ucat$ on some other 2-category where we know the size of the space of 2Homs.  The action of $\Ucat$ on the categories $\cal{V}^N$ gives a nice computable example where this can be proven.

 \item Showing that we have isomorphisms lifting the $\mathfrak{sl}_2$ relations, and finding 1-morphisms corresponding to 1-morphisms in $\U$ shows that the map $\gamma$ is a homomorphism into the split Grothendieck group $K_0(\UcatD)$. We need to show that this homomorphism is in fact an isomorphism.  To prove this we show that the indecomposable 1-morphisms in $\UcatD$ coincide (up to grading shift) with elements in Lusztig's canonical basis.  This can be proven using properties of the semilinear form and the fact that we have lifted the semilinear form to the space of graded 2Homs in $\UcatD$.
\end{itemize}

Hence, we have the following theorem from \cite{Lau1}.

\begin{thm}[Categorification of $\UA$] \hfill
\begin{itemize}
  \item The map
\begin{eqnarray}
  \gamma \maps \UA &\to&  K_0(\UcatD) \\
   E_{\ep} \onen &\mapsto& [\cal{E}_{\ep}\onen] \nn
\end{eqnarray}
for any divided power sequence $\ep$ is an isomorphism of $\Z[q,q^{-1}]$-algebras.

\item The indecomposable 1-morphisms of $\UcatD$ taken up to grading shift give rise to a basis in $K_0(\UcatD)$ with structure constants in $\N[q,q^{-1}]$.  Under the isomorphism $\gamma$ this basis coincides with the Lusztig canonical basis of $\UA$.

\item The graded dimension of $\HOM_{\Ucat}(\onen,\onen)$ is generated by products of nonnested dotted bubbles of the same orientation, so that
\begin{equation}
  \gdim\HOM_{\Ucat}(\onen,\onen) = \pi:= \prod_{a=1}^{\infty} \frac{1}{1-q^{2a}}. \nn
\end{equation}
Furthermore, we have
\begin{equation}
  \gdim\HOM_{\Ucat}(\cal{E}_{\ep}\onen, \cal{E}_{\ep'}\onen) = \la E_{\ep}1_n, E_{\ep'} 1_n \ra \cdot \pi. \nn
\end{equation}
\end{itemize}
\end{thm}

In the definition of the 2-category $\Ucat$ we took $\Bbbk$-linear combinations of diagrams for $\Bbbk$ a field.  In \cite{KLMS} it was shown that the isomorphism $\gamma$ holds when taking $\Z$-linear combinations of diagrams in the definition of $\Ucat$.

In \cite{KL3} a 2-category $\UcatD(\mathfrak{g})$ was defined for any symmetrizable Kac-Moody algebra.  It was conjectured that this 2-category categorifies $\U(\mathfrak{g})$. Recent unpublished work of Webster~\cite{Web} gives a proof of this conjecture.

%
\section{The higher structure of categorified quantum groups} \label{sec_higher}
%

In this section we briefly review some recent work exploring the implications of the higher relations in $\UcatD$.

%
\subsection{Derived equivalences and geometric actions} \label{subsec_derived}
%

In their seminal work Chuang and Rouquier showed that the higher structure of natural transformations in categorical $\mathbf{U}(\mathfrak{sl}_2)$-actions played an important role in the construction of derived equivalences~\cite{CR}.  In their framework the affine Hecke algebra acts as natural transformations between functors $\cal{E}^a$.  They called the additional structure of these natural transformations a {\em strong} categorical $\mf{sl}_2$-action.  They showed that if abelian categories $\cal{V_{-N}}, \cal{V}_{-N+2}, \dots, \cal{V}_N$ admit such an action, with $E$ and $F$ acting by exact functors, then it is possible to categorify the reflection element to give derived equivalences $T \maps D(\cal{V}_{-n}) \to D(\cal{V}_n)$ between the derived categories corresponding to opposite weight spaces.
\[
 \xy (0,10)*{};
  (11,0)*+{\cal{V}_{n-2}}="2";
  (27,0)*+{\cal{V}_{n}}="3";
  (43,0)*+{\cal{V}_{n+2}}="4";
  (54,0)*{\cdots};
  (0,0)*{\cdots};
    {\ar@/^0.7pc/^{\cal{E}} "2";"3"};
    {\ar@/^0.7pc/^{\cal{E}} "3";"4"};
    {\ar@/^0.7pc/^{\cal{F}} "3";"2"};
    {\ar@/^0.7pc/^{\cal{F}} "4";"3"};
  (-11,0)*+{\cal{V}_{-n-2}}="2'";
  (-27,0)*+{\cal{V}_{-n}}="3'";
  (-43,0)*+{\cal{V}_{-n+2}}="4'";
  (-54,0)*{\cdots};
    {\ar@/^0.7pc/^{\cal{F}} "2'";"3'"};
    {\ar@/^0.7pc/^{\cal{F}} "3'";"4'"};
    {\ar@/^0.7pc/^{\cal{E}} "3'";"2'"};
    {\ar@/^0.7pc/^{\cal{E}} "4'";"3'"};
      (0,18)*{};
    \textcolor[rgb]{1.00,0.00,0.00}{{\ar@{<->}@/_3.6pc/_T "3"; "3'"}}
 \endxy
\]
They used the resulting derived equivalences in a highly nontrivial way to prove the abelian defect conjecture in the modular representation theory
of the symmetric group.

Passing from actions of the Lie algebra $\mathbf{U}(\mf{sl}_2)$ to actions of the quantum group $\U(\mf{sl}_2)$, we have seen that the nilHecke algebra plays a fundamental role, rather than the affine Hecke algebra.  Categorification of the reflection element in the quantum setting was studied by Cautis, Kamnitzer, and Licata.  They defined an analog of the derived equivalence introduced by Chuang and Rouquier in the context of $\U$~\cite{CKL3}.  In their setting it suffices for the weight categories $\cal{V}_n$ to be triangulated rather than abelian. They prove a rather striking result that in various geometric settings the nilHecke algebra action can be constructed from rather weak and simple to check geometric conditions~\cite{CKL2}.

The higher structure involving the nilHecke algebra turns out to be a key ingredient in constructing derived equivalences categorifying the reflection element.    Cautis, Kamnitzer, and Licata used this geometrically formulated categorification of the reflection element to give nontrivial derived equivalences between categories of coherent sheaves on cotangent bundles to Grassmannians
\begin{equation}
  DCoh(T^*Gr(k,N)) \to DCoh(T^*(Gr(N-k,N))). \nn
\end{equation}
The nilHecke algebra also appears in Rouquier's study of 2-representations of Kac-Moody algebras \cite{Rou2}.

%
\subsection{Diagrammatics for symmetric functions} \label{subsec_symbub}
%

The higher relations in the 2-category $\Ucat$ give rise to some surprising connections between the graded endomorphism algebra $\END_{\Ucat}(\onen)$ and the algebra of symmetric functions.  Part of this connection was explained in Section~\ref{subsec_symm}.

Recall that any 2-morphism in $\END_{\Ucat}(\onen)$ can be written as a linear combination of products of nonnested dotted bubbles with the same orientation.  Using the action of the 2-category $\Ucat$ on the cohomology of iterated flag varieties one can show that there are no relations among such products of dotted bubbles \cite[Proposition 8.2]{Lau1}. This implies that the map $\phi^n$ from \eqref{eq_symm_iso} is an algebra isomorphism. In particular, for $n>0$ we have
\begin{eqnarray}
 \phi^n \maps \Lambda(\underline{x})= \Z[e_1,e_2, \dots] &\to& \END_{\Ucat}(\onen),
  \nn \\ \label{eq_isom_h}
  e_r(\underline{x}) & \mapsto & \xy
 (0,0)*{\cbub{(n-1)+r}{}};
 (6,8)*{n};
 \endxy,
\end{eqnarray}
where $\underline{x}$ is an infinite set of variables.  This isomorphism can alternatively be described using fake bubbles
\begin{eqnarray}
 \phi^n \maps \Lambda(\underline{x}) &\to& \END_{\Ucat}(\onen),
 \nn \\ \label{eq_isom_e}
 h_r(\underline{x}) & \mapsto &  (-1)^r \;  \xy
 (0,0)*{\ccbub{(-n-1)+r}{}};
 (6,8)*{n};
 \endxy.
\end{eqnarray}

The interplay between closed diagrams in $\END_{\Ucat}(\onen)$ extends beyond this connection with elementary and complete symmetric functions.  As an abelian group the ring of  symmetric functions has a basis of Schur polynomials.  Given a partition $\lambda = (\lambda_1, \lambda_2 \dots, \lambda_a)$ with $\lambda_1 \geq \lambda_2 \geq \dots \geq \lambda_a$ define the Schur polynomial \begin{equation}
\pi_{\lambda}=\left|
\begin{array}{cccc}
h_{\lambda_1} & h_{\lambda_1+1} & \cdots & h_{\lambda_1+a-1} \\
h_{\lambda_2-1} & h_{\lambda_2} & \cdots & h_{\lambda_2+a-2} \\
\vdots & \vdots & \ddots & \vdots \\
h_{\lambda_a-(a-1)} & h_{\lambda_a-(a-2)} & \cdots & h_{\lambda_a}
\end{array}
\right|= |h_{\lambda_i+j-i}|.
\end{equation}
The Schur polynomials can also be defined by the dual formula
\begin{equation} \label{eq_Schur_e}
\pi_{\lambda}=\left|
\begin{array}{cccc}
e_{\bar{\lambda}_1} & e_{\bar{\lambda}_1+1} & \cdots & e_{\bar{\lambda}_1+b-1} \\
e_{\bar{\lambda}_2-1} & e_{\bar{\lambda}_2} & \cdots & e_{\bar{\lambda}_2+b-2} \\
\vdots & \vdots & \ddots & \vdots \\
e_{\bar{\lambda}_b-(b-1)} & e_{\bar{\lambda}_b-(b-2)} & \cdots & e_{\bar{\lambda}_b} \\
\end{array}\right|=
|e_{\bar{\lambda}_i+j-i}|.
\end{equation}
where $\bar{\lambda}$ is the conjugate partition.    These formulas are called the Jacobi-Trudi or Giambelli formulas. For more details about symmetric functions see~\cite{Mac}.

In \cite{KLMS} natural closed diagrams were identified in $\END_{\Ucat}(\onen)$ that when simplified reconstruct these formulas.  The diagrams are quite natural in the diagrammatic framework, they have the form:
\begin{equation}
\pi_{\overline{\lambda}}
\quad =  \quad \xy
 (0,0)*{\includegraphics[scale=0.45]{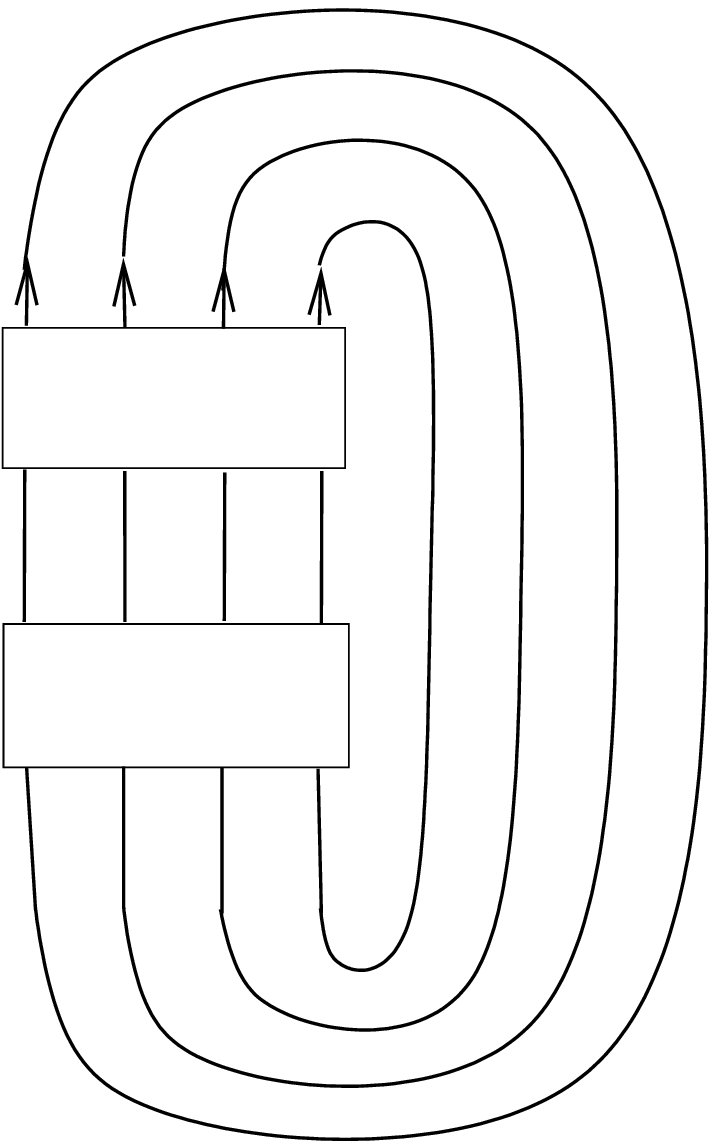}};
 (-8,-6)*{e_a};(-8,8)*{x^{\lambda+(n-a)}}; \nn
 \endxy
\end{equation}
where $x^{\lambda}+(n-a)=x_1^{\lambda_1+(n-a)}x_2^{\lambda_2+(n-a)}\dots x_a^{\lambda_a+(n-a)}$ denotes the diagram with $\lambda_1+(n-a)$ dots on the first strand, $\lambda_2+(n-a)$ dots on the second strand and so on.  The $(n-a)$ is a normalization, much like the $(n-1)$ needed to read off the degree of a clockwise oriented dotted bubble.  Reducing these diagrams into nonnested dotted bubbles with the same orientation using the higher relations in $\Ucat$ gives the same determinant formula for decomposing a Schur polynomial into complete or elementary symmetric functions.

\begin{example}
Using the higher relations in $\Ucat$ and a calculation analogous to the one in Remark~\ref{rem_other_inv} we have for $n>0$
\begin{align} \nn
\vcenter{\xy 0;/r.22pc/:
  (24,8)*{n};
  (-4,-4)*{}="lb";(4,4)*{}="rt" **\crv{(-4,-1) & (4,1)}?(1)*\dir{>};?(0)*\dir{>};
    (4,-4)*{}="rb";(-4,4)*{}="lt" **\crv{(4,-1) & (-4,1)}?(1)*\dir{>};
   (12,4)*{}="t2";(12,-4)*{}="t2'";
   "rt";"t2" **\crv{(4,8) & (12, 8)};
   "rb";"t2'" **\crv{(4,-8) & (12, -8)}?(.5)*\dir{}+(0,0)*{\bullet}
   +(-1,-3)*{\scs \textcolor[rgb]{1.00,0.00,0.00}{\lambda_1}+(n-1)};
   "lt";(20,4) **\crv{(-4,18) & (20, 18)};
   "lb";(20,-4) **\crv{(-4,-18) & (20, -18)}?(.5)*\dir{}+(0,0)*{\bullet}
   +(-1,-3)*{\scs  \textcolor[rgb]{1.00,0.00,0.00}{\lambda_2}+(n-2)};
   "t2'";"t2" **\dir{-} ;
   (20,4);(20,-4) **\dir{-};
   (9,0)*{};
 \endxy  }
\;\; &= \;\;
\xy 0;/r.18pc/: (0,0)*{\cbub{n-1+\lambda_1}{}};(5,7)*{n}; \endxy\;
\xy 0;/r.18pc/: (0,0)*{\cbub{n-1+\lambda_1}{}};(5,7)*{n}; \endxy \;\; - \;\;
\xy 0;/r.18pc/: (0,0)*{\cbub{n-1+\lambda_1+1}{}};(5,7)*{n}; \endxy\;
\xy 0;/r.18pc/: (0,0)*{\cbub{n-1+\lambda_1+1}{}};(5,7)*{n}; \endxy
\;\; = \;\;
\left|
\begin{array}{ccccc}
  \xy 0;/r.18pc/: (0,0)*{\cbub{n-1+\lambda_1}{}};(5,7)*{n}; \endxy &
  \xy 0;/r.18pc/: (0,0)*{\cbub{n-1+\lambda_1+1}{}};(5,7)*{n}; \endxy &
 \\ \\
 \xy 0;/r.18pc/: (0,0)*{\cbub{n-1+\lambda_2-1}{}};(5,7)*{n}; \endxy &
  \xy 0;/r.18pc/: (0,0)*{\cbub{n-1+\lambda_2}{}};(5,7)*{n}; \endxy &
\end{array}
\right|
\nn \\
\;\;& = \;\;
\left|
\begin{array}{ccccc}
  \phi^n(e_{\lambda_1}) & \phi^n(e_{\lambda_1+1}) \\ \\
  \phi^n(e_{\lambda_2-1})& \phi^n(e_{\lambda_2})
\end{array}
\right|
\;\; = \;\;
\phi^n(
s_{\textcolor[rgb]{1.00,0.00,0.00}{\overline{\lambda}_1},
\textcolor[rgb]{1.00,0.00,0.00}{\overline{\lambda}_2}}).
\nn
\end{align}
\end{example}

%
\subsection{Categorification of the Casimir element} \label{subsec_casimir}
%

Seemingly simple situations like the commutativity of an element in $\U$ with other elements become complex and more sophisticated when categorified.  The quantum Casimir element for $\mathbf{U}_q(\mathfrak{sl}_2)$ has the form
\begin{equation} \label{eq_usual_casimir_INTRO}
c := EF +\frac{q^{-1}K+qK^{-1}}{(q-q^{-1})^2} = FE +\frac{qK+q^{-1}K^{-1}}{(q-q^{-1})^2}.
\end{equation}
This element is central so that it commutes with other elements in $\mathbf{U}_q(\mathfrak{sl}_2)$. The integral idempotented version of the Casimir element $\dot{C}$ for $\U$ has the form:
\begin{eqnarray}
 \dot{C} &=& \prod_{n\in\Z}C1_n, \\
  C1_n =1_nC &:=& (-q^2+2-q^{-2})EF1_n-(q^{n-1}+q^{1-n})1_n, \label{eq_casimir_INTRO}
\end{eqnarray}
This element belongs to the center of $\U$, meaning that the family of elements $C1_n$ satisfy commutativity relations of the form $xC1_n=C1_m x$ for any $x \in 1_m \U 1_n$.

The minus signs in the definition of the Casimir element $C1_n$ suggest that a categorification of this element will require complexes.  Thus, a categorification $\cal{C}\onen$ of the Casimir element $C1_n$ will require us to work in the 2-category $Com(\UcatD)$ of bounded complexes over the 2-category $\UcatD$. The objects of $Com(\UcatD)$ are integers $n \in \Z$, the 1-morphisms are bounded complexes of 1-morphisms in $\UcatD$, and 2-morphisms are chain maps up to homotopy.  When passing to the split Grothendieck ring $K_0(Com(\UcatD))$ elements in odd homological degree acquire minus signs.

By considering the degrees of the 2-morphisms in $\Ucat$ there is a natural candidate for a complex lifting the Casimir element $C1_n$.  Namely, the complex
\begin{equation} \label{eq_casimir_EF_INTRO}
 \cal{C}\onen :=\xy
  (-55,15)*+{\cal{E}\cal{F} \onen \{2\}}="1";
  (-55,-15)*+{\onen \{1-n\}}="2";
  (0,15)*+{\underline{\cal{E}\cal{F} \onen }}="3";
  (0,-15)*+{\underline{\cal{E}\cal{F}\onen}}="4";
  (55,-15)*+{\onen \{n-1\}}="5";
  (55,15)*+{\cal{E}\cal{F} \onen \{-2\}}="6";
   {\ar^{\xUupdot\xUdown} "1";"3"};
   {\ar^{} "1";"4"};
   "1"+(18,-6)*{\xUup\xUdowndot};
   {\ar_(.35){\xUcupl} "2";"3"}; 
   {\ar_-{\xUcupl} "2";"4"};     
   {\ar^-<<<<<<<{\xUcapr} "3";"5"};
   "3"+(20,4)*{-\;\xUup\xUdowndot};
   {\ar "3";"6"};
   {\ar_{-\;\;\xUcapr} "4";"5"};
   {\ar_{} "4";"6"};
   "4"+(24,8)*{\xUupdot\xUdown};
   (-57,0)*{\bigoplus};(0,0)*{\bigoplus};(57,0)*{\bigoplus};
 \endxy
\end{equation}
centered in homological degree zero.  We call the above complex the Casimir complex.  It was shown in \cite{BKL} that the image of this complex in the Grothendieck ring of the category $Com(\UcatD)$ is the component $C1_n$ of the Casimir element $\dot{C}$.

In this categorified setting, the Casimir complex no longer commutes on the nose with other complexes in $Com(\UcatD)$. Rather, the complexes $\cal{C}\onen$ commute up to explicit chain homotopies.  Here the explicit form of the 2-morphisms and the higher relations between them are required to construct chain homotopies proving that the Casimir complex commutes with other complexes.  Hence, the notion of commutativity becomes much more interesting in the categorified context.

%
%
%
%

%
\section{Categorifying irreducible representations} \label{sec_irreps}
%

We have already seen that the cohomology rings of iterated flag varieties give a categorification of the irreducible $N+1$-dimensional representation $V^N$ of $\U$.  In this section we look at a different categorification of these representations that arise from the new structure introduced in the 2-category $\Ucat$.

%
\subsection{Cyclotomic quotients}
%

Recall the nilHecke algebra $\BNC_a$ defined by generators $\xi_i$ for $1 \leq i \leq a$ and $\partial_j$ for $1 \leq j \leq a-1$ and relations
\[
 \begin{array}{ll}
 \xi_i \xi_j =   \xi_j \xi_i , &  \\
   \partial_i \xi_j = \xi_j\partial_i \quad \text{if $|i-j|>1$}, &
   \partial_i\partial_j = \partial_j\partial_i \quad \text{if $|i-j|>1$}, \\
  \partial_i^2 = 0,  &
   \partial_i\partial_{i+1}\partial_i = \partial_{i+1}\partial_i\partial_{i+1},  \\
   \xi_i \partial_i - \partial_i \xi_{i+1}=1,  &   \partial_i \xi_i - \xi_{i+1} \partial_i =1.
  \end{array}
\]
This algebra can be viewed as the algebra of $a$ upward oriented strands in $\Ucat$ where we do not allow any bubbles and ignore the labels on regions.  For any positive number $N$, the algebra $\BNC_a$ admits a quotient called a {\em cyclotomic quotient}.  This quotient is defined by
\begin{equation}
  \BNC_a^N := \BNC_a / (\BNC_a x_1^N \BNC_a).
\end{equation}

In terms of diagrams $\BNC_a^N$ is defined by quotienting by the ideal generated by diagrams of the form
\begin{equation}
    \xy
  (16,4);(16,-4) **\dir{-}?(0)*\dir{<}+(2.3,0)*{};
  (10,0)*{\cdots};
  (4,4);(4,-4) **\dir{-}?(0)*\dir{<}+(2.3,0)*{} ;
  (-4,4);(-4,-4) **\dir{-}?(0)*\dir{<}+(2.3,0)*{}?(.55)*{\bullet}+(-4,1)*{N};
 \endxy
\end{equation}
so that anytime there are $N$ dots on the leftmost strand of a diagram it is zero.
Cyclotomic quotients can be defined in full generality given any highest weight $\Lambda$ of a symmetrizable quantum Kac-Moody algebra $\U(\mf{g})$.  In this generality it was conjectured in \cite{KL} that these cyclotomic quotients categorify irreducible highest weight representations for $\U(\mf{g})$.  In the case of $\mathfrak{sl}_2$ this conjecture states the following:

\paragraph{{\bf Cyclotomic Quotient Conjecture}}
\begin{enumerate}[(1)]
 \item There is an isomorphism of $\Q(q)$-vector spaces
\begin{equation}
  \bigoplus_{a} K_0(\BNC_a^N\pmod) \otimes_{\Z[q,q^{-1}]} \Q(q) \cong V^N,
\end{equation}
where $V^N$ is the irreducible highest weight representation of $\U$.

 \item The category $\oplus_a\BNC_a^N\pmod$ admits an action of the 2-category $\UcatD$.  In particular, there are functors between these module categories lifting the action of $E1_n$ and $F1_n$ so that on the split Grothendieck group we get an action of $\U$.
\end{enumerate}
\medskip

In the general setting of a symmetrizable Kac-Moody algebra part (1) of this conjecture was proven in \cite{LV,KRam2}.  The fact that these quotients admit functors lifting the $\U$ action appears in \cite{Web,SK}.

Here we give a new proof of (1) above by proving that there is an isomorphism
\begin{equation}
  \BNC_a^N \cong \Mat(a!, H_a),
\end{equation}
where, as in Section~\ref{subsec_flag}, $H_a:= H^*(Gr(a,N))$.   This would immediately imply that the split Grothendieck groups are 1-dimensional since the split Grothendieck groups of Morita equivalent rings are identical, and we have already explained that Grothendieck rings of cohomology rings of Grassmannians are 1-dimensional in Section~\ref{subsec_flag}.

%
\subsection{The polynomial algebra as a module over symmetric functions}
%

We write $\cal{P}_a=\Q[x_1,\dots,x_a]$ for the graded polynomial algebra with $\deg(x_i)=2$. For $\underline{x}=x_1, \cdots , x_a$ we will write
\begin{equation}
 \symm_a := \symm(\underline{x}) =\Q[x_1,x_2, \dots, x_a]^{S_a} \nn
\end{equation}
to emphasize the dependence on $a$.  As explained in Section~\ref{subsec_symbub} this algebra can be identified with the polynomial algebra $\Q[e_1,e_2, \dots, e_a]$. The algebra $\symm_a$ is graded with $\deg(e_j)=2j$.

The abelian subgroup $\cal{H}_a$ of $\cal{P}_a$ generated by monomials
$x_1^{\alpha_1} x_2^{\alpha_2} \cdots x_a^{\alpha_a}$, with
$\alpha_i \leq a-i$ for all $1\leq i \leq a$, has rank $a!$.  The multiplication in $\cal{P}_a$ induces a canonical isomorphism of
the tensor product $\cal{H}_a \otimes \symm_a$ with $\cal{P}_a$ as
graded $\symm_a$-modules.  In particular, any polynomial in variables $x_1, x_2, \dots, x_a$ can be expressed as a linear combination of monomials in $\cal{H}_a$ with coefficients in the ring of symmetric functions  $\symm_a$.  Hence, the  subgroup $\cal{H}_a$ is a basis for the polynomial ring $\cal{P}_a$ as a free graded module over $\symm_a$ of rank $a!$.  For details, see for example \cite[Chapter 2]{Man} and the references
therein.

\begin{example} Several examples are given below expressing an arbitrary polynomial in $\cal{P}_a$ in the basis $\cal{H}_a$ over the ring $\symm_a$.
\begin{itemize}
  \item $x_a=1\cdot e_1-x_1e_0-x_2e_0-\dots - x_{a-1}e_0$ with $e_0=1$ as in Section~\ref{subsec_symbub}.
  \item $x_1^{a} = \sum_{j=1}^{a}(-1)^{j+1} x_{1}^{a-j} e_{j}$.
\end{itemize}
\end{example}

%
\subsubsection{Remarks on gradings}
%
The graded rank of $\cal{P}_a$ as a free graded module over $\symm_a$ is the graded rank of $\cal{H}_a$. This is easily computed to be the nonsymmetric quantum factorial
\begin{equation}
(a)_{q^2}^{!} = (a)_{q^2} (a-1)_{q^2}\dots (2)_{q^2}(1)_{q^2}
\end{equation}
where
\begin{equation}
(a)_{q^2} := \frac{1-q^{2a}}{1-q^2} = 1+q^2+\cdots+q^{2(a-1)}. \nn
\end{equation}
Hence,
\begin{equation}
  \Q[x_1, \dots, x_a] \cong \bigoplus_{(a)_{q^2}^!} \symm_a \nn
\end{equation}
as $\symm_a$ modules. Furthermore, it follows that
$\Hom_{\symm_a}(\cal{P}_a,\cal{P}_a)$ is isomorphic as a graded ring to the ring $\Mat((a)_{q^2}^!; \symm_a )$ of $(a)_{q^2}^!\times (a)_{q^2}^!$-matrices with coefficients in $\symm_a$.  In what follows we suppress the grading and simplify notation by writing  $M_{a!}(\symm_a)$ for $\Mat((a)_{q^2}^!; \symm_a )$.

%
\subsection{NilHecke action on polynomials}
%

The nilHecke algebra acts as endomorphisms of $\cal{P}_a$ with $\partial_i \maps \cal{P}_a \to \cal{P}_a$ acting by divided difference operators
\begin{equation}
 \partial_i := \frac{1-s_i}{x_i-x_{i+1}}, \nn
\end{equation}
and $\xi_i$ acting by multiplication by $x_i$. From this definition it is clear that both the image and kernel of
the operator $\partial_i$ consists of polynomials which
are symmetric in $x_i$ and $x_{i+1}$.  It follows from the formula
\begin{equation}
\partial_i(fg) = (\partial_if) g + (s_i f) (\partial_ig)
\end{equation}
that $\partial_i$ acts as a $\symm_a$-module endomorphism of $\cal{P}_a$.

The action of $\BNC_a$ on $\cal{P}_a \cong \oplus_{(a)_{q^2}^!} \symm_a$ gives rise to a homomorphism
\begin{equation}
  \xymatrix{ \BNC_a \ar[rr]^-{\theta} && M_{a!}(\symm_a) \cong \End_{\symm_a}(\cal{P}_a,\cal{P}_a) }. \nn
\end{equation}
It was shown in \cite[Proposition 3.5]{Lau1} that this homomorphism is an isomorphism of graded algebras.  This isomorphism was given explicitly in \cite[Section 2.5]{KLMS} where diagrams for the elementary matrices $E_{i,j}$ in $M_{a!}(\symm_a)$ were identified.

%
\subsection{The action of $\xi_1$}
%

Given $\alpha_1$ and $\alpha=(\alpha_2, \alpha_3, \dots, \alpha_a)$ with $\alpha_i \leq a-i$ for all $1\leq i \leq a$  write $x_1^{\alpha_1}\underline{x}^{\alpha}=x_1^{\alpha_1}x_2^{\alpha_2}\dots x_{a}^{\alpha_a}$.  The set of sequences $\alpha$ partition the subgroup $\cal{H}_a$ into $(a-1)!$ ordered subsets
\begin{equation}
  B_{\alpha} = \left\{ x_1^{a-1}\underline{x}^{\alpha}, x_1^{a-2}\underline{x}^{\alpha} , \dots,
  x_1\underline{x}^{\alpha}, \underline{x}^{\alpha} \right\}
\end{equation}
with $a$ elements.  We can extend this order on $B_{\alpha}$ to a total order on $\cal{H}_a$ using lexicographic ordering on the sequences $\alpha$.

\begin{prop}
In the ordered basis defined above the matrix for $\theta(\xi_1)$ is block diagonal with $(a-1)!$ identical blocks
\begin{equation}
\theta(\xi_1)_{\alpha} = \left(
  \begin{array}{ccccc}
    e_1 & 1 & 0 &  & 0 \\
    -e_2 & 0 & 1 &  & \vdots \\
    e_3 & 0 & \ddots & \ddots & 0 \\
    \vdots & \vdots &  & 0 & 1 \\
    (-1)^{a+1}e_{a} & 0 &  & 0 & 0 \\
  \end{array}
\right)
\end{equation}
of size $a$.
\end{prop}

\begin{proof}
The action of $\xi_1$ on $\Q[x_1, x_2, \dots, x_a]$ is given by multiplication by $x_1$.  Hence, on each of the subset $B_{\alpha}$ the action of $\xi_1$  is given by
\begin{align}
  \xi_1 \;\maps\; x_1^{j}\underline{x}^{\alpha} \to
\left\{
\begin{array}{cl}
  x_1^{j+1}\underline{x}^{\alpha}  & \text{for $1 \leq j < a-1$,} \\
  x_1^{a}\underline{x}^{\alpha}  = \sum_{\ell=1}^{a}(-1)^{\ell+1} x_{1}^{a-\ell} e_{\ell}\underline{x}^{\alpha} ,  & \text{if $j=a-1$.}
\end{array}
\right.
\end{align}
Thus, $\xi_1$ acts by the same matrix on each of the $(a-1)!$ subset $B_{\alpha}$.
\end{proof}

\begin{prop} \label{prop_block}
The isomorphism $\theta \maps \BNC_a \to \Mat_{a!}(\symm_a)$ restricts to an isomorphism
\begin{equation}
 \theta^{N} \maps \BNC_a^{N} \to \Mat_{a!}(H_a)
\end{equation}
where $H_a=H^*(Gr(a,N))$ is the cohomology ring of the Grassmannian of $a$-planes in $N$-dimensional space.
\end{prop}

\begin{proof}
Let $c_j=e_{j}$ so that $\symm_a \cong \Q[c_1, c_2, \dots, c_a]$.
Then by Proposition~\ref{prop_block} the matrix $\theta(\xi_1)$ describing the action of $\xi_1$ on $\cal{P}_a$ decomposes into $(a-1)!$ identical blocks of  the form
\begin{equation}
    \theta(\xi_1)_{\alpha} \quad = \quad
\left(\begin{array}{ccccc}
    c_{1} & 1 & 0 & 0 & { 0} \\
    -c_{2} & 0 & 1 & \ddots & 0 \\
    c_{3} & 0 & \ddots & \ddots & 0 \\
    \vdots & \vdots & \ddots&  & 1 \\
   (-1)^{a+1}c_{a} & 0 &  &  &  0\\
  \end{array}
\right).
\end{equation}
Hence, the image of the cyclotomic ideal generated by $\xi_1^{N}$ in $\Mat_{a!}(\symm_a)$ is determined by the matrix equation $\theta(\xi_1)_{\alpha}^{N}=0$.  Notice that \begin{equation}
\theta(\xi_1)_{\alpha}^{N} =
\left(\begin{array}{ccccc}
    c_{1} & 1 & 0 & 0 & { 0} \\
    -c_{2} & 0 & 1 & \ddots & 0 \\
    c_{3} & 0 & \ddots & \ddots & 0 \\
    \vdots & \vdots & \ddots&  & 1 \\
   (-1)^{a+1}c_{a} & 0 &  &  &  0\\
  \end{array}\right)^{N-a+1}
  \cdot
\left(\begin{array}{ccccc}
    c_{1} & 1 & 0 & 0 & { 0} \\
    -c_{2} & 0 & 1 & \ddots & 0 \\
    c_{3} & 0 & \ddots & \ddots & 0 \\
    \vdots & \vdots & \ddots&  & 1 \\
   (-1)^{a+1}c_{a} & 0 &  &  &  0\\
  \end{array}\right)^{a-1} .
\end{equation}
But recall from \eqref{eq_grass_rels} that $H_a$ is the quotient of the polynomial ring $\Q[c_1,\dots,c_a]$ by the ideal generated by the terms in the first column of the matrix $\theta(\xi_1)_{\alpha}^{N-a+1}$.  It is easy to see that multiplying by $\theta(\xi_1)_{\alpha}^{a-1}$ has the effect of moving this first column to the last column.  Setting the resulting matrix equal to the zero matrix requires that all the elements of $\symm_a$ appearing in the last column are in the image of the cyclotomic ideal.  Furthermore, there are no other relations in the image of the cyclotomic ideal since the entries appearing in the other columns of the matrix $\theta(\xi_1)_{\alpha}^{N}$ factor into terms appearing in the $a$th column.  Hence, the quotient of $M_{a!}(\symm_a)$ by the image of the cyclotomic ideal is determined by taking the quotient of the coefficient ring $\symm_a$ by the Grassmannian ideal $I_{a,N}$.
\end{proof}

\begin{cor}
There is an isomorphism of $\Q(q)$-vector spaces
\begin{equation}
  \bigoplus_{a} K_0(\BNC_a^N\pmod) \otimes_{\Z[q,q^{-1}]} \Q(q) \cong V^N.
\end{equation}
\end{cor}


%

\begin{thebibliography}{KLMS10}

\bibitem[Bar08]{Bart}
B.~Bartlett.
\newblock The geometry of unitary 2-representations of finite groups and their
  2-characters.
\newblock 2008.
\newblock math.QA/0807.1329.

\bibitem[BD98]{BD}
J.C. Baez and J.~Dolan.
\newblock {\em Categorification}, volume 230 of {\em Contemp. Math.}
\newblock Amer. Math. Soc., Providence, RI, 1998.
\newblock arXiv:math/9802029.

\bibitem[BFK99]{BFK}
J.~Bernstein, I.~B. Frenkel, and M.~Khovanov.
\newblock A categorification of the {T}emperley-{L}ieb algebra and {S}chur
  quotients of {U}(sl(2)) via projective and {Z}uckerman functors.
\newblock {\em Selecta Math. (N.S.)}, 5(2):199--241, 1999.

\bibitem[BGG73]{BGG}
I.~Bernstein, I.~Gelfand, and S.~Gelfand.
\newblock Schubert cells, and the cohomology of the spaces ${G}/{P}$.
\newblock {\em Russian Math. Surveys}, 28:1--26, 1973.

\bibitem[BKL10]{BKL}
A.~Beliakova, M.~Khovanov, and A.~D. Lauda.
\newblock A categorification of the {C}asimir of quantum sl(2).
\newblock 2010.
\newblock arXiv:1008.0370,.

\bibitem[BL00]{BilLak}
S.~Billey and V.~Lakshmibai.
\newblock {\em Singular loci of {S}chubert varieties}, volume 182 of {\em
  Progress in {M}athematics}.
\newblock Birkh\"auser Boston Inc., Boston, MA, 2000.

\bibitem[BL03]{BL}
J.C. Baez and L.~Langford.
\newblock Higher-dimensional algebra. {IV}. 2-tangles.
\newblock {\em Adv. Math.}, 180(2):705--764, 2003.
\newblock arXiv:math/9811139.

\bibitem[BLM90]{BLM}
A.~Beilinson, G.~Lusztig, and R.~MacPherson.
\newblock A geometric setting for the quantum deformation of {${\rm GL}\sb n$}.
\newblock {\em Duke Math. J.}, 61(2):655--677, 1990.

\bibitem[BN02]{BN}
Dror Bar-Natan.
\newblock On {K}hovanov's categorification of the {J}ones polynomial.
\newblock {\em Alg.\ Geom.\ Top.}, 2:337--370, 2002.
\newblock math.QA/0201043.

\bibitem[Bor94]{Bor}
F.~Borceux.
\newblock {\em Handbook of categorical algebra. 1}, volume~50 of {\em
  Encyclopedia of Mathematics and its Applications}.
\newblock Cambridge University Press, Cambridge, 1994.

\bibitem[Cap08]{Cap}
C.~Caprau.
\newblock {$\rm sl(2)$} tangle homology with a parameter and singular
  cobordisms.
\newblock {\em Algebr. Geom. Topol.}, 8(2):729--756, 2008.
\newblock arXiv:0707.3051.

\bibitem[Cat]{Catsters}
The catsters youtube channel.
\newblock Available at {\texttt{http://www.youtube.com/user/TheCatsters}}.

\bibitem[CF94]{CF}
L.~Crane and I.~B. Frenkel.
\newblock Four-dimensional topological quantum field theory, {H}opf categories,
  and the canonical bases.
\newblock {\em J. Math. Phys.}, 35(10):5136--5154, 1994.

\bibitem[CKL09]{CKL3}
S.~Cautis, J.~Kamnitzer, and A.~Licata.
\newblock Derived equivalences for cotangent bundles of {G}rassmannians via
  categorical sl(2) actions.
\newblock 2009.
\newblock arXiv:0902.1797.

\bibitem[CKL10]{CKL2}
S.~Cautis, J.~Kamnitzer, and A.~Licata.
\newblock Coherent sheaves and categorical {$\mathfrak{sl}_2$} actions.
\newblock {\em Duke Math. J.}, 154(1):135--179, 2010.
\newblock arXiv:0902.1796.

\bibitem[CL10]{CL}
S.~Cautis, , and A.~Licata.
\newblock Heisenberg categorification and {H}ilbert schemes.
\newblock 2010.
\newblock arXiv:1009.5147.

\bibitem[CMW09]{CMW}
D.~Clark, S.~Morrison, and K.~Walker.
\newblock Fixing the functoriality of {K}hovanov homology.
\newblock {\em Geom. Topol.}, 13(3):1499--1582, 2009.
\newblock arXiv:math/0701339.

\bibitem[CR08]{CR}
J.~Chuang and R.~Rouquier.
\newblock Derived equivalences for symmetric groups and sl\_2-categorification.
\newblock {\em Ann. of Math.}, 167:245--298, 2008.
\newblock math.RT/0407205.

\bibitem[CRS97]{CRS}
J.~S. Carter, J.~Rieger, and M.~Saito.
\newblock A combinatorial description of knotted surfaces and their isotopies.
\newblock {\em Adv. Math.}, 127(1):1--51, 1997.

\bibitem[CS98]{CS}
J.S. Carter and M.~Saito.
\newblock {\em Knotted surfaces and their diagrams}, volume~55 of {\em
  Mathematical Surveys and Monographs}.
\newblock American Mathematical Society, Providence, RI, 1998.

\bibitem[CY98]{CY}
L.~Crane and D.~Yetter.
\newblock Examples of categorification.
\newblock {\em Cahiers Topologie G\'eom. Diff\'erentielle Cat\'eg.},
  39(1):3--25, 1998.
\newblock arXiv:q-alg/9607028.

\bibitem[Dem73]{Dem}
M.~Demazure.
\newblock Invariants sym\'etriques entiers des groupes de {W}eyl et torsion.
\newblock {\em Invent. Math.}, 21:287--301, 1973.

\bibitem[EK09]{EK}
B.~Elias and M.~Khovanov.
\newblock Diagrammatics for {S}oergel categories.
\newblock 2009.
\newblock arXiv:0902.4700.

\bibitem[FKS06]{FKS}
I.~B. Frenkel, M.~Khovanov, and C.~Stroppel.
\newblock A categorification of finite-dimensional irreducible representations
  of quantum sl(2) and their tensor products.
\newblock {\em Selecta Math. (N.S.)}, 12(3-4):379--431, 2006.
\newblock math/0511467.

\bibitem[FSS10]{FSS}
I.~Frenkel, C.~Stroppel, and J.~Sussan.
\newblock Categorifying fractional {E}uler characteristics, {J}ones-{W}enzl
  projector and $3j$-symbols with applications to {E}xts of {H}arish-{C}handra
  bimodules.
\newblock 2010.
\newblock arXiv:1007.4680.

\bibitem[GL92]{GL1}
I.~Grojnowski and G.~Lusztig.
\newblock On bases of irreducible representations of quantum {${\rm GL}\sb n$}.
\newblock 139:167--174, 1992.

\bibitem[Hil82]{Hiller}
H.~Hiller.
\newblock {\em Geometry of {C}oxeter groups}, volume~54 of {\em Research Notes
  in Mathematics}.
\newblock 1982.

\bibitem[Jac04]{Jac1}
M.~Jacobsson.
\newblock An invariant of link cobordisms from khovanov homology.
\newblock {\em Alg.\ Geom.\ Top.}, 4:1211--1251, 2004.
\newblock math.GT/0206303.

\bibitem[JS91]{js2}
A.~Joyal and R.~Street.
\newblock The geometry of tensor calculus. {I}.
\newblock {\em Adv. Math.}, 88(1):55--112, 1991.

\bibitem[Kam09]{Kam}
J.~Kamnitzer.
\newblock Lectures on geometric constructions of the irreducible
  representations of ${GL}_n$.
\newblock 2009.
\newblock arXiv:0912.0569.

\bibitem[Kho00]{Kh1}
M.~Khovanov.
\newblock A categorification of the {J}ones polynomial.
\newblock {\em Duke Math. J.}, 101(3):359--426, 2000.
\newblock math.QA/9908171.

\bibitem[Kho02]{Kh2}
M.~Khovanov.
\newblock A functor-valued invariant of tangles.
\newblock {\em Algebr. Geom. Topol.}, 2:665--741 (electronic), 2002.
\newblock math.QA/0103190.

\bibitem[Kho10a]{Kh4}
M.~Khovanov.
\newblock Categorifications from planar diagrammatics.
\newblock {\em Jpn. J. Math.}, 5(2):153--181, 2010.
\newblock arXiv:1009.3295.

\bibitem[Kho10b]{Kh5}
M.~Khovanov.
\newblock Heisenberg algebra and a graphical calculus.
\newblock 2010.
\newblock arXiv:1008.5084.

\bibitem[Kho11]{Kh3}
M.~Khovanov.
\newblock How to categorify one-half of quantum $gl(1|2)$.
\newblock 2011.
\newblock arXiv:1007.3517.

\bibitem[KK86]{KK}
B.~Kostant and S.~Kumar.
\newblock The nil {H}ecke ring and cohomology of {$G/P$} for a {K}ac-{M}oody
  group {$G$}.
\newblock {\em Adv. in Math.}, 62(3):187--237, 1986.

\bibitem[KL09]{KL}
M.~Khovanov and A.~Lauda.
\newblock A diagrammatic approach to categorification of quantum groups {I}.
\newblock {\em Represent. Theory}, 13:309--347, 2009.
\newblock math.QA/0803.4121.

\bibitem[KL10]{KL3}
M.~Khovanov and A.~Lauda.
\newblock A diagrammatic approach to categorification of quantum groups {III}.
\newblock {\em Quantum Topology}, 1:1--92, 2010.
\newblock math.QA/0807.3250.

\bibitem[KL11]{KL2}
M.~Khovanov and A.~Lauda.
\newblock A diagrammatic approach to categorification of quantum groups {II}.
\newblock {\em Trans. Amer. Math. Soc.}, 363:2685--2700, 2011.
\newblock math.QA/0804.2080.

\bibitem[KLMS10]{KLMS}
M.~Khovanov, A.~Lauda, M.~Mackaay, and M.~Sto\v{s}i\'c.
\newblock Extended graphical calculus for categorified quantum sl(2).
\newblock 2010.
\newblock arXiv:1006.2866.

\bibitem[KM11]{SK}
S.J. Kang and M.Kashiwara.
\newblock Categorification of highest weight modules via
  {K}hovanov-{L}auda-{R}ouquier algebras.
\newblock 2011.
\newblock arXiv:1102.4677.

\bibitem[KMS09]{KMS}
M.~Khovanov, V.~Mazorchuk, and C.~Stroppel.
\newblock A brief review of abelian categorifications.
\newblock {\em Theory Appl. Categ.}, 22:No. 19, 479--508, 2009.
\newblock math.RT/0702746.

\bibitem[KR08]{KRam2}
A.~Kleshchev and A.~Ram.
\newblock Representations of {K}hovanov-{L}auda {A}lgebras and combinatorics of
  {L}yndon words.
\newblock 2008.
\newblock arXiv:0909.1984.

\bibitem[KS10]{KR}
M.~Khovanov and R.~Sazdanovic.
\newblock Categorification of the polynomial ring.
\newblock 2010.
\newblock arXiv:1101.0293.

\bibitem[Kum02]{Kum}
S.~Kumar.
\newblock {\em Kac-{M}oody groups, their flag varieties and representation
  theory}, volume 204 of {\em Progress in Mathematics}.
\newblock Birkh\"auser Boston Inc., Boston, MA, 2002.

\bibitem[Lau06]{Lau0}
A.~D. Lauda.
\newblock Frobenius algebras and ambidextrous adjunctions.
\newblock {\em Theory Appl. Categ.}, 16:84--122, 2006.
\newblock math.CT/0502550.

\bibitem[Lau08]{Lau1}
A.~D. Lauda.
\newblock A categorification of quantum sl(2).
\newblock {\em Adv. Math.}, 225:3327--3424, 2008.
\newblock math.QA/0803.3652.

\bibitem[Lau11]{Lau2}
A.D. Lauda.
\newblock Categorified quantum sl(2) and equivariant cohomology of iterated
  flag varieties.
\newblock {\em Algebras and Representation Theory}, 14:253--282, 2011.
\newblock math.QA/0803.3848.

\bibitem[LOT08]{LOT2}
R.~Lipshitz, P.~Ozsv{\'a}th, and D.~Thurston.
\newblock Bordered {H}eegaard {F}loer homology: {I}nvariance and pairing.
\newblock 2008.
\newblock arXiv:0810.0687.

\bibitem[LOT09]{LOT}
R.~Lipshitz, P.~Ozsv{\'a}th, and D.~Thurston.
\newblock Slicing planar grid diagrams: a gentle introduction to bordered
  {H}eegaard {F}loer homology.
\newblock In {\em Proceedings of {G}\"okova {G}eometry-{T}opology {C}onference
  2008}, pages 91--119. G\"okova Geometry/Topology Conference (GGT), G\"okova,
  2009.
\newblock arXiv:0810.0695.

\bibitem[LS11]{LS}
A.~Licata and A.~Savage.
\newblock Hecke algebras, finite general linear groups, and {H}eisenberg
  categorification.
\newblock 2011.
\newblock arXiv:1101.0420.

\bibitem[LV09]{LV}
A.~D. Lauda and M.~Vazirani.
\newblock Crystals for categorified quantum groups.
\newblock 2009.
\newblock math.RT/0909.1810.

\bibitem[Mac95]{Mac}
I.~Macdonald.
\newblock {\em Symmetric functions and {H}all polynomials}.
\newblock Oxford Mathematical Monographs. The Clarendon Press Oxford University
  Press, New York, second edition, 1995.
\newblock With contributions by A. Zelevinsky, Oxford Science Publications.

\bibitem[Man01]{Man}
L.~Manivel.
\newblock {\em Symmetric functions, {S}chubert polynomials and degeneracy
  loci}, volume~6 of {\em SMF/AMS Texts and Monographs}.
\newblock AMS, Providence, RI, 2001.

\bibitem[Maz10]{Maz}
V.~Mazorchuk.
\newblock Lectures on algebraic categorification.
\newblock 2010.
\newblock arXiv:1011.0144.

\bibitem[MSV10]{MSV}
M.~Mackaay, M.~Stosic, and P.~Vaz.
\newblock A diagrammatic categorification of the q-{S}chur algebra.
\newblock 2010.
\newblock arXiv:1008.1348.

\bibitem[M{\"u}g03]{Muger}
M.~M{\"u}ger.
\newblock From subfactors to categories and topology. {I}. {F}robenius algebras
  in and {M}orita equivalence of tensor categories.
\newblock {\em J. Pure Appl. Algebra}, 180(1-2):81--157, 2003.
\newblock math.CT/0111204.

\bibitem[Ras10]{Ras}
J.~Rasmussen.
\newblock Khovanov homology and the slice genus.
\newblock {\em Invent. Math.}, 182(2):419--447, 2010.
\newblock arXiv:math/0402131.

\bibitem[Rou08]{Rou2}
R.~Rouquier.
\newblock 2-{K}ac-{M}oody algebras.
\newblock 2008.
\newblock arXiv:0812.5023.

\bibitem[RT90]{RT}
N.~Reshetikhin and V.~Turaev.
\newblock Ribbon graphs and their invariants derived from quantum groups.
\newblock {\em Comm. Math. Phys.}, 127(1):1--26, 1990.

\bibitem[Str95]{Street}
R.~Street.
\newblock Low-dimensional topology and higher-order categories.
\newblock In {\em Proceedings of CT95, Halifax, July 9-15}, 1995.

\bibitem[Str96]{Street2}
R.~Street.
\newblock {\em Categorical structures}.
\newblock North-Holland, Amsterdam, 1996.

\bibitem[Str05]{Strop}
C.~Stroppel.
\newblock Categorification of the {T}emperley-{L}ieb category, tangles, and
  cobordisms via projective functors.
\newblock {\em Duke Math. J.}, 126(3):547--596, 2005.

\bibitem[Sus07]{Sussan}
J.~Sussan.
\newblock {\em Category {O} and sl(k) link invariants}.
\newblock ProQuest LLC, Ann Arbor, MI, 2007.
\newblock Thesis (Ph.D.)--Yale University.

\bibitem[SVV09]{SVV}
P.~Shan, M.~Varagnolo, and E.~Vasserot.
\newblock Canonical bases and affine {H}ecke algebras of type {D}.
\newblock 2009.
\newblock arXiv:0912.4245.

\bibitem[VV09]{VV2}
M.~Varagnolo and E.~Vasserot.
\newblock Canonical bases and affine {H}ecke algebras of type {B}.
\newblock 2009.
\newblock arXiv:0911.5209.

\bibitem[Web10a]{Web}
B.~Webster.
\newblock Knot invariants and higher representation theory {I}: diagrammatic
  and geometric categorification of tensor products.
\newblock 2010.
\newblock arXiv:1001.2020.

\bibitem[Web10b]{Web2}
B.~Webster.
\newblock Knot invariants and higher representation theory {II}: the
  categorification of quantum knot invariants.
\newblock 2010.
\newblock arXiv:1005.4559.

\bibitem[Zhe08]{zheng}
H.~Zheng.
\newblock A geometric categorification of representations of {$U_q({\rm
  sl}_2)$}.
\newblock 12:348--356, 2008.
\newblock arXiv:0705.2630.

\end{thebibliography}

%

\vspace{0.1in}

\noindent A.L.:  { \sl \small Department of Mathematics, Columbia University, New
York, NY 10027} \newline \noindent
  {\tt \small email: lauda@math.columbia.edu}

%
\end{document}